\documentclass[11pt]{amsart}

\usepackage[numbered]{bookmark}
\usepackage{changepage} 
\usepackage[T1]{fontenc}
\usepackage[utf8]{inputenc}
\usepackage{graphics, graphicx}
\usepackage{amsmath, amsfonts, amsthm, amssymb, mathrsfs, mathtools, mathabx}
\usepackage{verbatim}
\usepackage{mathrsfs}
\usepackage{enumerate}
\usepackage{xcolor}
\usepackage{cite}
\usepackage{wrapfig}
\usepackage{subcaption}
\usepackage[font=small]{caption}
\usepackage{textcase}

\newtheorem{theorem}{Theorem}[section]
\newtheorem{definition}[theorem]{Definition}
\newtheorem{lemma}[theorem]{Lemma}
\newtheorem{proposition}[theorem]{Proposition}

\newtheorem{remark}{Remark}[section]
\newtheorem*{example}{Example}

\newtheorem{theoremx}{Theorem}

\newtheorem{corollaryx}{Corollary}

\newtheorem{theoremxx}{Theorem}

\numberwithin{equation}{section}

\newcommand{\N}{\mathbb{N}}

\newcommand{\R}{\mathbb{R}}
\newcommand{\C}{\mathbb{C}} 
\newcommand{\T}{\mathbb{T}}
\newcommand{\Z}{\mathbb{Z}}

\newcommand{\Lo}{\ensuremath{\mathcal{L}_\omega}}

\newcommand{\Mell}{\ensuremath{M_\Omega^{\mathrm{ell}}}}
\newcommand{\Mhyp}{\ensuremath{M_\Omega^{\mathrm{hyp}}}}

\newcommand{\Spol}{\ensuremath{\mathcal{S}^{\mathrm{pol},T}}}
\newcommand{\Sexp}{\ensuremath{\mathcal{S}^{\mathrm{exp},T}}}
\newcommand{\dec}{\ensuremath{l}}

\newcommand{\Lell}{\mathcal{L}}

\newcommand{\Function}[5]{\begin{array}{cccc} #1 : & #2 & \rightarrow & #3 \\ & #4 & \mapsto & #5 \end{array}}

\date{}
\author{Donato Scarcella}
\address{Departament de Matemàtiques i Informàtica, Universitat de Barcelona, Gran Via de les Corts Catalanes 585, 08007 Barcelona, Spain}
\email{donato.scarcella@ub.edu}

\author{Frank Trujillo}
\address{Centre de Mathématiques Laurent Schwartz (CMLS), CNRS, École polytechnique, Institut Polytechnique de Paris, Palaiseau, France.}
\email{frank.trujillo-amezquita@cnrs.fr}

\begin{document}
\title[Non-autonomous KAM theory I: Normally elliptic and hyperbolic cases]{Non-autonomous KAM theory for lower dimensional invariant tori (I):
\NoCaseChange{Normally elliptic and hyperbolic cases}}
\maketitle

\begin{abstract}
     Many physical phenomena are naturally modeled by dynamical systems subject to non-autonomous perturbations that decay in time, including the interaction of a laser pulse with a molecule, epidemiological models, nonlinear oscillatory systems, and celestial mechanics. In the present paper, we consider time-dependent perturbations decaying in time of Hamiltonian systems having a lower-dimensional isotropic normally elliptic (resp. hyperbolic) invariant torus with quasiperiodic solutions. Under suitable rates of decay in time of the perturbation, we prove the existence of invariant manifolds in the extended phase space, and determine the asymptotic behavior of the associated transverse dynamics. The above results are obtained for Hölder, smooth, and analytic Hamiltonian systems.
\end{abstract}

\section{Introduction}
Dynamical systems subject to non-autonomous perturbations decaying in time are relevant in describing many physical models, e.g., when considering the effect of a laser pulse
on a molecule~\cite{KSBAJCB07, BdlL11}, in epidemiological studies~\cite{TC23}, in nonlinear oscillatory systems~\cite{S24,S25} and in celestial mechanics~\cite{Sca25}.  For a recent survey on non-autonomous dynamical systems, we refer to~\cite{W25}.

In this work, we consider time-dependent perturbations of Hamiltonians having a lower-dimensional isotropic normally elliptic (resp. hyperbolic) invariant torus with quasiperiodic solutions. Under suitable time decay properties on the perturbations, we prove the existence of an \textit{asymptotic KAM torus} (see Definition \ref{def:Csigma_KAM_torus} below).  Heuristically, it consists of an invariant manifold in the extended phase space where the dynamics converge as time tends to infinity to the quasiperiodic solutions associated with the unperturbed system. Besides proving the existence of asymptotic KAM tori, the main novelty of the present work consists in investigating the solutions of the linearized Hamiltonian vector field along the trajectories contained in the asymptotic KAM torus. We prove that the corresponding cocycle is asymptotically conjugate to an explicit model cocycle displaying elliptic (resp. hyperbolic) behavior, showing that the asymptotic transverse dynamics is of the same type as that of the unperturbed invariant torus. We call these objects \textit{asymptotically elliptic (resp. hyperbolic) KAM tori}, see Definition \ref{def:ell_hyp_par_trasv_dyn_asym_KAM} below. This is the first part of two papers. In the second one~\cite{scarcella_KAMST2}, we consider time-dependent perturbations of Hamiltonians having a lower-dimensional isotropic normally parabolic invariant torus supporting quasiperiodic solutions.

The aim of these works is to develop a non-autonomous KAM theory for lower-dimensional invariant tori. Since the pioneering works of Kolmogorov~\cite{Kol54}, Arnold~\cite{Arn63a} and Moser \cite{M62}, the classical KAM theory shows the persistence of quasiperiodic solutions in nearly integrable Hamiltonian systems.
Over the past seventy years, KAM theory has undergone numerous developments, generalizations, and refinements. Several excellent surveys of the subject are available, including those by Bost~\cite{B86}, Pöschel~\cite{Posc01}, Chierchia~\cite{Chi03}, de la Llave~\cite{dlL01}, Sevryuk~\cite{S03} and Féjoz~\cite{Fej17}. We refer to Dumas's book for a fascinating historical report~\cite{D14}. This theory is interesting especially for its application to stability problems in celestial mechanics~\cite{Arn63b, Fe04, CP11}.

The persistence of lower-dimensional invariant tori has also been extensively investigated. In the normally elliptic case, the result was first announced by Melnikov~\cite{Mel65, Mel68} and later proved by Eliasson~\cite{El88}. For the persistence of normally hyperbolic invariant tori, we refer to the works of Graff~\cite{G74} and Zehnder~\cite{Zeh76}; see also Massetti~\cite{M18,M19} for the non-conservative setting. We also refer to the survey of Rüssmann~\cite{Ru01} for a comprehensive treatment of KAM theory, including both Kolmogorov's theorem on Lagrangian invariant tori and the persistence of lower-dimensional invariant tori.

The main difficulty of the KAM theory lies in the small denominators arising in the perturbation series, which make the convergence argument particularly delicate. The fundamental insight of Kolmogorov was to identify the appropriate arithmetic condition, together with a non-degeneracy assumption, in order to prove the convergence of these series using an iterative scheme based on a Newton algorithm characterized by superlinear convergence. We refer to~\cite{BGGS84} for a complete proof of the result of Kolmogorov. Another interesting approach consists in replacing the above-mentioned iterative method by a Nash-Moser Implicit Function Theorem approach on a suitable scale of Banach spaces. We refer to Zehnder~\cite{Zeh75,Zeh76}, Herman~\cite{B86}, Berti-Bolle~\cite{BB15} and Féjoz~\cite{Fe04}. 

The situation changes substantially when the perturbation depends aperiodically on time and decays as time tends to infinity. In this setting, one cannot generally expect the persistence of invariant tori. Instead, one can prove the existence of the asymptotic KAM tori; see Definition \ref{def:Csigma_KAM_torus}. 
In~\cite{Sca22}, the existence of asymptotic KAM tori is established for time-dependent perturbations that decay polynomially fast in time of Hamiltonians having a Lagrangian invariant torus with quasiperiodic solutions, thereby generalizing the previous works of Canadell-de la Llave~\cite{CdlL15} and Fortunati-Wiggins~\cite{FW14}. 

In the work~\cite{Sca25}, asymptotic dynamics of the planar three-body problem perturbed by a comet in a neighborhood of a Lagrangian invariant torus (whose existence is guaranteed by Arnold~\cite{Arn63b}; see also Féjoz~\cite{Fe02}) is studied. The interaction between the planets and the comet is modeled as a time-dependent perturbation.

In~\cite{Sca22b}, the existence of orbits asymptotic in forward and backward time to (possibly different) quasiperiodic solutions is proved for integrable or near-integrable Hamiltonians subject to time-dependent perturbations decaying polynomially fast in time.

Finally,~\cite{Sca22c} complements~\cite{Sca22} by considering arbitrary, rather than quasiperiodic, dynamics on the Lagrangian invariant torus associated with the unperturbed Hamiltonian, under exponentially decay in time of the perturbations.

The main difference between the non-autonomous and classical KAM theory is the absence of non-degeneracy conditions on the unperturbed Hamiltonian and of arithmetic assumptions on the frequencies of the unperturbed quasiperiodic motion. The latter stems from the absence of small denominators in the homological equations arising in time-dependent problems.

As mentioned before, in the present paper, we consider non-autonomous Hamiltonian vector fields converging as time tends to infinity to Hamiltonian vector fields having a lower-dimensional isotropic normally elliptic (resp. hyperbolic) invariant torus supporting quasiperiodic solutions. We prove the existence of asymptotic KAM tori and asymptotically elliptic (resp. hyperbolic) KAM tori for Hölder (Theorems \ref{Thm:elliptic_Csigma} and \ref{Thm:hyp_Csigma}), $C^\infty$ (Theorem \ref{thm:KAM_smooth}) and analytic (Theorems \ref{Thm:elliptic_analy} and \ref{Thm:hyp_analy}) Hamiltonians. %

In the elliptic case, we prove the existence of asymptotic KAM tori under polynomial decay in time of the perturbation (see Item \ref{thm:A_existence} of Theorem \ref{Thm:elliptic_Csigma}). Concerning the proof, the main challenge consists in reformulating the invariant equation in a functional setting suitable for applying an appropriate quantitative version of the Implicit Function Theorem. This requires a refined analysis of the associated Banach spaces chosen in order to solve the associated linearized problem through some cohomological equations. The presence of normal directions introduces additional nonlinear contributions to the invariance equation that have no counterpart in the Lagrangian case considered in~\cite{Sca22}, making the invariant equation and the associated linearized problem substantially more involved. 
We stress that no smallness on the perturbative terms is assumed. Instead, we prove the existence of asymptotic KAM tori for time sufficiently large. Then, if the flow is well-defined for all $t \in \R$ one could extend the set of definition using the flow (see Definition \ref{def:Csigma_KAM_torus} and~\eqref{eq:asymptotic_KAM_flow} below).

The main novelty of the present work consists in verifying that the above-mentioned asymptotic KAM torus is asymptotically elliptic (see Item \ref{thm:A_transverse} of Theorem \ref{Thm:elliptic_Csigma}). More specifically, under slightly stronger polynomial time decay assumptions, we prove that the transverse solutions of the linearized Hamiltonian vector field along the orbits contained in the asymptotic KAM torus have an asymptotic elliptic character. For this purpose, we prove the existence of a $C^1$ non-autonomous matrix that asymptotically conjugates the cocycle associated with the linearized Hamiltonian vector field on a trajectory contained in the asymptotic KAM torus with the cocycle of a suitable linear model displaying transversal elliptic dynamics. To achieve this, we introduce a nonlinear functional equation whose unknown is a time-dependent matrix defining the conjugacy between the two cocycles. The main difficulty is to identify the correct formulation of this equation so that its linearization can be analyzed in a suitable scale of weighted Banach spaces. This ultimately reduces to solving a system of cohomological equations with precise decay estimates, allowing the application of a quantitative Implicit Function Theorem to obtain the desired asymptotic conjugacy.

In addition, under slightly more restrictive polynomial decay assumptions, we prove that the asymptotic transverse dynamics is completely described by that of the limiting autonomous Hamiltonian system (see Corollary~\ref{cor:ell}). More precisely, assuming that the Hamiltonian converges as time tends to infinity to an autonomous Hamiltonian possessing a normally elliptic isotropic invariant torus with quasiperiodic solutions, we establish a one-to-one asymptotic correspondence between the transverse solutions of the perturbed and limiting linearized Hamiltonian vector fields.

Similar results are established in the hyperbolic case (see Theorem \ref{Thm:hyp_Csigma} and Corollary \ref{cor:hyp}). In this setting, we need exponential decay in time of the perturbation instead of polynomial decay. In particular, in Corollary \ref{cor:hyp}, which is the counterpart of Corollary \ref{cor:ell} in the hyperbolic setting, the correspondence is not one-to-one due to the unstable directions that destroy uniqueness. 

We believe that the analysis and the new methodologies developed in the present paper may be useful beyond the scope of the present work. In particular, they could be applied to the study of lower-dimensional invariant tori in dynamical systems where no small divisors arise. Examples of such systems include the fast-driven systems studied by Chen-Pinzari~\cite{CP21} and the normal form setting considered by Pinzari~\cite{P20}.

The paper is organized as follows. Section \ref{sc:results} contains the statements of the main results of this work. Section \ref{sc:analysis_transversal} studies the asymptotic properties of a suitable linear model displaying elliptic (resp. hyperbolic) transverse dynamics. Section \ref{sc:criteria_asym_dyn} provides sufficient conditions determining the asymptotic character of the transverse dynamics associated with an asymptotic KAM torus. Section \ref{sec:FS_Ref_Ham} establishes estimates for the higher-order terms in the Taylor expansion of the Hamiltonian systems under consideration. The proofs of Theorems \ref{Thm:elliptic_Csigma}, \ref{Thm:elliptic_analy} and Corollary \ref{cor:ell} are given in Section \ref{sec:Proof_Theorem_Ell} whereas those of Theorems \ref{Thm:hyp_Csigma}, \ref{Thm:hyp_analy} and Corollary \ref{cor:hyp} are contained in Section \ref{sec:Proof_Theorem_Hyp}.

\section{Main results}
\label{sc:results}
In this work, we are interested in studying the asymptotic dynamics of time-dependent Hamiltonian vector fields that converge, as time goes to infinity, to a Hamiltonian vector field having a lower-dimensional invariant torus supporting quasiperiodic solutions. We will prove analogous results under different regularity assumptions, namely, for systems that are Hölder continuous, $C^\infty$ or analytic with respect to the phase space variables and that are continuous with respect to the time parameter. For this reason, this section is divided into three parts. We refer to Section \ref{sec:Holder_Set} for the Hölder case, to Section \ref{sec:Smooth_Set} for the $C^\infty$ case, and to Section \ref{sec:Anal_Set} for the analytic case.

Let us introduce some notation that we will use throughout this work without explicit reference. Let $\mathsf{r}, \mathsf{s} \in \N$, we denote by $\mathcal{M}_{\mathsf{r} \times \mathsf{s}}(\R)$ the space of $\mathsf{r} \times \mathsf{s}$ matrices with real coefficients and write $\mathcal{M}_\mathsf{r}(\mathbb{R})$ when $\mathsf{r} = \mathsf{s}$. Given  $\mathsf{r} \in \N$, we define the standard symplectic matrix as
\begin{equation}\label{def:J}
     J_\mathsf{r} = \begin{pmatrix} 0 & \mathrm{Id}_\mathsf{r} \\ -\mathrm{Id}_\mathsf{r} & 0\end{pmatrix},
\end{equation}
where $\mathrm{Id}_\mathsf{r} \in \mathcal{M}_m(\R)$ denotes the $\mathsf{r} \times \mathsf{r}$ identity matrix. 

We will work on the (extended) phase space $\T^n \times \R^n \times \R^{2m} \times \R$, where $n \geq 1, \, m \geq 0$, and whose coordinates we denote by $(q,p,z,t)$. We sometimes denote $z = (x, y) \in \R^m \times \R^m$. Given $T \geq 0$, we introduce the interval $I_T=[T, +\infty) \subset \R$. We endow $\T^n \times \R^n \times \R^{2m}$ with the canonical symplectic form $\sum_{i = 1}^n dq_i\wedge dp_i + \sum_{j = 1}^m dx_i \wedge dy_i$ and denote 
\begin{equation}
    \label{eq:big_symplectic_matrix}
    J = \mathrm{diag}(J_n, J_m) = \begin{pmatrix}
        J_n & 0_{2m \times 2m} \\
        0_{2m \times 2n} & J_m
    \end{pmatrix}
\end{equation}
where refer to~\eqref{def:J} for the definition of $J_n$ and $J_m$.
Notice that for any Hamiltonian $H$ of class $C^1$ on $(q, p, z)$, defined on an open subset of $\T^n \times \R^n \times \R^{2m} \times \R$, the associated Hamiltonian vector field (with respect to the canonical symplectic form), which throughout this work we denote by $X_H$, is given by
\begin{equation}
    X_H = J \nabla H,
\end{equation}
where $J$ is given by \eqref{eq:big_symplectic_matrix}.

Let $B \subset \R^{n+2m}$ be an open ball centered at the origin.  Given $f:\T^n \times B \times I_T \to \R$, we define
\begin{equation}\label{def:f0}
    f_0:\T^n \times I_T \to \R, \quad f_0(q,t) = f(q,0,0,t)
\end{equation}
whereas, for a fixed $t \in I_T$, let
\begin{equation}\label{def:ft}
        f^t:\T^n \times B \to \R, \quad f^t(q,p, z) = f(q,p,z,t).
\end{equation}
The notation extends naturally to real-valued and matrix-valued functions. 
Throughout this article, we will often consider compositions of the form $f(g(\cdot, t), t)$ where $f(\cdot, t)$ and $g(\cdot, t)$ are two (possibly) time-dependent maps. Notice that, using the notation above, this map is formally defined as $(\cdot, t) \mapsto f^t \circ g^t (\cdot)$. However, to avoid cumbersome notation and since we will not consider any kind of time reparametrizations, given $f(\cdot, t)$ and $g(\cdot, t)$ we will denote simply by $f \circ g$ the map $(\cdot, t) \mapsto f(g(\cdot, t), t)$, whenever it is well-defined. In particular, when $g$ is time-independent, the map $f \circ g$ is given by $(\cdot, t) \mapsto f(g(\cdot), t).$

Let $\mathsf{r}, \mathsf{s}\in\N$ and let $M :\R^\mathsf{r}\times \R^\mathsf{s} \to \R$ be a bilinear form. In this work, we denote the action of $M$ on $(v_1, v_2)\in \R^\mathsf{r}\times \R^\mathsf{s}$ by $M \cdot (v_1, v_2)$. When, $\mathsf{r} = \mathsf{s}$ and $v_1 = v_2 = v$, we simply write $M \cdot v^2$. This notation extends naturally to $k$-linear forms.

In addition, we denote by $\pi_q$, $\pi_p$ and $\pi_z$ the projection of a vector in $\R^{2n+2m}$ onto the $q$, $p$ and $z$ components, respectively. More specifically, for every $v = (v_1,\dots, v_{2n+2m})^\top \in \R^{2n+2m}$ we have that 
\begin{equation*}
   \pi_q v =  (v_1,\dots, v_n)^\top, \quad \pi_p v =  (v_{n+1},\dots, v_{2n})^\top, \quad \pi_z v =  (v_{2n+1},\dots, v_{2n+2m})^\top.
\end{equation*}
 
\subsection{Hölder setting}\label{sec:Holder_Set}

  We denote by $C^\sigma$ the class of Hölder functions and by $|\cdot|_{C^\sigma}$ the associated norm (a brief introduction can be found in Appendix \ref{app:Holder}).

Let $\sigma \ge 1$, a positive integer $k \ge 0$ and a vector $\omega \in \R^n$, we consider time-dependent vector fields $X^t, X_0^t \in C^{\sigma+k}\left(\T^n \times B\right)$, for all fixed $t \in I_T$, continuous with respect to $t$, and a $C^\sigma$ embedding $\varphi_0:\T^n \to \T^n \times B$ such that   
\begin{equation}
\begin{aligned}\label{def:hyp_asym_KAM_tori}
    &\lim_{t \to +\infty} \left|X^t - X^t_0\right|_{C^{\sigma+k}} = 0,\\
  &X_0 \circ \varphi_0 = \partial_q\varphi_0 \cdot \omega. 
\end{aligned}
\end{equation}
In words, we are considering a time-dependent vector field $X$ converging as time tends to infinity to the vector field $X_0$. Furthermore, $X_0$ has an invariant torus $\varphi_0$ supporting quasiperiodic solutions of frequency vector $\omega$. We point out that the most natural setting is when $X_0$ is autonomous, but our results hold in this case as well. 

\begin{definition}\label{def:Csigma_KAM_torus}
    We assume that $(X, X_0, \varphi_0)$, defined on $\T^n \times \R^{n + 2m} \times I_T$, satisfy~\eqref{def:hyp_asym_KAM_tori}. A family of embeddings $\varphi:\T^n \times I_{T'} \to \T^n \times B$, with $T' \geq T$, is a $C^\sigma$ \emph{asymptotic KAM torus} associated with $(X, X_0, \varphi_0)$, for some $\sigma \geq 1$, if
    \begin{align}\label{def:cond1_asymKAMtorus}
            &\lim_{t \to +\infty}\left|\varphi^t -\varphi_0\right|_{C^\sigma}=0,\\ \label{def:cond2_asymKAMtorus}
             &X \circ \varphi = \partial_q\varphi \cdot \omega + \partial_t \varphi.
    \end{align}

We say that $\varphi$ is a $C^\infty$ asymptotic KAM torus associated with $(X, X_0, \varphi_0)$ if it is a $C^\sigma$ asymptotic KAM torus associated with $(X, X_0, \varphi_0)$, for all $\sigma \geq 1$.
\end{definition}

Let us point out that \eqref{def:cond2_asymKAMtorus} is equivalent to
    \begin{equation}
    \label{eq:asymptotic_KAM_flow}
        \psi_{t_0, X}^t \circ \varphi^{t_0} (q) = \varphi^t(q + \omega(t-t_0)), \qquad \text{ for any } t_0, t \in I_{T'} \text{ and any } q\in \T^n,
    \end{equation}
where $\psi_{t_0, X}^t$ denotes the flow at time $t$ with initial time $t_0$ associated with $X$.

In particular, the existence of a $C^\sigma$ asymptotic KAM torus provides the existence of orbits converging, as time tends to infinity, to the quasiperiodic solutions associated with the unperturbed system. More precisely, we have the following.

\begin{proposition}\label{prop:esint_solt_KAM_tori}
    Let $\varphi^t$ be a $C^\sigma$ asymptotic KAM torus associated with $(X^t, X_0^t, \varphi)$. Then, for all $x \in \bigcup_{t \ge T^{'}} \varphi^t(\T^n)$, there exists $(q,t_0) \in \T^n \times I_{T^{'}}$ such that 
\begin{equation*}
    \lim_{t \to +\infty}|\psi^t_{t_0, X}(x) - \varphi_0(q + \omega(t-t_0))|=0.
\end{equation*}
\end{proposition}
\begin{proof}
    If $x \in \bigcup_{t \ge T^{'}} \varphi^t(\T^n)$, then there exists $(q,t_0) \in \T^n \times I_{T^{'}}$ such that 
    $x = \varphi^{t_0}(q)$. By \eqref{eq:asymptotic_KAM_flow}, we have
    \begin{equation*}
        \psi^t_{t_0, X}(x) = \psi^t_{t_0, X}\circ \varphi^{t_0}(q) = \varphi^t(q + \omega(t-t_0)).
    \end{equation*}
    Hence, the result follows by \eqref{def:cond1_asymKAMtorus} and the equation above. 
\end{proof}

Roughly speaking, a $C^\sigma$ asymptotic KAM torus is a family of $C^\sigma$ embeddings $\{\varphi^t\}_{t \in [T',+\infty)}$,  describing the evolution by the flow associated to $X$ of the torus $\varphi^{T'}(\T^n) \times \{T'\}$, that converge to the $\psi_{X_0}$-invariant torus $\varphi_0$ as time tends to infinity (see Proposition \ref{prop:esint_solt_KAM_tori}).  Moreover, the dynamics on this family of embeddings converge in time to the quasiperiodic solutions associated with $X_0$ on the invariant torus $\varphi_0$.

In this work, in addition to proving the existence of asymptotic KAM tori, we would like to analyze the asymptotic behaviour of the solutions of the linearized vector field $DX$ along the orbits contained in the asymptotic KAM torus $\varphi^t$, that is, the asymptotic behaviour of the solutions of the time-dependent linear system $\partial_t \eta (t) = DX \circ \varphi^t(q + \omega t) \eta(t)$, with $q \in \T^n$. To describe this, we introduce the following matrices.
Given $\Omega_1,\dots,\Omega_m \in \R_{>0}$, we define $\Omega = \mathrm{diag}(\Omega_1,\dots, \Omega_m)$, and we introduce the following symmetric matrices
\begin{equation}\label{Ms}
    \Mell = \begin{pmatrix} \Omega & 0 \\ 0 & \Omega \end{pmatrix}, \quad \Mhyp = \begin{pmatrix} \Omega & 0 \\ 0 & -\Omega \end{pmatrix}.
\end{equation}

Recall that for autonomous Hamiltonians the archetypical models of lower dimensional invariant tori with a definite behaviour along the transverse variables, namely, \emph{elliptic}, or \emph{hyperbolic} behaviour, are given by the invariant torus $\T^n \times \{0\}$ associated with a Hamiltonian $H: \T^n \times B \subseteq \T^n \times \R^n \times \R^{2m} \to \R$ of the form
\begin{equation} \label{def:H_aut}
H(q, p, z) =  \omega \cdot p + \frac{1}{2} M\cdot z^2  + \mathcal{O}(p^2,pz,z^3), 
\end{equation}
with $M = \Mell$, or $M = \Mhyp$, respectively. We recall that $\frac{1}{2} M\cdot z^2$ denotes for the vector $z$ given twice as an argument of the symmetric bilinear form $\frac{1}{2} M$.  In these scenarios, the linearized system along the solutions contained in the invariant torus is of the form
\begin{equation}\label{def:DXH_aut}
DX_H(q + \omega t , 0, 0) =  \begin{pmatrix} 0_{n \times n} & a_{12}(q + \omega t ) & a_{13}(q + \omega t )^\top\\
                                0_{n \times n} & 0_{n \times n} & 0_{n 
                            \times 2m} \\
                                0_{2m \times n} & a_{13}(q + \omega t ) & J_mM \end{pmatrix}
\end{equation}
we refer to~\eqref{def:J} for the defition of $J_m$. 
In this case, the associated cocycle $\Phi(t; q, t_0)$ given by the fundamental solutions of \eqref{def:DXH_aut} (that is, for any $(\eta_0, t_0) \in \R^{2n + 2m} \times \R$  the map $\eta(t) = \Phi(t; q, t_0) \eta_0$ is the unique solution to the system $\eta'(t) = DX_H(q + \omega t, 0, 0)\eta(t)$ with initial conditions $\eta(t_0) = \eta_0$) has a explicit form displaying elliptic or hyperbolic behaviour, in accordance with the matrix $M$. See Section \ref{sc:analysis_transversal}.

Furthermore, as we shall see also in Section \ref{sc:analysis_transversal}, given a symmetric matrix $M \in \mathcal{M}_{2m}(\R)$ and a family of matrices $A:\T^n \times [1, +\infty) \to \mathcal{M}_{2n+2m}(\R)$ of the form
     \begin{equation}
     \label{def:A}
        A^t(q) =\begin{pmatrix} 0_{n \times n} & a_{12}^t(q) & a_{13}^t(q) \\
                                0_{n \times n} & 0_{n \times n} & 0_{n \times 2m} \\
                                0_{2m \times n} & a_{32}^t(q) & J_mM \end{pmatrix},
    \end{equation}
   with $a_{12}^t:\T^n \to \mathcal{M}_n(\R)$, $a_{13}^t :\T^n \to \mathcal{M}_{n\times 2m}(\R)$, and  $a_{32}^t:\T^n \to \mathcal{M}_{2m\times n}(\R)$, 
   satisfying
    \begin{equation}
    \label{def:cond1_ellhyppar_asymKAMtorus}
        \sup_{t \in I_{T^{''}}}|A^t|_{C^0}<\infty,
    \end{equation}
    the cocycle $\Phi_A(t; q, t_0)$ given by the fundamental solutions of the system
    \begin{equation}
        \label{eq:Trasv_dyn_AA}
        \partial_t \xi(t) = A^t(q + \omega t) \xi(t),
    \end{equation}
    displays, asymptotically, elliptic (resp. hyperbolic) behaviour, in the spirit of the behaviour associated with lower-dimensional invariant tori of autonomous Hamiltonians, if $M = \Mell$ (resp. $M = \Mhyp$).

    \begin{definition}
    \label{def:ell_hyp_par_trasv_dyn_asym_KAM}
    Let $(X, X_0, \varphi_0)$ defined on $\T^n \times B \subseteq \T^n \times \R^{n + 2m}$ satisfying \eqref{def:hyp_asym_KAM_tori} and let $\varphi: \T^n \times [T', +\infty) \to \T^n \times \R^{n + 2m}$ be a $C^\sigma$ asymptotic KAM torus associated with $(X, X_0, \varphi_0)$.
    
    We say that $\varphi$ is asymptotically \emph{elliptic} (resp. \emph{hyperbolic}), if there exist a continuous map $A:\T^n \times [T'', +\infty) \to \mathcal{M}_{2n+2m}(\R)$ of the form \eqref{def:A} with $M = \Mell$ (resp. $M = \Mhyp$) and $T'' \geq T'$, satisfying \eqref{def:cond1_ellhyppar_asymKAMtorus}, such that the cocycles $\Phi(t; q, t_0)$ and $\Phi_A(t; q, t_0)$, given by the fundamental solutions of the systems
    \[
    \begin{array}{ll}
    \dot \eta(t) = DX^t \circ \varphi^t (q + \omega t) \eta(t), & \qquad (q, t) \in \T^n \times [T'', +\infty), \\
    \dot \xi(t) = A^t(q + \omega t) \xi(t), & \qquad (q, t) \in \T^n \times [T'', +\infty),
    \end{array}\]
are \emph{conjugated} by a $C^1$ map $S: \T^n \times [T'', +\infty) \to GL(2n + 2m, \R)$,  that is,
    \begin{equation}
\Phi(t; q, t_0) = S(q + \omega t, t)\Phi_A(t; q, t_0)S(q + \omega t_0, t_0)^{-1},
        \label{eq:conjugated_cocycles}
    \end{equation} 
    and satisfy
    \begin{equation}\label{def:trasv_dyn_limit} 
        \lim_{t \to +\infty}\big|\pi_z \Phi(t; q,t_0) - \pi_z \Phi_A(t; q, t_0)S(q + \omega t_0, t_0)^{-1}\big| = 0.
    \end{equation}
\end{definition}

In order to quantify the regularity of time-dependent smooth functions, we introduce the following Banach spaces. Given $\sigma \ge 0$, a non-negative integer $k$, and parameters $T\ge 1$ and $\ell >0$, we define
\begin{align}\label{def:S}
\Spol_{(\sigma,k),\ell}
    = \left\{
    f : \T^n \times B \times I_T \to \R \;\middle|\;
    \begin{aligned}
        &f^t \in C^{\sigma+k}(\T^n \times B), \text{ for } t \in I_T; \\
        & \sup_{t \in I_T} \bigl(|f^t|_{C^{\sigma+k}}\, t^\ell \bigr) < \infty; \\
        &\partial^i_{(q,p,z)} f \in C(\T^n \times B \times I_T), \text{ for } 0 \le |i| \le k
    \end{aligned}
    \right\}
\end{align}
and endow it with the norm
\begin{equation}\label{def:norm_S}
    |f|^{\mathrm{pol}, T}_{\sigma+k, \ell} = \sup_{t \in I_T}|f^t|_{C^{\sigma+k}}t^\ell,
\end{equation}
where $\partial^{i}_{(q,p,z)}=\partial^{i_1}_{q_1} \cdots\partial^{i_n}_{q_n}\partial_{p_1}^{i_{n+1}}\cdots \partial_{p_1}^{i_{2n}}\partial_{z_1}^{i_{2n+1}}\cdots \partial_{z_{2m}}^{i_{2n+2m}}$ and $|i|= |i|_1,$ for any $i \in \N^{2n + 2m}$ and, as a convention, we adopt $\partial^0_{(q,p,z)}f = f$.

In general, we will use the same notation if the function is defined on $\T^n \times I_T$, as well as for vector-valued or matrix-valued functions. More precisely, we say that a vector-valued or matrix-valued function belongs to $\Spol_{(\sigma, k), \ell}$ if it is the case for each of its components.  The norms of a real-valued function or a matrix are defined as the maximum of those norms of its components. It will be specified by the context. %
When the need arises to specify the codomain of the function, we will denote the space of functions in $\Spol_{(\sigma, k), \ell}$ taking values in $\R^{\mathsf{p}}$ and $\mathrm{Mat}_{{\mathsf{p} \times \mathsf{r}}}(\R)$ by $\Spol_{(\sigma, k), \ell}(\R^\mathsf{p})$ and $\Spol_{(\sigma, k), \ell}(\R^{\mathsf{p} \times \mathsf{r}})$, respectively.

 In the following $\varphi_0$ denotes the trivial embedding
\begin{equation}\label{def:varphi0=(q,0,0)}
    \varphi_0:\T^n \to \T^n \times B, \qquad \varphi_0(q) = (q,0,0)
\end{equation}

The following theorem is proved in Section \ref{sec:Proof_Theorem_Ell}. 

\begin{theoremx}\label{Thm:elliptic_Csigma}

Fix $\sigma \ge 1$, $\ell >1$ and $l \ge 0$. Let $H_0: \T^n  \times B \times [1, +\infty) \to \R$ of the form
\begin{equation}
\label{eq:initial_hamiltonian}
 H_0(q, p, z, t) =  \omega \cdot p + \frac{1}{2}\Mell  \cdot z^2 + O(p^2,pz,z^3), \qquad \partial_{(p,z)}^2 H_0 \in \mathcal{S}_{(\sigma,3),0}^{\mathrm{pol},1},
\end{equation}
and $P: \T^n  \times B \times [1, +\infty) \to \R$ of  the form
\begin{equation}
\label{eq:perturbation_form}
P(q, p, z, t) = a(q, t) + b(q, t) \cdot p + c(q, t) \cdot z + \frac{1}{2}d(q, t)\cdot z^2.
\end{equation}

Then, for $H := H_0 + P$
the following holds.
\begin{enumerate}
    \item \label{thm:A_existence} Suppose that
    \begin{equation*}
        a \in \mathcal{S}^{\mathrm{pol},1}_{(\sigma, 2),0}, \hspace{2mm} \partial_qa \in \mathcal{S}_{(\sigma, 1), \ell+ \dec + 2}^{\mathrm{pol},1}, \hspace{2mm}  b \in \mathcal{S}_{(\sigma, 2), \ell}^{\mathrm{pol},1}, \hspace{2mm}  c \in \mathcal{S}_{(\sigma, 2), \ell+ \dec + 1}^{\mathrm{pol}, 1}, \hspace{2mm} d \in \mathcal{S}_{(\sigma, 2), \ell}^{\mathrm{pol}, 1}.
    \end{equation*}
    Then, there exists a $C^\sigma$ asymptotic KAM torus $\varphi$ associated with $(X_H, X_{H_0}, \varphi_0)$ of the form
    \begin{equation}
        \label{eq:asymp_KAM_torus_form}
    \begin{gathered}
        \varphi:\T^n \times I_T \to \T^n \times \R^n \times \R^{2m}, \qquad \varphi^t = (\textup{id}_{\T^n} + u^t, v^t, w^t), \qquad T \geq 1, \\
        u \in \mathcal{S}_{(\sigma, 0), \ell-1}^{\mathrm{pol},T}, \qquad  v \in \mathcal{S}_{(\sigma, 0), \ell+ \dec + 1}^{\mathrm{pol},T}, \qquad w \in \mathcal{S}_{(\sigma, 0), \ell+ \dec}^{\mathrm{pol},T}.
    \end{gathered}
        \end{equation}
        Moreover, any other $C^\sigma$ asymptotic KAM torus associated with $(X_H, X_{H_0}, \varphi_0)$ of the form \eqref{eq:asymp_KAM_torus_form} coincides with $\varphi$ in the intersection of their domains.

    \medskip
    
    \item \label{thm:A_transverse}  Suppose that $\sigma \ge 2$ and
    \begin{equation}
    \label{eq:decay_ell_2}
        a \in \mathcal{S}^{\mathrm{pol},1}_{(\sigma, 2),0}, \hspace{2mm}   \partial_qa \in \mathcal{S}_{(\sigma, 1), \ell+ \dec + 3}^{\mathrm{pol}, 1}, \hspace{2mm}  b \in \mathcal{S}_{(\sigma, 2), \ell}^{\mathrm{pol}, 1}, \hspace{2mm}  c \in \mathcal{S}_{(\sigma, 2), \ell+\dec+2}^{\mathrm{pol}, 1}, \hspace{2mm}  d \in \mathcal{S}_{(\sigma, 2), \ell+\dec+1}^{\mathrm{pol}, 1}.
    \end{equation}
    Then the $C^\sigma$ asymptotic KAM torus $\varphi$ associated with $(X_H, X_{H_0}, \varphi_0)$ given by the first statement is asymptotically elliptic.

\end{enumerate}
\end{theoremx}

\begin{remark}
The theorem above shows that \underline{any} perturbation $($not necessarily of the form \eqref{eq:perturbation_form}$)$ of a Hamiltonian with an invariant torus of the form \eqref{eq:initial_hamiltonian} admits an asymptotic elliptic KAM torus provided that the perturbation displays appropriate decrease rates.

Indeed, a general perturbation can be written in the form $P + Q$ where $P$ is of the form \eqref{eq:perturbation_form} and $Q = O(p^2, pz, z^3)$. Thus if $P$ verifies the assumptions of the theorem and $Q$  satisfies $\partial_{(p,z)}^2 Q \in \mathcal{S}_{(\sigma,3),0}^{\mathrm{pol},1}$, then it suffices to apply the theorem above to $H_0 + Q$ instead of $H_0$. 
\end{remark}
\begin{remark}
     The stronger decay assumptions on the perturbation terms $\partial_q a$ and $c$ (obtained by fixing $\ell$ and taking the parameter $\dec$ larger) yield the existence of asymptotic KAM tori whose components $v$ and $w$ exhibit faster decay rates as the parameter $\dec$ increases, while the decay behavior of the component $u$ remains unchanged, namely of order ${t^{-(\ell-1)}}$. 
\end{remark}

In the following corollary, we consider the case where the Hamiltonian $H$ in the statement of Theorem \ref{Thm:elliptic_Csigma} converges, as time tends to infinity, to an autonomous Hamiltonian $H^\infty$ whose zero section $\varphi_0$ in~\eqref{def:varphi0=(q,0,0)} is a normally elliptic invariant torus with quasiperiodic solutions. Under slightly stronger hypotheses than those of Item \ref{thm:A_transverse} of Theorem \ref{Thm:elliptic_Csigma}, we prove a one-to-one correspondence between the solutions, restricted to the transverse $z$-coordinates, of the linearized vector field $DX_H$ along orbits contained in the asymptotic KAM torus $\varphi$ and those associated with $DX_{H^\infty} \circ \varphi_0(q+\omega t)$. For the statement of this result, we consider the following projection \begin{equation}\label{def:Pi_p}
\Pi_p = \left(0_n \quad \mathrm{Id}_n \quad 0_{n \times 2m}\right) \in \mathcal{M}_{n \times 2n+2m}.
\end{equation}
\begin{corollaryx}\label{cor:ell}
    Under the hypotheses of Item \ref{thm:A_transverse} of Theorem \ref{Thm:elliptic_Csigma}, we assume the existence of a $C^2$ autonomous Hamiltonian $H^\infty: \T^n  \times B  \to \R$ of the form~\eqref{def:H_aut}, with $M = \Mell$, such that  
\begin{equation}\label{hyp:cor_ell_decay}
        \int_1^{+\infty} \sup_{q \in \T^n}|\partial_z\partial_pH^t(q,0,0) - \partial_z\partial_pH^\infty(q,0,0)| \, dt < \infty.
\end{equation}
Let $\varphi: \T^n \times [T'', +\infty) \to \T^n \times \R^{n + 2m}$ be the $C^\sigma$ asymptotically elliptic KAM torus associated with $(X_H, X_{H_0}, \varphi_0)$ given by Theorem \ref{Thm:elliptic_Csigma}, $S$ be the map defined in Definition \ref{def:ell_hyp_par_trasv_dyn_asym_KAM} and $\varphi_0$ be the trivial embedding in~\eqref{def:varphi0=(q,0,0)}. 
We consider the following systems 
\begin{align}\label{cor:syst_H}
    \dot \eta(t) = DX_H^t \circ \varphi^t (q + \omega t) \eta(t), & \qquad (q, t) \in \T^n \times [T'', +\infty), \\ \label{cor:syst_Hinfty}
    \dot \xi^\infty(t) = DX_{H^\infty}^t\circ \varphi_0(q + \omega t) \xi^\infty(t), & \qquad (q, t) \in \T^n \times [T'', +\infty).
\end{align}
    Then, if $\ell >2$, for every fixed $q\in\T^n$, $t_0\in I_{T^{''}}$ and $v_0 \in \R^n$ there exists a one-to-one correspondence between the sets 
     $$\left\{ \pi_z\eta \mid \eta \text{ is a solution  of ~\eqref{cor:syst_H}} \text{ with } \pi_p\eta(t_0) = v_0\right\},$$
    $$\left\{ \pi_z\xi^\infty \mid \xi^\infty \text{ is a solution  of ~\eqref{cor:syst_Hinfty}} \text{ with } \pi_p \xi^\infty(t_0) = \Pi_p S^{t_0}(q +\omega t_0)v_0\right\},$$ by the relation 
       \begin{equation*}
          \lim_{t \to +\infty}|\pi_z\eta(t)-\pi_z\xi^\infty(t)|=0.
    \end{equation*}
    \end{corollaryx}

To quantify the regularity and the exponential decay of time-dependent Hölder functions, we introduce the following Banach spaces. Given $\sigma \ge 0$, a non-negative integer $k$, a vector $i \in \N^{2n+2m}$, and parameters $T\ge 0$ and $\lambda \ge 0$, we define
\begin{align}\label{def:Sexp}
\Sexp_{(\sigma,k),\lambda}
    = \left\{
    f : \T^n \times B \times I_T \to \R \;\middle|\;
    \begin{aligned}
        &f^t \in C^{\sigma+k}(\T^n \times B), \text{ for } t \in I_T; \\
        & \sup_{t \in I_T} \bigl(|f^t|_{C^{\sigma+k}}\, e^{\lambda t} \bigr) < \infty; \\
        &\partial^i_{(q,p,z)} f \in C(\T^n \times B \times I_T), \text{ for } 0 \le |i| \le k
    \end{aligned}
    \right\}
\end{align}
with the following norm
\begin{equation}\label{def:norm_Sexp}
    |f|^{\mathrm{exp}, T}_{\sigma+k, \lambda} = \sup_{t \in I_T}|f^t|_{C^{\sigma+k}}e^{\lambda t}.
\end{equation}
We use analogous notation to that introduced for the spaces $\Spol_{(\sigma,k),\ell}$ for functions defined on $\T^m \times I_T$, as well as for vector- and matrix-valued functions.

We prove the following theorem in Section \ref{sec:Proof_Theorem_Hyp}.

\begin{theoremx}\label{Thm:hyp_Csigma}
Fix $\sigma \ge 1$ and $\lambda > 0$. Let $H_0: \T^n  \times B \times [0, +\infty) \to \R$ of the form
\begin{equation}
\label{eq:initial_hamiltonian_hyp}
 H_0(q, p, z, t) =  \omega \cdot p + \frac{1}{2}\Mhyp  \cdot z^2 + O(p^2,pz,z^3), \qquad \partial_{(p,z)}^2 H_0 \in \mathcal{S}_{(\sigma,3),0}^{\mathrm{exp},0},
\end{equation}
and $P: \T^n  \times B \times [0, +\infty) \to \R$ of  the form
\begin{equation}
\begin{gathered} \label{eq:perturbation_form_hyp}
P(q, p, z, t) = a(q, t) + b(q, t) \cdot p + c(q, t) \cdot z + \frac{1}{2}d(q, t)\cdot z^2, \\
a \in \mathcal{S}^{\mathrm{exp},0}_{(\sigma, 2),0}, \quad \partial_q a \in \mathcal{S}_{(\sigma, 1), \lambda}^{\mathrm{exp},0}, \qquad
b \in \mathcal{S}_{(\sigma, 2), \lambda}^{\mathrm{exp},0}, \qquad
c \in \mathcal{S}_{(\sigma, 2), \lambda}^{\mathrm{exp},0}, \quad d \in \mathcal{S}_{(\sigma, 2), \lambda}^{\mathrm{exp}, 0}.
\end{gathered}
\end{equation}
Then, for $H := H_0 + P$ the following holds.
\begin{enumerate}
    \item \label{thm:B_existence} Suppose that
    \begin{equation*}
\lambda > \max_{1 \le i \le m} \Omega_i.
    \end{equation*}
    Then, there exists a $C^\sigma$ asymptotic KAM torus $\varphi$ associated with $(X_H, X_{H_0}, \varphi_0)$ of the form
    \begin{equation}
        \label{eq:asymp_KAM_torus_form_hyp}
    \begin{gathered}
        \varphi:\T^n \times I_T \to \T^n \times \R^n \times \R^{2m}, \qquad \varphi^t = (\textup{id}_{\T^n} + u^t, v^t, w^t), \qquad T \geq 0, \\
        u \in \mathcal{S}_{(\sigma, 0), \lambda}^{\mathrm{exp},T}, \qquad  v \in \mathcal{S}_{(\sigma, 0), \lambda}^{\mathrm{exp},T}, \qquad w \in \mathcal{S}_{(\sigma, 0), \lambda}^{\mathrm{exp},T}.
    \end{gathered}
        \end{equation}
        Moreover, any other $C^\sigma$ asymptotic KAM torus associated with $(X_H, X_{H_0}, \varphi_0)$ of the form \eqref{eq:asymp_KAM_torus_form_hyp} coincides with $\varphi$ in the intersection of their domains.
    \medskip
     \item \label{thm:B_transverse}Suppose that $\sigma \ge 2$ and
    \begin{equation*}
    \lambda > 2\max_{1 \le i \le m} \Omega_i.
    \end{equation*}
    Then the $C^\sigma$ asymptotic KAM torus $\varphi$ associated with $(X_H, X_{H_0}, \varphi_0)$ given by the first statement is asymptotically hyperbolic.
\end{enumerate}
\end{theoremx}

The following corollary is the counterpart of Corollary \ref{cor:ell} in the hyperbolic setting. In contrast with the elliptic case, the correspondence between solutions is no longer one-to-one.
\begin{corollaryx}\label{cor:hyp}
    Under the hypotheses of Item \ref{thm:B_transverse} of Theorem \ref{Thm:hyp_Csigma}, we assume the existence of a $C^2$ autonomous Hamiltonian $H^\infty: \T^n  \times B  \to \R$ of the form~\eqref{def:H_aut}, with $M = \Mhyp$, such that  
\begin{equation}\label{hyp:cor_hyp_decay}
        \lim_{t\to+\infty}e^{-\Omega_j t}\int_1^t e^{\Omega_j \tau} \sup_{q \in \T^n}|\partial_{z_k}\partial_{p_i}H^\tau(q,0,0) - \partial_{z_k}\partial_{p_i}H^\infty(q,0,0)| \, d\tau =0
\end{equation}
for all $1 \le j \le m$, $1 \le k \le 2m$ and $1 \le i \le n$.
Let $\varphi: \T^n \times [T'', +\infty) \to \T^n \times \R^{n + 2m}$ be the $C^\sigma$ asymptotically elliptic KAM torus associated with $(X_H, X_{H_0}, \varphi_0)$ given by Theorem \ref{Thm:hyp_Csigma}, $S$ be the map defined in Definition \ref{def:ell_hyp_par_trasv_dyn_asym_KAM} and $\varphi_0$ be the trivial embedding in~\eqref{def:varphi0=(q,0,0)}. 
We consider the following systems 
\begin{align}\label{cor:syst_H_hyp}
    \dot \eta(t) = DX_H^t \circ \varphi^t (q + \omega t) \eta(t), & \qquad (q, t) \in \T^n \times [T'', +\infty), \\ \label{cor:syst_Hinfty_hyp}
    \dot \xi^\infty(t) = DX_{H^\infty}^t\circ \varphi_0(q + \omega t) \xi^\infty(t), & \qquad (q, t) \in \T^n \times [T'', +\infty).
\end{align}
    Then, for every fixed $q\in\T^n$, $t_0\in I_{T^{''}}$ and $v_0 \in \R^n$ there exists a correspondence between the sets 
     $$\left\{ \pi_z\eta \mid \eta \text{ is a solution  of ~\eqref{cor:syst_H_hyp}} \text{ with } \pi_p\eta(t_0) = v_0\right\},$$
    $$\left\{ \pi_z\xi^\infty \mid \xi^\infty \text{ is a solution  of ~\eqref{cor:syst_Hinfty_hyp}} \text{ with } \pi_p \xi^\infty(t_0) = \Pi_p S^{t_0}(q +\omega t_0)v_0\right\},$$ by the relation 
       \begin{equation*}
          \lim_{t \to +\infty}|\pi_z\eta(t)-\pi_z\xi^\infty(t)|=0,
    \end{equation*}
    we refer to~\eqref{def:Pi_p} for the definition of the projection $\Pi_p$.
\end{corollaryx}

\subsection{Smooth setting}\label{sec:Smooth_Set} In this section, we show that Theorems \ref{Thm:elliptic_Csigma} and \ref{Thm:hyp_Csigma} also hold when considering $C^\infty$ perturbations with appropriate decrease rates. More precisely, if we denote
\[ \Spol_{(\infty,k),\ell} = \bigcap_{\sigma \geq 1} \Spol_{(\sigma,k),\ell}, \qquad \Sexp_{(\infty,k),\lambda} = \bigcap_{\sigma \geq 1} \Sexp_{(\sigma,k),\lambda},\]
for any $k \in \N$ and any $T, \ell, \lambda \geq 0$, we have the following.

\begin{theoremx}
\label{thm:KAM_smooth}
    Theorems \ref{Thm:elliptic_Csigma} and \ref{Thm:hyp_Csigma} hold for $\sigma = +\infty$.
\end{theoremx}

Notice that although for $C^\infty$ perturbations Theorems \ref{Thm:elliptic_Csigma}-\ref{Thm:hyp_Csigma} yield immediately the existence of $C^\sigma$ asymptotic KAM tori, for $\sigma \geq 1$ arbitrarily large, we cannot deduce right away Theorem \ref{thm:KAM_smooth} since the domain of definition of these asymptotic KAM tori may depend on $\sigma$. However, we can easily overcome this difficulty by using the uniqueness properties of these asymptotic KAM tori. 

\begin{proof}[Proof of Theorem \ref{thm:KAM_smooth}]
Let us prove that Theorem \ref{Thm:elliptic_Csigma} holds for $\sigma = +\infty$. The proof in the case of Theorem \ref{Thm:hyp_Csigma} is completely analogous.

Notice that, since the first part of Theorem \ref{Thm:elliptic_Csigma} holds for $C^\infty$ perturbations with $\sigma = 2$ (satisfying the corresponding hypotheses) and as the assumptions for the second part of Theorem \ref{Thm:elliptic_Csigma} imply those in the first part, it suffices to show that the $C^2$ asymptotic KAM torus obtained in the first part of Theorem \ref{Thm:elliptic_Csigma} with $\sigma = 2$ is in fact a $C^\infty$ asymptotic KAM torus.

Let $P$ be a $C^\infty$ perturbation satisfying the hypotheses of the first part of Theorem \ref{Thm:elliptic_Csigma}, for any $\sigma \geq 2$, and let us denote by $\varphi_\sigma: \T^n \times [T_\sigma, +\infty) \to \T^n \times B$ the $C^\sigma$ asymptotic KAM torus given by the theorem. We may assume WLOG that $\sigma \mapsto T_\sigma$ is a non-decreasing function. 

We will now fix $\sigma \geq 2$ and show that $\varphi_2$ is in fact a $C^\sigma$ asympotic KAM torus. For this, it suffices to show that $\varphi_2$ satisfies \eqref{def:cond1_asymKAMtorus} and that $\varphi_2^t$ is of class $C^\sigma$, for any $t \in [T_2, +\infty)$.

By the uniqueness of the asymptotic KAM tori, it follows that
\[ \varphi_{2}\mid_{\T^n \times [T_\sigma, +\infty)} = \varphi_\sigma,\]
and thus $\varphi_2$ satisfies \eqref{def:cond1_asymKAMtorus}. Recalling that $\varphi_2$ satisfies \eqref{eq:asymptotic_KAM_flow}, the equation above yields
\[\varphi_2^t(q) = \Phi^t_{T_\sigma, X_H} \circ \varphi_\sigma^{T_\sigma}(q - \omega (t - T_\sigma)), \qquad \text{ for any } q \in \T^n \text{ and any } t \in [T_2, +\infty). \]

Hence, since $\Phi^t_{T_\sigma, X_H}$ is of class $C^\infty$, for any $t \in [T_2, +\infty)$ fixed, and $\varphi_\sigma^{T_\sigma}$ is of class $C^\sigma$, it follows that $\varphi_2^t$ is of class $C^\sigma$, for any $t \in [T_2, +\infty)$. 

As $\sigma \geq 2$ was arbitrary, this shows that $\varphi_2$ is a $C^\infty$ asymptotic KAM torus.
\end{proof}

\subsection{Analytic setting}\label{sec:Anal_Set} This section contains the real analytic version of the results established in Section \ref{sec:Holder_Set}. To this end, for some $\sigma >0$, we introduce the following complex domains
\begin{equation*}
    \T_\sigma^n = \{q \in \C^n /\Z^n \hspace{1mm}:\hspace{1mm} |\mathrm{Im}\, q| \le \sigma\}, \quad B_\sigma =\{(p,z) \in \C^{n+2m}\hspace{1mm}:\hspace{1mm} |(p,z)| \le \sigma\}.
\end{equation*}
Given $f:\T^n_\sigma \times B_\sigma \to \C$, we consider the following norm
\begin{equation}\label{def:norm_analy}
    |f|_\sigma = \sup_{(q,p,z) \in \T^n_\sigma \times B_\sigma}|f(q,p,z)|.
\end{equation}
We will use the same notation for functions defined only on $\T^n_\sigma$, as well as for vector-valued and matrix-valued functions.

Below, we provide the definition of analytic asymptotic KAM torus, which is the analytic version of Definition \ref{def:Csigma_KAM_torus}. %
Given $\sigma >0$ and $\omega \in \R^n$,  we consider time-dependent vector fields $X^t$ and $X_0^t$ real-analytic on $\T^n_\sigma \times B_\sigma$, for all fixed $t \in I_T$, and a real-analytic embedding $\varphi_0:\T^n_\sigma \to \T^n_\sigma \times B_\sigma$ such that 
\begin{equation}
\begin{aligned}\label{def:hyp_asym_KAM_tori_analy}
    &\lim_{t \to +\infty} \left|X^t - X^t_0\right|_{\sigma} = 0,\\
   &X_0 \circ \varphi_0 = \partial_q\varphi_0 \cdot \omega.
\end{aligned}
\end{equation}
\begin{definition}\label{def:analy_KAM_torus}
    We assume that $(X, X_0, \varphi_0)$, defined on $\T^n_\sigma \times B_\sigma \times I_T$, satisfy~\eqref{def:hyp_asym_KAM_tori_analy}. A family of embeddings $\varphi:\T^n \times I_{T'} \to \T^n \times B$, with $T' \geq T$, is an analytic \emph{asymptotic KAM torus} associated with $(X, X_0, \varphi_0)$, if for some $0<\sigma'< \sigma$
    \begin{align}\label{def:cond1_asymKAMtorus_analy}
            &\lim_{t \to +\infty}\left|\varphi^t -\varphi_0\right|_{\sigma'}=0,\\ \label{def:cond2_asymKAMtorus_analy}
             &X \circ \varphi = \partial_q\varphi \cdot \omega + \partial_t \varphi.
    \end{align}
\end{definition}

 Given $\sigma >0$, $T\ge 1$, and $\ell \ge 0$, to describe the regularity of time-dependent analytic functions decaying polynomially fast in time, we introduce the following Banach spaces
\begin{align}\label{def:S_pol_anal}
\mathscr{S}^{\mathrm{pol}, T}_{\sigma,\ell}
    = \left\{
    f : \T^n_\sigma \times B_\sigma \times I_T \to \C \;\middle|\;
    \begin{aligned}
        &f^t \mbox{ is real-analytic on $\T^n_\sigma \times B_\sigma$ for each }  t \in I_T; \\
        & \sup_{t \in I_T} \bigl(|f^t|_{\sigma}\, t^\ell \bigr) < \infty; \\
        & f \in C(\T^n_\sigma \times B_\sigma \times I_T)
    \end{aligned}
    \right\}
\end{align}
endowed with the norm
\begin{equation}\label{def:norm_S_analy}
    |f|^{\mathrm{pol}, T}_{\sigma, \ell} = \sup_{t \in I_T}|f^t|_{\sigma}t^\ell.
\end{equation}
We will use the same notation for functions defined on $\T^n_\sigma \times I_T$, as well as vector-valued or matrix-valued functions. It will be specified by the context. 
\begin{remark}
We stress that the space~\eqref{def:S_pol_anal} is the real-analytic version of  $\Spol_{(\sigma,k),\ell}$ defined in~\eqref{def:S}. 
However, unlike in the Hölder setting, we only require $f$ to be continuous on $\T^n_\sigma \times B_\sigma \times I_T$. Indeed, by Cauchy's estimates, for every $0<\sigma'<\sigma$, all partial derivatives of $f$ with respect to the variables $(q,p,z)$ are continuous on $\T^n_{\sigma'} \times B_{\sigma'} \times I_T$.
\end{remark}

We recall that $\varphi_0$ stands for the trivial embedding defined in~\eqref{def:varphi0=(q,0,0)}. The following theorems are the real-analytic counterparts of Theorems \ref{Thm:elliptic_Csigma} and \ref{Thm:hyp_Csigma}.

\begin{theoremxx}\label{Thm:elliptic_analy}
Fix $\sigma > 0$, $\ell >1$ and $l \ge 0$. Let $H_0 : \T^n_\sigma \times B_\sigma \to \C$ be real analytic of the form \eqref{eq:initial_hamiltonian} satisfying $\partial_{(p,z)}^2 H_0 \in \mathscr{S}_{\sigma,0}^{\mathrm{pol},1}$ and let $P : \T^n_\sigma \times B_\sigma \to \C$ be real analytic of the form~\eqref{eq:perturbation_form}. Then, for $H := H_0 + P$ the following holds.
\begin{enumerate}
    \item \label{thm:A_existence_analy} Suppose that
    \begin{equation*}
        a \in \mathscr{S}^{\mathrm{pol},1}_{\sigma,0}, \hspace{2mm} \partial_qa \in \mathscr{S}_{\sigma, \ell+ \dec + 2}^{\mathrm{pol},1}, \hspace{2mm}  b \in \mathscr{S}_{\sigma, \ell}^{\mathrm{pol},1}, \hspace{2mm}  c \in \mathscr{S}_{\sigma, \ell+ \dec + 1}^{\mathrm{pol}, 1}, \hspace{2mm} d \in \mathscr{S}_{\sigma, \ell}^{\mathrm{pol}, 1}.
    \end{equation*}
    Then, there exists an analytic asymptotic KAM torus $\varphi$ associated with $(X_H, X_{H_0}, \varphi_0)$ of the form
    \begin{equation}
    \begin{gathered}\label{eq:asymp_KAM_torus_form_ell_analy}
        \varphi:\T^n_{\sigma/2} \times I_T \to \T^n_\sigma \times \C^n \times \C^{2m}, \qquad \varphi^t = (\textup{id}_{\T^n} + u^t, v^t, w^t), \qquad T \geq 1, \\
        u \in \mathscr{S}_{{\sigma \over 2}, \ell-1}^{\mathrm{pol},T}, \qquad  v \in \mathscr{S}_{{\sigma \over 2}, \ell+ \dec + 1}^{\mathrm{pol},T}, \qquad w \in \mathscr{S}_{{\sigma \over 2}, \ell+ \dec}^{\mathrm{pol},T}.
    \end{gathered}
        \end{equation}
        Moreover, any other analytic asymptotic KAM torus associated with $(X_H, X_{H_0}, \varphi_0)$ of the form \eqref{eq:asymp_KAM_torus_form_ell_analy} coincides with $\varphi$ in the intersection of their domains.

    \medskip
    
    \item \label{thm:A_transverse_analy}  Suppose that 
    \begin{equation*}
        a \in \mathscr{S}^{\mathrm{pol},1}_{\sigma,0}, \hspace{2mm}   \partial_qa \in \mathscr{S}_{\sigma, \ell+ \dec + 3}^{\mathrm{pol}, 1}, \hspace{2mm}  b \in \mathscr{S}_{\sigma, \ell}^{\mathrm{pol}, 1}, \hspace{2mm}  c \in \mathscr{S}_{\sigma, \ell+\dec+2}^{\mathrm{pol}, 1}, \hspace{2mm}  d \in \mathscr{S}_{\sigma, \ell+\dec+1}^{\mathrm{pol}, 1}.
    \end{equation*}
    Then the analytic asymptotic KAM torus $\varphi$ associated with $(X_H, X_{H_0}, \varphi_0)$ given by the first statement is asymptotically elliptic.
    
\end{enumerate}
\end{theoremxx}

Given $\sigma >0$, $T\ge 0$, and $\lambda \ge 0$, we define the following Banach spaces of time-dependent real-analytic functions decaying exponentially fast in time.
\begin{align}\label{def:S_exp_anal}
\mathscr{S}^{\mathrm{exp}, T}_{\sigma,\lambda}
    = \left\{
    f : \T^n_\sigma \times B_\sigma \times I_T \to \C \;\middle|\;
    \begin{aligned}
        &f^t \mbox{ is real-analytic on $\T^n_\sigma \times B_\sigma$ for each }  t \in I_T; \\
        & \sup_{t \in I_T} \bigl(|f^t|_{\sigma}\, e^{\lambda t} \bigr) < \infty; \\
        & f \in C(\T^n_\sigma \times B_\sigma \times I_T)
    \end{aligned}
    \right\}
\end{align}
endowed with the norm
\begin{equation}\label{def:norm_S_analy_exp}
    |f|^{\mathrm{exp}, T}_{\sigma, \ell} = \sup_{t \in I_T}|f^t|_{\sigma}e^{\lambda t}.
\end{equation}

\begin{theoremxx}\label{Thm:hyp_analy}
Fix $\sigma > 0$. Let $H_0 : \T^n_\sigma \times B_\sigma \to \C$ be real analytic of the form \eqref{eq:initial_hamiltonian_hyp} satisfying $\partial_{(p,z)}^2 H_0 \in \mathscr{S}_{\sigma,0}^{\mathrm{exp},0}$ and let $P : \T^n_\sigma \times B_\sigma \to \C$ be real analytic of the form \eqref{eq:perturbation_form} with
    \begin{equation*}
        a, \partial_qa, b ,c , d \in \mathscr{S}_{\sigma, \lambda}^{\mathrm{exp}, 0}.
    \end{equation*}

Then, for $H := H_0 + P$ the following holds.
\begin{enumerate}
    \item \label{thm:B_existence_analy} Suppose that
    \[ \lambda > \max_{1 \le i \le m} \Omega_i.\]
    Then, there exists an analytic asymptotic KAM torus $\varphi$ associated with $(X_H, X_{H_0}, \varphi_0)$ of the form
    \begin{equation}
    \begin{gathered}\label{eq:asymp_KAM_torus_form_hyp_analy}
        \varphi:\T^n_{\sigma/2} \times I_T \to \T^n_\sigma \times \C^n \times \C^{2m}, \qquad \varphi^t = (\textup{id}_{\T^n} + u^t, v^t, w^t), \qquad T \geq 0, \\
        u, v, w \in \mathscr{S}_{{\sigma \over 2}, \lambda}^{\mathrm{exp},T}.
    \end{gathered}
        \end{equation}
        Moreover, any other analytic asymptotic KAM torus associated with $(X_H, X_{H_0}, \varphi_0)$ of the form \eqref{eq:asymp_KAM_torus_form_hyp_analy} coincides with $\varphi$ in the intersection of their domains.

    \medskip
    
    \item \label{thm:B_transverse_analy}  Suppose that 
    \[ \lambda > 2\max_{1 \le i \le m} \Omega_i.\]
    Then the analytic asymptotic KAM torus $\varphi$ associated with $(X_H, X_{H_0}, \varphi_0)$ given by the first statement is asymptotically hyperbolic.

\end{enumerate}
\end{theoremxx}

\begin{remark}
It follows from the proofs of Theorems \ref{Thm:elliptic_Csigma}, \ref{Thm:hyp_Csigma}, \ref{Thm:elliptic_analy}, and \ref{Thm:hyp_analy} that we have precise control over the regularity of the conjugating map $S$ and the linear cocycle $A$ in Definition \ref{def:ell_hyp_par_trasv_dyn_asym_KAM}. These objects are constructed in each proof, with the help of the Implicit Function Theorem, in order to determine the asymptotic behaviour of the transverse dynamics.

Since our main goal was to establish the asymptotic convergence, when restricted to the normal coordinates, of the solutions of the linearized cocycle associated with the perturbed system along the asymptotic KAM torus to those of the cocycle $A$ (which exhibits either elliptic or hyperbolic  behaviour depending on the setting), namely, Equation \eqref{def:trasv_dyn_limit}, we did not explicitly include these regularity properties in the statements of our theorems. %

However, it follows directly from the proofs that, in Theorems \ref{Thm:elliptic_Csigma} and \ref{Thm:hyp_Csigma}, the maps $A$ and $S$ are of class $C^{\sigma-1}$ with respect to the variables $(q, p, z)$ (while in Theorems \ref{Thm:elliptic_analy} and \ref{Thm:hyp_analy} they are analytic) and of class $C^1$ with respect to $t$.
\end{remark}

\section{Analysis of the transverse dynamics}
\label{sc:analysis_transversal}
In this section, we study the solutions of the time-dependent linear system~\eqref{eq:Trasv_dyn_AA}. The properties established below will be used to characterize the asymptotic dynamics of the linearized Hamiltonian systems along the orbits contained in the asymptotic KAM tori appearing in our main results.

Given $T\ge 1$, a symmetric matrix $M \in \mathcal{M}_{2m}(\R)$ and letting $\xi = (\xi_1, \xi_2, \xi_3, \xi_4) \in \R^n \times \R^n \times \R^m \times \R^m$, we rewrite~\eqref{eq:Trasv_dyn_AA} as
\begin{equation}
    \begin{cases}\label{eq:Trasv_dyn_A_2}
        \partial_t \xi_1 (t) = a^t_{12}(q+\omega t) \xi_2(t) +  a^{t}_{13}(q+\omega t) \begin{pmatrix} \xi_3(t) \\ \xi_4(t) \end{pmatrix},\\
        \partial_t \xi_2(t) = 0,\\
        \begin{pmatrix}\partial_t \xi_3 (t) \\ \partial_t \xi_4(t) \end{pmatrix} = a^{t}_{32}(q+\omega t) \xi_2(t) + J_m M \begin{pmatrix}\xi_3(t) \\ \xi_4(t) \end{pmatrix},
    \end{cases}
\end{equation}
for $(q,t) \in \T^n \times I_{T}$, where we refer to~\eqref{def:J} for the definition of $J_m$, and to~\eqref{def:A} for the terms $a^t_{12}, a^t_{13}$, and $a^t_{32}$. In what follows, we analyze the elliptic, and hyperbolic cases separately. These are treated in Sections \ref{sec:A_ell}, and \ref{sec:A_hyp}, respectively, and correspond to $M=\Mell$, and $\Mhyp$; see~\eqref{Ms} for the definitions of $\Mell$, and $\Mhyp$. Throughout this section, we denote by $C(\cdot)$ a generic positive constant depending on the parameter in brackets and by
\begin{equation}\label{not:proj_Pixy}
    \Pi_x = (\mathrm{Id}_m \quad 0_{m}), \, \Pi_y = (0_m \quad  \mathrm{Id}_m) \in \mathcal{M}_{m \times 2m}(\R)
\end{equation}
the projections onto the first and last $m$ rows.

Each solution $\xi(t;q,t_0)$ of the above system depends on the parameters $q\in\T^n$ and $t_0\in I_{T}$, and the initial condition $\xi(t_0;q)$ depends on $q\in\T^n$. To avoid cumbersome notation, throughout this section we omit these dependencies and simply write $\xi(t)$ and $\xi(t_0)$ whenever the meaning is clear from the context.

\subsection{Elliptic case: $M=\Mell$.}\label{sec:A_ell}
Here, the system~\eqref{eq:Trasv_dyn_A_2} can be written as
\begin{equation}
    \begin{cases}\label{eq:Trasv_dyn_A_2_ell}
        \partial_t \xi_1 (t) = a^t_{12}(q+\omega t) \xi_2(t) +  a^{t}_{13}(q+\omega t) \begin{pmatrix} \xi_3(t) \\ \xi_4(t) \end{pmatrix},\\
        \partial_t \xi_2(t) = 0,\\
        \begin{pmatrix}\partial_t \xi_3 (t) \\ \partial_t \xi_4(t) \end{pmatrix} = a^{t}_{32}(q+\omega t) \xi_2(t) +  \begin{pmatrix}\Omega \xi_4(t) \\ -\Omega\xi_3(t)\end{pmatrix},
    \end{cases}
\end{equation}
where we recall that $\Omega = \mathrm{diag}(\Omega_1,\dots, \Omega_m)$, for $\Omega_1,\dots,\Omega_m \in \R_{>0}$. 

The following proposition provides quantitative control on the solutions of~\eqref{eq:Trasv_dyn_A_2_ell} through growth estimates. 

\begin{proposition}\label{prop:trasv_dyn_ell_case}
    We assume that the non-autonomous linear system~\eqref{eq:Trasv_dyn_A_2_ell} satisfies
    \begin{equation}\label{def:cond1_ellhyppar_asymKAMtorus_components_A}
        a_{12}, \, a_{13}, \, a_{32} \in C^0(\T^n \times I_{T}), \qquad \sup_{t \in I_{T}}|a_{12}^t|_{C^0}, \, \sup_{t \in I_{T}}|a_{13}^t|_{C^0}, \, \sup_{t \in I_{T}}|a_{32}^t|_{C^0},< \infty.
    \end{equation}
    Then, for every fixed $(q,t_0) \in \T^n \times I_{T}$ and any initial condition $\xi(t_0) = (\xi_1(t_0),\xi_2(t_0),\xi_3(t_0),\xi_4(t_0))$ there exists a unique solution $\xi(t) = (\xi_1(t),\xi_2(t),\xi_3(t),\xi_4(t))$ and a positive constant $C$ depending on $|\xi(t_0)|$, $\sup_{t \in I_{T}}|a_{12}^t|_{C^0}, \, \sup_{t \in I_{T}}|a_{13}^t|_{C^0}, \, \sup_{t \in I_{T}}|a_{32}^t|_{C^0}$, $n$, $m$, $\Omega$ and $t_0$ such that 
    \begin{equation}\label{prop:div_xi_ell_A}
        |\xi_1(t)| \le C t^2, \quad |\xi_2(t)| \le C,  \quad |\xi_3(t)|, \,|\xi_4(t)| \le C t,
    \end{equation}
    for all $t \in I_{T}$.
\end{proposition}
\begin{proof}
    We observe that $\xi_2(t) = \xi_2(t_0)$ for all $t \in I_{T}$. This proves~\eqref{prop:div_xi_ell_A} for $\xi_2$. Now, we study the last equation of~\eqref{eq:Trasv_dyn_A_2_ell} where $\xi_2(t) = \xi_2(t_0)$ is known. For this purpose, we introduce the following complex notation. For all $t \in I_{T}$, let 
    \begin{align*}
        \zeta (t) &= \xi_3(t) + i\xi_4(t),\\
         F^t (q + \omega t) &=  \Pi_xa^{t}_{32}(q + \omega t)\xi_2(t_0) + i \Pi_y a^{t}_{32}(q+\omega t) \xi_2(t_0),
    \end{align*}
    where we refer to~\eqref{not:proj_Pixy} for the definition of $\Pi_x$ and $\Pi_y$.   We rewrite the last equation of~\eqref{eq:Trasv_dyn_A_2_ell} as
    \begin{equation*}
        \partial_t \zeta(t) = -i \Omega \zeta(t) +   F^t (q + \omega t)
    \end{equation*}
    and the solution of the latter is given by
    \begin{equation*}
        \zeta(t) = e^{-i\Omega(t-t_0)} \zeta(t_0) + \int_{t_0}^t e^{-i\Omega(t-s)}F^s(q+\omega s)ds
    \end{equation*}
    or in components, for every $1 \le j \le m$
     \begin{equation*}
        \zeta_j(t) = e^{-i\Omega_j(t-t_0)} \zeta_j(t_0) + \int_{t_0}^t e^{-i\Omega_j(t-s)}F^s_j(q+\omega s)ds
    \end{equation*}
    where $\xi_3 = \mathrm{Re} \zeta$ and $\xi_4 = \mathrm{Im} \zeta$. By~\eqref{def:cond1_ellhyppar_asymKAMtorus_components_A}, we have that $ \sup_{t \in I_{T}}|F^t|_{C^0}<\infty$ and hence trivial estimations show that $\xi_3(t)$ and $\xi_4(t)$ satisfy~\eqref{prop:div_xi_ell_A}. It remains to analyze the first equation of~\eqref{eq:Trasv_dyn_A_2_ell} where $\xi_2(t)$, $\xi_3(t)$, and $\xi_4(t)$ are known. We have that 
    \begin{equation*}
        \xi_1(t) = \xi_1(t_0) + \int_{t_0}^t a^s_{12}(q+\omega s) \xi_2(t_0) +  a^{s}_{13}(q+\omega s) \begin{pmatrix} \xi_3(s) \\ \xi_4(s) \end{pmatrix} \, ds
    \end{equation*}
    Combining~\eqref{def:cond1_ellhyppar_asymKAMtorus_components_A} and the above estimates proved for $\xi_3(t)$ and $\xi_4(t)$ with the latter, we prove that also $\xi_1(t)$ satisfies~\eqref{prop:div_xi_ell_A}.
\end{proof}
\begin{remark}
    We point out that in Theorem \ref{Thm:elliptic_Csigma} the unperturbed Hamiltonian $H_0$ in~\eqref{eq:initial_hamiltonian} has an invariant torus in $p=z=0$ supporting quasiperiodic solutions of frequency vector $\omega$. Therefore, it is natural to focus on the study of the transverse dynamics in a neighborhood of the invariant torus corresponding to $p=0$. This amounts to considering solutions $\xi(t)$ of system~\eqref{eq:Trasv_dyn_A_2_ell} with initial condition
    \begin{equation*}
        \xi(t_0)=(\xi_1(t_0),0,\xi_3(t_0),\xi_4(t_0)) 
    \end{equation*}
    which provides boundedness with respect to the motions transversal to the invariant torus associated with the unperturbed Hamiltonian $H_0$. These orbits are given by
     \begin{equation*}
    \begin{aligned}
        \xi_3(t) &= \cos \left(\Omega (t-t_0)\right)  \xi_3(t_0) +  \sin \left(\Omega (t - t_0)\right)\xi_4(t_0),\\
        \xi_4(t) &= \cos \left(\Omega (t-t_0)\right) \xi_4(t_0) -  \sin \left(\Omega (t - t_0)\right)\xi_3(t_0),
    \end{aligned}
    \end{equation*}
    where $\cos \big(\Omega (t-t_0)\big) = \mathrm{diag}\Big(\cos \big(\Omega_1 (t-t_0)\big), \dots, \cos \big(\Omega_m (t-t_0)\big)\Big)$ and $\sin \big(\Omega (t-t_0)\big) = \mathrm{diag}\Big(\sin \big(\Omega_1 (t-t_0)\big), \dots, \sin \big(\Omega_m (t-t_0)\big)\Big)$.  
\end{remark}

If we assume $\xi_2(t_0) \ne 0$, also in the autonomous case, as the one described in~\eqref{def:DXH_aut}, the motions transversal to the invariant torus in $p=z=0$ could exhibit an unexpected parabolic behavior. For completeness, in the following sections (see Sections \ref{sec:booundedness_ell_aut} and \ref{sec:booundedness_ell_non_aut}), we briefly analyze the corresponding autonomous setting, and we provide some hypotheses that ensure the existence of bounded transverse dynamics in the non-autonomous case. 
Although one typically associates elliptic dynamics with bounded and oscillatory motions, we shall retain this terminology, following the standard nomenclature in classical KAM theory concerning the persistence of normally elliptic isotropic invariant tori. 

Finally, in Section \ref{sec:limit_aut_ell}, %
we assume that $a_{32}^t$ converge as time tends to infinity to an autonomous matrix $a_{32}^\infty$.
We investigate under which assumption any transversal solution, described by $(\xi_3(t), \xi_4(t))$, of~\eqref{eq:Trasv_dyn_A_2_ell} is asymptotic to a transversal solution of the corresponding limiting autonomous system and vice versa. 
\subsubsection{On the Boundedness of the Transverse Dynamics: autonomous case}\label{sec:booundedness_ell_aut}
First, we need to prove the following technical lemma.
\begin{lemma}\label{lem:lemma_tecnico_ell_aut}
    Let $\sigma \ge 0$, $\Lambda >0$, $T>0$, $\omega \in \R^n$, and $G:\T^n \to \R$. We assume that  $$G \in C^{\sigma + n +1}(\T^n)$$   and we consider the following equation 
    \begin{equation}\label{eq:tec_lemma_aut_ell}
        \partial_t \chi(t) = -i\Lambda \chi(t) + G(q+\omega t)
    \end{equation}
    in the unknown $\chi:I_T \to \C$. 
    In addition, we assume the existence of $\gamma>0$, $\tau \ge 0$ such that, for all $k \in \Z^n$ and $1 \le j \le m$
    \begin{equation*}%
        |\omega \cdot k + \Lambda| \ge {\gamma \over \left(|k|+1 \right)^\tau},
    \end{equation*}
    where $|k| = |k|_1$. Then, for any fixed $(q,t_0) \in \T^n \times I_T$ there exists a unique solution $\chi$ of~\eqref{eq:tec_lemma_aut_ell} and a positive constant $C$ depending on $n$, $\sigma$, $\Lambda$, $\gamma$, $\tau$ and $|G|_{C^{\sigma +n +1}}$ such that 
    \begin{equation*}
        \chi(t_0) = 0, \quad |\chi(t)| \le C,
    \end{equation*}
    for all $t \in I_T$.
\end{lemma}
\begin{proof}
    In order to study~\eqref{eq:tec_lemma_aut_ell}, we first look for a particular solution of the form $\chi (t) = g(q + \omega t )$. Replacing it into~\eqref{eq:tec_lemma_aut_ell}, we obtain 
\begin{equation}\label{syst:trasv_ell_z_aut}
    \partial_q g(q + \omega t)\omega + i \Lambda g(q + \omega t) = G(q + \omega t).
\end{equation}

By Lemma $27$ in~\cite{Fe04} there exists a unique solution $g$ of~\eqref{syst:trasv_ell_z_aut} and a positive constant $C(n, \sigma, \Lambda)$ depending on $n$, $\sigma$, $\gamma$, $\tau$ and $\Lambda$ such that 
\begin{equation*}
    g (q + \omega t) = \sum_{k \in \Z^n} {G^{[k]} \over i \left(\omega \cdot k +  \Lambda\right)} e^{ik \cdot \left(q + \omega t\right)}, \qquad |g|_{C^\sigma} \le C(n, \sigma, \Lambda)|G|_{C^{\sigma + n +1}}
\end{equation*}
where $G^{[k]}$ stands for the $k$-th Fourier coefficient of $G$. Thus, the solution of equation~\eqref{eq:tec_lemma_aut_ell} with initial condition $\chi(t_0) = 0$ is given by
\begin{equation*}
    \chi(t) = g(q + \omega t) - e^{-i\Lambda(t-t_0)}g(q + \omega t_0). 
\end{equation*}
This concludes the proof of this lemma. 
\end{proof}
Given $T>0$, for $(q,t,t_0) \in \T^n \times I_T \times I_T$, we consider the following system 
\begin{equation}
    \begin{cases}\label{eq:Trasv_dyn_A_1_ell_aut}
        \partial_t \xi_1 (t) = a_{12}(q+\omega t) \xi_2(t) +  a_{13}(q+\omega t) \begin{pmatrix} \xi_3(t) \\ \xi_4(t) \end{pmatrix},\\
        \partial_t \xi_2(t) = 0,\\
        \begin{pmatrix}\partial_t \xi_3 (t) \\ \partial_t \xi_4(t) \end{pmatrix} = a_{32}(q+\omega t) \xi_2(t) +  \begin{pmatrix}\Omega \xi_4(t) \\ -\Omega\xi_3(t) \end{pmatrix},
    \end{cases}
\end{equation}
where now $a_{12}:\T^n \to \mathcal{M}_n(\R)$, $a_{13} :\T^n \to \mathcal{M}_{n\times 2m}(\R)$, and  $a_{32}:\T^n \to \mathcal{M}_{2m\times n}(\R)$  are functions that do not depend explicitly on time. 

\begin{proposition}\label{prop:bound_trasv_dyn_ell_aut} 
    Given $\sigma \ge 0$, we consider the autonomous linear system~\eqref{eq:Trasv_dyn_A_1_ell_aut} satisfying
    \begin{equation}\label{eq:Trasv_dyn_A_1_ell_aut_reg}
    a_{12}, \, a_{13} \in C^0(\T^n), \qquad  a_{32} \in C^{\sigma + n + 1}(\T^n).
\end{equation}
  We assume that there exist $\gamma>0$, $\tau \ge 0$ such that, for all $k \in \Z^n$ and $1 \le j \le m$
    \begin{equation}\label{diop_ass_ell_aut}
        |\omega \cdot k + \Omega_j| \ge {\gamma \over \left(|k|+1 \right)^\tau}
    \end{equation}
    where $|k| = |k|_1$. Then, for every fixed $(q,t_0) \in \T^n \times I_T$ and any initial condition $\xi(t_0)$ there exists a unique solution $\xi(t)$ of~\eqref{eq:Trasv_dyn_A_1_ell_aut} and a positive constant $C$ depending on $|\xi(t_0)|$, $|a_{12}|_{C^0}, \, |a_{13}|_{C^0}, \, |a_{32}|_{C^{\sigma + n+1}}$, $\sigma$, $n$, $m$, $\Omega$, $\gamma$, $\tau$ and $t_0$ such that 
    \begin{equation}\label{prop:div_xi_ell_aut_A}
        |\xi_1(t)| \le C t, \quad |\xi_2(t)|, \, |\xi_3(t)|, \,|\xi_4(t)| \le C 
    \end{equation}
    for all $t \in I_T$.
\end{proposition}
\begin{proof}
    For all $t \in I_T$, we have that $\xi_2(t) = \xi_2(t_0)$. Following the lines of the proof of Proposition \ref{prop:trasv_dyn_ell_case}, for all $t \in I_T$, we introduce the following complex notation
\begin{align*}
        \zeta (t) &= \xi_3(t) + i\xi_4(t),\\
         F (q + \omega t) &=  \Pi_xa_{32}(q+ \omega t) \xi_2(t_0) + i  \Pi_y a_{32}(q+ \omega t) \xi_2(t_0),
\end{align*}
and we rewrite the last line of~\eqref{eq:Trasv_dyn_A_1_ell_aut}, where now $\xi_2(t)$ is known, as
\begin{equation}\label{syst:trasv_dyn_aut_case_ell_2}
    \partial_t \zeta(t) = -i\Omega \zeta(t) + F(q + \omega t),
\end{equation}
or, in components, for every $1 \le j \le m$
\begin{equation}\label{syst:trasv_dyn_aut_case_ell_2_comp}
    \partial_t \zeta_j(t) = -i\Omega_j \zeta(t) + F_j(q + \omega t). 
\end{equation}
By Lemma \ref{lem:lemma_tecnico_ell_aut}, there exists a unique solution $\zeta_{\mathrm{part}}(t)$ of~\eqref{syst:trasv_dyn_aut_case_ell_2} satisfying
\begin{equation*}
    \zeta_{\mathrm{part}}(t_0) = 0, \qquad |\zeta_{\mathrm{part}}(t)| \le C(\sigma, n, m, \Omega, \gamma, \tau, |F|_{C^{\sigma +n +1}}),
\end{equation*}
for all $t \in I_T$, where $C(\sigma, n, \Omega, |F|_{C^{\sigma +n +1}})$ is a constant depending on $\sigma$, $n$, $m$ $\Omega$, $\gamma$, $\tau$, and $|F|_{C^{\sigma +n +1}}$. More specifically, the solution of equation~\eqref{syst:trasv_dyn_aut_case_ell_2} with initial condition $\zeta(t_0) = \xi_3(t_0) + i \xi_4(t_0)$ is given by
\begin{equation*}
    \zeta(t) = \zeta_{\mathrm{part}}(t) + e^{-i\Omega(t-t_0)}\zeta(t_0),
\end{equation*}
where $\xi_3(t) = \mathrm{Re} \zeta (t)$ and $\xi_4(t) = \mathrm{Im} \zeta(t)$. This proves~\eqref{prop:div_xi_ell_aut_A} for $\xi_3(t)$ and $\xi_4(t)$. Similarly to the proof of Proposition \ref{prop:trasv_dyn_ell_case}, one can verify~\eqref{prop:div_xi_ell_aut_A} also for $\xi_1(t)$.

\end{proof}
We point out that if the arithmetic condition~\eqref{diop_ass_ell_aut} on the frequencies $\omega$ and $\Omega$ fails, we can construct examples of autonomous linear systems as~\eqref{eq:Trasv_dyn_A_1_ell_aut} where the dynamics associated with $\xi_3(t)$ and $\xi_4(t)$ exhibit a parabolic behavior.
\begin{example}
     We consider the linear autonomous system in~\eqref{eq:Trasv_dyn_A_1_ell_aut}. For simplicity, we consider $m=1$. In this case, $\Omega \in \R_{>0}$. We assume the existence of $k^*  \in \Z^n$  such that 
    \begin{equation}\label{ex:resonance}
        \omega \cdot k^* + \Omega = 0. 
    \end{equation}
    For  fixed $(q,t_0) \in \T^n \times I_T$, we introduce the following notation
\begin{align*}
        \zeta (t) &= \xi_3(t) + i\xi_4(t),\\
        F(q + \omega(t-t_0)) &= \Pi_x a_{32}(q+ \omega t)\xi_2(t_0) + i  \Pi_y a_{32}(q+ \omega t) \xi_2(t_0).
\end{align*}
 We also assume that 
    \begin{equation}\label{ex:Fj*}
         F(q+ \omega t) = e^{i k^* \cdot (q+ \omega t)}.
    \end{equation}
     We can rewrite the last line of~\eqref{eq:Trasv_dyn_A_1_ell_aut}, where $\xi_2(t) = \xi_2(t_0) \ne 0$ is known, as
\begin{equation*}
    \partial_t \zeta(t) = -i\Omega \zeta(t) + F(q + \omega t).
\end{equation*}
 Using~\eqref{ex:resonance} and~\eqref{ex:Fj*}, a straightforward computation shows that the solution of the latter is given by  
\begin{equation*}
    \zeta (t) = e^{-i\Omega(t-t_0)}\left(\zeta(t_0) + e^{ik^* \cdot q}(t-t_0)\right)
\end{equation*}
and hence, using $\xi_3(t) = \mathrm{Re} \zeta (t)$ and $\xi_4(t) = \mathrm{Im} \zeta(t)$, we have that 
\begin{align*}
    \xi_3(t) &= \cos \left(\Omega(t-t_0)\right) \xi_3(t_0) + \sin \left(\Omega(t-t_0)\right) \xi_4(t_0) \\
    &+ (t-t_0)\left[\cos \left(\Omega(t-t_0)\right) \cos \left(k^*\cdot q\right) + \sin \left(\Omega(t-t_0)\right) \sin \left(k^*\cdot q\right)\right],\\
    \xi_4(t) &= \cos \left(\Omega(t-t_0)\right) \xi_4(t_0) - \sin \left(\Omega(t-t_0)\right) \xi_3(t_0) \\
    &+ (t-t_0)\left[\cos \left(\Omega(t-t_0)\right) \sin \left(k^*\cdot q\right) - \sin \left(\Omega(t-t_0)\right) \cos \left(k^*\cdot q\right)\right].
\end{align*}
\end{example}
\subsubsection{On the Boundedness of the Transverse Dynamics: non-autonomous case}\label{sec:booundedness_ell_non_aut}

Concerning the non-autonomous linear system~\eqref{eq:Trasv_dyn_A_2_ell}, we will show that, as in the autonomous case (see Proposition \ref{prop:bound_trasv_dyn_ell_aut}), suitable arithmetic and decay assumptions ensure that the components $\xi_3(t)$ and $\xi_4(t)$ remain bounded for all $t \in I_{T}$.
\begin{proposition}
\label{prop:bound_trasv_dyn_ell_non_aut} 
    Given $\sigma \ge 0$, we consider the non-autonomous linear system~\eqref{eq:Trasv_dyn_A_2_ell} satisfying
    \begin{equation}\label{eq:Trasv_dyn_A_1_ell_non_aut_reg}
      a_{12}, \, a_{13}, \, a_{32} \in C^0(\T^n \times I_{T}), \qquad \sup_{t \in I_{T}}|a_{12}^t|_{C^0}, \, \sup_{t \in I_{T}}|a_{13}^t|_{C^0}, \, \sup_{t \in I_{T}}|a_{32}^t|_{C^0},< \infty.
\end{equation}
     We assume that the pair $(\omega, \Omega)$ satisfy the Diophantine condition \eqref{diop_ass_ell_aut}, for some $\gamma > 0$ and $\tau \ge 0$. %
    In addition, we assume the existence of $a_{32}^\infty: \T^n \to \mathcal{M}_{2m\times n}(\R)$ such that 
    \begin{equation}\label{decay_ass_ell_non_aut}
       a_{32}^\infty \in C^{\sigma+n+1}(\T^n), \qquad  \int_{T}^{+\infty} |a_{32}^s - a_{32}^\infty|_{C^0} \,ds < \infty.
    \end{equation}
    Then, for fixed $(q,t_0) \in \T^n \times I_{T}$ and any initial condition $\xi(t_0)$ there exists a unique solution $\xi(t)$ of~\eqref{eq:Trasv_dyn_A_1_ell_aut} and a positive constant $C$ depending on $|\xi(t_0)|$, $\sup_{t \in I_{T}}|a_{12}^t|_{C^0}, \, \sup_{t \in I_{T}}|a_{13}^t|_{C^0}$, $|a_{32}^\infty|_{C^{\sigma+n+1}}$, $\int_{T}^\infty |a_{32}^s - a_{32}^\infty|_{C^0}$, $n$, $m$, $\sigma$, $\Omega$, $\gamma$, $\tau$ and $t_0$ such that 
    \begin{equation}\label{prop:div_xi_ell}
        |\xi_1(t)| \le C t, \quad |\xi_2(t)|, \, |\xi_3(t)|, \,|\xi_4(t)| \le C
    \end{equation}
    for all $t \in I_{T}$.
\end{proposition}
\begin{proof}
    For all $t \in I_{T}$, we have that $\xi_2(t) = \xi_2(t_0)$. The analysis of the last equation of~\eqref{eq:Trasv_dyn_A_2_ell} is more subtle. We introduce the following complex notation  
    \begin{align*}
        \zeta(t) &= \xi_3(t) + i\xi_4(t),\\
        F^{\infty}(q + \omega t) &= \Pi_x a^\infty_{32}(q+ \omega t)\xi_2(t_0) + i \Pi_y a^\infty_{32}(q+ \omega t)  \xi_2(t_0),\\
        F^t(q + \omega t) &= \Pi_x  a^t_{32}(q+ \omega t)\xi_2(t_0) + i \Pi_y a^t_{32}(q+ \omega t) \xi_2(t_0),\\
        R^t(q + \omega t) &= F^t(q+\omega t) - F^{\infty}(q + \omega t) 
    \end{align*}
    for all $t \in I_{T}$. Moreover, by~\eqref{decay_ass_ell_non_aut}, we have that $R^t$ satisfies
    \begin{equation}\label{property:Frest_ell_non_aut}
        \int_{T}^{+\infty} |R^s|_{C^0}\, ds <\infty.
    \end{equation}
    With the above notation, we rewrite the last equation of~\eqref{eq:Trasv_dyn_A_1_ell_aut}, where $\xi_2(t)$ is known, as
\begin{equation}\label{syst:trasv_dyn_non_aut_case_ell_2}
    \partial_t \zeta(t) = -i\Omega \zeta(t) + F^{\infty}(q + \omega t) + R^t(q + \omega t). 
\end{equation}
In order to study the solutions of the latter, first, we look for a particular solution $\zeta^\infty(t)$ of the following unperturbed autonomous system
\begin{equation}\label{syst:trasv_dyn_non_aut_case_ell_2_unpert}
    \partial_t \zeta^\infty(t) = -i\Omega \zeta^\infty(t) + F^{\infty}(q + \omega t) ,
\end{equation}
or, in components, for every $1 \le j \le m$
\begin{equation*}
    \partial_t \zeta^\infty_j(t) = -i\Omega \zeta^\infty_j(t) + F^{\infty}_j(q + \omega t).
\end{equation*}
By Lemma \ref{lem:lemma_tecnico_ell_aut}, there exists a unique solution $\zeta^\infty(t)$ of~\eqref{syst:trasv_dyn_non_aut_case_ell_2_unpert} such that 
\begin{equation}\label{zeta_infty_prop_ell_non_aut}
    \zeta^\infty(t_0)=0, \qquad |\zeta^\infty(t)| \le C(\sigma, n, \Omega, |F^\infty|_{C^{\sigma+n+1}}) ,
\end{equation}
for all $t \in I_{T}$. We now look for solutions of~\eqref{syst:trasv_dyn_non_aut_case_ell_2} of the form 
\begin{equation}\label{sol:zeta*zetaInft*zetaRest_ell_non_aut}
    \zeta(t) = \zeta^\infty(t) + \zeta_{\mathrm{rest}}(t),
\end{equation}
where $\zeta^\infty(t)$ is the solution of~\eqref{syst:trasv_dyn_non_aut_case_ell_2_unpert} satisfying~\eqref{zeta_infty_prop_ell_non_aut}. Replacing~\eqref{sol:zeta*zetaInft*zetaRest_ell_non_aut} into~\eqref{syst:trasv_dyn_non_aut_case_ell_2} we obtain that 
\begin{equation*}%
    \partial_t \zeta_{\mathrm{rest}}(t) = -i\Omega \zeta_{\mathrm{rest}}(t) + R^t(q + \omega(t-t_0)). 
\end{equation*}
The solution of the above equation with initial condition $\zeta_{\mathrm{rest}}(t_0) = \zeta(t_0) = \xi_3(t_0) + i \xi_4(t_0)$ is given by
\begin{equation*}
    \zeta_{\mathrm{rest}}(t) = e^{-i\Omega(t-t_0)} \zeta(t_0) + \int_{t_0}^te^{-i\Omega(t-s)}R^s(q + \omega s) \, ds.
\end{equation*}
Furthermore, the solution $\zeta(t)$ of~\eqref{syst:trasv_dyn_non_aut_case_ell_2} with initial condition $\zeta(t_0) = \xi_3(t_0) + i \xi_4(t_0)$ is given by
\begin{equation*}
    \zeta(t) = \zeta^\infty(t) + e^{-i\Omega(t-t_0)} \zeta(t_0) + \int_{t_0}^te^{-i\Omega(t-s)}R^s(q + \omega s) \, ds,
\end{equation*}
where we recall that $\zeta^\infty(t)$ is the solution of~\eqref{syst:trasv_dyn_non_aut_case_ell_2_unpert} satisfying~\eqref{zeta_infty_prop_ell_non_aut}. Thus, using~\eqref{property:Frest_ell_non_aut},~\eqref{zeta_infty_prop_ell_non_aut}, and $\xi_3(t) = \mathrm{Re}\zeta(t)$ and $\xi_4(t) = \mathrm{Im}\zeta(t)$, we prove~\eqref{prop:div_xi_ell} for $\xi_3(t)$ and $\xi_4(t)$. Similarly to the proof of Proposition \ref{prop:trasv_dyn_ell_case}, one can verify~\eqref{prop:div_xi_ell} also for $\xi_1(t)$.
\end{proof}

\subsubsection{Asymptotic correspondence with the limiting autonomous system}\label{sec:limit_aut_ell} In this section, we consider the following autonomous system 
\begin{equation}
    \begin{cases}\label{eq:Trasv_dyn_A_2_ell_syst_infty}
        \partial_t \xi^\infty_1 (t) = a^\infty_{12}(q+\omega t) \xi^\infty_2(t) +  a^\infty_{13}(q+\omega t) \begin{pmatrix} \xi^\infty_3(t) \\ \xi^\infty_4(t) \end{pmatrix},\\
        \partial_t \xi^\infty_2(t) = 0,\\
        \begin{pmatrix}\partial_t \xi^\infty_3 (t) \\ \partial_t \xi^\infty_4(t) \end{pmatrix} = a^\infty_{32}(q+\omega t) \xi^\infty_2(t) +  \begin{pmatrix}\Omega \xi^\infty_4(t) \\ -\Omega\xi^\infty_3(t) \end{pmatrix},
    \end{cases}
\end{equation}
for all $(q,t) \in \T^n \times I_T$, where $a^\infty_{12}:\T^n \to \mathcal{M}_n(\R)$, $a^\infty_{13} :\T^n \to \mathcal{M}_{n\times 2m}(\R)$, and  $a^\infty_{32}:\T^n \to \mathcal{M}_{2m\times n}(\R)$  do not depend explicitly on time.

In the following proposition, we provide some hypotheses ensuring a one-to-one correspondence between the transverse solutions of systems~\eqref{eq:Trasv_dyn_A_2_ell} and~\eqref{eq:Trasv_dyn_A_2_ell_syst_infty} and prove that the corresponding orbits are asymptotic as time tends to infinity. We recall that the parameter $T$ is defined in~\eqref{eq:Trasv_dyn_A_2} and $\pi_{z}$ stands for the projection onto the $z$-components.  

\begin{proposition}\label{prop:ell_1:1_asym_sol}
    We consider the systems~\eqref{eq:Trasv_dyn_A_2_ell} and~\eqref{eq:Trasv_dyn_A_2_ell_syst_infty} and we assume that 
    \begin{equation}\label{hyp:limit_aut_sys_ell}
        \int_{T}^{+\infty} |a_{32}^s - a_{32}^\infty|_{C^0} \, ds < \infty.
    \end{equation}
  Then, for every fixed $q\in\T^n$, $t_0\in I_{T}$ and $\xi_2^0 \in \R^n$ there exists a one-to-one correspondence between the sets $$\{ \pi_z\xi^\infty \mid \xi^\infty \text{ is a solution  of ~\eqref{eq:Trasv_dyn_A_2_ell_syst_infty}} \text{ with } \xi_2^\infty(t_0) = \xi_2^0\},$$ $$\{ \pi_z\xi \mid \xi \text{ is a solution  of ~\eqref{eq:Trasv_dyn_A_2_ell}} \text{ with } \xi_2(t_0) = \xi_2^0\},$$ by the relation 
       \begin{equation*}
        \lim_{t \to +\infty}|\pi_z\xi^\infty(t)-\pi_z\xi(t)|=0.
    \end{equation*}
\end{proposition}
\begin{proof}
    For all $t \in \times I_{T}$, we introduce the following complex notation
    \begin{align*}
        \zeta(t) &= \xi_3(t) + i \xi_4(t), \qquad \zeta^\infty(t) = \xi^\infty_3(t) + i \xi^\infty_4(t),\\
        F^{\infty}(q + \omega t) &= \Pi_x  a^\infty_{32}(q+ \omega t) \xi_2(t_0) + i \Pi_y a^\infty_{32}(q+ \omega t)\xi_2(t_0),\\
        F^t(q + \omega t) &= \Pi_x  a^t_{32}(q+ \omega t) \xi_2(t_0) + i \Pi_y a^t_{32}(q+ \omega t)\xi_2(t_0),\\
        R^t(q + \omega t) &=  F^t(q + \omega t) - F^{\infty}(q + \omega t),\\
    \end{align*}
    where~\eqref{hyp:limit_aut_sys_ell} implies
    \begin{equation}\label{hyp:limit_aut_sys_ell_F_rest}
        \int_{T}^{+\infty} |R^s|_{C^0}\, ds <\infty.
    \end{equation}
    Remembering that we are assuming $\xi_2(t_0) = \xi^\infty_2(t_0)$, we rewrite the last equation of~\eqref{eq:Trasv_dyn_A_2_ell} and~\eqref{eq:Trasv_dyn_A_2_ell_syst_infty} as
    \begin{align*}
        \partial_t \zeta(t) &= -i\Omega \zeta(t) + F^{\infty}(q + \omega t) + R^t(q + \omega t),\\
        \partial_t \zeta^\infty(t) &= -i\Omega \zeta^\infty(t) + F^{\infty}(q + \omega t).
    \end{align*}
    A straightforward computation shows that the difference of the formal solutions of the latter is equal to 
    \begin{align*}
        \zeta(t) - \zeta^\infty(t) &= e^{-i\Omega(t-t_0)}\left(\zeta(t_0) - \zeta^\infty(t_0) + \int_{t_0}^{+\infty}e^{-i\Omega(t_0-s)} R^s(q + \omega s)\, ds \right) \\
        &- \int_t^{+\infty} e^{-i\Omega(t-s)} R^s(q + \omega s)\, ds
    \end{align*}
By~\eqref{hyp:limit_aut_sys_ell_F_rest}, we have that $\int_{t_0}^{+\infty}e^{-i\Omega(t_0-s)} R^s(q + \omega s)\, ds$ is well defined and 
\begin{align*}
    &\lim_{t \to +\infty} \left|\int_t^{+\infty} e^{-i\Omega(t-s)} R^s(q + \omega s)\, ds\right| \le C(m, \Omega)\lim_{t \to +\infty}\int_t^{+\infty} | R^s|_{C^0}\, ds = 0.
\end{align*}
Therefore, 
\begin{equation*}
    \lim_{t \to +\infty}|\zeta(t) - \zeta^\infty(t)|=0 \hspace{5mm} \mbox{if and only if} \hspace{5mm} \zeta(t_0) = \zeta^\infty(t_0) - \int_{t_0}^{+\infty}e^{-i\Omega(t_0-s)} R^s(q + \omega s)\, ds. 
\end{equation*}
Remembering that $\pi_z \xi(t) = (\mathrm{Re}\zeta(t), \mathrm{Im}\zeta(t))$ and $\pi_z \xi^\infty(t) = (\mathrm{Re}\zeta^\infty(t), \mathrm{Im}\zeta^\infty(t))$, we conclude the proof of this proposition. 
\end{proof}

\subsection{Hyperbolic case: $M=\Mhyp$.}\label{sec:A_hyp}
In this case, the system~\eqref{eq:Trasv_dyn_A_2} can be written as
\begin{equation}
    \begin{cases}\label{eq:Trasv_dyn_A_2_hyp}
        \partial_t \xi_1 (t) = a^t_{12}(q+\omega  t) \xi_2(t) +  a^{t}_{13}(q+\omega t ) \begin{pmatrix} \xi_3(t) \\ \xi_4(t) \end{pmatrix},\\
        \partial_t \xi_2(t) = 0,\\
        \begin{pmatrix}\partial_t \xi_3 (t) \\ \partial_t \xi_4(t) \end{pmatrix} = a^{t}_{32}(q+\omega t) \xi_2(t) -  \begin{pmatrix}\Omega \xi_4(t) \\ \Omega\xi_3(t) \end{pmatrix},
    \end{cases}
\end{equation}
for all $(q,t) \in \T^n \times I_T$, where we recall that $\Omega = \mathrm{diag}(\Omega_1,\dots, \Omega_m)$, for $\Omega_1,\dots,\Omega_m \in \R_{>0}$ and the parameter $T$ is defined in~\eqref{eq:Trasv_dyn_A_2}.

This section is divided into two parts. First, in the following proposition, we provide growth estimates for the solutions of the system~\eqref{eq:Trasv_dyn_A_2_hyp}. Finally, in Proposition \ref{prop:hyp_1:1_asym_sol}, we assume that $a_{32}^t$ converge in time to an autonomous matrix $a^\infty_{32}$. We analyze under which hypotheses there is an asymptotic correspondence between the transversal solutions of~\eqref{eq:Trasv_dyn_A_2_hyp} and those of the limiting system.
\begin{proposition}\label{prop:trasv_dyn_hyp_case} 
    We assume that the non-autonomous linear system~\eqref{eq:Trasv_dyn_A_2_hyp} satisfies
    \begin{equation}\label{def:cond1_ellhyppar_asymKAMtorus_components}
        a_{12}, \, a_{13}, \, a_{32} \in C^0(\T^n \times I_{T}), \qquad \sup_{t \in I_{T}}|a_{12}^t|_{C^0}, \, \sup_{t \in I_{T}}|a_{13}^t|_{C^0}, \, \sup_{t \in I_{T}}|a_{32}^t|_{C^0},< \infty.
    \end{equation}
    Then, for every fixed $(q,t_0) \in \T^n \times I_{T}$ and any initial condition $\xi(t_0) = (\xi_1(t_0),\xi_2(t_0),\xi_3(t_0),\xi_4(t_0))$ there exists a unique solution $\xi(t) = (\xi_1(t),\xi_2(t),\xi_3(t),\xi_4(t))$ and a positive constant $C$ depending on $|\xi(t_0)|$, $\sup_{t \in I_{T}}|a_{12}^t|_{C^0}, \, \sup_{t \in I_{T}}|a_{13}^t|_{C^0}, \, \sup_{t \in I_{T}}|a_{32}^t|_{C^0}$, $n$, $m$, $\Omega$ and $t_0$ such that 
    \begin{equation}\label{prop:div_xi_hyp}
        |\xi_1(t)| \le C e^{(\max_{1 \le j \le m}\Omega_j)t}, \quad |\xi_2(t)| \le C,  \quad |\xi_3(t)|, \,|\xi_4(t)| \le C e^{(\max_{1 \le j \le m}\Omega_j)t},
    \end{equation}
    for all $t \in I_{T}$.
\end{proposition}
\begin{proof}
    We have $\xi_2(t) = \xi_2(t_0)$ for all $t \in I_{T}$. Hence,~\eqref{prop:div_xi_hyp} is satisfied for $\xi_2(t)$.  In order to analyze the last equation of~\eqref{eq:Trasv_dyn_A_2_hyp}, we introduce the following notation
    \begin{align*}
        \zeta_\pm(t) &= \xi_3(t) \pm \xi_4(t),\\
        F_\pm^t(q +\omega t) &= \Pi_x a_{32}(q+ \omega t)\xi_2(t_0) \pm  \Pi_y a_{32}(q+ \omega t)\xi_2(t_0),
    \end{align*}
    where we refer to~\eqref{not:proj_Pixy} for the definition of $\Pi_x$ and $\Pi_y$. 
    We rewrite the last equation of~\eqref{eq:Trasv_dyn_A_2_hyp} as
    \begin{equation*}
        \partial_t \zeta_+(t) = -\Omega \zeta_{+}(t) + F^t_+(q + \omega t), \qquad \partial_t \zeta_-(t) = \Omega \zeta_{-}(t) + F^t_-(q + \omega t)
    \end{equation*}
    or in components, for all $1 \le j \le m$
    \begin{equation*}
        \partial_t \zeta_{+,j}(t) = -\Omega_j \zeta_{+,j}(t) + F^t_{+,j}(q + \omega t), \qquad \partial_t \zeta_{+,j}(t) = \Omega_j \zeta_{+,j}(t) + F^t_{+,j}(q + \omega t).
    \end{equation*}
    For every initial condition $\zeta_\pm(t_0) = \xi_3(t_0) \pm \xi_4(t_0)$ there exists a unique solution of the latter given by
    \begin{equation*}
        \zeta_{\pm,j}(t) = e^{\mp\Omega_j(t-t_0)} \zeta_{\pm,j}(t_0) + \int_{t_0}^t e^{\mp \Omega(t-s)} F^s_{\pm,j}(q +\omega s) ds. 
    \end{equation*}
    for all $1 \le j \le m$. By~\eqref{eq:Trasv_dyn_A_2_hyp}, we have that $\sup_{t \in I_{T}}|F^t_\pm|_{C^0} < \infty$. Since $\xi_3(t) = {\zeta_+(t)  + \zeta_-(t) \over 2}$ and $\xi_4(t) = {\zeta_+(t)  - \zeta_-(t)\over 2}$, trivial estimates shows that  $\xi_3(t)$ and $\xi_4(t)$ satisfy~\eqref{prop:div_xi_hyp}. Similarly to the final part of the proof of Proposition~\ref{prop:trasv_dyn_ell_case}, the above estimates imply that \eqref{prop:div_xi_hyp} also holds for $\xi_1(t)$. 
\end{proof}
\begin{remark}
    We recall that in Theorem \ref{Thm:hyp_Csigma} the unperturbed Hamiltonian $H_0$ in~\eqref{eq:initial_hamiltonian_hyp} has an invariant torus in $p=z=0$ supporting quasiperiodic solutions. If we analyze transverse trajectories in a neighborhood of the invariant torus corresponding to $p=0$, that correspond to considering initial conditions $\xi(t_0)$ such that $\xi_2(t_0) =0$, we obtain the classical transversal hyperbolic orbits
     \begin{equation*}
    \begin{aligned}
        \xi_3(t) &= \cosh \left(\Omega (t-t_0)\right)  \xi_3(t_0) -  \sinh \left(\Omega (t - t_0)\right)\xi_4(t_0),\\
        \xi_4(t) &= \cosh \left(\Omega (t-t_0)\right) \xi_4(t_0) -  \sinh \left(\Omega (t - t_0)\right)\xi_3(t_0),
    \end{aligned}
    \end{equation*}
    where $\cosh \big(\Omega (t-t_0)\big) = \mathrm{diag}\Big(\cosh \big(\Omega_1 (t-t_0)\big), \dots, \cosh \big(\Omega_m (t-t_0)\big)\Big)$ and $\sinh \big(\Omega (t-t_0)\big) = \mathrm{diag}\Big(\sinh \big(\Omega_1 (t-t_0)\big), \dots, \sinh \big(\Omega_m (t-t_0)\big)\Big)$.  
\end{remark}

In the second part of this section, we consider the following autonomous system for all $(q,t) \in \T^n \times I_T$
\begin{equation}
    \begin{cases}\label{eq:Trasv_dyn_A_2_hyp_syst_infty}
        \partial_t \xi^\infty_1 (t) = a^\infty_{12}(q+\omega t) \xi^\infty_2(t) +  a^\infty_{13}(q+\omega t) \begin{pmatrix} \xi^\infty_3(t) \\ \xi^\infty_4(t) \end{pmatrix},\\
        \partial_t \xi^\infty_2(t) = 0,\\
        \begin{pmatrix}\partial_t \xi^\infty_3 (t) \\ \partial_t \xi^\infty_4(t) \end{pmatrix} = a^\infty_{32}(q+\omega t) \xi^\infty_2(t) -  \begin{pmatrix}\Omega \xi^\infty_4(t) \\ \Omega\xi^\infty_3(t) \end{pmatrix},
    \end{cases}
\end{equation}
 with $a^\infty_{12}:\T^n \to \mathcal{M}_n(\R)$, $a^\infty_{13} :\T^n \to \mathcal{M}_{n\times 2m}(\R)$, and  $a^\infty_{32}:\T^n \to \mathcal{M}_{2m\times n}(\R)$.

The following proposition is the counterpart of Proposition \ref{prop:ell_1:1_asym_sol} in the hyperbolic setting. Unlike the elliptic case, the correspondence between solutions is not one-to-one.
\begin{proposition}\label{prop:hyp_1:1_asym_sol}
     We consider the systems~\eqref{eq:Trasv_dyn_A_2_hyp} and~\eqref{eq:Trasv_dyn_A_2_hyp_syst_infty} and we assume that 
    \begin{equation}\label{hyp:limit_aut_sys_hyp}
        \lim_{t \to +\infty} e^{-\Omega_jt} \int_T^t e^{\Omega_j s}|a_{32}^s - a_{32}^\infty|_{C^0}\,ds = 0,
    \end{equation}
    for all $1 \le j \le m$. 
   Then, for every fixed $q\in\T^n$, $t_0\in I_{T}$ and $\xi_2^0 \in \R^n$ there exists a correspondence between the sets $$\{ \pi_z\xi^\infty \mid \xi^\infty \text{ is a solution  of ~\eqref{eq:Trasv_dyn_A_2_hyp_syst_infty}} \text{ with } \xi_2^\infty(t_0) = \xi_2^0\},$$ $$\{ \pi_z\xi \mid \xi \text{ is a solution  of ~\eqref{eq:Trasv_dyn_A_2_hyp}} \text{ with } \xi_2(t_0) = \xi_2^0\},$$ by the relation 
       \begin{equation*}
        \lim_{t \to +\infty}|\pi_z\xi^\infty(t)-\pi_z\xi(t)|=0.
    \end{equation*}
\end{proposition}
\begin{proof}
For all $t \in  I_T$, we introduce the following notation    
\begin{align*}
        \zeta_\pm(t) &= \xi_3(t) \pm  \xi_4(t), \qquad \zeta_\pm^\infty(t) = \xi^\infty_3(t) \pm  \xi^\infty_4(t),\\
        F^{\infty}_\pm(q + \omega t) &= \Pi_x a^\infty_{32}(q+ \omega t) \xi_2(t_0) \pm  \Pi_y a^\infty_{32}(q+ \omega t)\xi_2(t_0),\\
        F^{t}_\pm(q + \omega t) &= \Pi_x a^t_{32}(q+ \omega t) \xi_2(t_0) \pm  \Pi_y a^t_{32}(q+ \omega t)\xi_2(t_0),\\
        R^t_\pm(q + \omega t) &= F^{t}_\pm(q + \omega t) -F^{\infty}_\pm(q + \omega t).
    \end{align*}
    In particular,~\eqref{hyp:limit_aut_sys_hyp} implies that 
    \begin{equation}\label{proof_hyp:limit_aut_sys_hyp}
        \lim_{t \to +\infty} e^{-\Omega_jt} \int_T^t e^{\Omega_j s}|R^s_{ \pm}|_{C^0}\,ds = 0. 
    \end{equation}
    With the above notation, and using that $\xi_2(t_0) = \xi_2^\infty(t_0)$, we can rewrite the last equation of systems~\eqref{eq:Trasv_dyn_A_2_hyp} and~\eqref{eq:Trasv_dyn_A_2_hyp_syst_infty} as
    \begin{align*}
        \partial_t \zeta_\pm(t) &=  \mp \Omega \zeta_\pm(t) + F_\pm^\infty(q +\omega t) + R^t_{\pm}(q + \omega t),\\
        \partial_t \zeta^\infty_\pm(t) &=  \mp \Omega \zeta^\infty_\pm(t) + F_\pm^\infty(q +\omega t). 
    \end{align*}
    The difference between the corresponding formal solutions of the above equations can be expressed componentwise as
    \begin{equation}\label{proof:diff_sol_hyp_asym}
        \zeta_{\pm,j} (t) - \zeta^\infty_{\pm,j}(t) = e^{\mp \Omega_j t} \left[e^{\pm \Omega_j t_0}\left(\zeta_{\pm, j}(t_0) - \zeta^\infty_{\pm, j}(t_0)\right) + \int_{t_0}^t e^{\pm\Omega_j s} R^s_{\pm, j}(q +\omega s) \, ds \right]
    \end{equation}
    for all $1 \le j \le m$. We now consider the two cases separately. First, we observe that, for all $1 \le j \le m$
    \begin{equation*}
        |\zeta_{+,j} (t) - \zeta^\infty_{+,j}(t)| \le e^{-\Omega_j (t-t_0)} |\zeta_{+,j} (t_0) - \zeta^\infty_{+,j}(t_0)| + e^{- \Omega_j t}\int_{t_0}^t e^{\Omega_j s} |R^s_{+}|_{C^0} \, ds
    \end{equation*}
    and hence, by~\eqref{proof_hyp:limit_aut_sys_hyp}, 
    \begin{equation}\label{proof_hyp_asy:conclu_1}
        \lim_{t \to +\infty}|\zeta_{+,j} (t) - \zeta^\infty_{+,j}(t)| = 0
    \end{equation}
    for each choice of the initial condition $\zeta_{+,j} (t_0)$ and $\zeta^\infty_{+,j}(t_0)$. It remains to consider the other case. We rewrite the difference~\eqref{proof:diff_sol_hyp_asym} as
    \begin{align*}
        \zeta_{-,j} (t) - \zeta^\infty_{-,j}(t) &= e^{\Omega_j t} \left[e^{- \Omega_j t_0}\left(\zeta_{-, j}(t_0) - \zeta^\infty_{-, j}(t_0)\right) + \int_{t_0}^{+\infty} e^{-\Omega_j s} R^s_{-, j}(q +\omega s) \, ds \right]\\
        &-e^{\Omega_j t}\int_{t}^{+\infty} e^{-\Omega_j s} R^s_{-, j}(q +\omega s) \, ds,
    \end{align*}
    for all $1 \le j \le m$. We point out that, by~\eqref{proof_hyp:limit_aut_sys_hyp}, one can prove that, for all $1 \le j \le m$
    \begin{align*}
        &\int_{t_0}^{+\infty} e^{-\Omega_j s} R^s_{-, j}(q + \omega s) \, ds \qquad \mbox{is well defined},\\
        &\lim_{t\to+\infty}\left|e^{\Omega_j t}\int_{t}^{+\infty} e^{-\Omega_j s} R^s_{-, j}(q +\omega s) \, ds\right| \le C e^{\Omega_j t}\int_{t}^{+\infty} e^{-\Omega_j s} \left|R^s_{-}\right| \, ds = 0.
    \end{align*}
    This implies that 
    \begin{equation}
    \begin{aligned}\label{proof_hyp_asy:conclu_2}
    &\lim_{t \to +\infty}|\zeta_{-,j}(t) - \zeta^\infty_{-,j}(t)|=0 \hspace{3mm} \mbox{if and only if} \\
    &  e^{- \Omega_j t_0}\left(\zeta_{-, j}(t_0) - \zeta^\infty_{-, j}(t_0)\right) + \int_{t_0}^{+\infty} e^{-\Omega_j s} R^s_{-, j}(q +\omega s) \, ds=0.
    \end{aligned}
\end{equation}
Noticing that $\xi_3(t)={\zeta_+(t) + \zeta_-(t) \over 2}$, $\xi_4(t)={\zeta_+(t) - \zeta_-(t) \over 2}$ and using~\eqref{proof_hyp_asy:conclu_1} and~\eqref{proof_hyp_asy:conclu_2}, we conclude the proof of this proposition. 
\end{proof}
\begin{remark}\label{rmk:cond_corr_hyp}
    We note that a sufficient condition for hypothesis~\eqref{hyp:limit_aut_sys_hyp} to hold is the existence of a continuous function $f:I_T \to \R_{\ge 0}$ such that 
    \begin{equation*}
    \lim_{t \to +\infty} f(t) = 0,\quad\mbox{and} \quad |a_{32}^t - a_{32}^\infty|_{C^0} \le f(t) \quad \mbox{for all $t \in I_T$}.
    \end{equation*}
\end{remark}

\section{Criteria for transverse dynamics}\label{sc:criteria_asym_dyn}

 In this section we provide sufficient conditions for a $C^\sigma$-asymptotic KAM torus (see Definition \ref{def:Csigma_KAM_torus}) to be asymptotically elliptic/hyperbolic, in the sense of Definition \ref{def:ell_hyp_par_trasv_dyn_asym_KAM}. To simplify the notation, in the following, given $\omega 
\in \R^n$ we will denote by $\Lo$ the directional derivative

\begin{equation}
    \label{def:LoO}
 \Lo  = \partial_t + \omega \cdot \partial_q = \partial_t + \sum_{i = 1}^n \omega_i\partial_{q_i}.
\end{equation}

We start by pointing out that %
\eqref{eq:conjugated_cocycles} in Definition \ref{def:ell_hyp_par_trasv_dyn_asym_KAM} is equivalent to
\begin{equation}
    \label{eq:conjugated_cocycles_infinitesimal}
    DX \circ \varphi^t (q)  S(q, t)=  S(q, t)A(q, t)+ \Lo S(q, t).
\end{equation}
Indeed, denoting by $\tilde \Phi^t_{t_0}(q)$ the RHS of Equation \eqref{eq:conjugated_cocycles}, namely, $$\tilde \Phi^t_{t_0}(q) = S(q + \omega t, t)\Phi_A(t; q, t_0)S(q + \omega t_0, t_0)^{-1},$$ a direct calculation shows that its derivative with respect to $t$ is given by
\begin{align*}
     \frac{d}{dt} \tilde \Phi^t_{t_0}(q) & = \big(S(q + \omega t, t)A(q + \omega t) + \Lo S(q + \omega t, t) \big)\Phi_A(t; q, t_0)S(q + \omega t_0, t_0)^{-1} \\
     & = \big(S(q + \omega t, t)A(q + \omega t)  + \Lo S(q + \omega t, t) \big)S(q + \omega t, t)^{-1} \tilde \Phi^t_{t_0}(q),
\end{align*}
and thus Equation \eqref{eq:conjugated_cocycles} is satisfied (that is, $ \tilde \Phi^t_{t_0}$ is the fundamental solution of the system $ \dot \eta(t) = DX^t \circ \varphi^t (q + \omega t) \eta(t)$) if and only if Equation \eqref{eq:conjugated_cocycles_infinitesimal} is satisfied.

Using Equation \eqref{eq:conjugated_cocycles_infinitesimal}, we can rephrase Definition \ref{def:ell_hyp_par_trasv_dyn_asym_KAM} as the existence of a $C^1$ map $S: \T^n \times [T'', +\infty) \to GL(2n + 2m, \R)$ such that the transformation
\begin{equation}
\label{eq:A_as_function_of_S}
A(q, t) = S(q, t)^{-1} \big( DX \circ \varphi^t (q) S(q, t) + \Lo S(q, t) \big)
\end{equation}
is of the form \eqref{def:A}, with $M$ equal to either $\Mell$ or $\Mhyp$ (corresponding to the elliptic and hyperbolic scenarios, respectively) and satisfying \eqref{def:cond1_ellhyppar_asymKAMtorus} and \eqref{def:trasv_dyn_limit}.

Applying the arguments above to Definition \ref{def:ell_hyp_par_trasv_dyn_asym_KAM}, together with the analysis of the transversal asymptotic dynamics done in Section \ref{sc:analysis_transversal}, easily yields the following criterion.

\begin{proposition}
\label{prop:general_asymp_characterization}
    Let $(X, X_0, \varphi_0)$ defined on $\T^n \times B \times I_T\subseteq \T^n \times \R^{n + 2m}$ satisfying \eqref{def:hyp_asym_KAM_tori}, for some $\omega \in \R^n$, and let $\varphi: \T^n \times [T', +\infty) \to \T^n \times \R^{n + 2m}$ be a $C^\sigma$ asymptotic KAM torus associated with $(X, X_0, \varphi_0)$, for some $T' \geq T$ and $\sigma \geq 1$.

    Assume that there exists a $C^1$ map $S: \T^n \times [T'', +\infty) \to GL(2n + 2m, \R)$,
    \[ S = \begin{pmatrix}
        S_{qq} & S_{qp} & S_{qz} \\
        S_{pq} & S_{pp} & S_{pz} \\
        S_{zq} & S_{zp} & S_{zz} 
    \end{pmatrix}, \qquad S_{zz}  = \begin{pmatrix}
             S_{zx} &         S_{zy}
    \end{pmatrix},\]for some $T'' \geq T'$, satisfying
    \begin{equation}
    \label{eq:decrease_zp}
        \lim_{t \to +\infty} |S_{zp}^t| = 0,
    \end{equation}
    such that the transformation $A:  \T^n \times [T'', +\infty) \to \mathcal{M}_{2n + 2m}(\R)$ given by \eqref{eq:A_as_function_of_S} is of the form \eqref{def:A}, for some $M \in  \mathcal{M}_{2m}(\R)$, and satisfies  \eqref{def:cond1_ellhyppar_asymKAMtorus}.

    Let $\Omega_1,\dots,\Omega_m \in \R_{>0}$, $\varrho = \max_{1 \leq i \leq m} \Omega_i$ and we denote $\Omega = \mathrm{diag}(\Omega_1,\dots, \Omega_m)$. Then, the following holds.
    
    \begin{enumerate}
        \item If $M = \Mell$ and 
        \[ \lim_{t \to +\infty} \max \left\{ |S^t_{zq}|t^{2},   |S^t_{zz} - \mathrm{Id}_{2m}|t \right\} = 0, \]
        then $\varphi$ is an asymptotically elliptic KAM torus.
        \item If $M = \Mell$,  %
        $(\omega, \Omega)$ satisfy \eqref{diop_ass_ell_aut} and
\[ \lim_{t \to +\infty} \max \left\{ |S^t_{zq}|t,  |S^t_{zz} - \mathrm{Id}_{2m}| \right\} = 0,\]
        then $\varphi$ is an asymptotically elliptic KAM torus.
        \item If $M = \Mhyp$ and %
        \[ \lim_{t \to +\infty} \max \left\{ |S^t_{zq}|e^{\varrho t},  |S^t_{zz} - \mathrm{Id}_{2m}|e^{\varrho t} \right\} = 0, \]
        then $\varphi$ is an asymptotically hyperbolic KAM torus.
    \end{enumerate}
\end{proposition}

\begin{proof}
    By the arguments explained at the beginning of this section, the assumptions on the maps $A$ and $S$ imply that \eqref{eq:conjugated_cocycles} in Definition \ref{def:ell_hyp_par_trasv_dyn_asym_KAM} is satisfied. Thus, it suffices to check that in each of the cases above Equation \eqref{def:trasv_dyn_limit} holds. 

    Notice that, by \eqref{eq:conjugated_cocycles}, we have
    \begin{align*}
    \big|\pi_z & \Phi(t; q,t_0) -  \pi_z \Phi_A(t; q, t_0)S(q + \omega t_0, t_0)^{-1}\big| = \big|\pi_z (S(q + \omega t, t) - I)\Phi_A(t; q, t_0)S(q + \omega t_0, t_0)^{-1}\big| \\
    & \leq \big|(S_{zq}^t(q + \omega t) \quad S_{zp}^t(q + \omega t) \quad S_{zz}^t(q + \omega t) - \mathrm{Id}_{2m})   \Phi_A(t; q, t_0) \big| \big| S(q + \omega t_0, t_0)^{-1} \big|.
    \end{align*}

    By Propositions \ref{prop:trasv_dyn_ell_case}, \ref{prop:bound_trasv_dyn_ell_non_aut}, and \ref{prop:trasv_dyn_hyp_case} in each scenario, we can bound the growth rate of the norm of each entry of $\Phi_A(t; q, t_0)$ as $t$ goes to infinity. Then, in each case, these estimates together with the equation above and the assumptions on the decay rate of the different components of $S$ show that Equation \eqref{def:trasv_dyn_limit} is satisfied. This finishes the proof of the proposition.
\end{proof}

Therefore, to establish the asymptotically elliptic (resp. hyperbolic) behaviour of a given asymptotic KAM torus it suffices to construct a family of invertible matrices $S$ for which $S_{zq}^t$, $S_{zp}^t$ and $S_{zz} - \mathrm{Id}_{2m}$ have an appropriate decay (as $t$ goes to infinity) and for which the associated transformation \eqref{eq:A_as_function_of_S} is of the form \eqref{def:A} with $M = \Mell$ (resp. $M = \Mhyp$).

 In the remainder of this section, we describe an explicit way to construct such matrices and, as a consequence of Proposition \ref{prop:general_asymp_characterization}, obtain different criteria concerning the asymptotic character of the transverse dynamics along an asymptotic KAM torus.  
 
 The starting point will be the following lemma, which shows that given an isotropic torus in $\mathcal{T} \subseteq \T^n \times \R^{n + 2m}$ that is a graph over $\T^n \times \{ 0 \}$ there exists a symplectomorphism $\Psi$ sending the `zero section' $\T^n \times \{0\}$ to $\mathcal{T}$. Moreover, as we shall see below, denoting by $\Psi = (\Psi_q, \Psi_p, \Psi_z)$ the projections of $\Psi$ to the $q, p$ and $z$ coordinates, we can prescribe $\partial_z \Psi_z \circ \varphi_0: \T^n \to \mathrm{Sp}(2n + 2m, \R, J)$, where $\varphi_0: \T^n  \to \T^n \times \R^{n + 2m}$ is the trivial embedding $\varphi(q) = (q, 0)$ and $\mathrm{Sp}(2n + 2m, \R, J)$ denotes the space of symplectic matrices associated with $J$ as in \eqref{eq:big_symplectic_matrix}, namely,
 \[ \mathrm{Sp}(2n + 2m, \R, J) = \{ S \in \mathcal{M}_{2n + 2m}(\R) \mid S^T J S = J\}.\]

 In the following, we denote by $\textup{Sp}(2m, \R)$ the space of symplectic matrices associated to $J_m$ and by $\mathfrak{sp}(2m, \R)$ its associated Lie algebra, namely, 
\[ \mathfrak{sp}(2m, \R) := \{ G \in \mathcal{M}_{2m}(\R) \mid J_m G + G^{\top}J_m = 0 \}. \]
Recall that $e^G \in \textup{Sp}(2m, \R)$, for any $G \in  \mathfrak{sp}(2m, \R)$.

\begin{lemma}
\label{lem:embedding_symplectomorphism}
Fix $r \geq 1$. Let $\varphi = (\mathrm{Id} + u, v, w): \T^n \to \T^n \times \R^n \times \R^{2m}$ and $G: \T^n \to \mathfrak{sp}(2m, \R)$ of class $C^{r + 1}$. Assume that $\mathcal{T} := \varphi(\T^n)$ is isotropic and $\mathrm{id} + u \in \mathrm{Diff}^{r + 1}(\T^n)$.
 
Then, there exists a $C^{r}$ symplectomorphism $\Psi: \T^n \times \R^{n + 2m} \to \T^n \times \R^{n + 2m}$, preserving the symplectic form $J$ given by \eqref{eq:big_symplectic_matrix},  such that
\[\Psi(q, 0) = \varphi(q), \qquad \text{ for any } q \in \T^n,\]
and  $D \Psi \circ \varphi_0 =\boldsymbol{S}(\varphi, G)$, where $\varphi_0: \T^n  \to \T^n \times \R^{n + 2m}$ is the trivial embedding $\varphi_0(q) = (q, 0)$ and $\boldsymbol{S}(\varphi, G): \T^n \to \mathrm{Sp}(2n + 2m, \R, J)$ is a $C^r$ map given by 
\begin{equation}
    \label{eq:S_formula_autonomous}
    \boldsymbol{S}(\varphi, G) :=   \left(
\begin{array}{ccc}
\mathrm{Id}_n +\partial_q  u & 0 & 0 \\
\partial_q v & (\mathrm{Id}_n +\partial_q  u )^{-\top} & -(\mathrm{Id}_n +\partial_q  u )^{-\top} \partial_q w^\top J_m e^G \\
\partial_q w & 0 & e^G 
\end{array}
\right).
\end{equation}

Furthermore, if $G$ is of class $C^r$ (with $u, v, w$ still of class $C^{r + 1}$), the map $\boldsymbol{S}(\varphi, G): \T^n \to \mathrm{Sp}(2n + 2m, \R, J)$ above is still well-defined map and of class $C^r$.
\end{lemma}
\begin{proof}
    Let $u, v, w, G$ of class $C^{r + 1}$ as in the statement. Define $U = Id + u$, $V = v \circ U^{-1},$ $W = w\circ U^{-1}$. We will construct the symplectomorphism $\Psi$ as the composition of the following (symplectic) maps
    \begin{align*}
    \Psi_1(q, p, z) & = (q, p - \partial_q W(q)^{\top} J_m \big(z + \tfrac{1}{2
}W(q)\big) , z + W(q)),\\
     \Psi_2(q, p, z) &= (q, p + \tilde V(q), z), \qquad \tilde V(q) = V(q) + \tfrac{1}{2}\partial_q W(q)^{\top} J_m W(q),\\
     \Psi_3(q, p, z) &= (U(q), \partial_qU^{-T}(q)p, z),\\
     \Psi_4(q, p, z) &= \left(q, p + \tfrac{1}{2}\big( z^\top E_i(q) z\big)_{i = 1}^m, e^{G(q)}z
     \right), \qquad E_i(q) = \int_0^1 e^{sG(q)^\top}J_m\partial_{q_i}G(q)e^{sG(q)}ds.
     \end{align*}
     
Indeed, set $\Psi := \Psi_1 \circ \Psi_2 \circ \Psi_3 \circ \Psi_4$ and, assuming that the transformations above are well-defined $C^r$ symplectic diffeomorphisms, let us check that $\Psi$ verifies the desired properties.

Notice that 
\begin{gather*}
\Psi_1^{-1}(\mathcal{T}) =\{(q, \tilde V(q), 0, 0) \mid q \in \T^n\}, \qquad  (\Psi_1 \circ \Psi_2)^{-1}(\mathcal{T}) = \T^n \times \{0\},\\
\Psi_3^{-1}(\T^n \times \{0\}) = \T^n \times \{0\} = \Psi_4^{-1}(\T^n \times \{0\}).
\end{gather*}
Hence
\[\Psi^{-1}(\mathcal{T}) = \T^n \times \{0 \}.\]

Moreover, since 
\[\pi_q \circ \Psi (q, p, z) = U(q)\qquad \text{ for any } (q, p, z) \in \T^n \times \R^n \times \R^{2m},\] where $\pi_q$ is the projection onto the $q$ components introduced in Section \ref{sc:results}, 
we have
\[\Psi(q, 0) = (U(q), v(q), w(q)) = \varphi(q), \qquad \text{ for any } q \in \T^n.\]

A direct calculation then shows that
\begin{equation}
    \label{eq:explicit_sympectomorphism}
    \Psi(q, p, z) = \Big(U(q),   v(q) + \partial_qU(q)^{-T}\left(p - \partial_q w(q)^\top J_m e^{G(q)}z +  \tfrac{1}{2}\big( z^\top E_i(q) z\big)_{i = 1}^m \right), e^{G(q)}z + w(q)\Big),
\end{equation} 
and \eqref{eq:S_formula_autonomous} readily follows by taking derivatives in the above equation and evaluating at $(q, 0, 0)$.

Therefore, to show the first part of the lemma, it suffices to show that the maps $\Psi_1, \Psi_2, \Psi_3, \Psi_4$ defined above are indeed well-defined $C^r$ symplectic diffeomorphisms.  

Notice that $\Psi_1$ is the time one map of the Hamiltonian $F_1(q, z) =  -J_mW(q) \cdot z.$ %
Indeed, the associated Hamiltonian system is given by
\[ \left\{ \begin{array}{l} q'(t) = 0, \\ p'(t) = -\partial_q W(q(t))^T J_mz(t), \\ z'(t) = W(q(t)), \end{array} \right.\]
which, for initial conditions $(q^*, p^*, z^*)$, has an explicit solution given by
\[ \left\{ \begin{array}{l} q(t) = q^*, \\ p(t) = p^* - \partial_q W(q^*)^T J_m\big(tz^* + \frac{t^2}{2}W(q^*)\big), \\ z(t) = z^* + tW(q^*). \end{array} \right.\]
In particular, $\Psi_1$ is a symplectic map.

Letting $\alpha = dq \wedge dp + dx \wedge dy =  \sum_{i=1}^n dq_i \wedge dp_i + \sum_{j=1}^m dx_j \wedge dy_j$ be the standard symplectic form on $\T^n \times \R^{n+2m}$,  since $\Psi_1^{-1}(\mathcal T) = \{(q, \tilde V(q), 0, 0) \mid q \in \T^n\}$ is an isotropic torus,
\[ \Psi_2^*\alpha - \alpha = \sum_{i = 1}^ndq_i \wedge d\tilde V_i(q) = 0,\]
which shows that $\Psi_2$ is symplectic. 

A straightforward computation shows that $\Psi_3$ is symplectic. 

Finally, $\Psi_4$ can be realized as the time one map of the Hamiltonian $F_4(q, z) = -\frac{1}{2} J_m G(q)z \cdot z .$ %
Indeed, notice that since $G$ takes values in $\mathfrak{sp}(2m, \R)$ the matrix $J_m G(q)$ is symmetric, for any $q \in \T^n$. Thus the associated Hamiltonian system is given by
\[ \left\{ \begin{array}{l} q'(t) = 0, \\ p_i'(t) = \frac{1}{2}z(t)^\top J_m \partial_{q_i} G(q(t))z(t), \qquad 1 \leq i \leq m, \\ z'(t) = G(q(t))z, \end{array} \right.\]
which, for initial conditions $(q^*, p^*, z^*)$, has an explicit solution given by
\[ \left\{ \begin{array}{l} q(t) = q^*, \\ p_i(t) = p^*_i + \frac{1}{2}{z^*}^\top \left(\int_0^te^{sG(q^*)^\top} J_m \partial_{q_i}G(q^*)e^{sG(q^*)}ds \right) z^*,  \qquad 1 \leq i \leq m, \\ z(t) = e^{tG(q^*)}z^*. \end{array} \right.\]
In particular, $\Psi_4$ is a symplectic map.

To prove the second part of the lemma, let $\varphi = (\mathrm{Id} + u, v, w)$ as before (of class $C^{r + 1}$) and let us assume that $G$ is of class $C^r$. It follows directly from \eqref{lem:embedding_symplectomorphism} that $\boldsymbol{S}(\varphi, G)$ is of class $C^r$, so it suffices to check that it takes values in $\mathrm{Sp}(2n + 2m, \R, J)$.

A direct calculation shows that 
\[\boldsymbol{S}(\varphi, G)^\top J  \boldsymbol{S}(\varphi, G) = \boldsymbol{S}(\varphi, 0_{m \times m})^\top J  \boldsymbol{S}(\varphi,  0_{m \times m}).\]

Since the RHS of the equation above is constant and equal to $J$ by the first part of the lemma, it follows that $\boldsymbol{S}(\varphi, G)$ takes values in  $\mathrm{Sp}(2n + 2m, \R, J)$.
\end{proof}

We note that, under fairly general conditions, the tori in the one-parameter family associated with an asymptotic KAM torus are isotropic.

\begin{proposition}
    Let $\varphi^t$ be a $C^\sigma$ asymptotic KAM torus associated with $(X^t, X_0^t, \varphi_0)$. If $\varphi_0$ is isotropic then $\varphi^t$ is isotropic for all $t \in I_{T^{'}}$ 
\end{proposition}
\begin{proof}
 Recall that we denote by $\alpha = dq \wedge dp + dx \wedge dy =  \sum_{i=1}^n dq_i \wedge dp_i + \sum_{j=1}^m dx_j \wedge dy_j$ the standard symplectic form on $\T^n \times \R^{n+2m}$ and by $\psi^{t_0+t}_{t_0, \omega} (q) = q +\omega t$ the flow associated with the constant vector field $\omega$, for all $(q, t_0) \in \T^n \times I_{T'}$ and $t \ge 0$. Let $\psi_{t_0, X}^{t}$ be the flow at time $t$ with initial time $t_0$ of the vector field $X^t$. Using that $\psi_{t_0, X}^{t_0+t}$ is a symplectomorphism for all fixed $t_0 \in I_{T^{'}}$ and $t \ge 0$, we have that
    \begin{equation*}
        \left(\psi_{t_0,X}^{t_0+t}\circ \varphi^{t_0}\right)^*\alpha = \left(\varphi^{t_0}\right)^*\alpha
    \end{equation*}
    where here $^*$ denotes the pullback. By the latter and~\eqref{eq:asymptotic_KAM_flow}, we obtain that 
    \begin{equation}\label{proof_isotrop:eq_1}
        \left(\varphi^{t_0}\right)^*\alpha = \left(\psi^{t_0+t}_{t_0, \omega}\right)^*\left(\varphi^{t_0+t}\right)^*\alpha
    \end{equation}
    for all $t_0 \in I_{T^{'}}$ and $t \ge 0$. We will verify that $\left(\left(\varphi^{t_0}\right)^*\alpha\right)_q=0$ for all $q \in \T^n$ and $t_0 \in I_{T^{'}}$, where $\left(\left(\varphi^{t_0}\right)^*\alpha\right)_q$ stands for the symplectic form evaluated at $q\in\T^n$. To this end, we will prove that the RHS of~\eqref{proof_isotrop:eq_1} converges to zero when $t \to +\infty$. 

    We introduce the following notation, given the normed spaces $\left(X, |\cdot|_X\right)$, $\left(Y, |\cdot|_Y\right)$, and $\left(Z, |\cdot|_Z\right)$, and a continuous bilinear map $B:X \times Y \to Z$, we consider the following norm
    \begin{equation*}
        \|B\| = \sup_{|x|_X \le 1, \, |y|_Y \le 1} |B(x,y)|_Z.
    \end{equation*}
    Using that $\varphi_0$ is isotropic,  for all $q \in \T^n$ and $v_1, \, v_2 \in \R^n$, we have that
    
    \begin{align*}
        \left(\left(\varphi_0\right)^*\alpha\right)_q (v_1, v_2) = 0.
    \end{align*}
    Thanks to the latter and~\eqref{proof_isotrop:eq_1}, for all $(q, t_0) \in \T^n \times I_{T^{'}}$, $t \ge 0$ and $v_1, \, v_2 \in \R^n$, we obtain that 
    
     \begin{align*}
        \left|\left(\left(\varphi^{t_0}\right)^*\alpha\right)_q (v_1, v_2)\right| &= \left|\left(\left(\psi^{t_0+t}_{t_0, \omega}\right)^*\left(\varphi^{t_0+t}\right)^*\alpha\right)_q (v_1, v_2)\right|\\
        & = \left|\left(\left(\psi^{t_0+t}_{t_0, \omega}\right)^*\left(\varphi^{t_0+t}\right)^*\alpha\right)_q (v_1, v_2) - \left(\left(\varphi_0\right)^*\alpha\right)_q (v_1, v_2)\right|\\
        & \le \|dq\wedge dp + dx \wedge dy\| \left(|\varphi^{t_0 + t} - \varphi_0|_{C^1} + |\varphi^{t_0 + t} - \varphi_0|_{C^1}^2\right)|v_1| |v_2|.
    \end{align*}
    Using~\eqref{def:cond1_asymKAMtorus}, this concludes the proof of this proposition. 
\end{proof}

In the following, let us denote $\Pi_{xy} = \begin{pmatrix} 0_{2n\times 2m} \\ \mathrm{Id}_{2m}\end{pmatrix}.$

\subsection{A criterion for invariant KAM tori of autonomous Hamiltonians} 
Lemma \ref{lem:embedding_symplectomorphism} yields the following ellipticity/hyperbolicity criterion. 

\begin{proposition}
\label{prop:criteria_autonomous}
Fix $n, m \geq 1$, $\sigma \geq 2$ and $\omega \in \R^n$. Let $H: \T^n \times B \subseteq \T^n \times \R^{n + 2m}  \to \R$ be a Hamiltonian of class $C^\sigma$, where $B$ is an open ball centered at the origin. Assume that $\varphi = (\mathrm{id} + u, v, w): \T^n \to \T^n \times \R^n \times \R^{2m}$ parametrizes a $C^2$ invariant torus for $H$ with frequency $\omega$, that is, satisfies $X_H \circ \varphi = \partial_\omega \varphi$. 

Fix $G: \T^n \to \mathfrak{sp}(2, \R)$ of class $C^2$ and let $\Psi = \boldsymbol{\Psi}(\varphi, G)$ and $S = \boldsymbol{S}(\varphi, G)$ given by Lemma \ref{lem:embedding_symplectomorphism}. Then $H \circ \Psi$ is well defined on an open neighbourhood of $\T^n \times \{0\}$ and is of the form
\begin{equation}
    \label{eq:normal_form_hamiltonian_inv_torus}
    H \circ \Psi = cst. + \omega \cdot p  + \frac{1}{2} \zeta(q) \cdot z^2 + O(p^2,pz,z^3),
\end{equation} 
where $\zeta : \T^n \to \mathrm{Sym}(2m, \R)$ is given by
\begin{equation}
    \label{eq:formula_zeta_autonomous}
\zeta = \boldsymbol{\zeta}(H, \varphi, G) :=  \Pi_{xy}^{\top} S^{\top}[((D^2 H)\circ \varphi) S +J \partial_\omega S] \Pi_{xy}.
\end{equation}

In particular, given $\Omega_1,\dots,\Omega_m \in \R_{>0}$,  and denoting $\Omega = \mathrm{diag}(\Omega_1,\dots, \Omega_m)$,  the following holds.

\begin{enumerate}
    \item   If $\zeta = \Mell$, then $\varphi$ defines an elliptic invariant torus for $H$ with normal frequencies $\Omega_1, \dots, \Omega_m$.
    \item   If $\zeta = \Mhyp$, then $\varphi$ defines a hyperbolic invariant torus for $H$ with normal frequencies $\Omega_1, \dots, \Omega_m$.
\end{enumerate}

\end{proposition}

\begin{proof}[Proof of Proposition \ref{prop:criteria_autonomous}]
    Fix  $G: \T^n \to \mathfrak{sp}(2, \R)$ and $\Psi = \boldsymbol{\Psi}(\varphi, G)$ and $S = \boldsymbol{S}(\varphi, G)$ as in Lemma \ref{lem:embedding_symplectomorphism}. Recall that $X_{H \circ \Psi} = (D\Psi) ^{-1}(X_H \circ \Psi)$. Then
    \begin{equation*}
    \label{eq:vector_field_inv_torus}
    X_{H \circ \Psi}(q, 0) = D\Psi(q, 0)^{-1} X_H \circ \varphi =   D\Psi(q, 0)^{-1} \partial_\omega \varphi = \begin{pmatrix}
        \omega & 0 & 0
    \end{pmatrix}^\top.
    \end{equation*}
     In particular,
        \begin{equation}
    \label{eq:normal_form_inv_torus}
   \nabla (H \circ \Psi)(q, 0) = \begin{pmatrix}
        0 & \omega & 0
    \end{pmatrix}^\top.
    \end{equation}
    
     A direct calculation shows that
\begin{equation}
    \label{eq:derivative_conjugated_field}
    DX_{H \circ \Psi}(q, 0) = D\Psi^{-1}(q, 0)[(DX_{H}\circ \Psi(q, 0)) D\Psi(q, 0) - \partial_\omega D\Psi(q, 0)].
\end{equation}

By taking derivatives with respect to $q$ in $X_H \circ \varphi = \partial_\omega \varphi$, it follows that the first column of $(DX_{H}\circ \Psi(q, 0)) D\Psi(q,0) - \partial_\omega D\Psi(q, 0)$ (and hence of $DX_{H \circ \Psi}(q, 0)$) has all entries equal to zero. 

In particular, since $D^2(H \circ\Psi) = -J DX_{H \circ \Psi}$ is symmetric, it follows that $D^2 (H \circ \Psi)(q, 0)$ is of the form
\begin{equation}
\label{eq:normal_form_derivative_inv_torus}
D^2 (H \circ \Psi)(q, 0)  = \left(
\begin{array}{cccc}
0 & 0 & 0 & 0 \\
0 & * & * & * \\
0 & * & \multicolumn{2}{c}{\raisebox{-3pt}[1.2em][0pt]{$\zeta(q)$}} \\
0 & * &  &
\end{array}
\right),
\end{equation}
for some $\zeta: \T^n \to \textup{Sym}(2m, \R)$ smooth. 
Notice that $\zeta$ is given by,

\begin{align*}
\zeta(q) &= -\Pi_{xy}^{\top} J D\Psi^{-1}(q, 0)[(DX_{H}\circ \Psi(q, 0)) D\Psi(q, 0) - \partial_\omega D\Psi(q, 0)] \Pi_{xy}\\
& =\Pi_{xy}^{\top} D\Psi^{T}(q, 0)[((D^2 H)\circ \Psi(q, 0)) D\Psi(q, 0) +J \partial_\omega D\Psi(q, 0)] \Pi_{xy} \\
& = \Pi_{xy}^{\top} S^{\top}(q)[((D^2 H)\circ \varphi(q)) S(q) +J \partial_\omega S(q)] \Pi_{xy}.
\end{align*}

By \eqref{eq:normal_form_inv_torus} and \eqref{eq:normal_form_derivative_inv_torus}, it follows that $H \circ \Psi$ is of the form \eqref{eq:normal_form_hamiltonian_inv_torus}.
\end{proof}

\begin{remark}
    The map $\zeta = \boldsymbol{\zeta}(H, \varphi, G): \T^n \to \mathrm{Sym}(2m, \R)$ given by \eqref{eq:formula_zeta_autonomous} can be expressed explicitly as 
    \[\zeta = B^\top \partial_p^2 H\circ \varphi B + K^{\top} \partial_p\partial_zH\circ \varphi B + B^{\top} (\partial_p\partial_zH\circ \varphi)^\top K + K^{\top} \partial_z^2 H\circ \varphi K + K^{\top} J_m \partial_\omega K,\]
where $K = \exp(G)$,  $B = -(\mathrm{Id}_{n} +\partial_q  u )^{-T} \partial_q w^\top J_m K$ and $\partial_p\partial_zH\circ \varphi$ stands for the $2m \times n$ matrix having components $\partial_{p_i}\partial_{z_j}H\circ \varphi$ for $1 \le i \le n$ and $1 \le j \le 2m$.
\end{remark}

\subsection{A criterion for asymptotic KAM tori of non-autonomous Hamiltonians}

 In this section, we provide a series a criteria to describe the transverse dynamics of the linearized system associated with an asymptotic KAM torus.

Applying Lemma \ref{lem:embedding_symplectomorphism} to families of isotropic torus embeddings immediately yields the following.

\begin{lemma}
\label{lem:embedding_symplectomorphisms_family}
Fix $\sigma \geq 2$ and $T > 0$. Let $\varphi = (\mathrm{Id} + u, v, w): \T^n \times I_T\to \T^n \times \R^n \times \R^{2m}$ and $G: \T^n \times I_T \to \mathfrak{sp}(2m, \R)$ such that $u, v, w \in \mathcal{S}^{\mathrm{pol}, T}_{(\sigma, 0), 0}$ and $G \in  \mathcal{S}^{\mathrm{pol}, T}_{(\sigma - 1, 0), 0}$. Assume that $\varphi^t(\T^n)$ is an isotropic torus and $\mathrm{id} + u^t \in \mathrm{Diff}^\sigma(\T^n)$, for any $t \in I_T$.

Then 
\begin{equation}
\label{eq:S_formula}
\boldsymbol{S}(\varphi, G)  = \left(
\begin{array}{ccc}
\mathrm{Id}_n +\partial_q  u & 0 & 0 \\
\partial_q v & (\mathrm{Id}_n +\partial_q  u )^{-\top} & -(\mathrm{Id}_n +\partial_q  u )^{-\top} \partial_q w^\top J_m e^G\\
\partial_q w & 0 & e^G
\end{array}
\right),
\end{equation}
defines a map $\boldsymbol{S}(\varphi, G) : \T^n \times I_T \to \mathrm{Sp}(2n + 2m,\R, J)$ of class $C^{\sigma - 1}$ on the first variable and continuous on the second variable.

Moreover, if $\partial_q u, \partial_q v, \partial_q w$ and $G$ are of class $C^1$, then $\boldsymbol{S}(\varphi, G)$ is of class $C^1$.

\end{lemma}

\begin{proof}
    This follows directly from the second part of Lemma \ref{lem:embedding_symplectomorphism}, from Equation \eqref{eq:explicit_sympectomorphism}, and by the regularity assumptions on $u, v, w, G$.

\end{proof}

Applying the lemma above to asymptotic KAM tori, we can prove the following.

\begin{lemma}
\label{lem:matrix_A}
Fix $n, m \geq 1$, $T > 0$, $\sigma \geq 2$, $\omega \in \R^n$.  Let $H_0: \T^n \times B \times I_T  \to \R$ given by
\[ H_0(q, p, z, t) = \omega \cdot p +O^2(p, z), \qquad \partial_{(p,z)}^2H_0 \in \mathcal{S}^{\mathrm{pol}, T}_{(\sigma, 2), 0},\]
where $B \subseteq \R^{n + 2m}$ is an open ball centered at the origin.

Let $H: \T^n \times B \times I_T  \to \R$ such that $H^t$ is of class $C^\sigma$, for each $t \in I_T$, and 
\[ \lim_{t \to +\infty} |X_{H}^t - X_{H_0}^t|_{C^1} = 0. \]

Assume that $\varphi = (\mathrm{id} + u, v, w): \T^n \times I_T \to \T^n \times \R^n \times \R^{2m}$ is a $C^\sigma$ asymptotic KAM torus associated with $(X_H, X_{H_0}, \varphi_0)$, satisfying $u, v, w \in \mathcal{S}^{\mathrm{pol}, T}_{(\sigma, 0), 0}$.

Fix $G: \T^n \times I_T \to \mathfrak{sp}(2, \R)$ in $\mathcal{S}^{\mathrm{pol}, T}_{(\sigma - 1, 0), 0}$. Let $S = \boldsymbol{S}(\varphi, G): \T^n \times I_T \to \mathrm{Sp}(2n + 2m, \R)$ as in \eqref{eq:S_formula} and denote by $S^{-1}$ the map $(q, t) \mapsto S(q, t)^{-1}$.

 Assume that $\Lo S$ is well-defined. Then, the map $A = \boldsymbol{A}(H, \varphi, G): \T^n \times I_T \to \mathfrak{sp}(2n + 2m, \R)$ given by %
\begin{equation}
\label{eq:def_A_Hamiltonian}
A = S^{-1}\big( (DX_H \circ \varphi )S - \Lo S),
\end{equation}
is a well-defined uniformly bounded family of matrices of the form
     \begin{equation}
     \label{eq:form_A}
        A =\begin{pmatrix} 0_{n \times n} & * & * \\
                                0_{n \times n} & 0_{n \times n} & 0_{n \times 2m} \\
                                0_{2m \times n} & * & *\end{pmatrix}.
    \end{equation}

\end{lemma}

\begin{proof}
  Fix  $G: \T^n \times I_T \to \mathfrak{sp}(2, \R)$ and let $\Psi = \boldsymbol{\Psi}(\varphi, G)$, $S = \boldsymbol{S}(\varphi, G)$ as in Lemma \ref{lem:embedding_symplectomorphisms_family}. 

Let us start by checking that $A$ defines a family of matrices in $\mathfrak{sp}(2n + 2m, \R)$. Notice that this is equivalent to $JA$ taking values in $\mathrm{Sym}(2m + 2n, \R)$.
  
  Recalling that $J^\top = - J$ and that $S$ is a family of symplectic matrices (and thus satisfies $S^\top J S = J$), we have
  \begin{align*}
  JA & = JS^{-1}((DX_H \circ  \varphi) S - \Lo S) = S^T J(J(D^2H \circ \varphi) S - \Lo S) \\
  & = -S^T((D^2H \circ \varphi) S - J\Lo S).
  \end{align*}  
Hence, it suffices to check that $S^\top J\Lo S$ is a symmetric matrix. Applying $\Lo$ to $S S^{-1} = \mathrm{Id}_{2m + 2n}$ (notice that $\Lo S^{-1}$ is well-defined as $S^{-1} = - J S^{\top}J$), it easily follows that 
\[ (\Lo S)  S^{-1} = -S (\Lo S^{-1}).\]
Hence, 
\begin{align*}
 (S^{\top} J_m \Lo S)^{\top} & = - (\Lo S^{\top})J_m S = - J_m(\Lo S^{-1}) S = J_m S^{-1} \Lo S = S^{\top} J_m \Lo S,
\end{align*} 
and thus $JA$ is symmetric.

  Let us check now that $A$ is indeed of the form \eqref{eq:form_A}. Notice that since $JA$ is symmetric, it suffices to check that the first $n$ columns of $A$ have all entries equal to $0$. 
  
  Recall that $\varphi$ satisfies 
  \[X_H \circ \varphi = \Lo  \varphi.\]
  
  Since $\Psi(\cdot, 0) = \varphi(\cdot)$ and $S(\cdot) = D\Psi(\cdot, 0)$, by taking derivatives in the above equation, it follows that  the first $n$ columns of $(DX_{H}\circ \varphi) S - \Lo S$, (and thus of $A$, by \eqref{eq:def_A_Hamiltonian}), have all entries equal to zero. Therefore, $A$ is of the form \eqref{eq:form_A}.

Finally, noticing that $DX_H \circ \varphi$, $S$ and $S^{-1}$ are uniformly bounded, it follows that $A$ is also uniformly bounded.

\end{proof}

Similarly to the previous section, applying Proposition \ref{prop:general_asymp_characterization} and Lemmas \ref{lem:embedding_symplectomorphisms_family}, \ref{lem:matrix_A}, we can deduce the following criterion.

\begin{proposition}
\label{prop:criteria_non_autonomous}
Let $H_0, H : \T^n \times B \times I_T \to \R$ and  $\varphi = (\mathrm{Id} + u, v, w): \T^n \times I_T \to \T^n \times \R^{n} \times \R^{2m}$ as in Lemma \ref{lem:matrix_A}, with $\sigma \geq 2$. Fix $T' > T$ and $G: \T^n \times I_T \to \mathfrak{sp}(2, \R)$ in $\mathcal{S}^{\mathrm{pol}, T'}_{(\sigma - 1, 0), 0}$. Let $S = \boldsymbol{S}(\varphi, G) :  \T^n \times I_{T'} \to \mathrm{Sp}(2n + 2m, \R, J)$ as in \eqref{eq:S_formula}, assume that $S$ is of class $C^1$,  and let $\zeta = \boldsymbol{\zeta}(H, \varphi, G): \T^n \times I_{T'} \to \mathrm{Sym}(2m, \R)$ be given by
\begin{equation}
    \label{eq:formula_zeta}
    \boldsymbol{\zeta}(H, \varphi, G) = \Pi_{xy}^\top S^\top\big( (D^2H \circ \varphi)  S + J \Lo S\big) \Pi_{xy}.
\end{equation}
    Let $\Omega_1,\dots,\Omega_m \in \R_{>0}$, $\ell > 0$, $\lambda > \max_{1 \leq i \leq m} \Omega_i,$ and we denote $\Omega = \mathrm{diag}(\Omega_1,\dots, \Omega_m)$. Then, the following holds.
    
  \begin{enumerate}
        \item If $\zeta = \Mell$ and 
        \[  w \in \mathcal{S}^{\mathrm{pol}, T}_{(\sigma, 0), \ell + 1}, \qquad G \in \mathcal{S}^{\mathrm{pol}, T'}_{(\sigma - 1, 0), \ell},\]
        then $\varphi$ is an asymptotically elliptic KAM torus.
        \item If $\zeta = \Mell$, $(\omega, \Omega)$ satisfy \eqref{diop_ass_ell_aut} and
        \[  w \in \mathcal{S}^{\mathrm{pol}, T}_{(\sigma, 0), \ell}, \qquad G \in \mathcal{S}^{\mathrm{pol}, T'}_{(\sigma - 1, 0), \ell - 1},\]
        then $\varphi$ is an asymptotically elliptic KAM torus.
        \item If $\zeta = \Mhyp$ and 
                \[   w \in \mathcal{S}^{\mathrm{exp}, T}_{(\sigma, 0), \lambda}, \qquad G \in \mathcal{S}^{\mathrm{exp}, T'}_{(\sigma - 1, 0), \lambda},\]
        then $\varphi$ is an asymptotically hyperbolic KAM torus.
    \end{enumerate}
\end{proposition}

\begin{proof}
Let  $A = \boldsymbol{A}(H, \varphi, G): \T^n \times I_{T'} \to \mathfrak{sp}(2n + 2m, \R)$ as in Lemma \ref{lem:matrix_A}, that is,  given by \eqref{eq:def_A_Hamiltonian}. We will prove the result by checking that, in each case, the matrices $A$ and $S$ fulfill the assumptions of Proposition \ref{prop:general_asymp_characterization}.

First, let us point out that, by definition of $A$, Equation \eqref{eq:A_as_function_of_S} (with $X = X_H$) is satisfied and that, by \eqref{eq:S_formula}, Equation \eqref{eq:decrease_zp} holds trivially.

 Notice that $\zeta$ is given by $-\Pi_{xy}^{\top} J A \Pi_{xy}$ and, since $JA$ takes values in the symmetric matrices, it is a well-defined family of symmetric matrices.  Indeed, we have
\begin{align*}
\Pi_{xy}^{\top} J A \Pi_{xy} & = \Pi_{xy}^{\top} J S^{-1}[((DX_{H})\circ \varphi) S -\Lo S] \Pi_{xy}\\
 &= \Pi_{xy}^{\top} S^\top J[((J D^2 H)\circ \varphi) S -\Lo S] \Pi_{xy}\\
& = \Pi_{xy}^{\top} S^\top[((D^2 H)\circ \varphi) S +J\Lo S] \Pi_{xy}\\
& = -\zeta.
\end{align*}

Thus, by Lemma \ref{lem:matrix_A}, it follows that $A$  satisfies  \eqref{def:cond1_ellhyppar_asymKAMtorus} and is of the form
\[                              A =\begin{pmatrix} 0_{n \times n} & * & * \\
                                0_{n \times n} & 0_{n \times n} & 0_{n \times 2m} \\
                                0_{2m \times n} & * & J_m\zeta\end{pmatrix}.
\]

Therefore, to prove the result, it suffices to show that in each case the assumptions on $w$ and $G$ imply the decay rates in the corresponding case of Proposition \ref{prop:general_asymp_characterization}. 

Noticing that, by \eqref{eq:S_formula}, $S_{zq} = \partial_q w$ and $S_{zz} = e^G$, the decay rates on $S_{zq}$ follow immediately from the assumptions on $w$, while the decay rates on $S_{zz} - \mathrm{Id}_{2m} = e^G - \mathrm{Id}_{2m}$ follow from the assumptions on $G$ together with Lemmas \ref{lem:exp_expansion_bounds}, and \ref{lem:exp_exp_bounds}, which will be proven in later sections.
\end{proof}

\begin{remark}
\label{rmk:formula_zeta}
    The map $\zeta = \boldsymbol{\zeta}(H, \varphi, G): \T^n \to \mathrm{Sym}(2m, \R)$ given by \eqref{eq:formula_zeta} can be expressed explicitly as 
    \[\zeta = B^\top \partial_p^2 H\circ \varphi B + K^{\top} \partial_p\partial_zH\circ \varphi B + B^{\top} (\partial_p\partial_zH\circ \varphi)^\top K + K^{\top} \partial_z^2 H\circ \varphi K + K^{\top} J_m \Lo K,\]
where $K = \exp(G)$,  $B = -(\mathrm{Id}_{n} +\partial_q  u )^{-T} \partial_q w^\top J_m K$ and $\partial_p\partial_zH\circ \varphi$ stands for the $2m \times n$ matrix having components $\partial_{p_i}\partial_{z_j}H\circ \varphi$ for $1 \le i \le n$ and $1 \le j \le 2m$.
\end{remark}

\section{Higher order terms of the unperturbed Hamiltonian}\label{sec:FS_Ref_Ham}

In this section, we obtain bounds (on the norms in appropriate Banach spaces) for the higher-order terms (denoted by $R$) in the Taylor expansion of the unperturbed Hamiltonian $H_0$ in the different  Theorems stated in Section \ref{sc:results}. %

For this purpose, we can express  $H_0 : \T^n \times B \times I_0 \longrightarrow \R$ in~\eqref{eq:initial_hamiltonian}, and~\eqref{eq:initial_hamiltonian_hyp}, as 
\begin{equation}
\label{eq:initial_Hamiltonian_NR}
        H_0(q,p,z,t) = N(p,z) + R(q,p,z,t), 
\end{equation}
where 
\begin{equation*}
\begin{array}{rcl}   N(p,z) &= & \omega \cdot p + {1 \over 2} M^* \cdot z^2, \qquad  M^*=\Mell,\ \text{or }\Mhyp\\
    R(q,p,z,t) &= &M(q,p,z,t) \cdot p^2 + m(q,z,t) \cdot (p,z) + L(q,z,t) \cdot z^3,
\end{array}
\end{equation*}
 and
\begin{equation}
\label{def:abcdMmL_terms}
\begin{aligned}
    M(q,p,z,t) &= \int_0^1 (1 -\tau) \partial_p^2 H_0(q, \tau p, z,t) d\tau ,\\
    m(q,z,t) &= \int_0^1 \partial_p\partial_z H_0(q, 0, \tau z,t) d\tau,\\
    L(q,z,t) &=  {1 \over 2} \int_0^1 (1-\tau)^2 \partial_z^3 H_0(q,0,\tau z, t) d\tau. 
\end{aligned}
\end{equation}
Notice that the RHS in each of the equalities in \eqref{def:abcdMmL_terms} remains the same if we change $H_0$ by $R$.

Let us recall that $M \cdot z^2$ denotes the symmetric bilinear form $M$ evaluated twice on the vector $z$. The same convention applies to the terms  $M(q,p,z,t) \cdot p^2$ and $m(q,z,t) \cdot (p,z)$. In particular, $m(q,z,t) \cdot (p,z)$ stands for the vectors $p$ and $z$ given as arguments of the bilinear form $m(q,z,t)$. Finally, $L(q,z,t) \cdot z^3$ means the symmetric trilinear form $L(q,z,t)$ evaluated three times on $z$.

\begin{lemma}\label{lemma:def_bar_M_m_L}
    For all $(q,p,z,t) \in \T^n \times B \times I_0$, we define
    \begin{equation}
    \begin{aligned}\label{def:barMmL}
       \overline{M}(q,p,z,t) &= \int_0^1 \partial^2_p H_0(q, \tau p , z, t) d \tau, \hspace{10mm} \overline{\overline{M}}(q,p,z,t) = \partial_p^2 H_0(q,p,z,t),\\
       \overline{m}(q,z,t) &=  \partial_p\partial_z H_0(q, 0 , z, t), \hspace{18mm} \overline{L}(q,z,t) = \int_0^1 (1 -\tau)\partial_z^3 H_0(q, 0 , \tau z, t) d \tau\\
       \overline{\overline{{L}}}(q, z, t)  &=  \int_0^1 \partial_z^3 H_0(q, 0 , \tau z, t) d \tau.
       \end{aligned}
    \end{equation}
    Then,
    \begin{equation}\label{eq:barMmL}
        \begin{aligned}
            &\overline{M}(q,p,z,t)\cdot p = \partial_p\left(M(q,p,z,t) \cdot p^2\right), \hspace{10mm} \overline{\overline{M}}(q,p,z,t) = \partial_p \left(\overline{M}(q,p,z,t)\cdot p\right)\\
            &\overline{m}(q,z,t)\cdot p = \partial_z\left(m(q,z,t) \cdot (p,z)\right),\hspace{15mm} \overline{L}(q,z,t)\cdot z^2 = \partial_z\left(L(q,z,t) \cdot z^3\right),\\
            &\overline{\overline{{L}}}(q, z, t) \cdot z = \partial_z\left(\overline{L}(q,z,t)\cdot z^2\right)
        \end{aligned}
    \end{equation}
    for all $(q,p,z,t) \in \T^n \times B \times I_{0}$. Moreover, let $\sigma \ge 0$, $k \in \mathbb{Z}_{\ge 1}$, and $* \in \{\mathrm{pol}, \mathrm{exp}\}$. If $\partial_{(p,z)}^2 R \in \mathcal{S}_{(\sigma,k),0}^{*,T}$, where $T \ge 1$ when $*=\mathrm{pol}$ and $T \ge 0$ when $*=\mathrm{exp}$, then $\overline{M} ,\ \overline{\overline{M}} ,\ \overline{m} \in \mathcal{S}_{(\sigma,k),0}^{*,T}$, and $\overline{L} ,\ \overline{\overline{L}} \in \mathcal{S}_{(\sigma,k-1),0}^{*,T}$ with 
    \begin{equation}\label{est:MmLbar_Holder}
        |\overline{M}|^{*, T}_{\sigma+k, 0} ,\ |\overline{\overline{M}}|^{*, T}_{\sigma+k, 0} ,\ |\overline{m}|^{*, T}_{\sigma+k, 0} ,\ |\overline{L}|^{*, T}_{\sigma+k-1, 0} ,\ |\overline{\overline{L}}|^{*, T}_{\sigma+k-1, 0} \le |\partial_{(p,z)}^2 H_0|^{*, T}_{\sigma+k, 0}.
    \end{equation}
    In addition, letting $0<\sigma'<\sigma$, if $\partial_{(p,z)}^2 R \in \mathscr{S}_{\sigma,0}^{*,T}$ then $\overline{M} ,\ \overline{\overline{M}} ,\ \overline{m} \in \mathscr{S}_{\sigma,0}^{*,T}$, and $\overline{L} ,\ \overline{\overline{L}} \in \mathscr{S}_{\sigma',0}^{*,T}$ with
     \begin{equation}\label{est:MmLbar_analy}
        |\overline{M}|^{*, T}_{\sigma, 0} ,\ |\overline{\overline{M}}|^{*, T}_{\sigma, 0} ,\ |\overline{m}|^{*, T}_{\sigma, 0} \le |\partial_{(p,z)}^2 H_0|^{*, T}_{\sigma, 0}, \qquad 
        |\overline{L}|^{*, T}_{\sigma', 0} ,\ |\overline{\overline{L}}|^{*, T}_{\sigma', 0} \le {1 \over(\sigma-\sigma')}|\partial_{(p,z)}^2 H_0|^{*, T}_{\sigma, 0}.
    \end{equation}
\end{lemma}
\begin{proof}
    We prove the first equality in~\eqref{eq:barMmL}. Expanding the Hamiltonian $H_0$ in~\eqref{eq:initial_Hamiltonian_NR} and its derivative $\partial_p H_0  $ in a neighborhood of $p=0$, we have that
    \begin{align}\label{proof:lemma_bar_H}
        &H_0(q,p,z,t) = H_0(q,0,z,t) + \partial_p H_0(q,0,z,t)p + \int_0^1 (1-\tau)\partial_p^2 H_0(q, \tau p, z, t) d\tau \cdot p^2,\\ \label{proof:lemma_bar_partialH}
        &\partial_p H_0(q,p,z,t) = \partial_p H_0(q,0,z,t) + \int_0^1 \partial^2_p H_0(q, \tau p, z, t) d\tau \cdot p.
    \end{align}
    Differentiating~\eqref{proof:lemma_bar_H} with respect to $p$ and combining it with~\eqref{proof:lemma_bar_partialH} we obtain that 
    \begin{equation*}
        \int_0^1 \partial^2_p H_0(q, \tau p, z, t) d\tau \cdot p = \partial_p \left(\int_0^1 (1-\tau)\partial_p^2 H_0(q, \tau p, z, t) d\tau \cdot p^2\right)
    \end{equation*}
    which concludes the proof of the first equality of~\eqref{eq:barMmL}. To prove the second equality in~\eqref{eq:barMmL}, we perform a Taylor expansion of $\partial_p H_0$ around $p=0$, which reads as in~\eqref{proof:lemma_bar_partialH}. Differentiating with respect to $p$ both sides of~\eqref{proof:lemma_bar_partialH} we have that 
    \begin{equation*}
        \partial_p^2 H_0(q,p,z,t) = \partial_p \left(\int_0^1 \partial_p^2 H_0(q,\tau p,z,t) d\tau \cdot p \right),
    \end{equation*}
    that concludes the proof of the second equality of~\eqref{eq:barMmL}.
    Similarly, one can verify the others. 
    The estimates~\eqref{est:MmLbar_Holder} and~\eqref{est:MmLbar_analy} follow directly from the definitions in~\eqref{def:barMmL}; in the analytic case, the bounds for $\overline L$ and $\overline{\overline L}$ are obtained by Cauchy's estimates.
\end{proof}

\section{Asymptotically Elliptic Case}%
\label{sec:Proof_Theorem_Ell}

In this section we prove Theorems \ref{Thm:elliptic_Csigma}, \ref{Thm:elliptic_analy}  and Corollary \ref{cor:ell}.

Theorem \ref{Thm:elliptic_analy}, which is simply the version of Theorem \ref{Thm:elliptic_Csigma} adapted to real analytic Hamiltonians, can be proved along the same lines as Theorem \ref{Thm:elliptic_Csigma}, with some minor modifications. Moreover, since Theorem \ref{Thm:elliptic_Csigma} applies under the hypotheses of Theorem \ref{Thm:elliptic_analy}, it suffices to prove only the first part of Theorem \ref{Thm:elliptic_analy}. The proof of Corollary \ref{cor:ell} is instead obtained by combining Theorem \ref{Thm:elliptic_Csigma} with the asymptotic correspondence established in Proposition \ref{prop:ell_1:1_asym_sol}.
For this reason, except for Subsection \ref{sc:proof_analytic_ell}, where we discuss the modifications needed to adapt the proof to the analytic setting, we will focus mostly on Theorem \ref{Thm:elliptic_Csigma} and on Hamiltonians with Hölder regularity. The proof of Corollary \ref{cor:ell} is given separately in Subsection \ref{sec:proof_cor_ell}.

Theorem \ref{Thm:elliptic_Csigma} consists of two statements, namely, the existence of a $C^\sigma$ asymptotic KAM torus and the existence of asymptotic elliptic transversal dynamics (for the asymptotic KAM torus given by the first part of the theorem, under stronger assumptions on the perturbation), which we prove separately in Sections \ref{sec:proof_thm_a_ell_1} and \ref{sec:proof_thm_a_ell_2}, respectively. In Section \ref{sec:proof_thm_a_ell_1}, we will consider $\sigma \ge 1$, whereas, in Section \ref{sec:proof_thm_a_ell_2}, we will assume $\sigma \ge 2$.

The proof of Theorem \ref{Thm:elliptic_Csigma} relies on applying twice the Implicit Function Theorem on appropriately defined Banach spaces of time-dependent functions exhibiting appropriate polynomial decay in time. We introduce these spaces in the next subsection. 

In order to verify that the hypotheses of the Implicit Function Theorem are fulfilled in our setting, we will be naturally led to consider the existence of solutions to several (cohomological) equations, namely,
\begin{equation}
\label{eq:coh_eqs_ell}
\left\{ \begin{array}{lll}
    \Lo \chi := (\partial_t + \omega \cdot \partial_q) \chi = g, & & g: \T^n \times I_T \to \R,\\
    \mathcal{L}^{\mathrm{ell}}_{\omega, \Omega} \boldsymbol{\chi} :=  J_mM^{\mathrm{ell}}_\Omega \boldsymbol{\chi} - \Lo \boldsymbol{\chi}  = \boldsymbol{g}, & & \boldsymbol{g}: \T^n \times I_T \to \R^{2m}, \\
    \mathfrak{L}^{\mathrm{ell}}_{\omega, \Omega} X := X^{\top}M^{\mathrm{ell}}_\Omega + M^{\mathrm{ell}}_\Omega X + J_m \Lo X =  G, & & G: \T^n \times I_T \to \mathcal{M}_{2m}(\R),
    \end{array} \right.
\end{equation}
in the unknowns $\chi: \T^n \times I_T \to \R$, $\boldsymbol{\chi}: \T^n \times I_T \to \R^{2m}$ and $X: \T^n \times I_T \to \mathcal{M}_{2m}(\R)$ (restricted to appropriate Banach spaces of functions with polynomial decay in time),  where $\omega \in \R^n$, $\Omega = \mathrm{diag}(\Omega_1,\dots,\Omega_m)$ with $\Omega_1,\dots,\Omega_m \in \R_{> 0}$, and the matrices $M^{\mathrm{ell}}_\Omega$ and $J_m$ are as in \eqref{Ms} and \eqref{def:J}, respectively. Recall that  $I_T = [T, + \infty)$, for any $T \geq 1$.

 The remainder of this section is structured as follows. In Section \ref{sc:functional_setting_ell}, we introduce several Banach spaces of functions with polynomial decay in time. In Section \ref{sc:initial_ham_ell} we write the unperturbed Hamiltonian of Theorem \ref{Thm:elliptic_Csigma} in a suitable form and set some notations. In Sections \ref{sec:proof_thm_a_ell_1} and \ref{sec:proof_thm_a_ell_2}, we prove items \eqref{thm:A_existence} and \eqref{thm:A_transverse} of Theorem \ref{Thm:elliptic_Csigma}, respectively, assuming several results concerning the cohomological equations mentioned above which, for the sake of clarity of exposition, we prove later in Section \ref{sc:cohom_eqs_ell}. In Section \ref{sc:proof_analytic_ell} we discuss the proof of Theorem \ref{Thm:elliptic_analy}. Section \ref{sec:proof_cor_ell} is dedicated to the proof of Corollary \ref{cor:ell}. Finally, Section \ref{sc:norm_properties_ell} contains several technical results about the norms and spaces introduced in Sections \ref{sc:functional_setting_ell} and \ref{sc:proof_analytic_ell}, which will often be referenced throughout this section.

\subsection{Functional setting}\label{sc:functional_setting_ell} Let us introduce the function spaces that will help us describe the perturbations considered in Theorem \ref{Thm:elliptic_Csigma} and the families of torus embeddings to which the asymptotic KAM tori will belong. Notice that all of the spaces we introduce will be Banach spaces with their respective norm.

In the following, we let $\omega \in \R^n$, $\Omega = \mathrm{diag}(\Omega_1,\dots,\Omega_m)$ with $\Omega_1,\dots,\Omega_m \in \R_{> 0}$, $\sigma \geq 0$, $k \in \Z_{\geq0}$, $\ell > 0$ and $T \geq 1$, and we use the notations in \eqref{eq:coh_eqs_ell} for the operators $\Lo, \mathcal{L}^{\mathrm{ell}}_{\omega, \Omega}$ and $\mathfrak{L}^{\mathrm{ell}}_{\omega, \Omega}.$

\subsubsection{Real-valued  functions with polynomial decay}\label{sc:real_pol_decay} Recall that the Banach space $\mathcal{S}^{\mathrm{pol}, T}_{(\sigma, k), \ell}$ is defined in~\eqref{def:S} and the associated norm $|\cdot|^{\mathrm{pol}, T}_{\sigma+k, \ell}$ in~\eqref{def:norm_S}. 

Letting $k \ge 1$, in order to control the regularity of the term in the perturbation independent of $p$ and $z$, we introduce the space
\begin{equation}\label{def: S0ell}
    \mathcal{S}^{\mathrm{pol},T}_{(\sigma,k), (0,\ell)} = \left\{f:\T^n \times B \times I_T \to \R  \hspace{1mm} \left| \hspace{1mm} f \in \mathcal{S}^{\mathrm{pol}, T}_{(\sigma,k), 0}, \hspace{2mm} \partial_q f \in \mathcal{S}^{\mathrm{pol}, T}_{(\sigma, k-1), \ell}\right\}\right.,
\end{equation}
endowed with the norm
\begin{equation}\label{def:norm_S0ell}
    |f|^{\mathrm{pol}, T}_{(\sigma,k), (0,\ell)} = \max\{|f|^{\mathrm{pol}, T}_{\sigma+k, 0}, |\partial_qf|^{\mathrm{pol}, T}_{\sigma+k-1, \ell}\},
\end{equation}

To quantify the regularity of the components of the $C^\sigma$  asymptotic KAM torus, we define 
\begin{equation}\label{def:U}
    \mathcal{U}^{\mathrm{pol}, T}_{\sigma, \omega, \ell} = \left\{u:\T^n \times I_T \to \R^n \hspace{1mm} \left| \hspace{1mm} u \in \mathcal{S}^{\mathrm{pol},T}_{(\sigma,0), \ell}, \hspace{2mm}\Lo u \in \mathcal{S}^{\mathrm{pol}, T}_{(\sigma,0), \ell+1}\right\}\right.,
\end{equation}
endowed with the norm
\begin{equation}\label{def:norm_U}
    |u|^{\mathrm{pol}, T}_{\sigma, \omega, \ell} = \max \left\{|u|^{\mathrm{pol}, T}_{\sigma, \ell}, |\Lo u|^{\mathrm{pol}, T}_{\sigma, \ell +1}\right\}.
\end{equation}

Furthermore, we consider 
\begin{equation}\label{def:W}
    \mathcal{W}^{\mathrm{ell}, T}_{\sigma, \omega, \Omega, \ell} = \left\{w:\T^n \times I_T \to \R^{2m} \hspace{1mm} \left| \hspace{1mm} w \in \mathcal{S}^{\mathrm{pol}, T}_{(\sigma,0), \ell}, \hspace{2mm}\mathcal{L}^{\mathrm{ell}}_{\omega,\Omega} w \in \mathcal{S}^{\mathrm{pol}, T}_{(\sigma,0), \ell+1}\right\}\right.,
\end{equation}
endowed with the norm
\begin{equation}\label{def:norm_W}
    |w|^{\mathrm{pol}, T}_{\sigma, \omega, \Omega, \ell} = \max \left\{|w|^{\mathrm{pol}, T}_{\sigma, \ell}, |\mathcal{L}^{\mathrm{ell}}_{\omega,\Omega} w|^{\mathrm{pol}, T}_{\sigma, \ell +1}\right\}.
\end{equation}

\begin{remark}\label{rmk:extra_regularity_ell}
    If $\sigma \geq 1$ then any $u \in \mathcal{U}^{\mathrm{pol}, T}_{\sigma, \omega, \ell}$ (resp. $w \in \mathcal{W}^{\mathrm{ell}, T}_{\sigma, \omega, \Omega, \ell}$) is of class $C^1$. Moreover, as we shall see in Proposition \ref{prop:Csigma_HEomega} (resp. Proposition \ref{prop:Csigma_HEomegaOmega}), if $\sigma \geq 2$ then $\partial_q u \in  \mathcal{U}^{\mathrm{pol}, T}_{\sigma - 1, \omega, \ell}$ (resp.  $\partial_q w \in \mathcal{W}^{\mathrm{ell}, T}_{\sigma, \omega, \Omega, \ell}$) and, in particular, $\partial_q u$ (resp. $\partial_q w$) is of class $C^1$.
\end{remark}

\subsubsection{Matrix-valued functions with polynomial decay}
\label{sc:matrix_spaces_ell} 
We also introduce spaces of functions taking values in $\mathcal{M}_{2m}(\R)$. These spaces will be useful for describing the normal coordinates of the asymptotic KAM torus. 

Let
\begin{equation}
\begin{aligned}\label{def:BS_ell_M_Sym}
     \mathcal{M}^{\mathrm{pol}, T}_{\sigma, \ell} &= %
     \mathcal{S}^{\mathrm{pol}, T}_{(\sigma,0), \ell}(\R^{2m \times 2m}), \\
          \mathcal{S}ym^{\mathrm{pol}, T}_{\sigma, \ell} &= \left\{M:\T^n \times I_T \to \textup{Sym}(2m, \R) \hspace{1mm} \left| \hspace{1mm}  M \in \mathcal{M}^{\mathrm{pol}, T}_{\sigma, \ell} \right\}\right.,
    \end{aligned}
\end{equation}      
endowed with the norm,
\[ |M|^{\mathrm{pol}, T}_{\sigma, \ell} := \max_{1 \leq i, j \leq N} |M_{ij}|^{\mathrm{pol}, T}_{\sigma, \ell},\]
 and
 \begin{align}\label{def:BS_ell_Sp}
 \mathcal{S}p^{\mathrm{ell}, T}_{\sigma, \omega, \Omega, \ell} %
    &  = \left\{M:\T^n \times I_T \to \mathfrak{sp}(2m, \R) \hspace{1mm} \left| \hspace{1mm} M \in \mathcal{M}^{\mathrm{pol}, T}_{\sigma, \ell}, \quad \mathfrak{L}_{\omega, \Omega}^\mathrm{ell} M \in \mathcal{M}^{\mathrm{pol}, T}_{\sigma, \ell + 1} \right\} \right.,
\end{align}
endowed with the norm
\[ |M|^{\mathrm{pol}, T}_{\sigma, \omega, \Omega, \ell} = \max\left \{ |M|_{\sigma, \ell}^{\mathrm{pol}, T} ,|\mathfrak{L}_{\omega, \Omega}^\mathrm{ell} M|^{\mathrm{pol}, T}_{\sigma, \ell + 1} \right\},\]
where $\textup{Sym}(2m, \R)$ denotes the space of $2m \times 2m$ symmetric matrices and $\mathfrak{sp}(2m, \R)$ denotes the Lie algebra associated with the space of $2m \times 2m$ symplectic matrices $\textup{Sp}(2m, \R),$ namely, 
\[ \mathfrak{sp}(2m, \R) := \{ G \in \mathcal{M}_{2m}(\R) \mid J_m G + G^{\top}J_m = 0 \}. \]

Notice that if $\sigma \geq 1$ then any $M \in  \mathcal{S}p^{\mathrm{ell}, T}_{\sigma, \omega, \Omega, \ell}$ is of class $C^1$. 

\subsubsection{Perturbations with polynomial decay}
\label{sc:perturbation_ell}
To describe the space of perturbations considered in Theorem \ref{Thm:elliptic_Csigma}, given $\Lell = (\ell_1,\ell_2, \ell_3, \ell_4) \in \left(\R_{\ge 0}\right)^4$ we define
\begin{equation}\label{def:space_P_ell}
\mathcal{P}^{\mathrm{ell}, T}_{\sigma, \Lell} = \mathcal{S}^{\mathrm{pol},T}_{(\sigma,2), (0,\ell_1)} \times \mathcal{S}^{\mathrm{pol},T}_{(\sigma,2), \ell_2}(\R^n)  \times \mathcal{S}^{\mathrm{pol},T}_{(\sigma,2), \ell_3}(\R^{2m})  \times \mathcal{S}^{\mathrm{pol},T}_{(\sigma,2), \ell_4}(\R^{2m \times 2m}),
\end{equation}
endowed with the sum norm associated with the product, which we denote by $| \cdot |^{\mathrm{ell}, T}_{\sigma, \Lell}$.  As an abuse of notation, we denote the elements of this product space simply as $$P = (a, b, c, d) \in \mathcal{P}^{\mathrm{ell}, T}_{\sigma, \Lell},$$ and use the same letter $P$ to denote the associated Hamiltonian given by \eqref{eq:perturbation_form}, namely,
\[ P(q,p,z,t) = a(q,t) + b(q,t) \cdot p + c(q,t) \cdot z + {1 \over 2}d(q,t)\cdot z^2.\]
Given $\ell \geq 0$ and $\dec \ge 0$, we denote
\begin{equation}
    \label{eq:index_regularity_perturbations}
    \mathcal{L}(\ell, \dec) =  (\ell + \dec + 2, \ell, \ell + \dec + 1, \ell), \qquad \mathcal{L}^*(\ell, \dec) = (\ell + \dec + 3, \ell , \ell + \dec + 2, \ell + 1).
\end{equation}

With these notations, the spaces of perturbations associated with Items \eqref{thm:A_existence} and \eqref{thm:A_transverse} of Theorem \ref{Thm:elliptic_Csigma} are given by  $\mathcal{P}^{\mathrm{ell}, 1}_{\sigma, \Lell(\ell, \dec)}$ and $\mathcal{P}^{\mathrm{ell}, 1}_{\sigma, \Lell^*(\ell, \dec)}$, with $\sigma \geq 1$ and $\sigma \geq 2$, respectively.%

\subsubsection{Torus embeddings with polynomial decay} 
\label{sc:embeddings_ell} 
Similarly to how we parametrized the perturbations space, to describe the space of torus embeddings (to which the asymptotic KAM tori given by the theorem will belong), for any $\mathcal{K}= (k_1, k_2, k_3) \in \left(\R_{\ge 0}\right)^3$  we define
\[   \mathcal{E}^{\mathrm{ell}, T}_{\sigma, \mathcal{K}} = \mathcal{U}^{\mathrm{pol},T}_{\sigma, \omega, k_1} \times \mathcal{U}^{\mathrm{pol},T}_{\sigma, \omega, k_2} \times \mathcal{W}^{\mathrm{ell},T}_{\sigma, \omega, \Omega, k_3},
\]
endowed with the sum norm associated with the product, which we denote by $| \cdot |^{\mathrm{ell}, T}_{\sigma, \mathcal{K}}$. As an abuse of notation, we denote the elements of this product space simply as 
$$\varphi = (u, v, w) \in   \mathcal{E}^{\mathrm{ell}, T}_{\sigma, \mathcal{K}},$$
and use the same letter $\varphi$ to denote the associated family of torus embeddings, namely,
\begin{equation}
\label{eq:formula_torus_embeddings}
\Function{\varphi}{\T^n \times I_T}{\T^n \times \R^{n} \times \R^{2m}}{(q,t)}{(q + u(q, t), v(q, t), w(q, t))}.
\end{equation}

In the following, for any $k \geq 1$ and $\mathsf{k} \ge 0$, we denote
\begin{equation}
\label{eq:index_regularity_tori} 
    \mathcal{K}(k, \mathsf{k}) := (k - 1, k + \mathsf{k} + 1, k + \mathsf{k}).
\end{equation}
As we shall see in Section \ref{sec:proof_thm_a_ell_1}, for any $P \in  \mathcal{P}^{\mathrm{ell}, 1}_{\sigma, \Lell(\ell, \dec)}$ with $\sigma \geq 1$, $\ell > 1$ and $\dec > 0$ there exists a $C^\sigma$ asymptotic KAM torus in $ \mathcal{E}^{\mathrm{ell}, T}_{\sigma, \mathcal{K}(\ell, \dec)}$ associated with $(X_{H_0 + P}, X_{H_0}, \varphi_0)$, for some $T \geq 1$.

\subsection{Initial Hamiltonian and conventions in notation}\label{sc:initial_ham_ell} For the remainder of this section, we fix $\Omega_1,\dots,\Omega_m \in \R_{> 0}$,  $\omega \in \R^n$, $\sigma \geq 1$ and $H_0: \T^n \times B \times I_1 \to \R$ of the form \eqref{eq:initial_hamiltonian} as in the statement of Theorem \ref{Thm:elliptic_Csigma}, where $B \subseteq \R^{n +2m}$ is a ball centred at $0$ and $I_T = [T, + \infty)$, for any $T \geq 1$. For the sake of simplicity, let us assume that $B = B_{\R^{n + 2m}}(1)$ is the unit ball in $\R^{n + 2m}$. We denote $\Omega = \mathrm{diag}(\Omega_1,\dots,\Omega_m)$.

Since all the norms considered in this section concern (products of) function spaces with polynomial decay, for the sake of clarity and to simplify the notation, we omit the superscripts $\mathrm{ell}$ and $\mathrm{pol}$ from all the norms whenever there is no risk of confusion.

Using the notations introduced in Section \ref{sec:FS_Ref_Ham},  we can express $H_0$ as in \eqref{eq:initial_Hamiltonian_NR}, namely,
\[         H_0(q,p,z,t) = N(p,z) + R(q,p,z,t), \]
where 
\begin{equation*}
\begin{array}{rcl}   N(p,z) &= & \omega \cdot p + {1 \over 2} \Mell \cdot z^2, \\
    R(q,p,z,t) &= &M(q,p,z,t) \cdot p^2 + m(q,z,t) \cdot (p,z) + L(q,z,t) \cdot z^3,
\end{array}
\end{equation*}
$\Mell$ is given by \eqref{Ms}, $M, m, L$ are given by \eqref{def:abcdMmL_terms} and satisfy
\[ |M|^1_{\sigma + 3, 0}, |m|_{\sigma + 3, 0}^1, |L|_{\sigma + 2, 0}^1 \le|\partial_{(p,z)}^2 H_0|^1_{\sigma+3, 0}. \] 
By assumption $\partial_{(p,z)}^2 H_0 \in \mathcal{S}_{(\sigma,3),0}^{\mathrm{pol},1}.$ Let us fix $\Upsilon > 0$ such that 
\begin{equation}\label{def:const_Upsilon_ell}
|\partial_{(p,z)}^2 H_0|^1_{\sigma+3, 0} \le \Upsilon.
\end{equation}

We recall that the trivial embedding $\varphi_0: \T^n \to \T^n \times \R^{n + 2m}$, given by \eqref{def:varphi0=(q,0,0)}, defines an invariant torus for $H_0$ supporting quasiperiodic solutions with frequency vector $\omega$.

Finally, for the sake of simplicity, we fix  $\ell > 1$ and $\dec \ge 0$, and denote $\mathcal{L}(\ell, \dec)$,  $\mathcal{L}^*(\ell, \dec)$ and $\mathcal{K}(\ell, \dec)$ (as in \eqref{eq:index_regularity_perturbations} and \eqref{eq:index_regularity_tori}) simply by $\mathcal{L}$, $\mathcal{L}^*$ and $\mathcal{K}$, respectively.

\subsection{Proof of Item \eqref{thm:A_existence} of Theorem \ref{Thm:elliptic_Csigma}: Construction of $C^\sigma$ asymptotic KAM tori}\label{sec:proof_thm_a_ell_1}

In the following, we will use freely the notations introduced at the beginning of this section and in the previous two subsections.

Recall that the space of perturbations associated with Item \eqref{thm:A_existence} of Theorem \ref{Thm:elliptic_Csigma} is given by  $\mathcal{P}^{\mathrm{ell}, 1}_{\sigma, \Lell}$ (see Section \ref{sc:perturbation_ell}). As we shall see in the proposition below, for any perturbation $P \in \mathcal{P}^{\mathrm{ell}, 1}_{\sigma, \Lell} $, we can define an asymptotic KAM torus $\boldsymbol{\varphi}^T(P) \in  \mathcal{E}^{\mathrm{ell}, T}_{\sigma, \mathcal{K}}$ associated with $(X_{H_0 + P}, X_{H_0}, \varphi_0)$, for some $T \geq 1$. For this purpose, given $r, \rho>0$ and $T \geq 1$ we denote by $B_{\mathcal{P}^{\mathrm{ell}, T}_{\sigma, \Lell}}(r)  \subseteq  \mathcal{P}^{\mathrm{ell}, T}_{\sigma, \Lell}$ and $B_{ \mathcal{E}^{\mathrm{ell}, T}_{\sigma, \mathcal{K}}}(\rho) \subseteq \mathcal{E}^{\mathrm{ell}, T}_{\sigma, \mathcal{K}}$ the open balls centered at the origin with radius $r$ and $\rho$, respectively.

We will show the following

\begin{proposition}
\label{prop:A_existence}
There exists $C_0 > 1$ such that for any $r > 0$ there exists $T_0 \geq 1$ satisfying the following. For any $T \geq T_0$, there exists a $C^1$-map
\[ \boldsymbol{\varphi}^T: B_{\mathcal{P}^{\mathrm{ell}, T}_{\sigma, \Lell}}(r)  \subseteq  \mathcal{P}^{\mathrm{ell}, T}_{\sigma, \Lell}  \to  B_{ \mathcal{E}^{\mathrm{ell}, T}_{\sigma, \mathcal{K}}}(C_0r) \subseteq \mathcal{E}^{\mathrm{ell}, T}_{\sigma, \mathcal{K}}\]
such that $\boldsymbol{\varphi}^T(P)$ defines an asymptotic KAM torus associated with $(X_{H_0 + P}, X_{H_0}, \varphi_0)$,  for any $P \in  \mathcal{P}^{\mathrm{ell}, T}_{\sigma, \Lell} $ with $|P|_{\sigma, \Lell}^T < r$.
\end{proposition}

Notice that Item \eqref{thm:A_existence} of Theorem \ref{Thm:elliptic_Csigma} follows directly from the proposition above since
\begin{equation*}
    \big|P\big|_{\sigma, \Lell}^{T} \leq |P|_{\sigma, \Lell}^{1}, \qquad \text{ for any } T \geq 1\,  \text{ and any } \, P \in \mathcal{P}^{\mathrm{ell}, 1}_{\sigma, \Lell},
\end{equation*}
and by definition of $\mathcal{K}$ the asymptotic KAM torus has the desired decay rates.

The proof of Proposition \ref{prop:A_existence}, which we provide in \S \ref{sc:thm_A_existence_proof}, will be a consequence of an appropriate quantitative version of the Implicit Function Theorem (see Theorem \ref{thm:QIFT}) applied to a functional characterizing asymptotic KAM tori for the system $H_0$ and that we define in \S \ref{sc:functional_existence}. 

Let us point out that the choices for $\Lell$, $\mathcal{K}$ and the spaces $\mathcal{P}^{\mathrm{ell}, T}_{\sigma, \Lell},$  $ \mathcal{E}^{\mathrm{ell}, T}_{\sigma, \mathcal{K}}$ appear as natural restrictions for the solvability of certain cohomological equations (moreover, for their invertibility when seen as linear operators that will arise naturally in the following sections (see Lemmas \ref{lemma:DF_invertible}, \ref{lemma:DG_invertible} and Propositions  \ref{prop:Csigma_HEomega}, \ref{lemma:HE_interm_LoO}, \ref{prop:Csigma_HEomegaOmega}).

\subsubsection{Functionals characterizing asymptotic KAM tori}\label{sc:functional_existence} We start by defining a functional $\mathcal{F}^T$, where $T \geq 1$, taking as arguments the perturbation $P  \in \mathcal{P}^{\mathrm{ell}, T}_{\sigma, \Lell} $ and the family of torus embeddings $\varphi \in \mathcal{E}^{\mathrm{ell}, T}_{\sigma, \mathcal{K}}$ such that $\mathcal{F}^T(P, \varphi) = 0$ if and only if $\varphi$ defines an asymptotic KAM torus for $H_0 + P$. 

Recalling that $\varphi$ is a $C^\sigma$ asymptotic KAM torus associated with $(X_{H_0 + P}, X_{H_0}, \varphi_0)$ if and only if \eqref{def:cond1_asymKAMtorus} and \eqref{def:cond2_asymKAMtorus} are satisfied, for any $T \geq 1$, let 
\begin{equation}\label{proof:Csigma_def_F_ell_1}
    \Function{\mathcal{F}^T}{ \mathcal{P}^{\mathrm{ell}, T}_{\sigma,\Lell} \times \mathcal{E}^{\mathrm{ell}, T}_{\sigma, \mathcal{K}}}{ \mathcal{S}^{\mathrm{pol},T}_{(\sigma,0), \ell}(\R^n) \times \mathcal{S}^{\mathrm{pol},T}_{(\sigma,0), \ell+\dec+2}(\R^n) \times \mathcal{S}^{\mathrm{pol},T}_{(\sigma,0), \ell+\dec+1}(\R^{2m})}{(P, \varphi)}{X_{H_0 + P} \circ \varphi - \Lo \varphi },
\end{equation}
and recall that, for the sake of simplicity in the notation, we write $f \circ g$ for the map $(\cdot,t)\mapsto f(g(\cdot,t),t)$. We endow the codomain of $\mathcal{F}^T$ with the sum norm associated with the product, which we denote by $|\cdot|_{\boldsymbol{{\sigma, \ell, \dec}}}^T$ (where we used bold symbols to differentiate this norm from the norm $|\cdot|^T_{\sigma, \ell}$).

Notice that, with these definitions, any $\varphi \in \mathcal{E}^{\mathrm{ell}, T}_{\sigma,\mathcal{K}}$ satisfies \eqref{def:cond1_asymKAMtorus}, while $\mathcal{F}^T(P, \varphi) = 0$ if and only if \eqref{def:cond2_asymKAMtorus} is satisfied. Therefore,
\begin{equation}
\label{eq:KAM_torus_characterization}
    \mathcal{F}^T(P, \varphi) = 0 \quad \mbox{if and only if $\quad \varphi$ defines a $C^\sigma$ asymptotic KAM torus for $H_0 + P$.}
\end{equation}

Let us point out that we will need to restrict the domain of $\mathcal{F}^T$ for it to be well-defined (otherwise the composition $X_{H_0 + P} \circ \varphi$ might be undefined). This will be formalized in Lemma \ref{lemma:F_well_defined}.

\subsubsection{Proof of the existence of an asymptotic KAM torus} \label{sc:thm_A_existence_proof} In this section we will provide a proof of Proposition \ref{prop:A_existence} assuming the following properties of the functional $\mathcal{F}^T$, which  we will prove in the next section. Here and in the following, $D_\varphi \mathcal{F}^T$ denotes the differential of the functional $\mathcal{F}^T$ with respect to $\varphi = (u,v,w)$.

\begin{proposition}
\label{prop:F_properties}
    The operator $\mathcal{F}^T$ given by \eqref{proof:Csigma_def_F_ell_1} satisfies the following.
    \begin{enumerate}
    \item \label{prop:F_well_defined} For any $T \geq 1$, there exists a neighbourhood $\mathcal{U}^T \subseteq \mathcal{P}^{\mathrm{ell}, T}_{\sigma,\Lell} \times \mathcal{E}^{\mathrm{ell}, T}_{\sigma, \mathcal{K}}$ of $(0, 0)$ such that $\mathcal{F}^T\mid_{\mathcal{U}^T}$ and $D_\varphi\mathcal{F}^T\mid_{\mathcal{U}^T}$ are well-defined continuous maps. 
    
    \noindent Moreover, for any $r, \rho \geq 0$ there exists $T_0 \geq 1$ such that, for any $T \geq T_0$,  
    \[B_{\mathcal{P}^{\mathrm{ell}, T}_{\sigma,\Lell}}(r) \times B_{\mathcal{E}^{\mathrm{ell}, T}_{\sigma, \mathcal{K}}}(\rho) \subseteq \mathcal{U}^T.\]

         \item \label{prop:F_bounded}There exists a constant $C$,
     such that for any $r > 0$ and any $T \geq 1$,
\[ \sup \left\{ |\mathcal{F}^T(P, 0)|_{\boldsymbol{\sigma, \ell, \dec}}^T \,\left|\, P \in B_{\mathcal{P}^{\mathrm{ell}, T}_{\sigma,\Lell}}(r); \,\, (P, 0) \in \mathcal{U}^T \right\}\right. \leq Cr.\] %

    \item \label{prop:DF_invertible} $D_\varphi \mathcal{F}^T(0, 0)$ is invertible, for any $T \geq 1$. Moreover, there exists a constant $\bar C$, independent of $T$, such that $\|D_\varphi \mathcal{F}^T(0, 0)^{-1}\| \leq \bar C$, where $\| \cdot\|$ denotes the operator norm. %

        \item \label{prop:DF_limit} Let $r, \rho > 0$. Then  
        \[\lim_{T \to +\infty} \| D_\varphi\mathcal{F}^T(P, \varphi) - D_\varphi\mathcal{F}^T(0, 0) \| = 0, \qquad \text{uniformly on } \quad B_{\mathcal{P}^{\mathrm{ell}, T}_{\sigma,\Lell}}(r) \times B_{\mathcal{E}^{\mathrm{ell}, T}_{\sigma, \mathcal{K}}}(\rho),\]%
where $\| \cdot\|$ denotes the operator norm. %
    \end{enumerate}
\end{proposition}

Let us point out that Properties \eqref{prop:F_well_defined}, \eqref{prop:F_bounded} and \eqref{prop:DF_limit} follow from straightforward calculations and are not very difficult to show. The heart of the proof of Proposition \ref{prop:F_properties} is to show Property \eqref{prop:DF_invertible}, that is, that the derivative $D_\varphi\mathcal{F}^T(0, 0)$ is invertible. This will require solving certain cohomological equations arising naturally when studying this invertibility problem. Moreover, taking into account the cohomological equations and the spaces of functions where the associated solutions are defined, the spaces considered in the definition of $\mathcal{F}^T$ appear as a natural choice to guarantee that, besides being well-defined, the linear operator $D_\varphi\mathcal{F}^T(0, 0)$ is invertible.

 Notice that from Properties \eqref{prop:F_well_defined}, \eqref{prop:DF_invertible} in  Proposition \ref{prop:F_properties}, and since $\mathcal{F}^T(0, 0) = 0 $, by the Implicit Function Theorem it follows that Item \eqref{thm:A_existence} of Theorem \ref{Thm:elliptic_Csigma} holds for perturbations with sufficiently small $|\cdot |^1_{\sigma, \Lell}$-norm, or, equivalently,  Proposition \ref{prop:A_existence}  holds for sufficiently small $r$. %
 
 However, to prove the result for any perturbation in $\mathcal{P}^{\mathrm{ell}, 1}_{\sigma,\Lell}$, we will show that the domain where we can apply the Implicit Function Theorem gets larger as $T$ increases. This will rely on a quantitative version of the Implicit Function Theorem (see Theorem \ref{thm:QIFT}) where Item \eqref{prop:DF_limit} in Proposition \ref{prop:F_properties} will play an essential role.

\begin{proof}[Proof of Proposition \ref{prop:A_existence}]
Let $r > 0$ and define $\rho = 2C\bar C r,$ where $C$ and $\bar C$ are constants given by Items \eqref{prop:F_bounded}, and \eqref{prop:DF_invertible} of Proposition \ref{prop:F_properties}.

By Items \eqref{prop:F_well_defined} and \eqref{prop:DF_limit} of Proposition \ref{prop:F_properties} there exists $T_0 \geq 1$ such that, for any $T \geq T_0$,  $B_{\mathcal{P}^{\mathrm{ell}, T}_{\sigma,\Lell}}(r) \times B_{\mathcal{E}^{\mathrm{ell}, T}_{\sigma, \mathcal{K}}}(\rho)$ is contained in the domain of definition of $\mathcal{F}^T$ and 
\[  \| D\mathcal{F}^T(P, \varphi) - D\mathcal{F}^T(0, 0) \| < \frac{1}{2\bar C}, \qquad \text{ for } \quad |P|_{\sigma, \Lell}^T < r, \quad |\varphi|_{\sigma, \omega, \Omega, \mathcal{K}}^T < \rho.\]

Let $T \geq T_0$. By Items \eqref{prop:F_bounded} and \eqref{prop:DF_invertible} of Proposition \ref{prop:F_properties},
\[\|D_\varphi \mathcal{F}^T(0, 0)^{-1}\| \leq \bar C, \qquad \sup \left\{ |\mathcal{F}^T(P, 0)|_{\sigma, \Lell}^T \,\left|\, P \in B_{\mathcal{P}^{\mathrm{ell}, T}_{\sigma,\Lell}}(r)  \right\}\right. \leq \frac{\rho}{2\bar C}.\]

Noticing that $\mathcal{F}^T(0, 0) = 0$ and recalling \eqref{eq:KAM_torus_characterization}, Proposition \ref{prop:A_existence} follows by Theorem \ref{thm:QIFT}. 
\end{proof}

The remainder of this subsection concerns the proof of Proposition \ref{prop:F_properties}.

\subsubsection{Properties of the functional $\mathcal{F}^T$} For the sake of clarity,  we split the proof of Proposition \ref{prop:F_properties} into four lemmas (Lemmas \ref{lemma:F_well_defined}, \ref{lemma:F_bounded}, \ref{lemma:DF_invertible} and \ref{lem:DF_limit}), one for each property in the statement.

\begin{remark}
  We point out that Lemmas \ref{lemma:F_well_defined}, \ref{lemma:F_bounded}, and \ref{lemma:DF_invertible} can be proved under the slightly weaker assumption $d \in \mathcal{S}^{\mathrm{pol},1}_{(\sigma, 2), 1}$ instead of $d \in \mathcal{S}^{\mathrm{pol},1}_{(\sigma, 2), \ell}$. On the other hand, the stronger assumption $d \in \mathcal{S}^{\mathrm{pol},1}_{(\sigma, 2), \ell}$ is needed to prove Lemma \ref{lem:DF_limit}.
\end{remark}

We start by checking that $\mathcal{F}^T$ as in \eqref{proof:Csigma_def_F_ell_1} is well-defined.

\begin{lemma}\label{lemma:F_well_defined}
   Let $T \geq 1$. The functional $\mathcal{F}^T$ given by \eqref{proof:Csigma_def_F_ell_1} and its differential $D_\varphi \mathcal{F}^T$ are well-defined continuous maps on 
   \begin{equation}
   \label{eq:domain_F_ell}
   \mathcal{U}^T =  \mathcal{P}^{\mathrm{ell}, T}_{\sigma,\Lell} \times B_{\mathcal{E}^{\mathrm{ell}, T}_{\sigma, \mathcal{K}}}(T^{\ell - 1}).
   \end{equation}
\end{lemma}
\begin{proof}
Fix $P \in \mathcal{P}^{\mathrm{ell}, T}_{\sigma,\Lell}$ and $\varphi = (u, v, w) \in \mathcal{E}^{\mathrm{ell}, T}_{\sigma, \mathcal{K}}$ with $| \varphi |_{\sigma, \mathcal{K}}^T \leq T^{\ell - 1}$. It follows that 
\[|u|^T_{\sigma, \omega, \ell - 1}, |v|^T_{\sigma, \omega, \ell + \dec + 1}, |w|^T_{\sigma, \omega, \Omega, \ell + \dec} < T^{\ell - 1}.\]
In particular,
\[ |u|_{C^0} < 1, \qquad |v|_{C^0} < \frac{1}{T^{2 + \dec}}, \qquad |w|_{C^0} < \frac{1}{T^{1 + \dec}}, \]
and thus the composition $X_{H_0 + P} \circ \varphi$ (and therefore $\mathcal{F}^T(P, \varphi)$) is well-defined. 

The three components of $\mathcal{F}^T(P, \varphi)$ can be explicitly written as
\begin{equation}
\label{eq:formula_F_ell_1}
    \mathcal{F}^T(P, \varphi)  = \begin{pmatrix}\partial_p  H \circ \varphi - \omega -\Lo u\\
    -\partial_q H \circ \varphi - \Lo v\\
    J_m\partial_z H \circ \varphi - \Lo w\end{pmatrix},
\end{equation}
or, more precisely,
\begin{equation}
\begin{aligned}\label{eq:terms_F_ell_1}
    \partial_p H \circ \varphi - \omega -\Lo u &= b\circ \varphi_q + \overline{M} \circ \varphi \cdot  v + m \circ \varphi_{qz}\cdot w - \Lo u,\\
    -\partial_q H\circ \varphi - \Lo v &= -\partial_q a \circ \varphi_q - \partial_q b\circ \varphi_q\cdot v -  \partial_q c\circ \varphi_q \cdot w\\
    &- {1 \over 2}\partial_q d\circ \varphi_q \cdot (w)^2 - \partial_q M\circ \varphi
    \cdot (v)^2 - \partial_q m \circ \varphi_{qz}\cdot (v,w)\\
    &- \partial_q L \circ \varphi_{qz} \cdot (w)^3 - \Lo v,\\
    J_m\partial_z H\circ \varphi - \Lo w &= J_mM^{\mathrm{ell}}_\Omega \cdot w + J_mc \circ \varphi_q + J_m d\circ \varphi_q \cdot w\\
    &+ J_m\partial_z M\circ \varphi \cdot (v)^2 + J_m\overline{m} \circ \varphi_{qz}\cdot v + J_m \overline{L}\circ \varphi_{qz} \cdot (w)^2 - \Lo w,
\end{aligned}
\end{equation}
where the matrices $M^{\mathrm{ell}}_\Omega$ and $J_m$ are defined in~\eqref{Ms} and~\eqref{def:J}, respectively, and $\overline{M}$, $\overline{m}$, and $\overline{L}$ in Lemma \ref{lemma:def_bar_M_m_L}. In the equation above, we denoted by an index the projection of $\varphi$ on the $q$, $p$, and $z$ components. 

Now, by \eqref{eq:terms_F_ell_1} and using the properties contained in Proposition \ref{prop:Csigma_prop_norms_pol}, a straightforward computation shows that $\mathcal{F}(P, \varphi) \in \mathcal{S}^{\mathrm{pol},T}_{(\sigma,0), \ell}(\R^n) \times \mathcal{S}^{\mathrm{pol},T}_{(\sigma,0), \ell+\dec+2}(\R^n) \times \mathcal{S}^{\mathrm{pol},T}_{(\sigma,0), \ell+\dec+1}(\R^{2m})$.

The regularity of $\mathcal{F}^T$ and $D_\varphi \mathcal{F}^T$ follows directly from~\eqref{eq:terms_F_ell_1}, together with the regularity assumptions on the perturbative terms  $P=(a,b,c,d) \in \mathcal{P}^{\mathrm{ell}, T}_{\sigma,\mathcal{L}}$ and the components of the $C^\sigma$ asymptotic KAM torus $\varphi = (u,v,w) \in \mathcal{E}^{\mathrm{ell}, T}_{\sigma, \mathcal{K}}$.
\end{proof}

\begin{lemma}
    \label{lemma:F_bounded}
   There exists a constant $C$, such that for any $r > 0$ and any $T \geq 1$,
\[ \sup \left\{ |\mathcal{F}^T(P, 0)|_{\boldsymbol{\sigma, \ell, \dec}}^T \mid P \in \mathcal{P}^{\mathrm{ell}, T}_{\sigma,\Lell}, \, |P|_{\sigma, \Lell}^T < r \right\} < Cr .\]
\end{lemma}
\begin{proof}
Since $\mathcal{F}^T(P, 0) = X_{H_0 + P} \circ \varphi_0$ (where $\varphi_0$ is the trivial embedding \eqref{def:varphi0=(q,0,0)}), for any $P \in \mathcal{P}^{\mathrm{ell}, T}_{\sigma, \Lell}$, it follows that there exists a constant $C$, independent of $T$, such that
\[ \sup \left\{ |\mathcal{F}^T(P, 0)|_{\boldsymbol{\sigma, \ell, \dec}}^T \mid P \in \mathcal{P}^{\mathrm{ell}, T}_{\sigma,\Lell}, \, |P|_{\sigma, \Lell}^T < r \right\} < Cr .\]

\end{proof}

We will now consider the differential of $\mathcal{F}^T$ with respect to $\varphi = (u, v, w)$ evaluated at $(0,0)$. Observe that $\mathcal{F}^T(0, 0) = 0$. A direct calculation shows that 
\begin{equation}
\label{proof:Csigma_def_ell_DF_1}
\Function{D_\varphi\mathcal{F}^T(0,0)}{\mathcal{E}^{\mathrm{ell}, T}_{\sigma, \mathcal{K}}}{ \mathcal{S}^{\mathrm{pol},T}_{(\sigma,0), \ell}(\R^n) \times \mathcal{S}^{\mathrm{pol},T}_{(\sigma,0), \ell+\dec+2}(\R^n) \times \mathcal{S}^{\mathrm{pol},T}_{(\sigma,0), \ell+\dec+1}(\R^{2m})}{(\hat u, \hat v, \hat w) }{\begin{pmatrix} \overline{M}_0 \cdot \hat v + m_0 \cdot \hat w - \Lo \hat u\\
    -\Lo \hat v\\
    \mathcal{L}^{\mathrm{ell}}_{\omega, \Omega} \hat w + J_m\overline{m}_0 \cdot \hat v
    \end{pmatrix}},
\end{equation}
where we used the notation (subscript $0$) introduced in ~\eqref{def:f0} and the functions $\bar m, \bar M$ are given by Lemma \ref{lemma:def_bar_M_m_L}. We refer to~\eqref{eq:coh_eqs_ell} for the definition of the linear operators $\Lo$ and $\mathcal{L}^{\mathrm{ell}}_{\omega, \Omega}$.

We recall that $\Upsilon$ is the positive constant defined by~\eqref{def:const_Upsilon_ell}. We further recall that, throughout this work, $C(\cdot)$ stands for a generic positive constant depending on the parameter in brackets.
\begin{lemma}\label{lemma:DF_invertible}
    The linear operator $D_{\varphi} \mathcal{F}^T(0,0)$ given by ~\eqref{proof:Csigma_def_ell_DF_1} is invertible. Moreover, there exits a constant $\bar C$, depending on $\sigma,$ $\ell$, $\dec$ and $\Upsilon$, such that
    \[ \| D_{\varphi} \mathcal{F}^T(0,0)^{-1}\| < \bar C(\sigma, \ell,\dec,\Upsilon),\]
    where $\| \cdot \|$ denotes the operator norm.
\end{lemma}
\begin{proof}
    Let $g = (g_1, g_2, g_3) \in \mathcal{S}^{\mathrm{pol},T}_{(\sigma,0), \ell}(\R^n) \times \mathcal{S}^{\mathrm{pol},T}_{(\sigma,0), \ell+\dec+2}(\R^n) \times \mathcal{S}^{\mathrm{pol},T}_{(\sigma,0), \ell+\dec+1}(\R^{2m})$. To show that $D_{\varphi} \mathcal{F}^T(0,0)$ is invertible it suffices to show that $D_{\varphi} \mathcal{F}^T(0,0)(\hat u, \hat v, \hat w) = (g_1, g_2, g_3)$ admits a unique solution. We can rewrite this problem in terms of finding a unique solution $(\hat u, \hat v, \hat w) \in \mathcal{E}^{\mathrm{ell}, T}_{\sigma, \mathcal{K}} = \mathcal{U}^{\mathrm{pol},T}_{\sigma, \omega, \ell-1} \times \mathcal{U}^{\mathrm{pol},T}_{\sigma, \omega, \ell+\dec+1} \times \mathcal{W}^{\mathrm{ell},T}_{\sigma, \omega, \Omega, \ell+\dec}$ of the following system
    \begin{equation}
    \begin{cases}\label{proof:lemma_Csigma_invDF_eq}
        &\overline{M}_0 \cdot \hat v + m_0 \cdot \hat w - \Lo \hat u = g_1,\\ 
    &-\Lo \hat v=g_2,\\ 
    & \mathcal{L}^{\mathrm{ell}}_{\omega, \Omega} \hat w + J_m\overline{m}_0 \cdot \hat v = g_3.
    \end{cases}
    \end{equation} 
    Using Proposition \ref{prop:Csigma_HEomega}, a unique solution $\hat v \in \mathcal{U}^{\mathrm{pol},T}_{\sigma, \omega, \ell+\dec+1}$ of the second equation in~\eqref{proof:lemma_Csigma_invDF_eq} exists and satisfies
    \begin{equation}\label{est:Csigma_proof_inv_DF_sol_hatv}
        |\hat v|^T_{\sigma,\omega, \ell+\dec+1} \le {1 \over \ell+\dec+1} |g_2|^T_{\sigma, \ell+\dec+2}. 
    \end{equation}
    Now, in the third equation of~\eqref{proof:lemma_Csigma_invDF_eq}, $\hat v$ is known. Hence, thanks to Proposition \ref{prop:Csigma_HEomegaOmega}, there exists a unique solution $\hat w \in \mathcal{W}^{\mathrm{ell},T}_{\sigma, \omega, \Omega, \ell+\dec}$ of the third equation of~\eqref{proof:lemma_Csigma_invDF_eq}. Moreover, using the properties contained in Proposition \ref{prop:Csigma_prop_norms_pol},~\eqref{est:Csigma_proof_inv_DF_sol_hatv}, %
    Lemma \ref{lemma:def_bar_M_m_L} and recalling the definition of the norms $|\cdot |^T_{\sigma, \omega, \ell}$ and $|\cdot |^T_{\sigma, \omega, \Omega, \ell}$ in~\eqref{def:norm_U} and~\eqref{def:norm_W}, respectively, we have that  
    \begin{align}
        |\hat w|^T_{\sigma,\omega, \Omega, \ell+\dec} &\le {1 \over \ell+\dec}|g_3 -  J_m\overline{m}_0 \cdot \hat v|^T_{\sigma, \ell+\dec+1} \le C(\sigma, \ell, \dec)  \left(|g_3|^T_{\sigma, \ell+\dec+1} + \Upsilon  |\hat v|^T_{\sigma, \ell+\dec+1}\right)\nonumber\\ \label{est:Csigma_proof_inv_DF_sol_hatw}
        &\le C(\sigma, \ell, \dec, \Upsilon) \left(|g_3|^T_{\sigma, \ell+\dec+1} +  |g_2|^T_{\sigma, \ell+\dec+2}\right). 
    \end{align}

    It remains to analyze the first equation of~\eqref{proof:lemma_Csigma_invDF_eq} where $\hat v$ and $\hat w$ are known. Using again Proposition \ref{prop:Csigma_HEomega} a unique solution $\hat u \in \mathcal{U}^{\mathrm{pol},T}_{\sigma, \omega, \ell-1}$ of the first equation of~\eqref{proof:lemma_Csigma_invDF_eq} exists and, similarly to the previous case, using the properties contained in Proposition \ref{prop:Csigma_prop_norms_pol}, Lemma \ref{lemma:def_bar_M_m_L}, %
     ~\eqref{est:Csigma_proof_inv_DF_sol_hatv}, \,~\eqref{est:Csigma_proof_inv_DF_sol_hatw} and the definition of the norms~\eqref{def:norm_U} and~\eqref{def:norm_W}, we can verify that 
    \begin{align*}
        |\hat u|^T_{\sigma,\omega, \ell-1} &\le  {1 \over \ell-1}|g_1 - \overline{M}_0 \cdot \hat v - m_0 \cdot \hat w|^T_{\sigma, \ell} \\
        &\le C(\sigma, \ell) \left(|g_1|^T_{\sigma, \ell} + \Upsilon |\hat v|^T_{\sigma, \ell+\dec+1} + \Upsilon|\hat w|^T_{\sigma, \ell+\dec}\right) \\
        &\le C(\sigma,\ell, \dec, \Upsilon) \left(|g_1|^T_{\sigma, \ell} +  |g_2|^T_{\sigma, \ell+\dec+2} + |g_3|^T_{\sigma, \ell+\dec+1}\right)\\
        &=  C(\sigma,\ell, \dec, \Upsilon) |g|_{\boldsymbol{\sigma, \ell, \dec}}^T.
    \end{align*}
    This concludes the proof of this lemma. 
\end{proof}

Now we will study the dependence of $\| D_\varphi\mathcal{F}^T(P, \varphi) - D_\varphi\mathcal{F}^T(0, 0) \|$ on $P$ and $\varphi$. 

\begin{lemma}
    \label{lem:DF_limit}
    Let $P = (a, b, c, d) \in  \mathcal{P}^{\mathrm{ell}, T}_{\sigma,\mathcal{L}}$ and $\varphi_* = (u_*,v_*,w_*) \in \mathcal{E}^{\mathrm{ell}, T}_{\sigma, \mathcal{K}}$ with $| \varphi_*|_{\sigma, \mathcal{K}}^T < T^{\ell - 1}$. Then there exists a constant $C$, depending continuously on $\Upsilon, |P|^T_{\sigma, \mathcal{L}},$ and $| \varphi_*|_{\sigma,\mathcal{K}}^T$, such that the operator
    \[  D_\varphi \mathcal{F}^T(P, \varphi_*) - D_\varphi \mathcal{F}^T(0, 0):   \mathcal{E}^{\mathrm{ell}, T}_{\sigma, \mathcal{K}} \to \mathcal{S}^{\mathrm{pol},T}_{(\sigma,0), \ell}(\R^n) \times \mathcal{S}^{\mathrm{pol},T}_{(\sigma,0), \ell+\dec+2}(\R^n) \times \mathcal{S}^{\mathrm{pol},T}_{(\sigma,0), \ell+\dec+1}(\R^{2m}) \]
    satisfies
\begin{equation}
\label{eq:DF_difference_bound}
\| D\mathcal{F}^T(0, 0) -  D\mathcal{F}^T(P, \varphi_*) \| \leq \frac{C}{T^{\ell - 1}},
\end{equation}
where $\| \cdot\|$ denotes the operator norm. 

In particular, Item \eqref{prop:DF_limit} of Proposition \ref{prop:F_properties} holds.
\end{lemma}

\begin{proof}
Let $\varphi = (u, v, w) \in \mathcal{E}^{\mathrm{ell}, T}_{\sigma, \mathcal{K}}$. We have
\begin{equation*}
    \left(D_\varphi\mathcal{F}^T(0,0) - D_\varphi\mathcal{F}^T(P, \varphi_*)\right)\varphi = \begin{pmatrix}
        I_1 & I_2 & I_3
    \end{pmatrix}^\top
\end{equation*}
with
\begin{align*}
    I_1 &= -\partial_q b \circ \varphi_{*, q} u - \partial_q \overline{M} \circ \varphi_* \cdot (v_*, u) - \partial_q m \circ \varphi_{*,qz} \cdot (w_*, u) - \left(\overline{M}\circ \varphi_* - \overline{M}_0\right)v\\
    &- \partial_p \overline{M}\circ \varphi_*\cdot (v_*, v) - \partial_z \overline{M} \circ \varphi_* \cdot (v_*, w) - \left(m \circ \varphi_{*,qz} - m_0\right)w - \partial_z m \circ \varphi_{*,qz}\cdot (w_*, w)\\
    I_2 &= \partial_q^2 a \circ \varphi_{*, q} u + \partial_q^2 b\circ \varphi_{*, q} \cdot (v_*, u) + \partial_q^2 c \circ \varphi_{*, q} \cdot (w_*, u) + {1 \over 2}\partial_q^2 d \circ \varphi_{*, q} \cdot (w_*, w_*, u)\\
    &+\partial_q^2 M \circ \varphi_*\cdot(v_*, v_*, u) + \partial_q^2 m \circ \varphi_{*, qz} \cdot (v_*, w_*, u) + \partial_q^2 L \circ \varphi_{*, qz}\cdot (w_*, w_*, w_*, u)\\
    &+ \partial_q b \circ \varphi_{*, q} v + 2\partial_q M \circ \varphi_* \cdot (v_*, v) + \partial_p \partial_q M \circ \varphi_* \cdot (v_*, v_*, v) + \partial_q m \circ \varphi_{*, qz}\cdot (w_*, v)\\
    &+ \partial_q c \circ \varphi_{*, q} w + \partial_q d \circ \varphi_{*, q} \cdot (w_*, w) + \partial_z \partial_q M \circ \varphi_* \cdot (v_*, v_*, w) + \partial_z \partial_q m \circ \varphi_{*, qz}\cdot (v_*, w_*, w)\\
    &+\partial_q m \circ \varphi_{*, qz}\cdot(v_*, w) + \partial_z \partial_q L \circ \varphi_{*, qz} \cdot (w_*, w_*, w_*, w) +3\partial_q L\circ \varphi_{*, qz}\cdot (w_*, w_*, w)\\
    I_3 &= - J_m\partial_q c\circ\varphi_{*, q} u - J_m \partial_q d \circ\varphi_{*, q} \cdot (w_*, u) - J_m \partial_q \partial_z M\circ \varphi_*\cdot (v_*, v_*, u) \\
    &- J_m \partial_q \overline{m}\circ\varphi_{*, qz}\cdot (v_*, u) - J_m \partial_q \overline{L}\circ \varphi_{*, qz}\cdot (w_*, w_*, u) - J_m \partial_p \partial_z M\circ\varphi_* \cdot (v_*, v_*, v)\\
    &- J_m\partial_z M \circ\varphi_* \cdot (v_*, v) - J_m \partial_z M\circ\varphi_*\cdot (v_*, v) - J_m \left(\overline{m}\circ\varphi_{*, qz} - \overline{m}_0\right) v - J_m d\circ\varphi_{*, q}w \\
    &- \partial_z^2 M \circ\varphi_* \cdot (v_*, v_*, w) - J_m \partial_z \overline{m}\circ\varphi_{*, qz}\cdot (v_*, w) - J_m \partial_z \overline{L}\circ\varphi_{*, qz}\cdot(w_*, w_*, w)\\
    &- 2 J_m \overline{L}\circ\varphi_{*, qz}\cdot (w_*, w).
\end{align*}
where we used the notation introduced in~\eqref{def:f0} and~\eqref{def:ft} and we denoted by a subscript the projection of $\varphi_*$ on the $q, p$ and $z$.

Noticing that the decrease rates on $t$ add up when taking products and that they are preserved after composing with $\varphi^t$, using Proposition \ref{prop:Csigma_prop_norms_pol} and Lemma \ref{lemma:def_bar_M_m_L}, a straightforward calculation shows that 
\[I_1 \in \mathcal{S}^{\mathrm{pol},T}_{(\sigma,0), 2\ell - 1},\qquad I_2 \in \mathcal{S}^{\mathrm{pol},T}_{(\sigma,0), 2\ell + \dec+ 1},\qquad  I_3 \in \mathcal{S}^{\mathrm{pol},T}_{(\sigma,0), 2\ell + \dec}.\] 
Moreover, their respective norms in these spaces are bounded by $C|\varphi|_{\sigma, \mathcal{K}}^T$, where $C$ is a constant depending continuously on $\sigma, \ell, \dec, m, n, \Upsilon, |P|_{\sigma, \Lell}^T,$ and $| \varphi_*|_{\sigma, \mathcal{K}}^T$. 

We will write in detail this calculation only for $I_1$, the other two being completely analogous. For this purpose, we observe that 
\begin{align*}
    \overline{M}\circ \varphi_* - \overline{M}_0 &= \int_0^1 \partial_q \overline{M}(\mathrm{id} + \tau u_*, \tau v_*, \tau w_*)u_* d\tau + \int_0^1\partial_p \overline{M}(\mathrm{id} + \tau u_*, \tau v_*, \tau w_*)v_*d \tau \\
    &+ \int_0^1\partial_z \overline{M}(\mathrm{id} + \tau u_*, \tau v_*, \tau w_*)w_* d\tau,
\end{align*}
and thus $ \overline{M}\circ \varphi_* - \overline{M}_0 \in \mathcal{S}^{\mathrm{pol},T}_{(\sigma ,0), \ell -1}$. Similarly $m \circ \varphi_{*,qz} - m_0 \in \mathcal{S}^{\mathrm{pol},T}_{(\sigma ,0), \ell -1}$. Notice that their norms in this space are bounded from above by $C(\sigma, \ell, m, n)\Upsilon |\varphi_*|_{\sigma, \mathcal{K}}^T$.

By  Proposition \ref{prop:Csigma_prop_norms_pol}, we have
\begin{align*}
\label{proof:ell_FP1_first_est_I_1}
    |I_1|^T_{\sigma, 2\ell - 1} &\le  C(\sigma, \ell,\dec,m,n)\Big(|\partial_q b \circ \varphi_{*, q}|^T_{\sigma,\ell} |u|^T_{\sigma,\ell - 1} + \big|\partial_q \overline{M} \circ \varphi_*\big|^T_{\sigma,0}|v_*|^T_{\sigma,\ell + \dec+1}|u|^T_{\sigma,\ell - 1} \\
    & \quad +|\partial_q m \circ \varphi_{*,qz}|^T_{\sigma,0}|w_*|^T_{\sigma,\ell+\dec} |u|^T_{\sigma,\ell - 1} + \big|\overline{M} \circ \varphi_* - \overline{M}_0\big|^T_{\sigma,\ell - 1}|v|^T_{\sigma,\ell +\dec+ 1} \\
    & \quad +\big| \partial_p \overline{M}\circ \varphi_*|^T_{\sigma,0}|v_*|^T_{\sigma,\ell + \dec+1}|v|^T_{\sigma,\ell +\dec+ 1} +\big| \partial_z \overline{M} \circ \varphi_*\big|^T_{\sigma,0} |v_*|^T_{\sigma,\ell +\dec+ 1}|w|^T_{\sigma,\ell+\dec} \\
    &  \quad + |m \circ \varphi_{*,qz}  + m_0|^T_{\sigma,\ell - 1}|w|^T_{\sigma,\ell+\dec} +  |\partial_z m \circ \varphi_{*,qz}|^T_{\sigma,0} |w_*|^T_{\sigma,\ell+\dec}|w|^T_{\sigma,\ell+\dec} \Big)\\
    & \leq C(\sigma, \ell,\dec, m, n, \Upsilon, |P|_{\sigma, \Lell}^T, | \varphi_*|_{\sigma, \mathcal{K}}^T)| \varphi|_{\sigma,  \mathcal{K}}^T,
\end{align*}
where $C$ depends continuously on $\Upsilon, |P|_{\sigma, \Lell}^T,$ and  $|\varphi_*|_{\sigma, \mathcal{K}}^T.$

Therefore,
\begin{align*}
 \big|(D_\varphi\mathcal{F}^T(0,0) - D_\varphi\mathcal{F}^T(P, \varphi_*))\varphi\big|_{\boldsymbol{\sigma, \ell,\dec}}^T & \le |I_1|^T_{\sigma, \ell} + |I_2|^T_{\sigma, \ell + \dec+2} + |I_3|^T_{\sigma, \ell +\dec+ 1}  \\
 & \leq \frac{1}{T^{\ell - 1}} \left( |I_1|^T_{\sigma, 2\ell - 1} + |I_2|^T_{\sigma, 2\ell +\dec + 1} + |I_3|^T_{\sigma, 2\ell +\dec} \right) \\
 & \leq \frac{C}{T^{\ell - 1}} |\varphi|_{\sigma, \mathcal{K}}^T.
\end{align*}
\end{proof}

\subsection{Proof of Item \eqref{thm:A_transverse} of Theorem \ref{Thm:elliptic_Csigma}: Existence of asymptotic elliptic transversal dynamics}\label{sec:proof_thm_a_ell_2}

In this section, we prove that the asymptotic KAM torus obtained in Item \eqref{thm:A_existence} of Theorem \ref{Thm:elliptic_Csigma} is asymptotically elliptic under slightly stronger assumptions on the regularity (more precisely, on $\sigma$) and on the decay rate of the perturbations (see \eqref{eq:decay_ell_2}).

In the following, we continue using the notation and parameters from the previous subsections, with the only exception that we now assume $\sigma \geq 2$.

Recall that the perturbations considered in Item \eqref{thm:A_transverse} of Theorem \ref{Thm:elliptic_Csigma} can be identified with the space $\mathcal{P}^{\mathrm{ell}, 1}_{\sigma, \Lell^*}$ (see Section \ref{sc:matrix_spaces_ell}), with $\mathcal{L}^*$ as fixed in Section \ref{sc:initial_ham_ell}. Moreover, for any $T \geq 1$, we have $\mathcal{P}^{\mathrm{ell}, T}_{\sigma, \Lell^*} \subseteq \mathcal{P}^{\mathrm{ell}, T}_{\sigma, \Lell(\ell, \dec + 1)}$, where $\mathcal{L}(\ell, \dec + 1) \in \left(\R_{\ge 0}\right)^4$ is given by \eqref{eq:index_regularity_perturbations}, and
\[ |P|^T_{\sigma, \Lell(\ell, \dec + 1)} \leq |P|^T_{\sigma, \Lell^*} , \qquad \text{ for any }P \in \mathcal{P}^{\mathrm{ell}, T}_{\sigma, \Lell^*}.\] 

Thus, denoting $\mathcal{K}^* = \mathcal{K}(\ell, \dec + 1)$ where $\mathcal{K}(\ell, \dec + 1) \in \left(\R_{\ge 0}\right)^3$ is given by \eqref{eq:index_regularity_tori}, by Proposition \ref{prop:A_existence}, for any $r > 0$ there exist $T_0 \geq 1$ such that, for any $T \geq T_0$, there exists a $C^1$-map
\[ \boldsymbol{\varphi}^T: B_{\mathcal{P}^{\mathrm{ell}, T}_{\sigma, \Lell^*}}(r)   \to \mathcal{E}^{\mathrm{ell}, T}_{\sigma, \mathcal{K}^*}\]
for which $\boldsymbol{\varphi}^T(P)$ defines a $C^\sigma$ asymptotic KAM torus associated with $(X_{H_0 + P}, X_{H_0}, \varphi_0)$,  for any $P \in  \mathcal{P}^{\mathrm{ell}, T}_{\sigma, \Lell^*} $ with $|P|^T_{\sigma, \Lell^*} < r$.

Recall that the asymptotic KAM torus $\boldsymbol{\varphi}^T(P)$ associated to the perturbation $P$ is said to be \emph{elliptic} if there exist a family of invertible matrices $S:\T^n \times I_{T'} \to \mathcal{M}_{2n+2m}(\R)$, for some $T' \geq T$ as in Definition \ref{def:ell_hyp_par_trasv_dyn_asym_KAM}.  We will construct these invertible matrices using Lemma \ref{lem:embedding_symplectomorphisms_family}, which to any continuous map $G: \T^n \times I_T \to \mathfrak{sp}(2m, \R)$ associates a continuous family of symplectic matrices $\boldsymbol{S}(\boldsymbol{\varphi}^T(P), G): \T^n \times I_T \to Sp(2m + 2n, \R, J)$ of the form \eqref{eq:S_formula}. Moreover, it follows from \eqref{eq:S_formula} that if  $\boldsymbol{\varphi}^T(P)$ and $G$ are of class $C^1$ then $\boldsymbol{S}(\boldsymbol{\varphi}^T(P), G)$ is also of class $C^1$.

Therefore, by Proposition \ref{prop:criteria_non_autonomous}, in order for $S = \boldsymbol{S}(\boldsymbol{\varphi}^T(P), G)$, given by Lemma \ref{lem:embedding_symplectomorphisms_family}, to satisfy the assumptions of Definition \ref{def:ell_hyp_par_trasv_dyn_asym_KAM}, it suffices to assume that $G$ and $\varphi = \boldsymbol{\varphi}^T(P)$ are sufficiently regular and satisfy suitable decay conditions, and to show that the function $\boldsymbol{\zeta}(H_0 + P, \varphi, G)$, defined by \eqref{eq:formula_zeta}, is equal to $\Mell$, namely,
\[
\Pi_{xy}^\top S^\top\big( (D^2 (H_0 + P) \circ \varphi) S + J \Lo S\big) \Pi_{xy} = \Mell.
\]

We will show the following.

\begin{proposition}
\label{prop:A_transverse}
Fix $r > 0$ and let $T_0 \geq 1$ be given by Proposition \ref{prop:A_existence}. There exist $T_1 \geq T_0$ such that, for any $T \geq T_1$,  there exists a $C^1$-map
\[ \boldsymbol{G}^T: B_{\mathcal{P}^{\mathrm{ell}, T}_{\sigma, \Lell^*}}(r)    \subseteq  \mathcal{P}^{\mathrm{ell}, T}_{\sigma, \Lell^*} \to \mathcal{S}p^{\mathrm{ell}, T}_{\sigma-1, \omega, \Omega, \ell + \dec}  \]
for which $\boldsymbol{S}(\boldsymbol{\varphi}^T(P),  \boldsymbol{G}^T(P))$ given by \eqref{eq:S_formula} is of class $C^1$ and
\begin{equation}
\label{eq:IFT_condition_ell}
\boldsymbol{\zeta}(H_0 + P, \boldsymbol{\varphi}^T(P),  \boldsymbol{G}^T(P)) =  \Mell, \qquad \text{for any } P \in B_{\mathcal{P}^{\mathrm{ell}, T}_{\sigma, \Lell^*}}(r),
\end{equation}
where $\boldsymbol{\zeta}$ is given by \eqref{eq:formula_zeta} and  $\boldsymbol{\varphi}^T: B_{\mathcal{P}^{\mathrm{ell}, T}_{\sigma, \Lell^*}}(r)   \to \mathcal{E}^{\mathrm{ell}, T}_{\sigma, \mathcal{K}^*}$  is given by Proposition \ref{prop:A_existence}.

In particular, by Proposition \ref{prop:criteria_non_autonomous}, for any $P \in B_{\mathcal{P}^{\mathrm{ell}, T}_{\sigma, \Lell^*}}(r)$, the asymptotic KAM torus $\boldsymbol{\varphi}^T(P)$ is asymptotically elliptic.
\end{proposition}

Notice that Item \eqref{thm:A_transverse} of Theorem \ref{Thm:elliptic_Csigma} follows directly from the proposition above, since
\begin{equation*}
    \big|P \big|_{\sigma, \Lell^*}^{T} \leq |P|_{\sigma, \Lell^*}^{1}, \qquad \text{ for any } T \geq 1\,  \text{ and any } \, P \in \mathcal{P}^{\mathrm{ell}, 1}_{\sigma, \Lell^*},
\end{equation*}
and, by definition of $\mathcal{K}^\ast$, the asymptotic KAM torus has the desired decay rates.

We refer the reader to Section \ref{sc:matrix_spaces_ell} for the definition of the Banach space of symplectic matrices $\big( \mathcal{S}p^{\mathrm{ell}, T}_{\sigma-1, \omega, \Omega, \ell+ \dec}, |\cdot|^T_{\sigma - 1, \omega, \Omega, \ell + \dec} \big)$.

To prove Proposition \ref{prop:A_transverse} we will follow the same strategy as in the previous section, that is, we will define a functional encoding the asymptotic ellipticity of the transversal dynamics and we will obtain our result as a consequence of an appropriate quantitative version of the Implicit Function Theorem (see Theorem \ref{thm:QIFT}).

The remainder of this section concerns the proof of Proposition \ref{prop:A_transverse}.

In the following, we fix $r > 0$ and let $T_0 \geq 1$ be given by Proposition \ref{prop:A_existence}. Moreover, for any $T \geq T_0$, we will denote by $\boldsymbol{\varphi}^T: B_{\mathcal{P}^{\mathrm{ell}, T}_{\sigma, \Lell^*}}(r)   \to \mathcal{E}^{\mathrm{ell}, T}_{\sigma, \mathcal{K}^*}$ the map given by Proposition \ref{prop:A_existence}.

\subsubsection{Functionals describing elliptic asymptotic KAM tori}\label{sc:functional_transversal}

We will define a family of functionals $\mathcal{G}^T$, for $T \geq T_0$, having as arguments a perturbation $P \in \mathcal{P}^{\mathrm{ell}, T}_{\sigma, \Lell^*} $ and a matrices $G \in \mathcal{S}p^{\mathrm{ell}, T}_{\sigma-1, \omega, \Omega, \ell + \dec}$, such that, for $\boldsymbol{S}(\boldsymbol{\varphi}^T(P), G)$ given by Lemma \ref{lem:embedding_symplectomorphisms_family}, we have
\begin{equation}
\label{eq:functional_asymptotic_characterization_ell}
      \mathcal{G}^T(P, G) = 0 \quad \mbox{if and only if $\boldsymbol{\zeta}(H_0 + P, \boldsymbol{\varphi}^T(P),  \boldsymbol{G}^T(P)) =  \Mell$}.
\end{equation}
where $\boldsymbol{\zeta}$ is given by \eqref{eq:formula_zeta}. Taking into account the explicit formula for $\boldsymbol{\zeta}$ given in Remark \ref{rmk:formula_zeta}, we define 
\[\mathcal{G}^T: B_{\mathcal{P}^{\mathrm{ell}, T}_{\sigma,\Lell^*}}(r) \times \mathcal{S}p^{\mathrm{ell}, T}_{\sigma-1, \omega, \Omega, \ell+\dec}   \to  \mathcal{S}ym^{\mathrm{pol}, T}_{\sigma-1, \ell+\dec+1}\]
as
\begin{equation}
\label{eq:formula_F_ell_2}
    \begin{aligned}
    \mathcal{G}^T(P, G) =&  B^\top \partial_p^2 H\circ \varphi B + K^{\top} \partial_p\partial_zH\circ \varphi B + B^{\top} \left(\partial_p\partial_zH\circ \varphi\right)^\top K \\
    & + K^{\top} \partial_z^2 H\circ \varphi K + K^{\top} J_m \Lo K - M^{\mathrm{ell}}_\Omega,
    \end{aligned}
\end{equation}
where $H = H_0 + P$, $\varphi = \boldsymbol{\varphi}^T(P) = (u, v, w)$, $K = \exp(G)$, the matrix $M^{\mathrm{ell}}_\Omega$ is defined in~\eqref{Ms},  $B = -(\mathrm{Id}_{2n} +\partial_q  u )^{-\top} (\partial_q w)^\top J_m K$, and $\partial_p\partial_zH\circ \varphi$ stands for the $2m \times n$ matrix having components $\partial_{p_i}\partial_{z_j}H\circ \varphi$ for $1 \le i \le n$ and $1 \le j \le 2m$. We refer the reader to Section \ref{sc:matrix_spaces_ell} for the definition of the space $\big(\mathcal{S}ym^{\mathrm{pol}, T}_{\sigma-1, \ell+\dec+1}, |\cdot|^T_{\sigma - 1, \ell +\dec+ 1}  \big)$ taking values in the set of $2m \times 2m$ symmetric matrices. 

Therefore, with these definitions,
\begin{equation}
\label{eq:elliptic_characterization}
    \text{ If }  \quad \mathcal{G}^T(P, G) = 0 \quad \text{ then } \quad  \boldsymbol{\varphi}^T(P) \text{ is an asymptotically elliptic KAM torus.}
\end{equation}

\subsubsection{Proof of the existence of elliptic transversal dynamics} \label{sc:thm_A_transverse_proof} In this section we will provide a proof of Proposition \ref{prop:A_transverse} assuming the following properties of the functional $\mathcal{G}^T$, which we will prove in the next section. For the rest of this section, we denote by $D_G\mathcal{G}^T$ the differential of the functional $\mathcal{G}^T$ with respect to $G$.

\begin{proposition}
\label{prop:G_properties}
    The operator $\mathcal{G}^T$ given by \eqref{eq:formula_F_ell_2} satisfies the following.
    \begin{enumerate}
    \item \label{prop:G_well_defined} $\mathcal{G}^T$ and $D_G\mathcal{G}^T$ are well-defined continuous maps, for any $T \geq T_0$. 
    
    \item \label{prop:G_bounded} There exists a constant $C$ such that, for any $T \geq T_0$,
\[ \sup \left\{ |\mathcal{G}^T(P, 0)|^T_{\sigma - 1, \ell +\dec+ 1} \,\left|\, P \in B_{\mathcal{P}^{\mathrm{ell}, T}_{\sigma,\Lell^*}}(r) \right\}\right. \leq Cr.\] %

    \item \label{prop:DG_invertible} $D_G \mathcal{G}^T(0, 0)$ is invertible, for any $T \geq T_0$. Moreover, there exists a constant $\bar C$, independent of $T$, such that $\|D_G \mathcal{G}^T(0, 0)^{-1}\| < \bar C$. 
        \item \label{prop:DG_limit} Let $r, \rho > 0$. Then  
        \[\lim_{T \to +\infty} \| D_G \mathcal{G}^T(P, G) - D_G\mathcal{G}^T(0, 0) \| = 0, \qquad \text{uniformly on } \quad B_{\mathcal{P}^{\mathrm{ell}, T}_{\sigma,\Lell^*}}(r) \times B_{\mathcal{S}p^{\mathrm{ell}, T}_{\sigma-1, \omega, \Omega, \ell + \dec}}(\rho),\]
where $\| \cdot\|$ denotes the operator norm and $B_{\mathcal{S}p^{\mathrm{ell}, T}_{\sigma-1, \omega, \Omega, \ell + \dec}}(\rho) \subset \mathcal{S}p^{\mathrm{ell}, T}_{\sigma-1, \omega, \Omega, \ell + \dec}$ stands for a ball of radius $\rho$ centered at the origin.
    \end{enumerate}
\end{proposition}

The properties above will be proven in Lemmas \ref{lemma:G_well_defined}, \ref{lemma:G_bounded}, \ref{lemma:DG_invertible} and \ref{lem:DG_limit}, respectively. As for  Proposition \ref{prop:F_properties}, Properties \eqref{prop:G_well_defined}, \eqref{prop:G_bounded}  and \eqref{prop:DG_limit} of Proposition \ref{prop:G_properties} follow by straightforward calculations, while Property \eqref{prop:DG_invertible} is much more subtle. 

Assuming Proposition \ref{prop:G_properties}, the proof of Proposition \ref{prop:A_transverse} is very similar to that of Proposition \ref{prop:A_existence} (for which we assumed Proposition \ref{prop:F_properties}). For the sake of completeness, we reproduce it here.

\begin{proof}[Proof of Proposition \ref{prop:A_transverse}]
Let $T \geq T_0$ and define $\rho = 2C\bar C r,$ where $C$ and $\bar C$ are the constants given by Items \eqref{prop:G_bounded} and \eqref{prop:DG_invertible} of Proposition \ref{prop:G_properties}.

By Items \eqref{prop:G_well_defined} and \eqref{prop:DG_limit} of Proposition \ref{prop:G_properties} there exists $T_1 \geq T_0$ such that, for any $T \geq T_1$,  $B_{\mathcal{P}^{\mathrm{ell}, T}_{\sigma,\Lell^*}}(r) \times \mathcal{S}p^{\mathrm{ell}, T}_{\sigma-1, \omega, \Omega, \ell+\dec}$ is contained in the domain of definition of $\mathcal{G}^T$ and 
\[  \| D\mathcal{G}^T(P, \varphi) - D\mathcal{G}^T(0, 0) \| < \frac{1}{2\bar C}, \qquad \text{ for } \quad |P|^T_{\sigma, \Lell^*} < r, \quad |G|_{\sigma - 1, \omega, \Omega, \ell+\dec}^T < \rho.\]

Let $T \geq T_1$. By Items \eqref{prop:G_bounded} and \eqref{prop:DG_invertible} of Proposition \ref{prop:G_properties},
\[\|D_\varphi \mathcal{G}^T(0, 0)^{-1}\| \leq \bar C, \qquad \sup \left\{ |\mathcal{G}^T(P, 0)|^T_{\sigma - 1, \ell +\dec+ 1} \,\left|\, P \in B_{\mathcal{P}^{\mathrm{ell}, T}_{\sigma,\Lell^*}}(r)  \right\}\right. \leq \frac{\rho}{2\bar C}.\]

Noticing that $\mathcal{G}^T(0,0)=0$ and recalling \eqref{eq:functional_asymptotic_characterization_ell}, the existence of a map
\[
\boldsymbol{G}^T: B_{\mathcal{P}^{\mathrm{ell}, T}_{\sigma, \Lell^*}}(r) \to \mathcal{S}p^{\mathrm{ell}, T}_{\sigma-1, \omega, \Omega, \ell + \dec}
\]
satisfying \eqref{eq:IFT_condition_ell} follows from Theorem \ref{thm:QIFT}. 

Since $\boldsymbol{\varphi}^T(P) = (\mathrm{id} + u, v, w) \in \mathcal{E}^{\mathrm{ell}, T}_{\sigma, \mathcal{K}^*}$, the functions $\partial_q u$, $\partial_q v$, and $\partial_q w$ are of class $C^1$ (see Remark \ref{rmk:extra_regularity_ell}). Moreover, by construction, $\boldsymbol{G}^T(P) \in \mathcal{S}p^{\mathrm{ell}, T}_{\sigma-1, \omega, \Omega, \ell + \dec}$ and is therefore also of class $C^1$. Hence, $\boldsymbol{S}(\boldsymbol{\varphi}^T(P), \boldsymbol{G}^T(P))$, given by \eqref{eq:S_formula}, is of class $C^1$. Finally, it suffices to notice that $w \in \mathcal{S}^{\mathrm{pol}, T}_{(\sigma, 0), \ell + \dec + 1}$ and $G \in \mathcal{S}^{\mathrm{pol}, T'}_{(\sigma - 1, 0), \ell + \dec}$, and then apply Proposition \ref{prop:criteria_non_autonomous}.
\end{proof}

The remainder of this subsection concerns the proof of Proposition \ref{prop:G_properties}.

\subsubsection{Properties of the functional $\mathcal{G}^T$} We start by checking that $\mathcal{G}^T$ as in \eqref{eq:formula_F_ell_2} is well-defined.

\begin{lemma}\label{lemma:G_well_defined}
   Let $T \geq T_0$. The functional $\mathcal{G}^T$ given by \eqref{eq:formula_F_ell_2} and its partial derivative $D_G\mathcal{G}^T$ are well-defined continuous maps.
\end{lemma}

\begin{proof}
   Let us verify that $\mathcal{G}^T$ takes values in $\mathcal{S}ym^{\mathrm{pol}, T}_{\sigma-1, \ell+\dec+1}.$ First, we check that, for any $P, G$, the matrix $\mathcal{G}^T(P, G)(q, t)$ given by \eqref{eq:formula_F_ell_2} is symmetric. In the following, if there is no risk of confusion, we omit writing the variables $P, G, q, t$. By \eqref{eq:formula_F_ell_2}, it suffices to check that $K^{\top} J_m \Lo K$ is symmetric.
  
By taking derivatives with respect to $q$ and $t$ in $K(q, t) K(q, t)^{-1} = \mathrm{Id}_{2m}$, it easily follows that 
\[ (\Lo K)  K^{-1} = -K (\Lo K^{-1}).\]
Hence, recalling that $K^{\top} J_m = J_m K^{-1}$, we have
\begin{align*}
 (K^{\top} J_m \Lo K)^{\top} & = - (\Lo K^{\top})J_m K = - J_m(\Lo K^{-1}) K = J_mK^{-1} \Lo K = K^{\top} J_m \Lo K,
\end{align*} 
and thus $\mathcal{G}^T(P, G)(q, t)$ is symmetric.

Now, let us check that $\mathcal{G}^T(P, G) \in\mathcal{M}^{\mathrm{pol}, T}_{\sigma-1, \ell+\dec+1}.$ We have
\begin{equation}
\begin{aligned}\label{eq:der_order_2_Ham}
    \partial_p^2 H & = \overline{\overline{M}},\\
    \partial_p\partial_z H & = \overline{m} + \partial_z \overline{M}\cdot p,  \\
    \partial_z^2 H & = M^{\mathrm{ell}}_\Omega + d +  \partial_z \overline{m} \cdot p + \partial_z^2 M^t \cdot p^2 + \overline{\overline{L}}^t \cdot z ,
\end{aligned}
\end{equation}
where the matrix $M^{\mathrm{ell}}_\Omega$ is defined in~\eqref{Ms}, and $\overline{M}$, $\overline{\overline{M}}$, $\overline{m}$, and $\overline{\overline{L}}$ in Lemma \ref{lemma:def_bar_M_m_L}.

Since $B_{ij} \in \mathcal{S}^{\mathrm{pol}, T}_{(\sigma-1, 0), \ell+\dec+1}$ for all $1 \le i \le n$ and $1 \le j \le 2m$, it is easy to check that all the terms containing the matrix $B$ in the definition of $\mathcal{G}^T$ belong to $\mathcal{M}^{\mathrm{pol}, T}_{\sigma-1, \ell+\dec+1}.$

Thus it remains to verify that  $K^{\top} \partial_z^2 H K   + K^{\top} J_m \Lo K - M^{\mathrm{ell}}_\Omega \in \mathcal{M}^{\mathrm{pol}, T}_{\sigma-1, \ell+\dec+1}.$ For this purpose, we notice that, 
\begin{align*}
    &K^\top \partial_z^2 H\circ \varphi K   + K^\top J_m \Lo K - M^{\mathrm{ell}}_\Omega \\
    &= K^\top \left(\partial_z^2 H\circ \varphi - M^{\mathrm{ell}}_\Omega\right)K \\
    &+K^\top \left[(K - \mathrm{Id}_{2m})^\top M_\Omega^{\mathrm{ell}} + M_\Omega^{\mathrm{ell}}(K - \mathrm{Id}_{2m}) + J_m \Lo (K - \mathrm{Id}_{2m})\right]\\
    &-(K - \mathrm{Id}_{2m})^\top(K - \mathrm{Id}_{2m})^\top M_{\Omega}^{\mathrm{ell}},
\end{align*}
where the latter is obtained by adding and subtracting $K^\top M^{\mathrm{ell}}_\Omega K$ and $K^\top \left(K - \mathrm{Id}_{2m}\right)^\top M_{\Omega}^{\mathrm{ell}}$ and using the trivial equality $\Lo K = \Lo(K - \mathrm{Id}_{2m})$.

Since $$K - \mathrm{Id}_{2m} - G = G^2 \int_0^1(1 - s)\exp(sG)ds \in \mathcal{M}_{\sigma - 1, 2\ell + 2\dec}^{T} \subseteq \mathcal{M}_{\sigma - 1, \ell + \dec+1}^{T}$$ and $G^\top M_\Omega^{\mathrm{ell}} + M_\Omega^{\mathrm{ell}} G^\top + J_m\Lo G \in  \mathcal{M}_{\sigma - 1, \ell + \dec+1}^{T}$ (see~\eqref{def:BS_ell_Sp}), it follows that $$(K - \mathrm{Id}_{2m})^\top M_\Omega^{\mathrm{ell}} + M_\Omega^{\mathrm{ell}}(K - \mathrm{Id}_{2m}) + J_m \Lo (K - \mathrm{Id}_{2m}) \in \mathcal{M}^{T}_{\sigma-1, \ell +\dec+1}.$$

Noticing that $\partial_z^2 H\circ \varphi - M^{\mathrm{ell}}_\Omega \in \mathcal{M}^{\mathrm{pol}, T}_{\sigma, \ell +\dec+1}$ and $K - \mathrm{Id}_{2m} \in \mathcal{M}^{\mathrm{pol}, T}_{\sigma-1, \ell+\dec}$,  we conclude that the functional $\mathcal{G}^T$ is well-defined. Concerning the regularity of $\mathcal{G}^T$ and $D_G \mathcal{G}^T$, it follows directly from~\eqref{eq:formula_F_ell_2}, together with the regularity assumptions on the perturbative terms  $P=(a,b,c,d) \in \mathcal{P}^{\mathrm{ell}, T}_{\sigma,\mathcal{L}^*}$, the components of the $C^\sigma$ asymptotic KAM torus $\varphi =  \boldsymbol{\varphi}^T(P) = (u,v,w) \in \mathcal{E}^{\mathrm{ell}, T}_{\sigma, \mathcal{K}^*}$ and the matrix $G \in \mathcal{S}p^{\mathrm{ell}, T}_{\sigma-1, \omega, \Omega, \ell+\dec}$.
\end{proof}

\begin{lemma}
    \label{lemma:G_bounded}
   There exists a constant $C$, such that for any $T \geq T_0$,
\[ \sup \left\{ |\mathcal{G}^T(P, 0)|^T_{\sigma - 1, \ell + \dec+ 1} \,\left|\, P \in B_{\mathcal{P}^{\mathrm{ell}, T}_{\sigma,\Lell^*}}(r) \right\}\right. \leq Cr.\]
\end{lemma}
\begin{proof}

For any $T \geq T_0$ and any $P \in B_{\mathcal{P}^{\mathrm{ell}, T}_{\sigma, \Lell^*}}(r)$, we have
\[  \mathcal{G}^T(P, 0) =  B^\top \partial_p^2 H\circ \varphi B +  \partial_p\partial_zH\circ \varphi B + B^{\top}  \left(\partial_p\partial_zH\circ \varphi\right)^\top  +  \partial_z^2 H\circ \varphi  - M^{\mathrm{ell}}_\Omega, \] 
where $H = H_0 + P$, $\varphi = \boldsymbol{\varphi}^T(P) = (u, v, w)$, and  $B = -(\mathrm{Id}_{2n} +\partial_q  u )^{-T} \partial_q w J_m K.$ 

Recall that by definition of $\boldsymbol{\varphi}^T$ and Proposition \ref{prop:A_existence} we have $| \boldsymbol{\varphi}^T(P)|^T_{\sigma, \mathcal{K}^*} < C_0r$, where $C_0$ is a constant independent of $T$, $P$ and $r$. Hence there exists a constant $C$, independent of $T$ and $r$, such that
\[ \sup \left\{ |\mathcal{G}^T(P, 0)|^T_{\sigma - 1, \ell + \dec+1} \,\left|\, P \in B_{\mathcal{P}^{\mathrm{ell}, T}_{\sigma,\Lell^*}}(r) \right\}\right.  < Cr \]
\end{proof}

We will now consider the differential of $\mathcal{G}^T$ with respect to $G$ evaluated at $(0,0)$. Observe that $\mathcal{G}^T(0, 0) = 0$. A direct calculation shows that 
\[ \mathcal{G}^T(0, G) = \exp(G)^{\top} M^{\mathrm{ell}}_\Omega \exp(G) + \exp(G)^{\top} J_m \Lo \exp(G) - M^{\mathrm{ell}}_\Omega.\]
Thus, the differential of $\mathcal{G}^T$ with respect to $G$ evaluated at $(0,0)$ is given by
\begin{equation}\label{eq:diff_form_F_ell_2}
\Function{D_{G}\mathcal{G}^T(0,0)}{\mathcal{S}p^{\mathrm{ell}, T}_{\sigma-1, \omega, \Omega, \ell+\dec}}{ \mathcal{S}ym^{\mathrm{pol}, T}_{\sigma-1, \ell+\dec+1}}{G}{\mathfrak{L}^{\mathrm{ell}}_{\omega, \Omega} (G) = G^{\top}M^{\mathrm{ell}}_\Omega + M^{\mathrm{ell}}_\Omega G + J_m \Lo G},
\end{equation}
where we refer to~\eqref{eq:coh_eqs_ell} for the definition of the linear operator $\mathfrak{L}^{\mathrm{ell}}_{\omega, \Omega} $.

We recall that, throughout this work, we denote by $C(\cdot)$ a generic positive constant depending on the parameter in brackets. 
\begin{lemma}\label{lemma:DG_invertible}
    Let $T \geq T_0$. The linear operator $D_G \mathcal{G}^T(0,0)$ given by \eqref{eq:diff_form_F_ell_2} is invertible. Moreover, there exits a constant $\bar C$, depending only on $\ell$ and $\dec$, such that
    \[ \| D_{G} \mathcal{G}^T(0,0)^{-1}\| < \bar C(\ell, \dec),\]
    where $\| \cdot \|$ denotes the operator norm.
\end{lemma}

\begin{proof}
This follows directly from Proposition \ref{prop:Csigma_HEomegaOmegaOmega}.
  \end{proof}

Now we will study the dependence of $\| D_G\mathcal{G}^T(P, G) - D_G\mathcal{G}^T(0, 0) \|$ on $P$ and $G$.

\begin{lemma}
    \label{lem:DG_limit}
    Let $P \in  \mathcal{P}^{\mathrm{ell}, T}_{\sigma, \Lell^*}$ and $ G_* \in \mathcal{S}p^{\mathrm{ell},T}_{\sigma - 1, \omega, 
    \Omega, \ell+\dec}$ with $|P|_{\sigma, \Lell^*}^T < r$. Then there exists a constant $C_*$, depending continuously on $\Upsilon, |P|^T_{\sigma, \Lell^*},$ and $| G_*|^T_{\sigma - 1, \omega,\Omega, \ell+\dec}$, such that the operator
    \[  D_G\mathcal{G}^T(P, G_*) - D_G\mathcal{G}^T(0, 0):   \mathcal{S}p^{\mathrm{ell}, T}_{\sigma-1, \omega, \Omega, \ell+\dec} \to  \mathcal{S}ym^{\mathrm{pol}, T}_{\sigma - 1, \ell +\dec+ 1} \]
    satisfies
\begin{equation}
\label{eq:DG_difference_bound}
\| D\mathcal{G}^T(0, 0) -  D\mathcal{G}^T(P, G_*) \| \leq \frac{C_*}{T^{\ell + \dec - 1}},
\end{equation}
where $\| \cdot\|$ denotes the operator norm. 

In particular, Item \eqref{prop:DG_limit} of Proposition \ref{prop:G_properties} holds.
\end{lemma}

\begin{proof}

    Recall that (see \eqref{eq:formula_F_ell_2}) $\mathcal{G}^T(P, G)$ is of the form
   \[  \mathcal{G}^T(P, G) =  e^{G^\top}\beta(P)e^G + e^{G^\top}\Mell e^G +  e^{G^\top} J_m \Lo  e^{G} - M^{\mathrm{ell}}_\Omega,
\]
where $\beta(P) \in \mathcal{M}^{\mathrm{pol},T}_{\sigma,\ell + \dec+ 1}$ depends only on $P$ (more precisely, on $H_0$, $P$ and the $C^\sigma$ asymptotic KAM torus $\boldsymbol{\varphi}^T(P)$). Hence, by \eqref{eq:derivative_exp}, we have
\begin{align*}
D_G \mathcal{G}(P, G_*)G & = \mathcal{A}(G_*, G)^\top e^{G^\top_*}\beta(P)e^{G_*} + e^{G^\top_*}\beta(P)e^{G_*}\mathcal{A}(G_*, G) \\
& +  \mathcal{A}(G_*, G)^\top e^{G^\top_*}\Mell e^{G_*} +  e^{G^\top_*}\Mell e^{G_*}\mathcal{A}(G_*, G)\\
& +  \mathcal{A}(G_*, G)^\top e^{G^\top_*} J_m \Lo  e^{G_*} +   e^{G^\top_*}  J_m \Lo  (e^{G_*} \mathcal{A}(G_*, G)),
\end{align*}
where $\mathcal{A}(\cdot, \cdot)$ is given by \eqref{eq:adjoint_formula}.

Therefore,
\begin{equation}
\label{eq:DG_difference}
\begin{aligned}
(D_G \mathcal{G}(P, G_*) & - D_G \mathcal{G}(0, 0))G  = \mathcal{A}(G_*, G)^\top e^{G^\top_*}\beta(P)e^{G_*} + e^{G^\top_*}\beta(P)e^{G_*}\mathcal{A}(G_*, G) \\
& + (\mathcal{A}(G_*, G)^\top e^{G^\top_*}\Mell e^{G_*} - G^\top \Mell) +  (e^{G^\top_*}\Mell e^{G_*}\mathcal{A}(G_*, G) - \Mell G)\\
& +  \mathcal{A}(G_*, G)^\top e^{G^\top_*} J_m \Lo  e^{G_*} +  (e^{G^\top_*}  J_m \Lo  (e^{G_*} \mathcal{A}(G_*, G)) - J_m \Lo G).
\end{aligned}
\end{equation}

Using  Proposition \ref{prop:Csigma_prop_norms_matrix} and Lemma \ref{lem:exp_expansion_bounds}, a straightforward calculation shows that 
\begin{equation}
\label{eq:better_decay_DG}
(D_G \mathcal{G}(P, G_*)  - D_G \mathcal{G}(0, 0))G \in \mathcal{S}ym^{\mathrm{pol},T}_{\sigma - 1, 2\ell +2\dec}.
\end{equation}
Moreover, its norm is bounded by $C_*| G|^T_{\sigma - 1, \omega, \Omega, \ell+\dec}$, where $C_*$ is a constant depending continuously on $\sigma, \ell, m, n, \Upsilon, |P|_{\sigma, \Lell^*},$ and $|G_*|^T_{\sigma - 1, \omega, \Omega, \ell +\dec}$. 

Indeed, to check \eqref{eq:better_decay_DG}, we will obtain an explicit bound of the $|\cdot|_{\sigma - 1, 2\ell + 2 \dec}^T$-norm that for each of the terms in \eqref{eq:DG_difference} (where we consider the differences inside the parentheses as a single term). %

In the following, we denote simply by $C$ constants appearing in the inequalities and that depend only on $\sigma, \ell, \dec, m$. 

By Proposition \ref{prop:Csigma_prop_norms_matrix} and Lemma \ref{lem:exp_expansion_bounds}, we have
\[ |\mathcal{A}(G_*, G)^\top e^{G^\top_*}\beta(P)e^{G_*} |^T_{\sigma-1, 2\ell +2\dec} \leq Ce^{C|G_*|^T_{\sigma-1, 0}} | G|^T_{\sigma-1, \omega, \Omega, \ell + \dec} |\beta(P)|^T_{\sigma, \ell+\dec+1}.\]

The term $e^{G^\top_*}\beta(P)e^{G_*}\mathcal{A}(G_*, G)$ is treated similarly.

By triangular inequality, Proposition \ref{prop:Csigma_prop_norms_matrix} and Lemma \ref{lem:exp_expansion_bounds}, we have
\begin{align*}
    |\mathcal{A}(G_*, G)^\top e^{G^\top_*}\Mell e^{G_*}  & - G^\top \Mell|^T_{\sigma -1, 2\ell+2\dec}  \leq C \Big( |\mathcal{A}(G_*, G)^\top (e^{G^\top_*} - \mathrm{Id}_m)\Mell e^{G_*} |^T_{\sigma -1, 2\ell +2\dec}  \\ 
    & + |\mathcal{A}(G_*, G)^\top \Mell (e^{G_*} - \mathrm{Id}_m) |^T_{\sigma -1, 2\ell +2\dec} +  |(\mathcal{A}(G_*, G)^\top - G^\top)\Mell |^T_{\sigma -1, 2\ell +2\dec} \Big)\\
    & \leq Ce^{C|G_*|^T_{\sigma-1, 0}} | G|^T_{\sigma -1, \ell+\dec} \Big(|e^{G^\top_*} - \mathrm{Id}_m|^T_{\sigma -1, \ell +\dec} + |e^{G_*} - \mathrm{Id}_m|^T_{\sigma -1, \ell +\dec}\Big)\\
    &+  C|\mathcal{A}(G_*, G)^\top - G^\top|^T_{\sigma-1, 2\ell+2\dec} \\
    & \leq Ce^{C|G_*|^T_{\sigma-1, 0}} | G|^T_{\sigma-1, \omega, \Omega, \ell+\dec}| G_*|^T_{\sigma-1, \ell+\dec}
\end{align*}

The term $e^{G^\top_*}\Mell e^{G_*}\mathcal{A}(G_*, G) - \Mell G$ is treated similarly.

By Proposition \ref{prop:Csigma_prop_norms_matrix} and Lemma \ref{lem:exp_expansion_bounds},
\begin{align*}
| \mathcal{A}(G_*, G)^\top e^{G^\top_*} J_m \Lo  e^{G_*} |_{\sigma-1, 2\ell+2\dec} & \leq Ce^{C|G_*|^T_{\sigma-1, 0}}  | G|^T_{\sigma-1, \ell+\dec} |\Lo G_*|^T_{\sigma-1, \ell+\dec} \\
& \leq Ce^{C|G_*|^T_{\sigma-1, 0}}  | G|^T_{\sigma-1,\omega, \Omega, \ell+\dec} |G_*|^T_{\sigma-1, \omega, \Omega, \ell+\dec}.
\end{align*}

Finally, again by Proposition \ref{prop:Csigma_prop_norms_matrix} and Lemma \ref{lem:exp_expansion_bounds},
\begin{align*}
|e^{G^\top_*} & J_m \Lo  (e^{G_*} \mathcal{A}(G_*, G))  - J_m \Lo G|_{\sigma-1, 2\ell+2\dec}  \leq C \Big( |e^{G^\top_*}  J_m \Lo  (e^{G_*}) \mathcal{A}(G_*, G)|_{\sigma-1, 2\ell+2\dec} \\
& \quad +| e^{G^\top_*}  J_m  e^{G_*} \Lo (\mathcal{A}(G_*, G)) -  e^{G^\top_*} J_m \Lo G|_{\sigma-1, 2\ell+2\dec} + | e^{G^\top_*} J_m \Lo G -  J_m \Lo G|_{\sigma-1, 2\ell+2\dec}\Big) \\
& \leq Ce^{C|G_*|^T_{\sigma-1, 0}}(| \Lo (e^{G_*})|_{\sigma-1, \ell+\dec})|\mathcal{A}(G_*, G)|_{\sigma-1, \ell+\dec} + |\Lo(\mathcal{A}(G_*, G) - G)|_{\sigma-1, 2\ell+2\dec}\\
&\quad + |e^{G^\top_*} - \mathrm{Id}_m|_{\sigma-1, \ell+\dec} |\Lo G|_{\sigma-1, \ell+\dec} \\
& \leq Ce^{C|G_*|^T_{\sigma-1, 0}} | G|^T_{\sigma-1,\omega, \Omega, \ell+\dec} |G_*|^T_{\sigma-1, \omega, \Omega, \ell+\dec}.
\end{align*}

Therefore, there exists $C_*$ depending continuously on $\sigma, \ell, \dec, m, n, \Upsilon, |P|^T_{\sigma, \Lell^*},$ and $|G_*|^T_{\sigma - 1, \omega, \Omega, \ell+\dec}$ such that
\begin{align*}
 \big|(D_G\mathcal{G}^T(0,0) - D_G\mathcal{G}^T(P, G_*))G\big|^T_{\sigma - 1, \ell + \dec+1} & \leq  \frac{1}{T^{\ell +\dec- 1}}\big|(D_G\mathcal{G}^T(0,0) - D_G\mathcal{G}^T(P, G_*))G\big|^T_{\sigma - 1, 2\ell+2\dec}  \\
 & \leq \frac{C_*}{T^{\ell +\dec - 1}} |G|^T_{\sigma - 1, \omega, \Omega, \ell+\dec}.
\end{align*}

\end{proof}

\subsection{Proof of Theorem \ref{Thm:elliptic_analy}}
\label{sc:proof_analytic_ell} 
 As mentioned at the beginning of Section \ref{sec:Proof_Theorem_Ell}, it suffices to prove the first part of Theorem \ref{Thm:elliptic_analy} (since the second part will follow immediately from Theorem \ref{Thm:elliptic_Csigma}). 

Moreover, since the proof of the first part of Theorem \ref{Thm:elliptic_analy} follows the same lines and structure as that of Item \eqref{thm:A_existence} in Theorem \ref{Thm:elliptic_Csigma} (which is given in Section \ref{sec:proof_thm_a_ell_1}), but considering spaces of analytic instead of Hölder functions, we will only define the spaces and discuss the very minor changes needed to carry out the proof. Let us point out that, for the sake of completeness, the cohomological equations appearing in Section \ref{sec:proof_thm_a_ell_1} are also considered when restricted to the analytic setting in Section \ref{sc:cohom_eqs_ell} and that the analytic counterparts of the results concerning the norms and the Banach spaces appearing in Section \ref{sec:proof_thm_a_ell_1} are also proved in Section \ref{sc:norm_properties_ell} (see Proposition \ref{prop:properties_analy_pol}).

Recall that the space of real analytic functions with polynomial decay in time $\mathscr{S}^{\mathrm{pol},T}_{\sigma,\ell}$ and its associated norm $    |\cdot |^T_{\sigma, \ell}$ (which as an abuse of notation we denote as the one associated to $\mathcal{S}^{\mathrm{pol},T}_{\sigma,\ell}$) were defined in \eqref{def:S_pol_anal} and \eqref{def:norm_S_analy}, respectively.

Given $\sigma >0$, $\ell\ge 0$ and $T\ge 1$, where $\sigma$ will now denote the width of the complex neighbourhood, we define the counterparts of the spaces $\mathcal{S}^{\mathrm{pol},T}_{(\sigma, \ast), (0,\ell)}$, $\mathcal{U}^{\mathrm{pol}, T}_{\sigma, \omega, \ell}$ and $\mathcal{W}^{\mathrm{ell}, T}_{\sigma, \omega, \Omega, \ell}$, given by \eqref{def: S0ell},~\eqref{def:U} and~\eqref{def:W}, as
\begin{equation}\label{def: S0ell_analy}
    \mathscr{S}^{\mathrm{pol},T}_{\sigma, (0,\ell)} = \left\{f:\T^n_\sigma \times B_\sigma \times I_T \to \C  \hspace{1mm} | \hspace{1mm} f \in \mathscr{S}^{\mathrm{pol}, T}_{\sigma, 0}, \hspace{2mm} \partial_q f \in \mathscr{S}^{\mathrm{pol}, T}_{\sigma, \ell}\right\},
\end{equation}
\begin{equation}\label{def:U_analy}
    \mathscr{U}^{\mathrm{pol}, T}_{\sigma, \omega, \ell} = \left\{u:\T^n_\sigma \times I_T \to \C \hspace{1mm} | \hspace{1mm} u \in \mathscr{S}^{\mathrm{pol},T}_{\sigma, \ell}, \hspace{2mm}\Lo u \in \mathscr{S}^{\mathrm{pol}, T}_{\sigma, \ell+1}\right\},
\end{equation}
\begin{equation}\label{def:W_analy}
    \mathscr{W}^{\mathrm{ell}, T}_{\sigma, \omega, \Omega, \ell} = \left\{w:\T^n_\sigma \times I_T \to \C^{2m} \hspace{1mm} | \hspace{1mm} w \in \mathscr{S}^{\mathrm{pol}, T}_{\sigma, \ell}, \hspace{2mm}\mathcal{L}^{\mathrm{ell}}_{\omega,\Omega} w \in \mathscr{S}^{\mathrm{pol}, T}_{\sigma, \ell+1}\right\},
\end{equation}
endowed with the norms
\begin{equation}\label{def:norm_S0ell_analy}
    |f|^T_{\sigma, (0,\ell)} = \max\{|f|^T_{\sigma, 0}, |\partial_qf|^T_{\sigma, \ell}\},
\end{equation}
\begin{equation}\label{def:norm_U_analy}
    |u|^T_{\sigma, \omega, \ell} = \max \left\{|u|^T_{\sigma, \ell}, |\Lo u|^T_{\sigma, \ell +1}\right\},
\end{equation}
\begin{equation}\label{def:norm_W_analy}
    |w|^{T}_{\sigma, \omega, \Omega, \ell} = \max \left\{|w|^T_{\sigma, \ell}, |\mathcal{L}^{\mathrm{ell}}_{\omega,\Omega} w|^T_{\sigma, \ell +1}\right\},
\end{equation}
respectively.

With these notations, the proof follows the exact same lines as the proof of Item \eqref{thm:A_existence} in Theorem \ref{Thm:elliptic_Csigma} given in Section \ref{sec:proof_thm_a_ell_1}, by defining the counterparts $\mathscr{P}^{\mathrm{ell}, T}_{\sigma, \Lell}$, $\mathscr{E}^{\mathrm{ell}, T}_{\sigma, \mathcal{K}}$ of the spaces $\mathcal{P}^{\mathrm{ell}, T}_{\sigma, \Lell}$, $\mathcal{E}^{\mathrm{ell}, T}_{\sigma, \mathcal{K}}$ in the obvious way and by considering the functional $\mathcal{F}^T$ in an appropriate subset of $\mathscr{P}^{\mathrm{ell}, T}_{\sigma, \Lell} \times \mathscr{E}^{\mathrm{ell}, T}_{\sigma, \mathcal{K}}$, for example $ \mathscr{P}^{\mathrm{ell}, T}_{\sigma,\Lell} \times B_{\mathscr{E}^{\mathrm{ell}, T}_{\sigma/2, \mathcal{K}}}\big(T^{(\ell - 1)/2}\big)$ (compare with \eqref{eq:domain_F_ell} in Lemma \ref{lemma:F_well_defined}). 

Let us mention that the choice of this subset is the only major difference in the proof, as for $\mathcal{F}^T$ to be well-defined we will have to consider a smaller open neighbourhood of the origin in $\mathscr{E}^{\mathrm{ell}, T}_{\sigma, \mathcal{K}}$ with a smaller analiticity width (for example $\sigma/2$). 

\subsection{Proof of Corollary \ref{cor:ell}}\label{sec:proof_cor_ell}
 We point out that the proof of Theorem \ref{Thm:elliptic_Csigma} is constructive. Once we have proved the existence of a $C^\sigma$ asymptotic elliptic KAM torus associated with $(X_H, X_{H_0}, \varphi_0)$, then the matrix $A$ in~\eqref{def:A} has an explicit form depending on the asymptotic KAM torus $\varphi=(\mathrm{id}_{\T^n} + u,v,w)$ and the components of the matrix $S$ (see~\eqref{eq:A_as_function_of_S}). In particular, a straightforward computation shows that the component $a_{32}$ of the matrix $A$ is given by
  \begin{equation*}
     a_{32} = e^{G^\top}
\left(\partial_z\partial_pH\circ\varphi\right)
(\mathrm{id}_{\T^n}+\partial_qu)^{-\top}- e^{G^\top} J_m\partial_qw
(\mathrm{id}_{\T^n}+\partial_qu)^{-1}
\left(\partial_p^2H\circ\varphi\right)
(\mathrm{id}_{\T^n}+\partial_qu)^{-\top}
 \end{equation*}
 where, for $\ell >2$, $\dec \ge 0$ and $T'' \ge 1$
 \begin{equation}\label{cor_proof:decay_uvwG}
     u \in \mathcal{S}_{(\sigma, 0), \ell-1}^{\mathrm{pol},T''}, \qquad  v \in \mathcal{S}_{(\sigma, 0), \ell+ \dec + 2}^{\mathrm{pol},T''}, \qquad w \in \mathcal{S}_{(\sigma, 0), \ell+ \dec+1}^{\mathrm{pol},T''}, \qquad G \in \mathcal{S}_{(\sigma, 0), \ell+\dec}^{\mathrm{pol},T''}.
 \end{equation}
We refer to~\eqref{eq:asymp_KAM_torus_form}, Lemma \ref{lem:embedding_symplectomorphisms_family}, Proposition \ref{prop:criteria_non_autonomous} and Section \ref{sec:proof_thm_a_ell_2}. We recall that here we are assuming $\ell >2$ instead of $\ell >1$. 

The proof of Corollary \ref{cor:ell} follows directly from Proposition \ref{prop:ell_1:1_asym_sol} together with~\eqref{eq:conjugated_cocycles} and~\eqref{def:trasv_dyn_limit} in Definition \ref{def:ell_hyp_par_trasv_dyn_asym_KAM}. Thus, it suffices to verify the hypothesis~\eqref{hyp:limit_aut_sys_ell} of Proposition \ref{prop:ell_1:1_asym_sol} that, in this case, translates into
\begin{equation}\label{proof_cor:cond_to_verify}
        \int_{T''}^{+\infty} |a_{32}^s - \partial_z\partial_pH^\infty\circ\varphi_0|_{C^0} ds < \infty,
    \end{equation}
    where $\varphi_0$ is the trivial embedding defined by~\eqref{def:varphi0=(q,0,0)}. 

    To this end,  we observe that
    \begin{align*}
        &(\mathrm{id}_{\T^n}+\partial_qu)^{-\top} = \mathrm{id}_{\T^n} - \int_0^1 (\mathrm{id}_{\T^n}+\tau\partial_qu)^{-\top}(\partial_qu)^\top(\mathrm{id}_{\T^n}+\tau\partial_qu)^{-\top}\,d\tau,\\
        &e^{G^\top} = \mathrm{Id} + {G^\top}\int_0^1e^{\tau {G^\top}}\,d\tau,\\
        &\partial_z\partial_pH\circ\varphi(q) - \partial_z\partial_pH^\infty\circ\varphi_0(q) = \left(\partial_z\partial_pH\right)_0 - \left(\partial_z\partial_pH^\infty\right)_0 + \int_0^1 \partial_q\partial_z\partial_pH(\mathrm{id}_{\T^n} + \tau u, \tau v, \tau w) d \tau \, u\\
        &\qquad \qquad + \int_0^1 \partial_p\partial_z\partial_pH(\mathrm{id}_{\T^n} + \tau u, \tau v, \tau w) d \tau \, v +  \int_0^1 \partial_z\partial_z\partial_pH(\mathrm{id}_{\T^n} + \tau u, \tau v, \tau w) d \tau \, w
    \end{align*}
    where, in the last equality, we used the notation (subscript $0$) introduced in ~\eqref{def:f0}. 
    It follows that
    \begin{align*}
        &e^{G^\top}
\left(\partial_z\partial_pH\circ\varphi\right)
(\mathrm{id}_{\T^n}+\partial_qu)^{-\top} - \partial_z\partial_pH^\infty\circ\varphi_0  \\
&\qquad \qquad =\left(\partial_z\partial_pH\right)_0 - \left(\partial_z\partial_pH^\infty\right)_0 + \int_0^1 \partial_q\partial_z\partial_pH(\mathrm{id}_{\T^n} + \tau u, \tau v, \tau w) d \tau \, u\\
        &\qquad \qquad + \int_0^1 \partial_p\partial_z\partial_pH(\mathrm{id}_{\T^n} + \tau u, \tau v, \tau w) d \tau \, v +  \int_0^1 \partial_z\partial_z\partial_pH(\mathrm{id}_{\T^n} + \tau u, \tau v, \tau w) d \tau \, w  \\
&\qquad \qquad +  {G^\top}\int_0^1e^{\tau {G^\top}}\,d\tau \,\left(\partial_z\partial_pH\circ\varphi\right) - \partial_z\partial_pH\circ\varphi \int_0^1 (\mathrm{id}_{\T^n}+\tau\partial_qu)^{-\top}(\partial_qu)^{\top}(\mathrm{id}_{\T^n}+\tau\partial_qu)^{-\top}\,d\tau\\
&\qquad\qquad - {G^\top}\int_0^1e^{\tau {G^\top}}\,d\tau \,\left(\partial_z\partial_pH\circ\varphi\right)\int_0^1 (\mathrm{id}_{\T^n}+\tau\partial_qu)^{-\top}(\partial_qu)^{\top}(\mathrm{id}_{\T^n}+\tau\partial_qu)^{-\top}\,d\tau.
    \end{align*}
Using~\eqref{cor_proof:decay_uvwG}, hypothesis~\eqref{hyp:cor_ell_decay} and the properties contained in Section \ref{sc:norm_properties_ell}, we have that 
\begin{equation}\label{proof_cor:est_1}
    \begin{gathered}
       \scalebox{0.97}{$e^{G^\top}
\left(\partial_z\partial_pH\circ\varphi\right)
(\mathrm{id}_{\T^n}+\partial_qu)^{-\top} - \partial_z\partial_pH^\infty\circ\varphi_0 
 - \left(\left(\partial_z\partial_pH\right)_0  - \left(\partial_z\partial_pH^\infty\right)_0\right) \in \mathcal{S}_{(\sigma, 0), \ell-1}^{\mathrm{pol},T''},$}\\
 \int_{T''}^{+\infty} |\left(\partial_z\partial_pH\right)_0  - \left(\partial_z\partial_pH^\infty\right)_0|_{C^0} ds < \infty.
    \end{gathered}
\end{equation}
    Similarly, one can prove that 
    \begin{equation} \label{proof_cor:est_2}
        e^{G^\top} J_m\partial_qw
(\mathrm{id}_{\T^n}+\partial_qu)^{-1}
\left(\partial_p^2H\circ\varphi\right)
(\mathrm{id}_{\T^n}+\partial_qu)^{-\top} \in \mathcal{S}_{(\sigma, 0), \ell+ \dec+1}^{\mathrm{pol},T''}. 
    \end{equation}
  Thanks to~\eqref{proof_cor:est_1},~\eqref{proof_cor:est_2} and remembering that $\ell >2$, we prove~\eqref{proof_cor:cond_to_verify}. This concludes the proof of Corollary \ref{cor:ell}. 

\subsection{Cohomological Equations}
\label{sc:cohom_eqs_ell}

In this section, we study in detail the cohomological equations appearing in the proof of Theorems \ref{Thm:elliptic_Csigma} and \ref{Thm:elliptic_analy}, namely, the three equations in \eqref{eq:coh_eqs_ell} restricted to the Hölder setting (see Sections \ref{sec:proof_thm_a_ell_1} and \ref{sec:proof_thm_a_ell_2}) and the first two equations in \eqref{eq:coh_eqs_ell} in the analytic setting (see Section \ref{sc:proof_analytic_ell}), respectively.

To avoid unnecessary repetition, we present the proofs in full detail only for the Hölder setting, since the arguments in the analytic setting are very similar. For the latter, we only indicate the main steps and the modifications required by the analytic framework.

First, we analyze the linear operator $\Lo$. Throughout this section, $\|\cdot\|$ denotes the operator norm. In the following proposition, we use the Banach spaces defined in~\eqref{def:S},~\eqref{def:U}~\eqref{def:S_pol_anal}, and~\eqref{def:U_analy}. 
\begin{proposition}\label{prop:Csigma_HEomega} 
       Given $\sigma \ge 0$, $T \ge 1$, and $\ell >0$,  the operator
    \begin{align*}
        \Lo : \mathcal{U}^{\mathrm{pol}, T}_{\sigma, \omega, \ell} \longrightarrow \mathcal{S}^{\mathrm{pol}, T}_{(\sigma, 0), \ell +1}(\R^n)
    \end{align*}
    is well-defined, invertible and satisfies
    \[ \| \Lo ^{-1}\| \leq C(\ell),\]
    for some $C > 0$ depending only on $\ell$.
    
   In addition, if $\sigma \ge 2$, then, for any $u \in \mathcal{U}^{\mathrm{pol}, T}_{\sigma, \omega, \ell}$
    \begin{equation}\label{eq:prop_dist_Lomega_U}
        \partial_q\left(\Lo u\right) = \Lo \partial_q  u .
    \end{equation}
    Moreover, given $\sigma > 0$, $T \ge 1$, and $\ell >0$, the same well-definedness, invertibility, and estimate hold if one replaces the Banach spaces $\mathcal{U}^{\mathrm{pol}, T}_{\sigma, \omega, \ell}$ and $\mathcal{S}^{\mathrm{pol}, T}_{(\sigma, 0), \ell +1}$ by $\mathscr{U}^{\mathrm{pol}, T}_{\sigma, \omega, \ell}$ and $\mathscr{S}^{\mathrm{pol}, T}_{\sigma, \ell +1}$, respectively. Moreover, in this case, the equality~\eqref{eq:prop_dist_Lomega_U} holds for any $\sigma>0$ and $u \in \mathscr{U}^{\mathrm{pol}, T}_{\sigma, \omega, \ell}$.
\end{proposition}
\begin{proof} 
    The proof of the first part of this proposition is similar to that contained in~\cite{Sca22}. For this reason, it is sketched. Given $g \in \mathcal{S}^{\mathrm{pol}, T}_{(\sigma, 0), \ell +1}(\R^n)$ it suffices to find the unique solution $u \in \mathcal{U}^{\mathrm{pol}, T}_{\sigma, \omega, \ell}$ of
    \begin{equation}\label{eq:proof_inv_Lo}
        \Lo u(q,t) = g(q,t) 
    \end{equation}
    for all $(q,t) \in \T^n \times I_T$. 
    One can verify that the above equation admits a unique solution $u \in \mathcal{U}^{\mathrm{pol}, T}_{\sigma, \omega, \ell}$ of the form
    \begin{equation}\label{eq:proof_inv_sol_Lo}
        u(q,t) = - \int_t^{+\infty} g (q + \omega(\tau-t), \tau)d\tau
    \end{equation}
    for all $(q,t) \in \T^n \times I_T$. Moreover, by the latter, a straightforward computation shows that 
    \begin{equation*}
        |u|^T_{\sigma, \ell} \le {1 \over \ell} |g|_{\sigma, \ell + 1}^T.
    \end{equation*}
    Recalling the definition of the norm $|\cdot|_{\sigma, \omega, \ell}^T$ in~\eqref{def:norm_U}, using the above inequality and~\eqref{eq:proof_inv_Lo} we prove that $u \in \mathcal{U}^{\mathrm{pol}, T}_{\sigma, \omega, \ell}$, which concludes the proof of the first part of this proposition. We need to verify~\eqref{eq:prop_dist_Lomega_U}.  To this end, we claim that it suffices to prove that, if $\sigma \ge 2$, then for any $u \in \mathcal{U}^{\mathrm{pol}, T}_{\sigma, \omega, \ell}$ one has that 
    \begin{equation}\label{eq:partialtqu=partialqtu}
        \partial_t \partial_q u = \partial_q \partial_t u.
    \end{equation}
    Indeed, if we assume~\eqref{eq:partialtqu=partialqtu} then
    \begin{equation}
        \partial_q\left(\Lo u\right) = \partial_q(\partial_q u \omega + \partial_tu) = \partial^2_q u \omega + \partial_q\partial_tu = \partial^2_q u \omega + \partial_t\partial_qu = \Lo \partial_q u,
    \end{equation}
    which proves~\eqref{eq:prop_dist_Lomega_U}. We verify~\eqref{eq:partialtqu=partialqtu}. By the definition~\eqref{def:U} of the Banach space $\mathcal{U}^{\mathrm{pol}, T}_{\sigma, \omega, \ell}$, if $u \in \mathcal{U}^{\mathrm{pol}, T}_{\sigma, \omega, \ell}$, then there exists $g \in  \mathcal{S}^{\mathrm{pol}, T}_{(\sigma, 0), \ell +1}$ such that $\Lo u = g$. Thus, we can write $u$ as in~\eqref{eq:proof_inv_sol_Lo}. A straightforward computation shows that $u$ is differentiable with respect to the variable $t$ and
    \begin{align*}
        \partial_t u (q,t) &= g(q,t) + \int_t^{+\infty} \partial_q g(q + \omega(\tau-t), \tau)\omega d\tau,\\
        \partial_q u (q,t) &= - \int_t^{+\infty} \partial_q g(q + \omega(\tau-t), \tau) d\tau,\\
        \partial_t\partial_q u (q,t) &= \partial_q\partial_t u (q,t) = \partial_qg(q,t) + \int_t^{+\infty} \partial^2_q g(q + \omega(\tau-t), \tau)\omega d\tau,
    \end{align*}
for all $(q,t) \in\T^n \times I_T$. This concludes the proof of this proposition in the Hölder setting. The proof in the analytic setting is identical. We therefore omit the details.
\end{proof}

Before addressing the analysis of the linear operator $\mathcal{L}^{\mathrm{ell}}_{\omega, \Omega}$, we need to prove the following two auxiliary lemmas 

\begin{lemma}\label{lemma:HE_interm_LoO}
    We fix $\sigma \ge 0$, $T \ge 1$, $\ell >0$, and a positive $\Lambda \in \R$. We consider the function $g:\T^n \times I_T \to \C$ and we assume that $\mathrm{Re} g$, $\mathrm{Im} g \in \mathcal{S}^{\mathrm{pol},T}_{(\sigma, 0), \ell +1}$, where $\mathrm{Re} g$ and $\mathrm{Im} g$ stand for the real and imaginary part of $g$, respectively. 

    We consider the following equation 
    \begin{equation}\label{eq:HE_Interm_1dim_LoO}
        \pm i\Lambda z - \Lo z = g
    \end{equation}
    in the unknown $z:\T^n \times I_T \to \C$. The latter admit a unique solution $z$ satisfying %
    \begin{equation}\label{est:HE_Interm_1dim_LoO}
        \mathrm{Re} z, \mathrm{Im} z \in \mathcal{S}^{\mathrm{pol},T}_{(\sigma, 0), \ell}, \hspace{2mm}\mbox{and}\hspace{3mm} |\mathrm{Re} z|_{\sigma, \ell}^T, |\mathrm{Im} z|_{\sigma, \ell}^T \le {1 \over \ell} \left(|\mathrm{Re} g|_{\sigma, \ell+1}^T + |\mathrm{Im} g|_{\sigma, \ell+1}^T\right),
    \end{equation}
    where $\mathrm{Re} z$ and $\mathrm{Im} z$ stand for the real and imaginary part of $z$. In addition, if $\sigma \ge 2$, then 
    \begin{equation}\label{eq:prop_dist_Lomega_z_W}
        \partial_q \left(\Lo \mathrm{Re}z\right) = \Lo \left( \partial_q \mathrm{Re}z\right), \quad \partial_q \left(\Lo \mathrm{Im}z\right) = \Lo \left( \partial_q \mathrm{Im}z\right).
    \end{equation}
\end{lemma}
\begin{proof}
    We consider the following transformation
    \begin{equation*}
        \phi:\T^n \times I_T \to \T^n \times I_T, \quad \phi(q,t) = (q-\omega t, t).
    \end{equation*}
    We claim that it is enough to find a solution to the following simpler equation
    \begin{equation}\label{eq:HE_Interm_1dim_LoO_2}
        \pm i \Lambda k(q,t) - \partial_t k(q,t) = g\circ \phi^{-1}(q,t). 
    \end{equation}
    Let $k$ be a solution of~\eqref{eq:HE_Interm_1dim_LoO_2}, we want to prove that $z = k \circ \phi$ is a solution of~\eqref{eq:HE_Interm_1dim_LoO}. Indeed, replacing $k \circ \phi$ on the left hand side of~\eqref{eq:HE_Interm_1dim_LoO}, we obtain that 
    \begin{equation*}
        \pm i \Lambda k\circ \phi - \partial_q \left( k\circ \phi\right)\omega - \partial_t\left( k\circ \phi\right) = \pm i \Lambda k\circ \phi - \partial_t k \circ \phi = g
    \end{equation*}
    where we used~\eqref{eq:HE_Interm_1dim_LoO_2} the last equality of the latter.
    Vice versa, let $z$ be a solution of~\eqref{eq:HE_Interm_1dim_LoO}, we verify that $k = z \circ \phi^{-1}$ is a solution of~\eqref{eq:HE_Interm_1dim_LoO_2}. To this end, we replace $z \circ \phi^{-1}$ on the left hand side of~\eqref{eq:HE_Interm_1dim_LoO_2}, and we can see that
    \begin{equation*}
        \pm i \Lambda z \circ \phi^{-1} - \partial_t\left(z \circ \phi^{-1}\right)  = \pm i \Lambda z \circ \phi^{-1} - \partial_q z \circ \phi^{-1}\omega - \partial_t z \circ \phi^{-1} = g \circ \phi^{-1}
    \end{equation*}
    where the last equality of the latter is due to~\eqref{eq:HE_Interm_1dim_LoO}.
    This concludes the proof of the claim. 

    A solution of~\eqref{eq:HE_Interm_1dim_LoO_2} can be provided by integration. Hence, for all $(q,t) \in \T^n \times I_T$ and $t_0 \in I_T$, a solution of~\eqref{eq:HE_Interm_1dim_LoO_2} exists and it is given by
    \begin{equation*}
        k(q,t) = e^{\mp i\Lambda (t_0 -t)}\left(k_0(q) - \int_{t_0}^t e^{\mp i \Lambda(\tau - t_0)} g(q + \omega \tau, \tau)\, d\tau \right) 
    \end{equation*}
    where $k_0:\T^n \to \C$ is free. 

    In this work, we are interested in functions that decay in time. In particular, we are looking for solutions $k$ of~\eqref{eq:HE_Interm_1dim_LoO_2} satisfying $\lim_{t \to +\infty} |k(q,t)| = 0$, where $|\cdot|$ denotes the modulus. For this reason, the only possible choice of $k_0$ is 
    \begin{equation*}
        k_0(q,t) = \int_{t_0}^{+\infty} e^{\mp i \Lambda(\tau - t_0)} g(q + \omega \tau, \tau)\, d\tau 
    \end{equation*}
    and hence a unique solution of~\eqref{eq:HE_Interm_1dim_LoO_2} satisfying the above-mentioned asymptotic condition exists, and it has the form
    \begin{equation*}
        k(q,t) =  \int_t^{+\infty} e^{\mp i \Lambda(\tau - t)} g(q + \omega \tau, \tau)\, d\tau 
    \end{equation*}
    for all $(q,t)\in \T^n \times I_T$. 
    Moreover, 
    \begin{equation*}
        z(q,t) = k \circ \phi(q,t) = \int_t^{+\infty} e^{\mp i \Lambda(\tau - t)} g(q + \omega (\tau -t), \tau)\, d\tau 
    \end{equation*}
    is the unique solution of~\eqref{eq:HE_Interm_1dim_LoO} we are looking for. It remains to verify that $z$ satisfies~\eqref{est:HE_Interm_1dim_LoO}.

    A straightforward computation shows that
    \begin{align*}
        &\mathrm{Re} z (q,t)=  \int_t^{+\infty} \cos \left(\mp \Lambda(\tau -t)\right) \mathrm{Re} g(q + \omega(\tau -t), \tau) - \sin \left(\mp \Lambda(\tau -t)\right) \mathrm{Im} g(q + \omega(\tau -t), \tau)\, d\tau \\
         &\mathrm{Im} z (q,t)=  \int_t^{+\infty} \cos \left(\mp \Lambda(\tau -t)\right) \mathrm{Im} g(q + \omega(\tau -t), \tau) + \sin \left(\mp \Lambda(\tau -t)\right) \mathrm{Re} g(q + \omega(\tau -t), \tau)\, d\tau 
    \end{align*}
    for all $(q,t) \in \T^n \times I_T$. Using the trivial estimates $|\cos \left(\mp \Lambda(\tau -t)\right)|, |\sin \left(\mp \Lambda(\tau -t)\right)| \le 1$, and remembering the definition of the norms $|\cdot|^T_{\sigma, \ell}$ in~\eqref{def:norm_S}, we prove~\eqref{est:HE_Interm_1dim_LoO} and this concludes the proof of this part of this lemma. In order to prove the second part, we assume $\sigma \ge 2$. By the same argument used in the second part of the proof of Proposition \ref{prop:Csigma_HEomega}, one can verify that if
    \begin{equation}\label{eq:partialtqz=partialqtz}
        \partial_t \partial_q \mathrm{Re} z(q,t) = \partial_q\partial_t \mathrm{Re} z(q,t), \quad \partial_t \partial_q \mathrm{Im} z(q,t) = \partial_q\partial_t \mathrm{Im} z(q,t)
    \end{equation}
    for all $(q,t)\in\T^n \times I_t$,
    then~\eqref{eq:prop_dist_Lomega_z_W} holds. We only verify~\eqref{eq:partialtqz=partialqtz} for $\mathrm{Re}z$ being the proof for $\mathrm{Im}z$ similar. A straightforward computation shows that $\mathrm{Re}z$ is differentiable with respect to $t$ and
    \begin{align*}
        \partial_t \mathrm{Re}z(q,t) &= -  \mathrm{Re}g(q+\omega(\tau-t), \tau) + \int_t^{+\infty}\bigg[\left(\mp \Lambda \sin \left(\mp \Lambda (\tau-t)\right)\right) \mathrm{Re}g(q+\omega(\tau-t), \tau)\\
        &-\cos \left(\mp \Lambda \left(\tau - t\right)\right)\partial_q \mathrm{Re}g(q+\omega(\tau-t), \tau)\omega \mp \Lambda \cos \left(\mp \Lambda \left(\tau - t\right)\right) \mathrm{Im}g(q+\omega(\tau-t), \tau)\\
        &- \sin\left(\mp \Lambda \left(\tau - t\right)\right) \partial_q \mathrm{Im}g(q+\omega(\tau-t), \tau)\omega \bigg]d \tau,\\
        \partial_q \mathrm{Re}z(q,t) &= \int_t^{+\infty}\bigg[\cos \left(\mp \Lambda \left(\tau - t\right)\right)\partial_q \mathrm{Re}g(q+\omega(\tau-t), \tau) \\
        &- \sin \left(\mp \Lambda (\tau-t)\right)\partial_q \mathrm{Im}g(q+\omega(\tau-t), \tau)\bigg]d \tau,\\
        \partial_t \partial_q \mathrm{Re}(q,t) &= \partial_q \partial_t \mathrm{Re}(q,t) = -\partial_q\mathrm{Re}g(q+\omega(\tau-t), \tau) \\
        &+ \int_t^{+\infty}\bigg[\left(\mp \Lambda \sin \left(\mp \Lambda (\tau-t)\right)\right) \partial_q\mathrm{Re}g(q+\omega(\tau-t), \tau)\\
        &-\cos \left(\mp \Lambda \left(\tau - t\right)\right)\partial^2_q \mathrm{Re}g(q+\omega(\tau-t), \tau)\omega \mp \Lambda \cos \left(\mp \Lambda \left(\tau - t\right)\right) \partial_q\mathrm{Im}g(q+\omega(\tau-t), \tau)\\
        &- \sin\left(\mp \Lambda \left(\tau - t\right)\right) \partial^2_q \mathrm{Im}g(q+\omega(\tau-t), \tau)\omega \bigg]d \tau
    \end{align*}
    for all $(q,t) \in \T^n \times I_T$, which proves~\eqref{eq:partialtqz=partialqtz} for $\mathrm{Re}z$. This concludes the proof of this lemma.
\end{proof}

The following lemma is the real-analytic counterpart of Lemma \ref{lemma:HE_interm_LoO}. 
\begin{lemma}\label{lemma:HE_interm_LoO_analy}
    We fix $\sigma > 0$, $T \ge 1$, $\ell >0$, and a positive $\Lambda \in \R$. Given $g:\T^n_\sigma \times I_T \to \C$ such that $g \in \mathscr{S}^{\mathrm{pol},T}_{\sigma, \ell +1}$, we consider the following equation 
    \begin{equation}\label{eq:HE_Interm_1dim_LoO_analy}
        \pm i\Lambda z - \Lo z = g
    \end{equation}
    in the unknown $z:\T^n_\sigma \times I_T \to \C$. The latter admits a unique solution $z \in \mathscr{S}^{\mathrm{pol},T}_{\sigma, \ell}$. Moreover, there exists a positive constant $C(\ell)$ depending on $\ell$ such that
    \begin{equation}\label{est:HE_Interm_1dim_LoO_analy}
         |z|_{\sigma, \ell}^T \le C(\ell) | g|_{\sigma, \ell+1}^T.
    \end{equation}
    In addition, 
    \begin{equation}\label{eq:prop_dist_Lomega_z_W_analy}
        \partial_q \left(\Lo z\right) = \Lo \left( \partial_q z\right).
    \end{equation}
\end{lemma}
\begin{proof}
    As in the proof of Lemma \ref{lemma:HE_interm_LoO}, 
    \begin{equation*}
        z(q,t) = k \circ \phi(q,t) = \int_t^{+\infty} e^{\mp i \Lambda(\tau - t)} g(q + \omega (\tau -t), \tau)\, d\tau
    \end{equation*}
    is the unique solution of~\eqref{eq:HE_Interm_1dim_LoO_analy} satisfying $z \in \mathscr{S}^{\mathrm{pol},T}_{\sigma, \ell}$. A straightforward computation gives~\eqref{est:HE_Interm_1dim_LoO_analy} and~\eqref{eq:prop_dist_Lomega_z_W_analy}.
\end{proof}

The following proposition is devoted to the analysis of the linear operator $\mathcal{L}^{\mathrm{ell}}_{\omega, \Omega}$. We recall that the spaces $\mathcal{W}^{\mathrm{ell},T}_{\sigma, \omega, \Omega, \ell}$ and $\mathscr{W}^{\mathrm{ell},T}_{\sigma, \omega, \Omega, \ell}$ are defined in~\eqref{def:W} and~\eqref{def:W_analy}, respectively. 

\begin{proposition}\label{prop:Csigma_HEomegaOmega}  
    Given $\sigma \ge 0$, $T \ge 1$, and $\ell >0$, the operator
    \begin{align*}
         \mathcal{L}^{\mathrm{ell}}_{\omega, \Omega} : \mathcal{W}^{\mathrm{ell},T}_{\sigma, \omega, \Omega, \ell} \longrightarrow \mathcal{S}^{\mathrm{pol},T}_{(\sigma, 0), \ell +1}(\R^{2m})
    \end{align*}
    as in \eqref{eq:coh_eqs_ell} is well-defined, invertible and satisfies
    \begin{equation}\label{est:HE_ell_Holder_analy}   
    \| (\mathcal{L}^{\mathrm{ell}}_{\omega, \Omega})^{-1} \| \leq C(\ell),
\end{equation}
    for some $C > 0$ depending only on $\ell$.%
    
In addition, if $\sigma \ge 2$, then, for any $\chi \in \mathcal{W}^{\mathrm{ell},T}_{\sigma, \omega, \Omega, \ell}$
    \begin{equation}\label{eq:prop_dist_Lomega_W}
        \partial_q\left(\Lo \chi\right) = \Lo \partial_q  \chi .
    \end{equation}
    Moreover, given $\sigma > 0$, $T \ge 1$, and $\ell >0$, the same well-definedness, invertibility, and estimate hold if one replaces the Banach spaces $\mathcal{W}^{\mathrm{ell}, T}_{\sigma, \omega, \Omega,\ell}$ and $\mathcal{S}^{\mathrm{pol}, T}_{(\sigma, 0), \ell +1}$ by $\mathscr{W}^{\mathrm{ell}, T}_{\sigma, \omega, \Omega, \ell}$ and $\mathscr{S}^{\mathrm{pol}, T}_{\sigma, \ell +1}$, respectively. Moreover, in this case, the equality~\eqref{eq:prop_dist_Lomega_W} holds for any $\sigma>0$ and $\chi \in \mathscr{W}^{\mathrm{ell}, T}_{\sigma, \omega, \Omega, \ell}$.
\end{proposition}
\begin{proof}
    We recall that, for any $\chi \in \mathcal{W}^{\mathrm{ell},T}_{\sigma, \omega, \Omega, \ell}$, the operator $\mathcal{L}^{\mathrm{ell}}_{\omega, \Omega} \chi = J_mM^{\mathrm{ell}}_\Omega \chi - \Lo \chi$, where the matrices $M^{\mathrm{ell}}_\Omega$ and $J_m$ are defined in~\eqref{Ms} and~\eqref{def:J}, respectively. Then, for all $g = (g_1,\dots, g_{2m})^\top \in \mathcal{S}^{\mathrm{pol}, T}_{(\sigma, 0), \ell +1}$, we can rewrite this dynamical problem in terms of finding the unique solution $\chi = (\chi_1,\dots, \chi_{2m})^\top \in \mathcal{W}^{\mathrm{ell},T}_{\sigma, \omega, \Omega, \ell}$ of the following equation
   \begin{equation}\label{proof:HE_eq_LoO}
       J_mM^{\mathrm{ell}}_\Omega \chi - \Lo \chi = g.
   \end{equation}
   For this purpose, we observe that the matrix $JM^{\mathrm{ell}}_\Omega$ is diagonalizable in the complex with eigenvalues $\pm i\Omega_j$ for $j=1,\dots, m$. Therefore, there exists a matrix $P$ such that 
   \begin{equation}\label{proof:HE_diagMO}
       J_mM^{\mathrm{ell}}_\Omega  = P\begin{pmatrix} i\Omega & 0 \\ 0 &  -i\Omega\end{pmatrix}P^{-1} \quad \mbox{with} \hspace{2mm} P =\begin{pmatrix} \mathrm{Id}_m & \mathrm{Id}_m \\ i \mathrm{Id}_m &  -i \mathrm{Id}_m\end{pmatrix}
   \end{equation}
   where $\mathrm{Id}_m$ stands for the identity $m \times m$ matrix and we recall that $\Omega=\mathrm{diag}\left(\Omega_1, \dots, \Omega_m\right)$ with $\Omega_1, \dots, \Omega_m \in \R_{>0}$. 
   
   We introduce the following notation
   \begin{equation}\label{not:chi_HE_ell}
       \chi=(\chi^{(1)}, \chi^{(2)})^\top, \quad g=(g^{(1)}, g^{(2)})^\top
   \end{equation}
   where $\chi^{(1)} = (\chi_1,\dots, \chi_{m})^\top$, $\chi^{(2)} = (\chi_{m+1},\dots, \chi_{2m})^\top$, $g^{(1)} = (g_1,\dots, g_{m})^\top$ and $g^{(2)} = (g_{m+1},\dots, g_{2m})^\top$.
   
   Replacing~\eqref{proof:HE_diagMO} into~\eqref{proof:HE_eq_LoO} and multiplying on the left both sides of the equation by $P^{-1}$, a straightforward computation shows that we can rewrite~\eqref{proof:HE_diagMO} in terms of the following $2m$ equations
   \begin{equation}
        \begin{aligned}\label{syst:HE_ell_1}
           &i\Omega \left(\chi^{(1)} - i\chi^{(2)}\right) - \Lo\left(\chi^{(1)} - i\chi^{(2)}\right) = \left(g^{(1)} - ig^{(2)}\right),\\
           -&i\Omega \left(\chi^{(1)} + i\chi^{(2)}\right) - \Lo\left(\chi^{(1)} + i\chi^{(2)}\right) = \left(g^{(1)} + ig^{(2)}\right).
        \end{aligned}   
   \end{equation}
   Introducing the complex variables $Z = (Z_1,\dots, Z_m) =  \chi^{(1)} - i\chi^{(2)}$, and $\bar Z =  (\bar Z_1,\dots, \bar Z_m) = \chi^{(1)} + i\chi^{(2)}$, the above system can be rewritten as $2m$ complex uncoupled equations
   \begin{equation}\label{syst:Eq_Z_inv_W}
        \begin{aligned}
           &i\Omega_k Z_k - \Lo Z_k = \left(g^{(1)}_k - ig^{(2)}_{m+k}\right),\\
           -&i\Omega_k \bar Z_k - \Lo\bar Z_k = \left(g^{(1)}_k + ig^{(2)}_{m+k}\right).
        \end{aligned}   
   \end{equation}
   for $k=1,\dots m$.
   Thanks to Lemma \ref{lemma:HE_interm_LoO}, there exists a unique solution of the latter satisfying 
   \begin{equation}\label{proof:HE_LoO_sol_Z}
       \mathrm{Re} Z, \mathrm{Im} Z \in \mathcal{S}^{\mathrm{pol},T}_{(\sigma, 0), \ell} \hspace{2mm}\mbox{and}\hspace{3mm} |\mathrm{Re} Z|_{\sigma, \ell}^T, |\mathrm{Im} Z|_{\sigma, \ell}^T \le {2 \over \ell} |g|_{\sigma, \ell+1}^T.
    \end{equation}
    Furthermore, returning to the original variables,
    \begin{equation}\label{eq:sol_eq_inv_W}
        \chi^{(1)} = {\bar Z + Z \over 2} = \mathrm{Re}Z, \quad \chi^{(2)} = {\bar Z - Z \over 2i} = \mathrm{Im}Z
    \end{equation}
    and using~\eqref{proof:HE_LoO_sol_Z}, we prove that $\chi =(\chi^{(1)}, \chi^{(2)})^\top \in \mathcal{S}^{\mathrm{pol},T}_{(\sigma, 0), \ell} $ and $|\chi|^T_{\sigma, \ell} \le {2 \over \ell}|g|^T_{\sigma, \ell+1}$. This concludes the proof of the first part of this proposition.  We need to verify~\eqref{eq:prop_dist_Lomega_W}. To this end, we assume $\sigma \ge 2$ and we observe that, using the definition~\eqref{def:W}, if $\chi \in \mathcal{W}^{\mathrm{ell},T}_{\sigma, \omega, \Omega, \ell}$, then there exists $g \in \mathcal{S}^{\mathrm{pol},T}_{(\sigma, 0), \ell +1}$ such that $\mathcal{L}^{\mathrm{ell}}_{\omega, \Omega} \chi = g$. Then, we can write $\chi$ as in~\eqref{eq:sol_eq_inv_W} where $Z$ is the unique solution of the system~\eqref{syst:Eq_Z_inv_W}. Thanks to~\eqref{eq:prop_dist_Lomega_z_W} in Lemma \ref{lemma:HE_interm_LoO}, we have that 
    \begin{equation*}
        \partial_q \left(\Lo \mathrm{Re}Z\right) = \Lo \left( \partial_q \mathrm{Re}Z\right), \quad \partial_q \left(\Lo \mathrm{Im}Z\right) = \Lo \left( \partial_q \mathrm{Im}Z\right).
    \end{equation*}
    Using~\eqref{eq:sol_eq_inv_W} and the latter, we prove~\eqref{eq:prop_dist_Lomega_W}.

    In the analytic setting, we need to solve the following dynamical problem. For any $g \in \mathscr{S}^{\mathrm{pol}, T}_{\sigma, \ell +1}$, we need to find a unique solution $\chi \in \mathscr{W}^{\mathrm{ell},T}_{\sigma, \omega, \Omega, \ell}$ of equation~\eqref{proof:HE_eq_LoO}. Using the notation~\eqref{not:chi_HE_ell} and following the proof in the Hölder case, we observe that solving  equation~\eqref{proof:HE_eq_LoO} is equivalent to solving the system~\eqref{syst:HE_ell_1}. 

    We introduce the following variables $Z = (Z_1,\dots, Z_m) =  \chi^{(1)} - i\chi^{(2)}$, and $W =  (W_1,\dots, W_m) = \chi^{(1)} + i\chi^{(2)}$, and we rewrite system~\eqref{syst:HE_ell_1} as a system of $2m$ uncoupled equations
   \begin{equation}\label{syst:Eq_Z_inv_W_analy}
        \begin{aligned}
           &i\Omega_k Z_k - \Lo Z_k = \left(g^{(1)}_k - ig^{(2)}_{m+k}\right),\\
           -&i\Omega_k W_k - \Lo W_k = \left(g^{(1)}_k + ig^{(2)}_{m+k}\right).
        \end{aligned}   
   \end{equation}
   for $k=1,\dots m$. By Lemma \ref{lemma:HE_interm_LoO_analy}, there exists a unique solution of the latter satisfying
   \begin{equation*}
        Z, \, W \in \mathscr{S}^{\mathrm{pol},T}_{\sigma, \ell} \hspace{2mm}\mbox{and}\hspace{3mm} |Z|_{\sigma, \ell}^T, |W|_{\sigma, \ell}^T \le {2 \over \ell} |g|_{\sigma, \ell+1}^T.
   \end{equation*}
   Returning to the original variables $\chi^{(1)} = {W + Z \over 2}$ and $\chi^{(2)} = {W - Z \over 2i}$, we prove the invertibility of $\mathcal{L}^{\mathrm{ell}}_{\omega, \Omega}$ in this setting and estimate~\eqref{est:HE_ell_Holder_analy}. The commutation identity~\eqref{eq:prop_dist_Lomega_W} follows in the analytic setting by the same argument as in the Hölder case.
\end{proof}

Finally, in the following proposition, we study the invertibility of the linear operator $\mathfrak{L}^{\mathrm{ell}}_{\omega, \Omega}$.  We consider only the Hölder setting, as it is the only case required for the proofs of our theorems. We use the Banach spaces defined in~\eqref{def:BS_ell_M_Sym} and~\eqref{def:BS_ell_Sp}.
\begin{proposition}\label{prop:Csigma_HEomegaOmegaOmega}
    Given $\sigma \ge 0$, $T\ge 1$, and $\ell >0$, the following operator
    \begin{equation*}
        \mathfrak{L}^{\mathrm{ell}}_{\omega, \Omega} : \mathcal{S}p^{\mathrm{ell}, T}_{\sigma, \omega, \Omega, \ell} \longrightarrow \mathcal{S}ym^{\mathrm{pol}, T}_{\sigma, \ell+1}
    \end{equation*}
    is well-defined and invertible.   Moreover, 
    \begin{equation*}\label{est:LOoo_ell} 
\| (\mathfrak{L}^{\mathrm{ell}}_{\omega, \Omega})^{-1} \| \leq C(\ell),
    \end{equation*}
    for some $C > 0$ depending only on $\ell$.

\end{proposition}
\begin{proof}
For any $G \in \mathcal{S}p^{\mathrm{ell}, T}_{\sigma, \omega, \Omega, \ell}$, we recall that $\mathfrak{L}^{\mathrm{ell}}_{\omega, \Omega} G = G^{\top}M^{\mathrm{ell}}_\Omega + M^{\mathrm{ell}}_\Omega G + J_m \Lo G$. %
Thus, for any $g \in \mathcal{S}ym^{\mathrm{pol}, T}_{\sigma, \ell+1}$, this dynamical problem is equivalent of finding the unique solution $G \in \mathcal{S}p^{\mathrm{ell}, T}_{\sigma, \omega, \Omega, \ell}$ of the following equation
\begin{equation}\label{proof:HE_eq_LoOO}
   G^{\top}M^{\mathrm{ell}}_\Omega + M^{\mathrm{ell}}_\Omega G + J_m \Lo G = g.
\end{equation}

    To this end, we fix $g \in \mathcal{S}ym^{\mathrm{pol}, T}_{\sigma, \ell+1}$. Notice that $g$ can be written as
\[ g = \begin{pmatrix}
X & Y \\ Y^{\top} & -Z
\end{pmatrix},  \qquad  \begin{array}{ll}  Y : \T^n \times I_T \to \mathcal{M}_m(\R), \quad X, Z : \T^n \times I_T \to  \textup{Sym}(m, \R), \\
X_{ij}, Y_{ij}, Z_{ij} \in \mathcal{S}^{\mathrm{pol}, T}_{(\sigma, 0), \ell+1}, \qquad \text{ for any } 1 \leq i, j \leq m. \end{array}\]

Similarly, notice that any $G \in \mathcal{S}p^{\mathrm{ell}, T}_{\sigma, \omega, \Omega, \ell } $ can be written as 
\[ G = \begin{pmatrix}
P & Q \\ R & -P^{\top}
\end{pmatrix}, \qquad   \begin{array}{ll}  P : \T^n \times I_T \to \mathcal{M}_m(\R), \quad Q, R : \T^n \times I_T \to  \textup{Sym}(m, \R), \\ P_{ij}, Q_{ij}, R_{ij} \in \mathcal{S}^{\mathrm{pol},T}_{(\sigma, 0), \ell} \qquad \text{ for any } 1 \leq i, j \leq m.  \end{array} \]

With these notations, the equation~\eqref{proof:HE_eq_LoOO} becomes
\[ \left\{ \begin{array}{llll}
P^{\top}\Omega + \Omega P + \Lo R = X, \\
R \Omega + \Omega Q - \Lo P^{\top} = Y, \\
Q \Omega + \Omega R - \Lo P= Y^{\top}, \\
P \Omega + \Omega P^{\top} + \Lo Q = Z.
 \end{array} \right.  \]

 Denoting
 \[  \mathfrak{F} = P - P^{\top} + i (R - Q), \qquad {}\mathfrak{G} = P + P^{\top} + i (R + Q),\]

the system of equations above can be written as
\[ \left\{ \begin{array}{llll}
 i\Omega \mathfrak{F} - i \mathfrak{F} \Omega + \Lo \mathfrak{F} = (Y - Y^{\top}) + i(X - Z),  \\
 i\Omega \mathfrak{G} + i \mathfrak{G} \Omega + \Lo \mathfrak{G} =  -(Y + Y^\top) + i(X + Z),
 \end{array} \right. \]
 or equivalently,
 \[ \left\{ \begin{array}{llll}
 i(\Omega_j - \Omega_k)\mathfrak{F}_{jk} + \Lo \mathfrak{F}_{jk} = Y_{jk} - Y_{kj} + i(X_{jk} - Z_{jk}),  \qquad 1 \leq j,k \leq m,\\
  i(\Omega_j + \Omega_k)\mathfrak{G}_{jk} + \Lo \mathfrak{G}_{jk} = -Y_{jk} - Y_{kj} + i(X_{jk} + Z_{jk}),  \qquad 1 \leq j,k \leq m.
 \end{array} \right.\]
 
 By Proposition \ref{lemma:HE_interm_LoO}, each of the equations above admits a unique solution whose real and imaginary parts belong to $\mathcal{S}^{\mathrm{pol}, T}_{(\sigma, 0), \ell}$.  Noticing that
 \[ \begin{pmatrix}
P & Q \\ R & -P^{\top}
\end{pmatrix} = \frac{1}{2}
\begin{pmatrix}
\mathrm{Re}(\mathfrak{F} + \mathfrak{G}) & \mathrm{Im}(\mathfrak{G} - \mathfrak{F}) \\ \mathrm{Im}(\mathfrak{F} + \mathfrak{G}) & \mathrm{Re}(\mathfrak{F} - \mathfrak{G})
\end{pmatrix}, \]
it follows that $\mathfrak{L}^{\mathrm{ell}}_{\omega, \Omega}$ is invertible. Moreover, by~\eqref{est:HE_Interm_1dim_LoO} in Proposition \ref{lemma:HE_interm_LoO}, there exists a positive constant depending on $\ell$ such that \[ |(\mathfrak{L}^{\mathrm{ell}}_{\omega, \Omega})^{-1}g|^T_{\sigma, \omega, \Omega, \ell} \le C(\ell) |g|^T_{\sigma, \ell+1},  \]
which finishes the proof of the proposition. 
\end{proof}

\subsection{Norm properties}
\label{sc:norm_properties_ell}
In this section, we prove several properties for the norms of the Banach spaces introduced in Sections \ref{sc:functional_setting_ell} and \ref{sc:proof_analytic_ell}. Throughout this section, we denote by $C(\cdot)$ a generic positive constant depending on the dimensions $n$ and $m$, and the parameter in brackets.

\begin{proposition}\label{prop:Csigma_prop_norms_pol}
    Given $\sigma$, $\ell$, $d \ge 0$, $T\ge 1$ and a non-negative integer $k$, for all $f \in \mathcal{S}^{\mathrm{pol},T}_{(\sigma, k), \ell}$ and $g \in \mathcal{S}^{\mathrm{pol}, T}_{(\sigma, k), d}$ we have the following properties.
\begin{enumerate}
\item For all $s\ge 0$ and  $\beta \in \N^{2(n + m)}$, if $|\beta| + s \le \sigma + k$, then  
\begin{equation*}
\left|\partial^{\beta}_{(q,p,z)} f \right|^T_{s, \ell} \le C |f|^T_{\sigma +k, \ell}
\end{equation*}
\item For all $\ell' \ge 0$, if $f \in \mathcal{S}^{\mathrm{pol},T}_{(\sigma, k), \ell+\ell'}$ then $|f|^T_{\sigma+k, \ell}  \le T^{-\ell'} |f|^T_{\sigma+k, \ell+\ell'}$,
\item $|fg|^T_{\sigma+k, \ell+d} \le C(\sigma, k)\left(|f|^T_{0,\ell}|g|^T_{\sigma+k,d} + |f|^T_{\sigma+k,\ell}|g|^T_{0,d}\right)$. 
\item We consider $\sigma \ge 1$,  and we assume that $g:\T^n \times B \times I_T \to \T^n \times B$.  %
Letting $\tilde g: \T^n \times B \times I_T \to \T^n \times B \times I_T$ such that $\tilde g(q,p,z,t) = (g(q,p,z,t), t)$, 
then $f \circ  \tilde g \in \mathcal{S}^{\mathrm{pol}, T}_{\sigma+k, \ell}$ and 
\begin{equation*}
|f \circ \tilde g|^T_{\sigma+k, \ell} \le C(\sigma, k) \left(|f|^T_{\sigma+k,\ell}\left(|\partial_{(q,p,z)} g|^T_{0,d}\right)^\sigma + |f|^T_{1,\ell}|\partial_{(q,p,z)}  g|^T_{\sigma+k-1,d} +  |f|^T_{0, \ell}  \right).
 \end{equation*}
\end{enumerate}
\end{proposition}
\begin{proof}
    We refer to~\cite{Sca25} for the proof. 
\end{proof}

 As a direct consequence of Proposition \ref{prop:Csigma_prop_norms_pol}, we have the following.

\begin{proposition}\label{prop:Csigma_prop_norms_matrix}
    Let $\sigma, \ell, d\ge 0$ and $T\ge 1$. For any $F \in  \mathcal{M}^{\mathrm{pol}, T}_{\sigma, \ell}$ and $G \in  \mathcal{M}^{\mathrm{pol}, T}_{\sigma, d}$ we have the following properties.
\begin{enumerate}
\item For all $s\ge 0$ and  $\beta \in \N^{2(n + m)}$, if $|\beta| + s \le \sigma $, then  
\begin{equation*}
\left|\partial^{\beta}_{(q,p,z)} F \right|^T_{s, \ell} \le C |F|^T_{\sigma , \ell}
\end{equation*}
\item \label{prop:crescita_indici_norm_ell} For all $\ell' \ge 0$,  if $F \in \mathcal{M}^{\mathrm{pol}, T}_{\sigma, \ell+\ell'}$ then $|F|^T_{\sigma, \ell}  \le T^{-\ell'} |F|^T_{\sigma, \ell+\ell'}$, 
\item \label{prop:bound_norm_product_matrices} $|FG|^T_{\sigma, \ell+d} \le C(\sigma)\left(|F|^T_{0,\ell}|G|^T_{\sigma,d} + |F|^T_{\sigma,\ell}|G|^T_{0,d}\right)$. %
\end{enumerate}
\end{proposition}

Recall that the derivative of the exponential operator $\exp: \mathcal{M}_{2m}(\R) \to \mathcal{M}_{2m}(\R)$ at $F \in \mathcal{M}_{2m}$ and evaluated at $G \in  \mathcal{M}_{2m}$ is given by
\begin{equation}
\label{eq:derivative_exp}
D \exp(F)G = e^{F} \mathcal{A}(F, G),
\end{equation}
where
\begin{equation}
\label{eq:adjoint_formula}
\mathcal{A}(F, G) := \int_0^1e^{-sF}Ge^{sF}ds.
\end{equation}
Similarly, given a differentiable map $s \mapsto G(s)$, the derivative of the map $s \mapsto e^{G(s)}$ is given by
\begin{equation}
\label{eq:derivative_exp_parameter}
\frac{d }{ds} \exp (F(s)) = e^{F(s)}\mathcal{A}(F(s), F'(s )).
\end{equation}

The following properties follow easily from the definitions together with Proposition \ref{prop:Csigma_prop_norms_matrix}.

\begin{lemma}
\label{lem:exp_expansion_bounds}
      Let $\sigma, \ell, d \ge 0$ and $T\ge 1$. For any $B \in \mathcal{M}^{\mathrm{pol}, T}_{\sigma, 0},$ $F \in  \mathcal{M}^{\mathrm{pol}, T}_{\sigma, \ell},$ and $G \in  \mathcal{M}^{\mathrm{pol}, T}_{\sigma, d},$ the following statements hold.

      \begin{enumerate}
          \item \label{it:bound_expF} $|e^B|_{\sigma, 0}^T \leq e^{C(\sigma)|B|_{\sigma, 0}^T}.$
          \item \label{it:bound_expF-id} $|e^F - \mathrm{Id}_{2m}|_{\sigma, \ell}^T \leq  |F|_{\sigma, \ell}^Te^{C(\sigma)|F|_{\sigma, 0}^T}$.       
          \item \label{it:bound_expF-id-F} $|e^F - \mathrm{Id}_{2m} - F|_{\sigma, 2\ell}^T \leq (|F|_{\sigma, \ell}^T)^2e^{C(\sigma)|F|_{\sigma, 0}^T}$.
          \item \label{it:bound_adjoint} $|\mathcal{A}(B, G)|_{\sigma, d}^T \leq C(\sigma)|G|_{\sigma, d}^Te^{C(\sigma)|B|_{\sigma, 0}^T}$.
          \item \label{it:bound_adjoint-G} $|\mathcal{A}(F, G) - G|_{\sigma, \ell + d}^T \leq C(\sigma)|F|_{\sigma, \ell}^T |G|_{\sigma, d}^T e^{C(\sigma)|F|_{\sigma, 0}^T}.$
          \item \label{it:bound_operator_L} If $\Lo B \in \mathcal{M}_{\sigma, \ell}^{\mathrm{pol}, T}$ then $|\Lo (e^B)|_{\sigma, \ell}^T \leq C(\sigma) e^{C(\sigma)|B|^T_{\sigma, 0}}|\Lo B|_{\sigma, \ell}^T$.
          \item \label{it:bound_operator_LA} If $\Lo B \in \mathcal{M}_{\sigma, 1}^{\mathrm{pol},T}$ and $\Lo G \in \mathcal{M}_{\sigma, d + 1}^{\mathrm{pol},T}$ then $$|\Lo (\mathcal{A}(B, G))|_{\sigma, d + 1}^T \leq C(\sigma) e^{C(\sigma)|B|^T_{\sigma, 0}}(|\Lo B|_{\sigma, 1}^T|G|_{\sigma, d}^T + |\Lo G|_{\sigma, d + 1}^T).$$
          \item \label{it:bound_operator_LA-G} If $\Lo F \in \mathcal{M}_{\sigma, \ell + 1}^{\mathrm{pol},T}$ and $\Lo G \in \mathcal{M}_{\sigma, d + 1}^{T}$ then $$|\Lo (\mathcal{A}(F, G) - G)|_{\sigma, \ell + d + 1}^{\mathrm{pol},T} \leq C(\sigma) e^{C(\sigma)|Fe|^T_{\sigma, 0}}(|\Lo F|_{\sigma, \ell + 1}^T|G|_{\sigma, d}^T + |F|_{\sigma, \ell}^T |\Lo G|_{\sigma, d + 1}^T).$$
      \end{enumerate}
      
\end{lemma}
\begin{proof}
Along the proof, for the sake of simplicity, we will denote by $C$ a generic positive constant depending only on $n, m$ and $\sigma$. %

By Item \eqref{prop:bound_norm_product_matrices} of Proposition \ref{prop:Csigma_prop_norms_matrix} and since  $e^B = \sum_{k \geq 0}\frac{B^k}{k!}$,
    \begin{align*}
    |e^B|_{\sigma, 0}^T & \leq \sum_{k \geq 0}\frac{|B^k|^T_{\sigma, 0}}{k!}  \leq 1 + \sum_{k \geq 1}\frac{C^{k - 1}(|B|^T_{\sigma, 0})^k}{k!}  \leq e^{C|B|^T_{\sigma, 0}}, 
    \end{align*}
    which proves Property \eqref{it:bound_expF}. 
    
    Similarly,
    \begin{align*}
    |e^F - \mathrm{Id}_{2m}|_{\sigma, \ell}^T & \leq  \sum_{k \geq 1}\frac{|F^k|^T_{\sigma, \ell}}{k!}  \leq \sum_{k \geq 1}\frac{C|F|^T_{\sigma, \ell}|F^{k - 1}|^T_{\sigma, 0}}{k!}  \leq  \sum_{k \geq 1}\frac{C^{k}|F|^T_{\sigma, \ell}(|F|^T_{\sigma, 0})^{k - 1}}{k!}   \\ & \leq |F|^T_{\sigma, \ell}  e^{C|F|^T_{\sigma, 0}},
    \end{align*}
    which proves Property \eqref{it:bound_expF-id}.
    
    Furthermore,
    \begin{align*}
    |e^F - \mathrm{Id}_{2m} - F|_{\sigma, 2\ell}^T & \leq  \sum_{k \geq 2}\frac{|F^k|^T_{\sigma, 2\ell}}{k!}  \leq \sum_{k \geq 2}\frac{{C|F|^T_{\sigma, \ell}|F^{k - 1}|^T_{\sigma, \ell}}}{k!}  \leq  \sum_{k \geq 2}\frac{C^{k}(|F|^T_{\sigma, \ell})^2(|F|^T_{\sigma, 0})^{k - 2}}{k!}   \\ & \leq (|F|^T_{\sigma, \ell})^2  e^{C|F|^T_{\sigma, 0}},
    \end{align*}
    which proves Property \eqref{it:bound_expF-id-F}.

   By Item \eqref{prop:bound_norm_product_matrices} of Proposition \ref{prop:Csigma_prop_norms_matrix} and Property \eqref{it:bound_expF} of this Lemma, we have
    \begin{align*}
    |\mathcal{A}(B, G)|_{\sigma, d}^T & \leq C^2\int_0^1 |e^{-sB}|_{\sigma, 0}^T |G|_{\sigma, d}^T  |e^{sB}|_{\sigma, 0}^T \leq C^2  |G|_{\sigma, d}^T e^{2C |B|_{\sigma, 0}^T},
    \end{align*}
    which proves Property \eqref{it:bound_adjoint}.

     By Item \eqref{prop:bound_norm_product_matrices} of Proposition \ref{prop:Csigma_prop_norms_matrix} and Properties \eqref{it:bound_expF}, \eqref{it:bound_expF-id} of this Lemma,
    \begin{align*}
    |\mathcal{A}(F, G) -G|_{\sigma, \ell + d}^T & = \left|  \int_0^1 (e^{-sF} - \mathrm{Id}_{2m})Ge^{sF} + G(e^{sF} - \mathrm{Id}_{2m})ds  \right|_{\sigma, \ell + d}^T \\
    & \leq \int_0^1 \left( C^2|e^{-sF} - \mathrm{Id}_{2m}|_{\sigma, \ell}^T |G|_{\sigma, d}^T  |e^{sF}|_{\sigma, 0}^T  + C  |G|_{\sigma, d}^T |e^{sF} - \mathrm{Id}_{2m}|_{\sigma, \ell}^T \right)ds \\
    &\leq C^2 |F|_{\sigma, \ell}^T e^{C|F|_{\sigma, 0}^T}|G|_{\sigma, d}^T e^{C|F|_{\sigma, 0}^T} + C|G|_{\sigma, d}^T|F|_{\sigma, \ell}^T e^{CF|_{\sigma, 0}^T},
    \end{align*}
    which proves Property \eqref{it:bound_adjoint-G}.

    By Item \eqref{prop:bound_norm_product_matrices} of Proposition \ref{prop:Csigma_prop_norms_matrix}, Properties \eqref{it:bound_expF}-\eqref{it:bound_adjoint-G} of this Lemma and  \eqref{eq:derivative_exp_parameter}, we have
    \[ |\Lo (e^B)|^T_{\sigma, \ell} = |e^B \mathcal{A}(B, \Lo B)|^T_{\sigma, \ell} \leq  C C e^{(C + C)|B|^T_{0, \ell}}|\Lo B|_{\sigma, \ell}^T,\]
        which proves Property \eqref{it:bound_operator_L}.

    By Item \eqref{prop:bound_norm_product_matrices} of Proposition \ref{prop:Csigma_prop_norms_matrix} and Properties \eqref{it:bound_expF}-\eqref{it:bound_operator_L} of this Lemma,
    \begin{align*}
    |\Lo(\mathcal{A}(B, G))|_{\sigma, d + 1}^T & = \left|  \int_0^1 e^{-sB} \left( \mathcal{A}(-sB, -s\Lo B) Ge^{sB} + \Lo Ge^{sB} + Ge^{sB} \mathcal{A}(sB, s\Lo B) \right)ds  \right|_{\sigma, \ell + d}^T \\
    &\leq C^4 e^{C|B|_{\sigma, 0}}\left(2C |\Lo B|_{\sigma, 1}^T|G|^T_{\sigma, d} +  |\Lo G|_{\sigma, d + 1}^T\right)e^{C|B|_{\sigma, 0}^T},
    \end{align*}
        which proves Property \eqref{it:bound_operator_LA}.
        
    Similarly, by Item \eqref{prop:bound_norm_product_matrices} of Proposition \ref{prop:Csigma_prop_norms_matrix}, Properties \eqref{it:bound_expF}-\eqref{it:bound_operator_LA},
    \begin{align*}
    | & \Lo(\mathcal{A}(F, G)  - G)|_{\sigma, \ell + d + 1}^T = \left|  \int_0^1  \Lo \left( (e^{-sF} - \mathrm{Id}_{2m})Ge^{sF} + G(e^{sF} - \mathrm{Id}_{2m}) \right) ds  \right|_{\sigma, \ell + d + 1}^T \\
    & = \left|  \int_0^1 e^{-sF} \mathcal{A}(-sF, -s\Lo F)Ge^{sF} + (e^{-sF} - \mathrm{Id}_{2m})(\Lo G e^{sF} + G e^{sF} \mathcal{A}(sF, s\Lo F )) \right. \\
    & \quad + \Lo G(e^{sF} - \mathrm{Id}_{2m}) + Ge^{sF}\mathcal{A}(sF, s\Lo F ) ds  \Bigg|_{\sigma, \ell + d + 1}^T \\
    & \leq C^4\Big( C e^{(2C + C)|F|_{\sigma, 0}^T} |\Lo F|_{\sigma, \ell + 1}^T |G|_{\sigma, \ell}^T +  C|F|_{\sigma, \ell}^Te^{(C + C + C)|F|_{\sigma, 0}^T} (|\Lo G|_{\sigma, d + 1}^T + |G|_{\sigma, d}^T|\Lo F|_{\sigma, \ell + 1}^T )  \\
    & \quad + e^{C|F|_{\sigma, 0}^T}|\Lo G|_{\sigma, d + 1}^T |F|_{\sigma, \ell}^T + C e^{C|F|_{\sigma, 0}^T} |G|_{\sigma, d}^T|\Lo F|_{\sigma, \ell + 1}^T \Big) 
    \end{align*}
            which proves Property \eqref{it:bound_operator_LA-G}.
\end{proof}

\begin{proposition}\label{prop:properties_analy_pol}
Given $\sigma>0$, $\ell, \, d \ge 0$ and $T\ge 1$,  for all $f \in \mathscr{S}^{\mathrm{pol}, T}_{\sigma, \ell}$ and $g \in \mathscr{S}^{\mathrm{pol}, T}_{\sigma, d}$, we have the following properties    
\begin{enumerate}
    \item For all $0<\sigma'<\sigma$ and $\beta \in \N^{2(n+m)}$, then
    \begin{equation*}
        \left|\partial^{\beta}_{(q,p,z)} f \right|^T_{\sigma', \ell} \le {C \over (\sigma-\sigma')^{|\beta|}} |f|^T_{\sigma, \ell}.
    \end{equation*}
    \item For all $\ell'\ge 0$, if $f \in \mathscr{S}^{\mathrm{pol}, T}_{\sigma, \ell+\ell'}$ then $|f|^T_{\sigma, \ell} \le T^{-\ell'}|f|^T_{\sigma, \ell+\ell'}$.
    \item $|fg|^T_{\sigma, \ell +d} \le |f|^T_{\sigma, \ell}|g|^T_{\sigma, d}$. %
    \item For all $0<\sigma'<\sigma$, we consider $g:\T^n_{\sigma'} \times B_{\sigma'} \times I_T \to \T^n_\sigma \times B_\sigma$ such that $g \in \mathscr{S}^{\mathrm{pol}, T}_{\sigma',0}$. Letting $\tilde g: \T^n_{\sigma'} \times B_{\sigma'} \times I_T \to \T^n_\sigma \times B_\sigma \times I_T$ such that $\tilde g(q,p,z,t) = (g(q,p,z,t), t)$, 
then $f \circ  \tilde g \in \mathscr{S}^{\mathrm{pol}, T}_{\sigma', \ell}$ and 
\begin{equation*}
    |f\circ \tilde g|^T_{\sigma', \ell} \le |f|^T_{\sigma, \ell}.
\end{equation*}
\end{enumerate}
\end{proposition}
\begin{proof}
    The proof of the first property is a straightforward application of Cauchy's estimates. The remaining ones follow directly from the definitions of the norms~\eqref{def:norm_analy} and~\eqref{def:norm_S_analy}.
\end{proof}

\section{Asymptotically Hyperbolic Case}
\label{sec:Proof_Theorem_Hyp}

In this section we prove Theorems \ref{Thm:hyp_Csigma}, \ref{Thm:hyp_analy} and Corollary \ref{cor:hyp}, which are the hyperbolic counterparts of the results proven in Section \ref{sec:Proof_Theorem_Ell}. 

We will structure this section similarly to Section \ref{sec:Proof_Theorem_Ell} focusing mostly on the proof of Theorem \ref{Thm:hyp_Csigma}, which follows the same proof strategy as Theorem \ref{Thm:elliptic_Csigma} and relies on a version of the Implicit Function Theorem applied to appropriately defined Banach spaces of time-dependent functions exhibiting exponential decay in time. By doing this, we will be naturally led to consider the existence of solutions to the following (cohomological) equations, which are the analogs of those in \eqref{eq:coh_eqs_ell}, namely,
\begin{equation}
\label{eq:coh_eqs_hyp}
\left\{ \begin{array}{lll}
    \Lo \chi := (\partial_t + \omega \cdot \partial_q) \chi = g, & & g: \T^n \times I_T \to \R,\\
    \mathcal{L}^{\mathrm{hyp}}_{\omega, \Omega} \boldsymbol{\chi} :=  J_mM^{\mathrm{hyp}}_\Omega \boldsymbol{\chi} - \Lo \boldsymbol{\chi}  = g, & & \boldsymbol{g}: \T^n \times I_T \to \R^{2m}, \\
    \mathfrak{L}^{\mathrm{hyp}}_{\omega, \Omega} X := X^{\top}M^{\mathrm{hyp}}_\Omega + M^{\mathrm{hyp}}_\Omega X + J_m \Lo X =  G, & & G: \T^n \times I_T \to \mathcal{M}_{2m}(\R),
    \end{array} \right.
\end{equation}
in the unknowns $\chi: \T^n \times I_T \to \R$, $\boldsymbol{\chi}: \T^n \times I_T \to \R^{2m}$ and $X: \T^n \times I_T \to \mathcal{M}_{2m}(\R)$ (restricted to appropriate Banach spaces of functions with exponential decay),  where $\omega \in \R^n$, $\Omega = \mathrm{diag}(\Omega_1,\dots,\Omega_m)$ with $\Omega_1,\dots,\Omega_m \in \R_{> 0}$, and the matrices $M^{\mathrm{hyp}}_\Omega$ and $J_m$ are as in \eqref{Ms} and~\eqref{def:J}, respectively. Recall that  $I_T = [T, + \infty)$, for any $T \geq 1$.

 The remainder of this section is structured as follows. In Section \ref{sc:functional_setting_hyp}, we introduce several Banach spaces of functions with exponential decay in time. In Section \ref{sc:initial_ham_hyp} we write the unperturbed Hamiltonian of Theorem \ref{Thm:hyp_Csigma} in a suitable form and set some notations. In Sections \ref{sec:proof_Thm_hyp_Csigma_1} and \ref{sec:proof_thm_b_hyp_2}, we prove items \eqref{thm:B_existence} and \eqref{thm:B_transverse} of Theorem \ref{Thm:hyp_Csigma}, respectively, assuming several results concerning the cohomological equations mentioned above which, for the sake of clarity of exposition, we prove later in Section \ref{sc:cohom_eqs_hyp}. In Section \ref{sc:proof_analytic_hyp} we prove Theorem \ref{Thm:hyp_analy}.  Section \ref{sec:proof_cor_hyp} is dedicated to the proof of Corollary \ref{cor:hyp}. Finally, Section \ref{sc:norm_properties_hyp} contains several technical results about the norms and spaces introduced in Sections \ref{sc:functional_setting_hyp} and \ref{sc:proof_analytic_hyp}.%

\subsection{Functional setting}\label{sc:functional_setting_hyp}  In the following, we let $\omega \in \R^n$, $\Omega = \mathrm{diag}(\Omega_1,\dots,\Omega_m)$ with $\Omega_1,\dots,\Omega_m \in \R_{> 0}$, $\sigma \geq 1$, $k \in \Z_{>0}$, $\lambda \geq 0$ and $T \geq 0$, and we use the notations in \eqref{eq:coh_eqs_hyp} for the operators $\Lo, \mathcal{L}^{\mathrm{hyp}}_{\omega, \Omega}$ and $\mathfrak{L}^{\mathrm{hyp}}_{\omega, \Omega}.$

\subsubsection{Real-valued functions with exponential decay} \label{sc:functions_exp_decay} Recall that, the Banach space $\Sexp_{(\sigma, k), \lambda}$ is defined in~\eqref{def:Sexp} and the associated norm $|\cdot |^{\mathrm{exp}, T}_{\sigma+k, \lambda}$ in \eqref{def:norm_Sexp}.

In what follows, we introduce the counterparts of the spaces defined in~\eqref{def: S0ell},~\eqref{def:U}, and~\eqref{def:W} for functions decaying exponentially fast in time. Given $k \ge 1$, we define
\begin{equation}\label{def: S0exp}
    \mathcal{S}^{\mathrm{exp},T}_{(\sigma,k), (0,\lambda)} = \left\{f:\T^n \times B \times I_T \to \R  \hspace{1mm} \left| \hspace{1mm} f \in \mathcal{S}^{\mathrm{exp}, T}_{(\sigma,k), 0}, \hspace{2mm} \partial_q f \in \mathcal{S}^{\mathrm{exp}, T}_{(\sigma, k-1), \lambda}\right\}\right.,
\end{equation}
\begin{equation}\label{def:Uexp}
    \mathcal{U}^{\mathrm{exp}, T}_{\sigma, \omega, \lambda} = \left\{u:\T^n \times I_T \to \R \hspace{1mm} \left| \hspace{1mm} u \in \mathcal{S}^{\mathrm{exp},T}_{(\sigma,0), \lambda}, \hspace{2mm}\Lo u \in \mathcal{S}^{\mathrm{exp}, T}_{(\sigma,0), \lambda}\right\}\right.,
\end{equation}
\begin{equation}\label{def:Wexp}
    \mathcal{W}^{\mathrm{hyp}, T}_{\sigma, \omega, \Omega, \ell} = \left\{w:\T^n \times I_T \to \R^{2m} \hspace{1mm} \left| \hspace{1mm} w \in \mathcal{S}^{\mathrm{exp}, T}_{(\sigma,0), \lambda}, \hspace{2mm}\mathcal{L}^{\mathrm{hyp}}_{\omega,\Omega} w \in \mathcal{S}^{\mathrm{pol}, T}_{(\sigma,0), \lambda}\right\}\right.,
\end{equation}
endowed with the norms
\begin{equation}\label{def:norm_S0exp}
    |f|^{\mathrm{exp}, T}_{(\sigma,k), (0,\lambda)} = \max\{|f|^{\mathrm{exp}, T}_{\sigma+k, 0}, |\partial_qf|^{\mathrm{exp}, T}_{\sigma+k-1, \lambda}\},
\end{equation}
\begin{equation}\label{def:norm_Uexp}
    |u|^{\mathrm{exp}, T}_{\sigma, \omega, \lambda} = \max \left\{|u|^{\mathrm{exp}, T}_{\sigma, \lambda}, |\Lo u|^{\mathrm{exp}, T}_{\sigma, \lambda}\right\},
\end{equation}
\begin{equation}\label{def:norm_Wexp}
    |w|^{\mathrm{exp}, T}_{\sigma, \omega, \Omega, \lambda} = \max \left\{|w|^{\mathrm{exp}, T}_{\sigma, \lambda}, |\mathcal{L}^{\mathrm{hyp}}_{\omega,\Omega} w|^{\mathrm{exp}, T}_{\sigma, \lambda}\right\},
\end{equation}
respectively.

\subsubsection{Matrix-valued functions with exponential decay}
\label{sc:matrices_exp_decay}
In this part, we introduce a variant of the Banach spaces of matrix-valued functions defined earlier in this section that retains the same structure but incorporates exponential time decay rather than polynomial decay. We define
\begin{equation}
\begin{aligned}\label{def:M_Sym_exp_Holder}
     \mathcal{M}^{\mathrm{exp}, T}_{\sigma, \lambda} &= \left\{M:\T^n \times I_T \to \mathcal{M}_{2m}( \R) \hspace{1mm} \left| \hspace{1mm}  \forall 1 \leq i, j \leq m, \quad M_{ij}  \in \mathcal{S}^{\mathrm{exp}, T}_{(\sigma,0), \lambda} \right\}\right. , \\
          \mathcal{S}ym^{\mathrm{exp}, T}_{\sigma, \lambda} &= \left\{M:\T^n \times I_T \to \textup{Sym}(2m, \R) \hspace{1mm} \left| \hspace{1mm}  M \in \mathcal{M}^{\mathrm{exp}, T}_{\sigma, \lambda} \right\}\right. ,
          \end{aligned}
          \end{equation}
endowed with the norm,
\[ |M|^{\mathrm{exp}, T}_{\sigma, \lambda} := \max_{1 \leq i, j \leq m} |M_{ij}|^{\mathrm{exp}, T}_{\sigma, \lambda},\]
 and
 \begin{align}\label{def:Sp_exp_Holder}
\mathcal{S}p^{\mathrm{hyp}, T}_{\sigma, \omega, \Omega, \lambda}
&  = \left\{M:\T^n \times I_T \to \mathfrak{sp}(2m, \R) \hspace{1mm} \left| \hspace{1mm} M \in \mathcal{M}^{\mathrm{exp}, T}_{\sigma, \lambda}, \quad \mathfrak{L}_{\omega, \Omega}^{\mathrm{hyp}} M \in \mathcal{M}^{\mathrm{exp}, T}_{\sigma, \lambda}  \right\} \right.,
\end{align}
endowed with the norm 
\[ |M|^{\mathrm{exp}, T}_{\sigma, \omega, \Omega, \lambda} = \max\left \{ |M|_{\sigma, \lambda}^{\mathrm{exp}, T} ,|\mathfrak{L}_{\omega, \Omega}^{\mathrm{hyp}} M|^{\mathrm{exp}, T}_{\sigma, \lambda} \right\}.\]

\subsubsection{Perturbations and torus embeddings with exponential decay}
\label{sc:perturbation_hyp}
To describe the space of perturbations considered in Theorem \ref{Thm:hyp_Csigma} and the families of torus embeddings with exponential decay, we define
\begin{equation}\label{proof_hyp_1_def_spaces_P_E}
    \begin{aligned}
    &\mathcal{P}^{\mathrm{hyp}, T}_{\sigma, \lambda} = \mathcal{S}^{\mathrm{exp},T}_{(\sigma,2), (0,\lambda)}(\R) \times \mathcal{S}^{\mathrm{exp},T}_{(\sigma,2), \lambda}(\R^n) \times \mathcal{S}^{\mathrm{exp},T}_{(\sigma,2), \lambda}(\R^{2m}) \times \mathcal{S}^{\mathrm{exp},T}_{(\sigma,2), \lambda}(\R^{2m \times 2m}),\\
    &\mathcal{E}^{\mathrm{hyp}, T}_{\sigma, \lambda} = \mathcal{U}^{\mathrm{exp},T}_{\sigma, \omega, \lambda} \times \mathcal{U}^{\mathrm{exp},T}_{\sigma, \omega, \lambda} \times \mathcal{W}^{\mathrm{hyp}, T}_{\sigma, \omega, \Omega, \lambda},
    \end{aligned}
\end{equation}
each endowed with the sum norm associated with the corresponding product space. With a slight abuse of notation, we use the notation $| \cdot |^{\mathrm{hyp}, T}_{\sigma, \lambda}$ for both norms, whenever no confusion can arise. Notice that these are well-defined Banach spaces. 

Similarly to Sections \ref{sc:embeddings_ell} and \ref{sc:perturbation_ell}, as an abuse of notation, we denote the elements of these product spaces simply as $P = (a, b, c, d) \in \mathcal{P}^{\mathrm{hyp}, T}_{\sigma, \lambda}$ and $\varphi = (u, v, w) \in \mathcal{E}^{\mathrm{hyp}, T}_{\sigma, \lambda}$. Moreover, we will use the same letters to denote the associated perturbation $P$ (given by \eqref{eq:perturbation_form_hyp}) and family of torus embeddings $\varphi$ (given by \eqref{eq:formula_torus_embeddings}). 

With these notations, the space of perturbations associated with Items \eqref{thm:B_existence} and \eqref{thm:B_transverse} of Theorem \ref{Thm:hyp_Csigma} is given by, $\mathcal{P}^{\mathrm{hyp}, 0}_{\sigma, \lambda}$ with 
\[\sigma \geq 1, \quad \lambda > \max_{1 \leq i \leq m} \Omega_i, \qquad \text{and} \qquad \sigma \geq 2, \quad \lambda > 2\max_{1 \leq i \leq m} \Omega_i,\] respectively. 

Moreover, as we shall see in Section \ref{sec:proof_Thm_hyp_Csigma_1}, for any $P \in  \mathcal{P}^{\mathrm{hyp}, 0}_{\sigma, \lambda}$ with $\sigma \geq 1$ and $\lambda > \max_{1 \leq i \leq m} \Omega_i$ there exists a $C^\sigma$ asymptotic KAM torus in $ \mathcal{E}^{\mathrm{hyp}, T}_{\sigma, \lambda}$ associated with $(X_{H_0 + P}, X_{H_0}, \varphi_0)$, for some $T \geq 0$.

\subsection{Initial Hamiltonian and conventions in notation}\label{sc:initial_ham_hyp}

For the remainder of this section, we fix $\Omega_1,\dots,\Omega_m \in \R_{> 0}$, $\omega \in \R^n$, $\sigma \geq 1$ and $H_0: \T^n \times B \times I_1 \to \R$ of the form \eqref{eq:initial_hamiltonian} as in the statement of Theorem \ref{Thm:hyp_Csigma}, where $B \subseteq \R^{n +2m}$ is a ball centred at $0$ and $I_T = [T, + \infty)$, for any $T \geq 0$. For the sake of simplicity, let us assume that $B = B_{\R^{n + 2m}}(1)$ is the unit ball in $\R^{n + 2m}$. We denote $\Omega = \mathrm{diag}(\Omega_1,\dots,\Omega_m)$.

Since all the norms considered in this section concern (products of) function spaces with polynomial decay, for the sake of clarity and to simplify the notation, we omit the superscripts $\mathrm{hyp}$ and $\mathrm{exp}$ from all the norms whenever there is no risk of confusion.

Using the notations introduced in Section \ref{sec:FS_Ref_Ham},  we can express $H_0$ as in \eqref{eq:initial_Hamiltonian_NR}, namely,
\[         H_0(q,p,z,t) = N(p,z) + R(q,p,z,t), \]
where 
\begin{equation*}
\begin{array}{rcl}   N(p,z) &= & \omega \cdot p + {1 \over 2} \Mhyp \cdot z^2, \\
    R(q,p,z,t) &= &M(q,p,z,t) \cdot p^2 + m(q,z,t) \cdot (p,z) + L(q,z,t) \cdot z^3,
\end{array}
\end{equation*}
$\Mhyp$ is given by \eqref{Ms}, $M, m, L$ are given by \eqref{def:abcdMmL_terms} and satisfy
\begin{equation}
  \label{def:const_Upsilon_hyp}
|M|^{0}_{\sigma + 3, 0}, |m|_{\sigma + 3, 0}^{0}, |L|_{\sigma + 2, 0}^{ 0} \le|\partial_{(p,z)}^2 H_0|^{0}_{\sigma+3, 0} \leq \Upsilon, \end{equation}
for some $\Upsilon > 0$ that we fix for the rest of the section.

We recall that the trivial embedding $\varphi_0: \T^n \to \T^n \times \R^{n + 2m}$, given by \eqref{def:varphi0=(q,0,0)}, defines an invariant torus for $H_0$ supporting quasiperiodic solutions with frequency vector $\omega$.

As mentioned before, we prove the two statements of Theorem \ref{Thm:hyp_Csigma} separately. The structure of the proof is similar to that of Theorem \ref{Thm:elliptic_Csigma}. For this reason, some parts of the proof are only sketched.

\subsection{Proof of Item \eqref{thm:B_existence} of Theorem \ref{Thm:hyp_Csigma}: Construction of $C^\sigma$ asymptotic KAM tori}\label{sec:proof_Thm_hyp_Csigma_1} 
In the following, we will use freely the notations introduced at the beginning of this section and in the previous two subsections.

Let us fix  $\lambda > \max_{1 \le i \le m} \Omega_i$. Recall that the space of perturbations associated with Item \eqref{thm:B_existence} of Theorem \ref{Thm:elliptic_Csigma} is given by  $\mathcal{P}^{\mathrm{hyp}, 0}_{\sigma, \lambda}$ (see Section \ref{sc:perturbation_hyp}).

 Given $r, \rho>0$ and $T \geq 0$, we denote by $B_{\mathcal{P}^{\mathrm{hyp}, T}_{\sigma, \lambda}}(r)  \subseteq  \mathcal{P}^{\mathrm{hyp}, T}_{\sigma, \lambda}$ and $B_{ \mathcal{E}^{\mathrm{hyp}, T}_{\sigma, \lambda}}(\rho) \subseteq \mathcal{E}^{\mathrm{hyp}, T}_{\sigma, \lambda}$ the open balls centered at the origin with radius $r$ and $\rho$, respectively. In the spirit of the proof of Item \eqref{thm:A_existence} of Theorem \ref{Thm:elliptic_Csigma}, contained in Section \ref{sec:proof_thm_a_ell_1}, we will show the following.

\begin{proposition}
\label{prop:B_existence}
There exists $C_0 > 1$ such that for any $r > 0$ there exists $T_0 \geq 0$ satisfying the following. For any $T \geq T_0$, there exists a $C^1$-map
\[ \boldsymbol{\varphi}^T: B_{\mathcal{P}^{\mathrm{hyp}, T}_{\sigma, \lambda}}(r)  \subseteq  \mathcal{P}^{\mathrm{hyp}, T}_{\sigma, \lambda}  \to  B_{ \mathcal{E}^{\mathrm{hyp}, T}_{\sigma, \lambda}}(C_0r) \subseteq \mathcal{E}^{\mathrm{hyp}, T}_{\sigma, \lambda}\]
such that $\boldsymbol{\varphi}^T(P)$ defines an asymptotic KAM torus associated with $(X_{H_0 + P}, X_{H_0}, \varphi_0)$,  for any $P \in  \mathcal{P}^{\mathrm{hyp}, T}_{\sigma, \lambda} $ with $|P|^T_{\sigma, \lambda} < r$.
\end{proposition}

Notice that Item \eqref{thm:B_existence} of Theorem \ref{Thm:hyp_Csigma} follows directly from the proposition above since
\begin{equation*}
    \big|P\big|_{\sigma, \lambda}^{T} \leq |P|_{\sigma, \lambda}^{0}, \qquad \text{ for any } T \geq 0\,  \text{ and any } \, P \in \mathcal{P}^{\mathrm{hyp}, 0}_{\sigma, \lambda}. 
\end{equation*}

To prove Proposition \ref{prop:B_existence} we consider the functional 
\begin{equation}\label{def:F_hyp_1}
    \Function{\mathcal{F}^T}{\mathcal{P}^{\mathrm{hyp}, T}_{\sigma, \lambda} \times \mathcal{E}^{\mathrm{hyp}, T}_{\sigma, \lambda}}{\mathcal{S}^{\mathrm{exp},T}_{(\sigma,0), \lambda}(\R^n) \times \mathcal{S}^{\mathrm{exp},T}_{(\sigma,0), \lambda}(\R^n) \times \mathcal{S}^{\mathrm{exp},T}_{(\sigma,0), \lambda}(\R^{2m})}{(P, \varphi)}{X_{H_0 + P} \circ \varphi - \Lo \varphi },
\end{equation}
and endow the codomain of $\mathcal{F}^T$, namely $\mathcal{S}^{\mathrm{exp},T}_{(\sigma,0), \lambda}(\R^n) \times \mathcal{S}^{\mathrm{exp},T}_{(\sigma,0), \lambda}(\R^n) \times \mathcal{S}^{\mathrm{exp},T}_{(\sigma,0), \lambda}(\R^{2m})$ with the sum norm associated with the product, which we denote by $|\cdot|^T_{\boldsymbol{\sigma, \lambda}}$, where the parameters are written in bold symbols.

Notice that this functional is defined by the same formula as the one introduced in Section \ref{sec:proof_thm_a_ell_1} and given in \eqref{proof:Csigma_def_F_ell_1}, but with a different unperturbed Hamiltonian $H_0$, and with domain and codomain consisting of functions with exponential rather than polynomial decay. Moreover, by the arguments given in Section \ref{sec:proof_thm_a_ell_1} it follows that, for $(P, \varphi)$ in the domain of definition, $\mathcal{F}^T(P, \varphi) = 0 $ if and only if $\varphi$ defines a $C^\sigma$ asymptotic KAM torus for $H_0 + P$.

The differential of $\mathcal{F}^T$ in the variables $\varphi = (u,v,w)$ evaluated at $(0,0)$ is given by
\begin{equation}\label{proof:Csigma_def_hyp_DF_1}
\Function{D_{\varphi}\mathcal{F}(0,0)}{\mathcal{E}^{\mathrm{hyp}, T}_{\sigma,\lambda}}{ \mathcal{S}^{\mathrm{exp},T}_{(\sigma,0), \lambda}(\R^n) \times \mathcal{S}^{\mathrm{exp},T}_{(\sigma,0), \lambda}(\R^n) \times \mathcal{S}^{\mathrm{exp},T}_{(\sigma,0), \lambda}(\R^{2m})}{(\hat u, \hat v, \hat w) }{\begin{pmatrix} \overline{M}_0 \cdot \hat v + m_0 \cdot \hat w - \Lo \hat u\\
    -\Lo \hat v\\
    \mathcal{L}^{\mathrm{hyp}}_{\omega, \Omega} \hat w + J_m\overline{m}_0 \cdot \hat v
    \end{pmatrix}},
\end{equation}
where we used the notation (subscript $0$) introduced in ~\eqref{def:f0} and the functions $\bar m, \bar M$ are given by Lemma \ref{lemma:def_bar_M_m_L}. We refer to~\eqref{eq:coh_eqs_hyp} for the definition of the linear operators $\Lo$ and $\mathcal{L}^{\mathrm{hyp}}_{\omega, \Omega}$. Notice that this is given precisely \eqref{proof:Csigma_def_ell_DF_1} (that is, the differential of the map $D_\varphi \mathcal{F}$ in the elliptic case) but replacing the operator $\mathcal{L}^{\mathrm{ell}}_{\omega, \Omega}$ by $\mathcal{L}^{\mathrm{hyp}}_{\omega, \Omega}$.

Proposition \ref{prop:B_existence} will be a consequence of the following analogue of Proposition \ref{prop:F_properties}.

\begin{proposition}
\label{prop:F_properties_hyp}
    The operator $\mathcal{F}^T$ given by \eqref{def:F_hyp_1} satisfies the following.
    \begin{enumerate}
    \item \label{prop:F_well_defined_hyp} For any $T \geq 0$, there exists a neighbourhood $\mathcal{U}^T \subseteq \mathcal{P}^{\mathrm{hyp}, T}_{\sigma,\lambda} \times \mathcal{E}^{\mathrm{hyp}, T}_{\sigma, \lambda}$ of $(0, 0)$ such that $\mathcal{F}^T\mid_{\mathcal{U}^T}$ and $D_\varphi\mathcal{F}^T\mid_{\mathcal{U}^T}$ are well-defined continuous maps. 
    
    \noindent Moreover, for any $r, \rho \geq 0$ there exists $T_0 \geq 0$ such that, for any $T \geq T_0$,  
    \[B_{\mathcal{P}^{\mathrm{hyp}, T}_{\sigma,\lambda}}(r) \times B_{\mathcal{E}^{\mathrm{hyp}, T}_{\sigma, \lambda}}(\rho) \subseteq \mathcal{U}^T.\]

         \item \label{prop:F_bounded_hyp}There exists a constant $C$,
     such that for any $r > 0$ and any $T \geq 0$,
\[ \sup \left\{ |\mathcal{F}^T(P, 0)|_{\boldsymbol{\sigma, \lambda}}^T \,\left|\, P \in B_{\mathcal{P}^{\mathrm{hyp}, T}_{\sigma,\lambda}}(r); \,\, (P, 0) \in \mathcal{U}^T \right\}\right. \leq Cr.\]

    \item \label{prop:DF_invertible_hyp} $D_\varphi \mathcal{F}^T(0, 0)$ is invertible, for any $T \geq 0$. Moreover, there exists a constant $\bar C$, independent of $T$, such that $\|D_\varphi \mathcal{F}^T(0, 0)^{-1}\| \leq \bar C$, where $\| \cdot\|$ denotes the operator norm. 

        \item \label{prop:DF_limit_hyp} Let $r, \rho > 0$. Then  
        \[\lim_{T \to +\infty} \| D_\varphi\mathcal{F}^T(P, \varphi) - D_\varphi\mathcal{F}^T(0, 0) \| = 0, \qquad \text{uniformly on } \quad B_{\mathcal{P}^{\mathrm{hyp}, T}_{\sigma,\lambda}}(r) \times B_{\mathcal{E}^{\mathrm{hyp}, T}_{\sigma, \lambda}}(\rho),\]
where $\| \cdot\|$ denotes the operator norm.
    \end{enumerate}
\end{proposition}

\begin{proof}[Proof of Proposition \ref{prop:B_existence}]
Assuming Proposition \ref{prop:F_properties_hyp}, the proof of Proposition \ref{prop:B_existence} follows exactly the same lines as that of Proposition \ref{prop:A_existence}, which relies on Proposition \ref{prop:F_properties} and uses a quantitative version of the Implicit Function Theorem (Theorem \ref{thm:QIFT}), and thus we omit it here. 
\end{proof}

Finally, let us sketch the proof of Proposition \ref{prop:F_properties_hyp}.

\begin{proof}[Proof of Proposition \ref{prop:F_properties_hyp}] As in Proposition \ref{prop:F_properties}, Properties \eqref{prop:F_well_defined_hyp}, \eqref{prop:F_bounded_hyp} and \eqref{prop:DF_limit_hyp} follow from straightforward calculations while the proof of Property \eqref{prop:DF_invertible_hyp} requires the invertibility of certain linear operator (namely $\mathcal{L}^{\mathrm{hyp}}_{\omega, \Omega}$, see Proposition \ref{prop:Csigma_HEomegaOmega_exp}) which appears naturally in the cohomological equations related to the invertibility of $D_\varphi \mathcal{F}^T(0, 0)$ (see the proof of Lemma \ref{lemma:DF_invertible}). 

Furthermore, the proofs of the properties above follow closely the proofs of the corresponding properties in Proposition \ref{prop:F_properties}, by considering norms with exponential decay rather than with polynomial decay and with very minor variations. 

More precisely, to prove Property \eqref{prop:F_well_defined_hyp} it suffices to replace $T^{\ell - 1}$ by $e^{\lambda T}$ in \eqref{eq:domain_F_ell} (that is, to consider $\mathcal{U}^T =  \mathcal{P}^{\mathrm{hyp}, T}_{\sigma,\lambda} \times B_{\mathcal{E}^{\mathrm{hyp}, T}_{\sigma, \lambda}}({e^{\lambda T}})$) and follow the proof of Lemma \ref{lemma:F_well_defined}. Property \eqref{prop:F_bounded_hyp} is a straightforward computation as in Lemma \ref{lemma:F_bounded}. The proof of Property \eqref{prop:DF_invertible_hyp} follows the exact same lines as the proof of Lemma \ref{lemma:DF_invertible} but using  Proposition \ref{prop:Csigma_HEomegaOmega_exp} (which guarantees the invertibility of $\mathcal{L}^{\mathrm{hyp}}_{\omega, \Omega}$) rather than Proposition \ref{prop:Csigma_HEomega} (which guarantees the invertibility of $\mathcal{L}^{\mathrm{ell}}_{\omega, \Omega}$). Finally, in Property \eqref{prop:DF_limit_hyp}, it suffices to replace $T^{\ell - 1}$ by $e^{\lambda T}$ in \eqref{eq:DF_difference_bound} and follow the proof of Lemma \ref{lem:DF_limit}.
\end{proof}

\subsection{Proof of Item \eqref{thm:B_transverse} of Theorem \ref{Thm:hyp_Csigma}: Existence of asymptotic hyperbolic transversal dynamics}\label{sec:proof_thm_b_hyp_2} 
In this section, we prove that the asymptotic KAM torus obtained in Item \eqref{thm:B_existence} of Theorem \ref{Thm:hyp_Csigma} is asymptotically hyperbolic under slightly stronger assumptions on the regularity and on the decay rate of the perturbations (more precisely, on $\sigma$ and $\lambda$).

In the following, we continue using the notation and parameters from the previous subsections, with the only exceptions that we now assume $\sigma \geq 2$ and $\lambda > 2\max_{1 \le i \le m} \Omega_i$.

Recall that, with these notations, the perturbations considered in Item \eqref{thm:B_transverse} can be identified with the space $\mathcal{P}^{\mathrm{hyp}, 0}_{\sigma, \lambda}$ (see Section \ref{sc:perturbation_hyp}).

In the spirit of the proof of Item \eqref{thm:A_transverse} of Theorem \ref{Thm:elliptic_Csigma}, contained in Section \ref{sec:proof_thm_a_ell_2}, we will show the following.

\begin{proposition}
\label{prop:B_transverse}
Fix $r > 0$ and let $T_0 \geq 0$ be given by Proposition \ref{prop:B_existence}. There exist $T_1 \geq T_0$ such that, for any $T \geq T_1$,  there exists a $C^1$-map
\[ \boldsymbol{G}^T: B_{\mathcal{P}^{\mathrm{hyp}, T}_{\sigma, \lambda}}(r)    \subseteq  \mathcal{P}^{\mathrm{hyp}, T}_{\sigma, \lambda} \to \mathcal{S}p^{\mathrm{hyp}, T}_{\sigma-1, \omega, \Omega, \lambda}\]
for which $\boldsymbol{S}(\boldsymbol{\varphi}^T(P),  \boldsymbol{G}^T(P))$ given by \eqref{eq:S_formula} is of class $C^1$ and
\begin{equation*}
\label{eq:IFT_condition_hyp}
\boldsymbol{\zeta}(H_0 + P, \boldsymbol{\varphi}^T(P),  \boldsymbol{G}^T(P)) =  \Mhyp, \qquad \text{for any } P \in B_{\mathcal{P}^{\mathrm{hyp}, T}_{\sigma, \lambda}}(r),
\end{equation*}
where $\boldsymbol{\zeta}$ is given by \eqref{eq:formula_zeta} and  $\boldsymbol{\varphi}^T: B_{\mathcal{P}^{\mathrm{hyp}, T}_{\sigma, \lambda}}(r)   \to \mathcal{E}^{\mathrm{hyp}, T}_{\sigma, \lambda}$  is given by Proposition \ref{prop:B_existence}.

In particular, by Proposition \ref{prop:criteria_non_autonomous}, for any $P \in B_{\mathcal{P}^{\mathrm{hyp}, T}_{\sigma, \lambda}}(r)$, the asymptotic KAM torus $\boldsymbol{\varphi}^T(P)$ is asymptotically hyperbolic.
\end{proposition}

 We refer the reader to  Section \ref{sc:matrices_exp_decay} for the definition of the space $\big( \mathcal{S}p^{\mathrm{hyp}, T}_{\sigma-1, \omega, \Omega, \lambda}, |\cdot|^T_{\sigma - 1, \omega, \Omega, \lambda} \big)$. 

Notice that Item \eqref{thm:B_transverse} of Theorem \ref{Thm:hyp_Csigma} follows directly from the proposition above, since
\begin{equation*}
    \big|P\big|_{\sigma, \lambda}^{T} \leq |P|_{\sigma, \lambda}^{0}, \qquad \text{ for any } T \geq 0\,  \text{ and any } \, P \in \mathcal{P}^{\mathrm{hyp}, 0}_{\sigma, \lambda}. 
\end{equation*}

In the following, we fix $r > 0$ and let $T_0 \geq 0$ be given by Proposition \ref{prop:B_existence}. Moreover, for any $T \geq T_0$, we will denote by $\boldsymbol{\varphi}^T: B_{\mathcal{P}^{\mathrm{hyp}, T}_{\sigma, \lambda}}(r)   \to \mathcal{E}^{\mathrm{hyp}, T}_{\sigma, \lambda}$ the map given by Proposition \ref{prop:B_existence}. 

To prove Proposition \ref{prop:B_transverse}, we consider the functional 
$$\mathcal{G}^T: B_{\mathcal{P}^{\mathrm{hyp}, T}_{\sigma,\lambda}}(r) \times \mathcal{S}p^{\mathrm{hyp}, T}_{\sigma-1, \omega, \Omega, \lambda}   \to  \mathcal{S}ym^{\mathrm{exp}, T}_{\sigma-1, \lambda}$$ 
as
\begin{equation}
\label{eq:formula_F_hyp_2}
    \begin{aligned}
    \mathcal{G}^T(P, G) =&  B^\top \partial_p^2 H\circ \varphi B + K^{\top} \partial_p\partial_zH\circ \varphi B + B^{\top}  {\left(\partial_p\partial_zH\circ \varphi\right)^\top} K \\
    & + K^{\top} \partial_z^2 H\circ \varphi K + K^{\top} J_m \Lo K - M^{\mathrm{hyp}}_\Omega,
    \end{aligned}
\end{equation}
where $H = H_0 + P$, $\varphi = \boldsymbol{\varphi}^T(P) = (u, v, w)$, $K = \exp(G)$,  $B = -(\mathrm{Id}_{2n} +\partial_q  u )^{-\top} (\partial_q w)^\top J_m K$ and {$\partial_p\partial_zH\circ \varphi$ stands for the $2m \times n$ matrix having coomponents $\partial_{p_i}\partial_{z_j}H\circ \varphi$ for $1 \le i \le n$ and $1 \le j \le 2m$}. We refer the reader to Section \ref{sc:matrices_exp_decay} for the definition of the Banach spaces of symmetric matrices $\big(\mathcal{S}ym^{\mathrm{exp}, T}_{\sigma-1, \lambda}, |\cdot|^T_{\sigma - 1, \lambda}  \big)$. 

Notice that this functional is defined by substituting $M^{\mathrm{ell}}_\Omega$ by $M^{\mathrm{hyp}}_\Omega$ in the functional \eqref{eq:formula_F_ell_2} introduced in Section \ref{sec:proof_thm_a_ell_2} and by restricting its domain to functions with exponential decay. Moreover, by the same arguments provided in Section \ref{sec:proof_thm_a_ell_2}, it follows that 
\begin{equation}
\label{eq:hyperbolic_characterization}
    \text{ If }  \quad \mathcal{G}^T(P, \varphi) = 0 \quad \text{ then } \quad  \boldsymbol{\varphi}^T(P) \text{ is a asymptotically hyperbolic KAM torus.}
\end{equation}

We observe that $\mathcal{G}^T(0,0) = 0$ and that the differential of $\mathcal{G}^T$ with respect to $G$ evaluated at $(0,0)$ is given by
\begin{equation}\label{eq:diff_form_F_hyp_1}
\Function{D_{G}\mathcal{G}^T(0,0)}{\mathcal{S}p^{\mathrm{hyp}, T}_{\sigma-1, \omega, \Omega, \lambda}}{ \mathcal{S}ym^{\mathrm{exp}, T}_{\sigma-1, \lambda}}{\hat G}{\mathfrak{L}^{\mathrm{hyp}}_{\omega, \Omega} (\hat G) = \hat G^{\top}M^{\mathrm{hyp}}_\Omega + M^{\mathrm{hyp}}_\Omega\hat G + J_m \Lo \hat G}.
\end{equation}

Proposition \ref{prop:B_transverse} will be a consequence of the following.

\begin{proposition}
\label{prop:G_properties_hyp}
    The operator $\mathcal{G}^T$ given by \eqref{eq:formula_F_hyp_2} satisfies the following.
    \begin{enumerate}
    \item \label{prop:G_well_defined_hyp} $\mathcal{G}^T$ and $D_\varphi\mathcal{G}^T$ are well-defined continuous maps, for any $T \geq T_0$. 
    
    \item \label{prop:G_bounded_hyp} There exists a constant $C$ such that, for any $T \geq T_0$,
\[ \sup \left\{ |\mathcal{G}^T(P, 0)|^T_{\sigma - 1, \lambda} \,\left|\, P \in B_{\mathcal{P}^{\mathrm{hyp}, T}_{\sigma,\lambda}}(r) \right\}\right. \leq Cr.\] %

    \item \label{prop:DG_invertible_hyp} $D_G \mathcal{G}^T(0, 0)$ is invertible, for any $T \geq T_0$. Moreover, there exists a constant $\bar C$, independent of $T$, such that $\|D_G \mathcal{G}^T(0, 0)^{-1}\| < \bar C$. 
        \item \label{prop:DG_limit_hyp} Let $r, \rho > 0$. Then  
        \[\lim_{T \to +\infty} \| D_G \mathcal{G}^T(P, G) - D_G\mathcal{G}^T(0, 0) \| = 0, \qquad \text{uniformly on } \quad B_{\mathcal{P}^{\mathrm{hyp}, T}_{\sigma,\lambda}}(r) \times B_{\mathcal{S}p^{\mathrm{hyp}, T}_{\sigma-1, \omega, \Omega, \lambda}}(\rho),\]
where $\| \cdot\|$ denotes the operator norm and $B_{\mathcal{S}p^{\mathrm{hyp}, T}_{\sigma-1, \omega, \Omega, \lambda}}(\rho) \subset \mathcal{S}p^{\mathrm{hyp}, T}_{\sigma-1, \omega, \Omega, \lambda}$ stands for a ball of radius $\rho$ centered at the origin.
    \end{enumerate}
\end{proposition}

\begin{proof}[Proof of Proposition \ref{prop:B_transverse}]
Assuming Proposition \ref{prop:G_properties_hyp}, the proof of Proposition \ref{prop:B_transverse} follows exactly the same lines as that of Proposition \ref{prop:A_transverse}, which relies on Proposition \ref{prop:G_properties} and uses a quantitative version of the Implicit Function Theorem (Theorem \ref{thm:QIFT}), and thus we omit it here. 
\end{proof}

We now provide a sketch of the proof of Proposition \ref{prop:G_properties_hyp}.
\begin{proof}[Proof of Proposition \ref{prop:G_properties_hyp}] 

The proof of Properties \eqref{prop:G_well_defined_hyp}, \eqref{prop:G_bounded_hyp} and \eqref{prop:DG_limit_hyp} of Proposition \ref{prop:G_properties_hyp} will follow closely the proof of the corresponding properties in Proposition \ref{prop:G_properties}. However, the proof  Property \eqref{prop:DG_invertible_hyp} relies on the analysis of a particular linear operator $\mathfrak{L}^{\mathrm{hyp}}_{\omega, \Omega}$ whose invertibility is proven in Proposition \ref{prop:Csigma_HEomegaOmegaOmega_hyp}.

 Considering norms with exponential decay rather than with polynomial decay, Properties \eqref{prop:G_well_defined_hyp} and \eqref{prop:G_bounded_hyp} can be shown by following the proofs of Lemmas \ref{lemma:G_well_defined} and \ref{lemma:G_bounded}, respectively. Property \eqref{prop:DF_invertible_hyp} follows by Proposition \ref{prop:Csigma_HEomegaOmegaOmega_hyp}. Finally, to prove Property \eqref{prop:DG_limit_hyp}, it suffices to replace $T^{\ell - 1}$ by $e^{\lambda T}$ in \eqref{eq:DG_difference_bound} and follow the proof of Lemma \ref{lem:DG_limit}.
\end{proof}

\subsection{Proof of Theorem \ref{Thm:hyp_analy}}
\label{sc:proof_analytic_hyp} 

 Notice that to prove Theorem \ref{Thm:hyp_analy} it suffices to prove only the first Item in the statement, since the second part will follow immediately from Theorem \ref{Thm:hyp_Csigma}. 

Similarly to Section \ref{sc:proof_analytic_ell}, the proof of the first part of Theorem \ref{Thm:hyp_analy} follows the same lines and structure of that of Item \ref{thm:B_existence} in Theorem \ref{Thm:hyp_Csigma} (which is given in Section \ref{sec:proof_Thm_hyp_Csigma_1}) but considering spaces of analytic instead of Hölder functions. For this reason, we will only define the spaces and discuss the very minor changes needed to carry out the proof. Let us point out that, for the sake of completeness, the cohomological equations appearing in Section \ref{sec:proof_Thm_hyp_Csigma_1} are also considered when restricted to the analytic setting in Section \ref{sc:cohom_eqs_ell} and that the analytic counterparts of the results concerning the norms and the Banach spaces appearing in Section \ref{sec:proof_Thm_hyp_Csigma_1} are also proved in Section \ref{sc:norm_properties_hyp} (see Proposition \ref{prop:properties_analy_hyp}).

Recall that the space of real analytic functions with exponential decay in time $\mathscr{S}^{\mathrm{exp},T}_{\sigma,\ell}$ and its associated norm $|\cdot |^T_{\sigma, \ell}$ (which as an abuse of notation we denote as the one associated to $\mathcal{S}^{\mathrm{exp},T}_{\sigma,\ell}$) were defined in \eqref{def:S_exp_anal} and \eqref{def:norm_S_analy_exp}, respectively.

Given $\sigma >0$, $\lambda \ge 0$ and $T\ge 0$, where $\sigma$ will now denote the width of the complex neighbourhood, we define the counterparts of the spaces $\mathcal{S}^{\mathrm{exp},T}_{(\sigma, \ast), (0,\ell)}$, $\mathcal{U}^{\mathrm{exp}, T}_{\sigma, \omega, \ell}$ and $\mathcal{W}^{\mathrm{ell}, T}_{\sigma, \omega, \Omega, \ell}$, given by \eqref{def: S0exp},~\eqref{def:Uexp} and~\eqref{def:Wexp}, as

\begin{equation}\label{def: S0hyp_analy}
    \mathscr{S}^{\mathrm{exp},T}_{\sigma, (0,\lambda)} = \left\{f:\T^n_\sigma \times B_\sigma \times I_T \to \C  \hspace{1mm} | \hspace{1mm} f \in \mathscr{S}^{\mathrm{exp}, T}_{\sigma, 0}, \hspace{2mm} \partial_q f \in \mathscr{S}^{\mathrm{exp}, T}_{\sigma, \lambda}\right\},
\end{equation}
\begin{equation}\label{def:U_analy_exp}
    \mathscr{U}^{\mathrm{exp}, T}_{\sigma, \omega, \ell} = \left\{u:\T^n_\sigma \times I_T \to \C \hspace{1mm} | \hspace{1mm} u \in \mathscr{S}^{\mathrm{exp},T}_{\sigma, \lambda}, \hspace{2mm}\Lo u \in \mathscr{S}^{\mathrm{exp}, T}_{\sigma, \lambda}\right\},
\end{equation}
\begin{equation}\label{def:W_analy_exp}
    \mathscr{W}^{\mathrm{hyp}, T}_{\sigma, \omega, \Omega, \lambda} = \left\{w:\T^n_\sigma \times I_T \to \C^{2m} \hspace{1mm} | \hspace{1mm} w \in \mathscr{S}^{\mathrm{exp}, T}_{\sigma, \lambda}, \hspace{2mm}\mathcal{L}^{\mathrm{hyp}}_{\omega,\Omega} w \in \mathscr{S}^{\mathrm{hyp}, T}_{\sigma, \lambda}\right\},
\end{equation}
endowed with the norms
\begin{equation}\label{def:norm_S0hyp}
    |f|^T_{\sigma, (0,\lambda)} = \max\{|f|^T_{\sigma, 0}, |\partial_qf|^T_{\sigma, \lambda}\},
\end{equation}
\begin{equation}\label{def:norm_U_analy_exp}
    |u|^T_{\sigma, \omega, \lambda} = \max \left\{|u|^T_{\sigma, \lambda}, |\Lo u|^T_{\sigma, \lambda}\right\},
\end{equation}
\begin{equation}\label{def:norm_W_analy_exp}
    |w|^{T}_{\sigma, \omega, \Omega, \lambda} = \max \left\{|w|^T_{\sigma, \lambda}, |\mathcal{L}^{\mathrm{hyp}}_{\omega,\Omega} w|^T_{\sigma, \ell +1}\right\},
\end{equation}
respectively.

With these notations, the proof follows the exact same lines as the proof of Item \ref{thm:B_existence} in Theorem \ref{Thm:hyp_Csigma}. Again, as in the proof of Theorem \ref{Thm:elliptic_analy}, the only major difference when carrying out the proof is the need for a different choice of domain for the functional $\mathcal{F}^T$, but this can be easily done as explained in Section \ref{sc:proof_analytic_ell}.

\subsection{Proof of Corollary \ref{cor:hyp}}\label{sec:proof_cor_hyp}
The proof of this corollary is similar to that of Corollary \ref{cor:ell} contained in Section \ref{sec:proof_cor_ell}. We recall that the proof of Theorem \ref{Thm:hyp_Csigma} is constructive, hence the component $a_{32}$ of the matrix $A$ in~\eqref{def:A} in terms of the asymptotic KAM torus $\varphi=(\mathrm{id}_{\T^n} + u,v,w)$ and the components of the matrix $S$ (see~\eqref{eq:A_as_function_of_S}) has the following form
  \begin{equation*}
     a_{32} = e^{G^\top}
\left(\partial_z\partial_pH\circ\varphi\right)
(\mathrm{id}_{\T^n}+\partial_qu)^{-\top}- e^{G^\top} J_m(\partial_qw)^\top
(\mathrm{id}_{\T^n}+\partial_qu)^{-1}
\left(\partial_p^2H\circ\varphi\right)
(\mathrm{id}_{\T^n}+\partial_qu)^{-\top}
 \end{equation*}
 where, for $\lambda > 2 \max_{1 \le i \le m}\Omega_i$ and $T'' \ge 1$
 \begin{equation*}%
     u, \, v, \, w, \, G  \in \mathcal{S}_{(\sigma, 0), \lambda}^{\mathrm{pol},T''}.
 \end{equation*}
The proof of Corollary \ref{cor:hyp} is an immediate consequence of Proposition \ref{prop:hyp_1:1_asym_sol},~\eqref{eq:conjugated_cocycles}, and~\eqref{def:trasv_dyn_limit}. It only remains to verify hypothesis~\eqref{hyp:limit_aut_sys_hyp} that, in the present setting, this amounts to proving that
\begin{equation}\label{proof_cor:cond_to_verify_hyp}
        \lim_{t \to +\infty} e^{-\Omega_jt} \int_{T''}^t e^{\Omega_j s}|a_{32}^s - \partial_z\partial_pH^\infty\circ\varphi_0|_{C^0}\,ds = 0,
    \end{equation}
    for all $1 \le j \le m$,  where $\varphi_0$ is the trivial embedding defined by~\eqref{def:varphi0=(q,0,0)}.

     As in the proof of Corollary \ref{cor:ell} in Section \ref{sec:proof_cor_ell}, one can verify that 
     \begin{equation}\label{proof_cor:est_hyp}
    \begin{gathered}
       \scalebox{0.97}{$e^{G^\top}
\left(\partial_z\partial_pH\circ\varphi\right)
(\mathrm{id}_{\T^n}+\partial_qu)^{-\top} - \partial_z\partial_pH^\infty\circ\varphi_0 
 - \left(\left(\partial_z\partial_pH\right)_0  - \left(\partial_z\partial_pH^\infty\right)_0\right) \in \mathcal{S}_{(\sigma, 0), \lambda}^{\mathrm{exp},T''},$}\\
 e^{G^\top} J_m(\partial_qw)^\top
(\mathrm{id}_{\T^n}+\partial_qu)^{-1}
\left(\partial_p^2H\circ\varphi\right)
(\mathrm{id}_{\T^n}+\partial_qu)^{-\top} \in \mathcal{S}_{(\sigma, 0), \lambda}^{\mathrm{exp},T''}.
    \end{gathered}
\end{equation}
Moreover, by hypothesis~\eqref{hyp:cor_hyp_decay}, we have that 
\begin{equation}\label{proof_cor_hyp:cons_hyp_decay}
        \lim_{t\to+\infty}e^{-\Omega_j t}\int_1^t e^{\Omega_j \tau} |\left(\partial_z\partial_pH\right)_0  - \left(\partial_z\partial_pH^\infty\right)_0|_{C^0} \, d\tau =0.
\end{equation}
Using~\eqref{proof_cor:est_hyp},~\eqref{proof_cor_hyp:cons_hyp_decay} and remembering Remark \ref{rmk:cond_corr_hyp}, we prove~\eqref{proof_cor:cond_to_verify_hyp}. This concludes the proof of Corollary \ref{cor:hyp}.

\subsection{Cohomological Equations}
\label{sc:cohom_eqs_hyp}

In this section, we study in detail the cohomological equations appearing in the proof of Theorems \ref{Thm:hyp_Csigma} and \ref{Thm:hyp_analy}, namely, the three equations in \eqref{eq:coh_eqs_hyp} restricted to the Hölder setting (see Sections \ref{sec:proof_Thm_hyp_Csigma_1}  and \ref{sec:proof_thm_b_hyp_2}) and the first two equations in \eqref{eq:coh_eqs_hyp} in the analytic setting (see Section \ref{sc:proof_analytic_hyp}), respectively.

As in Section \ref{sc:cohom_eqs_ell}, we only present the proofs in detail in the Hölder setting, indicating only the necessary modifications in the analytic one.

In the following proposition, we study the linear operator $\Lo$ on the Banach spaces defined in~\eqref{def:Sexp},~\eqref{def:Uexp},~\eqref{def:S_exp_anal}, and~\eqref{def:U_analy_exp}.

\begin{proposition}\label{prop:Csigma_HEomega_exp}
       Given $\sigma \ge 0$, $T \ge 0$, and $\lambda >0$,  the operator
    \begin{align*}
        \Lo : \mathcal{U}^{\mathrm{exp}, T}_{\sigma, \omega, \lambda} \longrightarrow \mathcal{S}^{\mathrm{exp}, T}_{(\sigma, 0), \lambda}(\R^n)
    \end{align*}
    is well-defined, invertible and satisfies
    \[ \| \Lo ^{-1}\| \leq C(\ell),\]
    for some $C > 0$ depending only on $\ell$.
    
    In addition, if $\sigma \ge 2$, then, for any $u \in \mathcal{U}^{\mathrm{exp}, T}_{\sigma, \omega, \ell}$
    \begin{equation}\label{eq:prop_dist_Lomega_U_exp}
        \partial_q\left(\Lo u\right) = \Lo \partial_q  u .
    \end{equation}
    Moreover, given $\sigma > 0$, $T \ge 0$, and $\lambda >0$, the same well-definedness, invertibility, and estimate hold if one replaces the Banach spaces $\mathcal{U}^{\mathrm{exp}, T}_{\sigma, \omega, \lambda}$ and $\mathcal{S}^{\mathrm{exp}, T}_{(\sigma, 0), \lambda}$ by $\mathscr{U}^{\mathrm{exp}, T}_{\sigma, \omega, \lambda}$ and $\mathscr{S}^{\mathrm{exp}, T}_{\sigma, \lambda}$, respectively. Moreover, in this case, the equality~\eqref{eq:prop_dist_Lomega_U_exp} holds for any $\sigma>0$ and $u \in \mathscr{U}^{\mathrm{exp}, T}_{\sigma, \omega, \lambda}$.
\end{proposition}
\begin{proof}
    The proof is similar to that of Proposition \ref{prop:Csigma_HEomega}. For this reason, it is omitted. 
\end{proof}

Also in this case, before analyzing the linear operator $\mathcal{L}^{\mathrm{hyp}}_{\omega, \Omega}$, we need to prove the following quantitative lemma. Its statement is divided into two parts.
The first part concerns the Hölder setting, while the second deals with the analytic one.
\begin{lemma}\label{lemma:HE_interm_LoO_exp}
    We fix $\sigma \ge 0$, $T \ge 0$, $\lambda >0$, and $\Lambda >0$. Let $g_\pm:\T^n \times I_T \to \R$ be functions such that $g_\pm \in \Sexp_{(\sigma, 0), \lambda}$. We consider the following equations
    \begin{align}\label{lemma:he_exp_auxlemma_1}
        \Lambda z_+ - \Lo z_+ &= g_+\\ \label{lemma:he_exp_auxlemma_2}
        -\Lambda z_- - \Lo z_- &= g_-
    \end{align}
    in the unknowns $z_\pm: \T^n \times I_T \to \R$. Then, there exists a unique solution $z_+$ of~\eqref{lemma:he_exp_auxlemma_1} satisfying 
    \begin{equation}\label{lemma:he_exp_z+_properties}
        z_+ \in \Sexp_{(\sigma, 0), \lambda} \hspace{3mm} \mbox{and} \hspace{3mm} |z_+|_{\sigma, \lambda}^T \le {1 \over \lambda + \Lambda} |g_+|_{\sigma, \lambda}^T.
    \end{equation}
    Moreover, if we assume that 
    \begin{equation*}
        \lambda > \Lambda
    \end{equation*}
    there exists a unique solution $z_-$ of~\eqref{lemma:he_exp_auxlemma_2} satisfying
    \begin{equation}\label{lemma:he_exp_z-_properties}
        z_- \in \Sexp_{(\sigma, 0), \lambda} \hspace{3mm} \mbox{and} \hspace{3mm} |z_-|_{\sigma, \lambda}^T \le {1 \over \lambda - \Lambda} |g_-|_{\sigma, \lambda}^T.
    \end{equation}
    In addition, if $\sigma \ge 2$, then
    \begin{equation}\label{eq:prop_dist_Lomega_z_W_exp}
        \partial_q \left(\Lo z_\pm\right) = \Lo \left(\partial_q z_\pm\right).
    \end{equation}
    In this second part, we fix $\sigma, \, \lambda, \, \Lambda >0$, and $T\ge 0$. We consider functions $g_\pm:\T^n \times I_T \to \C$ such that $g_\pm \in \mathscr{S}^{\mathrm{exp}, T}_{\sigma, \lambda}$. Then, there exists a unique solution $z_+$ of~\eqref{lemma:he_exp_auxlemma_1} satisfying~\eqref{lemma:he_exp_z+_properties} with $\Sexp_{(\sigma, 0), \lambda}$ replaced by $\mathscr{S}^{\mathrm{exp}, T}_{\sigma, \lambda}$. In addition, if $\lambda > \Lambda$, there exists a unique solution $z_-$ of~\eqref{lemma:he_exp_auxlemma_2} satisfying~\eqref{lemma:he_exp_z-_properties} with $\Sexp_{(\sigma, 0), \lambda}$ replaced by $\mathscr{S}^{\mathrm{exp}, T}_{\sigma, \lambda}$. Moreover, in this case, equality~\eqref{eq:prop_dist_Lomega_z_W_exp} holds for any $\sigma >0$.
\end{lemma}
 \begin{proof}
     For the first part of this proof, we rewrite equations~\eqref{lemma:he_exp_auxlemma_1} and~\eqref{lemma:he_exp_auxlemma_2} as
     \begin{equation*}
         \pm \Lambda z_\pm - \Lo z_\pm = g_\pm. 
     \end{equation*}
     As in the proof of Proposition \ref{lemma:HE_interm_LoO}, one can see that, for all $t_0 \in I_T$, formal solutions of the above equations are given by
     \begin{equation}\label{proof_lemma:formsolpm_exp_auxlemma}
         z_\pm (q,t) = e^{\pm \Lambda(t - t_0)} \left(z_{0, \pm}(q -\omega t) - \int_{t_0}^t e^{\pm \Lambda(t_0-\tau)} g_\pm (q + \omega(\tau - t), \tau)d\tau\right)
     \end{equation}
     for all $(q, t) \in \T^n \times I_T$, where $z_{0, \pm}:\T^n \to \R$ is free. We analyze $z_+$ and $z_-$ separately. 
     First, we consider $z_+$. We observe that $e^{\Lambda(t - t_0)} \to +\infty$ if $t \to +\infty$. For this reason, in order to ensure the decay in time of $z_+$, the only possible choice of $z_{0,+}$ is
     \begin{equation*}
         z_{0,+}(q,t) = \int_{t_0}^{+\infty} e^{\Lambda(t_0- \tau)}g_+(q + \omega(\tau - t), \tau)d\tau.
     \end{equation*}
     Hence, replacing the latter into~\eqref{proof_lemma:formsolpm_exp_auxlemma}, we obtain that 
     \begin{equation}\label{eq:z+_inv_hyp}
         z_+ (q,t) =  \int_{t}^{+\infty} e^{\Lambda(t - \tau)} g_+ (q + \omega(\tau - t), \tau)d\tau
     \end{equation}
     is the unique solution of~\eqref{lemma:he_exp_auxlemma_2} satisfying $\lim_{t\to +\infty}|z_+(q,t)|=0$, where $|\cdot|$ denotes the absolute value.  Remembering the definition of the norm $|\cdot|^T_{\sigma, \lambda}$ in~\eqref{def:norm_S0exp}, thanks to the latter, we prove~\eqref{lemma:he_exp_z+_properties}. 

     The analysis of $z_-$ is more subtle. First, we point out that $e^{- \Lambda(t - t_0)}\to 0$ if $t \to +\infty$. This implies that, in this case, there is significantly more freedom in the choice of $z_{0, -}$ in~\eqref{proof_lemma:formsolpm_exp_auxlemma} in order to ensure that $\lim_{t\to +\infty}|z_-(q,t)|=0$. For this reason, we first choose $z_{0,-}$ in~\eqref{proof_lemma:formsolpm_exp_auxlemma} so that $z_-$ satisfies~\eqref{lemma:he_exp_z-_properties}, and then we establish uniqueness.

     If $\lambda > \Lambda$, we take
     \begin{equation*}
         z_{0,-}(q,t) = \int_{t_0}^{+\infty} e^{-\Lambda(t_0 - \tau)}g_-(q + \omega(\tau - t), \tau)d\tau
     \end{equation*}
     and hence, a straightforward computation shows that a solution of~\eqref{lemma:he_exp_auxlemma_2} satisfying~\eqref{lemma:he_exp_z-_properties} is given by 
     \begin{equation}\label{eq:z-_inv_hyp}
         z_- (q,t) =  \int_{t}^{+\infty} e^{-\Lambda(t - \tau)} g_- (q + \omega(\tau - t), \tau)d\tau
     \end{equation}
     for all $(q,t) \in \T^n \times I_T$. It remains to prove that this solution is unique. Suppose there exist $z_{0,-}^{(1)}, z_{0,-}^{(2)}:\T^n \to \R$ in such a way that 
     \begin{equation}
     \begin{aligned}\label{proof:uniqueness_hyp}
         z_-^{(1)} (q,t) &= e^{- \Lambda(t - t_0)} \left(z^{(1)}_{0, -}(q -\omega t) - \int_{t_0}^t e^{- \Lambda(t_0-\tau)} g_- (q + \omega(\tau - t), \tau)d\tau\right),\\
         z_-^{(2)} (q,t) &= e^{- \Lambda(t - t_0)} \left(z^{(2)}_{0, -}(q -\omega t) - \int_{t_0}^t e^{- \Lambda(t_0-\tau)} g_- (q + \omega(\tau - t), \tau)d\tau\right),
     \end{aligned}
     \end{equation}
     are solutions of~\eqref{lemma:he_exp_auxlemma_2} satisfying~\eqref{lemma:he_exp_z-_properties}. We define $w = z_-^{(1)} - z_-^{(2)}$ and thanks to~\eqref{proof:uniqueness_hyp} we obtain that 
     \begin{equation}\label{proof:he_aproxLemma_hyp_w1}
         w(q,t) = e^{- \Lambda(t - t_0)} \left(z^{(1)}_{0, -} - z^{(2)}_{0, -} \right)(q -\omega t)
     \end{equation}
     for all $(q,t)\in\T^n \times I_T$. On the other hand, by the definition of $z_-^{(1)}$ and $z_-^{(2)}$, we have that $w \in \Sexp_{(\sigma, 0), \lambda}$ and hence
     \begin{equation}\label{proof:he_aproxLemma_hyp_w2}
         |w(q,t)| \le C|w|^T_{\sigma, \lambda} e^{-\lambda t}
     \end{equation}
     for all $(q,t) \in \T^n \times I_T$ and a positive constant $C$. Now, combining~\eqref{proof:he_aproxLemma_hyp_w1} and~\eqref{proof:he_aproxLemma_hyp_w2}, we obtain that 
     \begin{equation*}
         \left|\left(z^{(1)}_{0, -} - z^{(2)}_{0, -} \right)(q -\omega t)\right| \le C e^{\lambda t_0}|w|^T_{\sigma, \lambda} e^{-(\lambda-\Lambda) t}
     \end{equation*}
     for all $t_0 \in I_T$ and $(q,t) \in \T^n \times I_T$. We stress that $\lambda >\Lambda$ implies that the right-hand side of the latter goes to zero if $t\to +\infty$. This proves that $z^{(1)}_{0, -} = z^{(2)}_{0, -}$ (and hence $z_-^{(1)}=z_-^{(2)}$), which concludes the proof of the first part of this lemma. Now, we assume $\sigma \ge 2$ and we prove~\eqref{eq:prop_dist_Lomega_z_W_exp}. Using the same argument in the second part of the  proof of Proposition \ref{prop:Csigma_HEomega}, one can see that the property~\eqref{eq:prop_dist_Lomega_z_W_exp} is a straightforward consequence of
     \begin{equation}\label{eq:partialtqzexp=partialqtzexp}
        \partial_t \partial_q  z_\pm = \partial_q\partial_t  z_\pm. 
    \end{equation}
    In order to prove the latter, we observe that, using~\eqref{eq:z+_inv_hyp} and~\eqref{eq:z-_inv_hyp}, we can write $z_\pm$ in the following form
    \begin{equation*}
         z_\pm (q,t) =  \int_{t}^{+\infty} e^{\pm\Lambda(t - \tau)} g_\pm (q + \omega(\tau - t), \tau)d\tau
     \end{equation*}
    for all $(q,t) \in \T^n \times I_T$ and an easy computation shows that $z_\pm$ are differentiable with respect to the variable $t$ and
    \begin{align*}
        \partial_t z_\pm (q,t) &= g_\pm(q,t) + \int_t^{+\infty} \left(\pm \Lambda\right)e^{\pm\Lambda(t-\tau) }g_\pm (q + \omega(\tau-t), \tau) - e^{\Lambda(t-\tau)}\partial_q g(q + \omega(\tau-t), \tau) \omega d\tau,\\
        \partial_q z_\pm(q,t) &= \int_t^{+\infty} e^{\pm\Lambda(t -\tau)} \partial_q g_\pm (q+\omega(\tau-t), \tau)d\tau,\\
        \partial_t \partial_q z_\pm(q,t) &= \partial_q \partial_t z_\pm(q,t) = \partial_q g_\pm(q,t) + \int_t^{+\infty} \bigg[\left(\pm \Lambda\right)e^{\pm\Lambda(t-\tau) }\partial_qg_\pm (q + \omega(\tau-t), \tau)\\
        &- e^{\Lambda(t-\tau)}\partial^2_q g(q + \omega(\tau-t), \tau) \omega \bigg]d\tau.
    \end{align*}
    This concludes the proof of~\eqref{eq:partialtqzexp=partialqtzexp}.

  The proof of the second part in the real-analytic setting follows exactly the same argument. We therefore omit it and leave to the reader the straightforward modifications needed to adapt the above proof.
\end{proof}

We emphasize that, unlike equation~\eqref{lemma:he_exp_auxlemma_2}, the existence of a solution to~\eqref{lemma:he_exp_auxlemma_1} does not require any assumption of exponential decay in time on $g_+$. Indeed, any time decay exhibited by $g_+$ is inherited by the solution $z_+$. However, a more refined analysis of the decay of $z_+$ does not lead to any improvement in our results, and we therefore stated Lemma \ref{lemma:HE_interm_LoO_exp} in that form.

 In the following proposition, we study the linear operator $ \mathcal{L}^{\mathrm{hyp}}_{\omega, \Omega}$. We recall that the spaces $\mathcal{W}^{\mathrm{hyp}, T}_{\sigma, \omega, \Omega, \lambda}$ and $\mathscr{W}^{\mathrm{hyp}, T}_{\sigma, \omega, \Omega, \lambda}$ are defined in~\eqref{def:Wexp} and~\eqref{def:W_analy_exp}, respectively. 

\begin{proposition}\label{prop:Csigma_HEomegaOmega_exp}
    Given $\sigma \ge 0$, $T \ge 0$, and $\lambda >0$,  the following operator
    \begin{align*}
         \mathcal{L}^{\mathrm{hyp}}_{\omega, \Omega} : \mathcal{W}^{\mathrm{hyp}, T}_{\sigma, \omega, \Omega, \lambda} \longrightarrow \mathcal{S}^{\mathrm{exp},T}_{(\sigma, 0), \lambda}(\R^{2m})
    \end{align*}
    is well-defined. If $\lambda > \max_{1 \le i \le m}\Omega_i$, then it is invertible and satisfies
    \[ \| (\mathcal{L}^{\mathrm{hyp}}_{\omega, \Omega})^{-1} \| \leq {2 \lambda \over \lambda^2 - \max_{1\le i \le m}\Omega_i^2}.\]
    
In addition, if $\sigma \ge 2$, then, for any $\chi \in \mathcal{W}^{\mathrm{hyp},T}_{\sigma, \omega, \Omega, \lambda}$
    \begin{equation}\label{eq:prop_dist_Lomega_W_exp}
        \partial_q\left(\Lo \chi\right) = \Lo \partial_q  \chi .
    \end{equation}
Moreover, given $\sigma > 0$, $T \ge 0$, and $\lambda > \max_{1 \le i \le m}\Omega_i$, the same well-definedness, invertibility, and estimate hold if one replaces the Banach spaces $\mathcal{W}^{\mathrm{hyp}, T}_{\sigma, \omega, \Omega,\lambda}$ and $\mathcal{S}^{\mathrm{exp}, T}_{(\sigma, 0), \lambda}$ by $\mathscr{W}^{\mathrm{hyp}, T}_{\sigma, \omega, \Omega, \lambda}$ and $\mathscr{S}^{\mathrm{exp}, T}_{\sigma, \lambda}$, respectively. Moreover, in this case, the equality~\eqref{eq:prop_dist_Lomega_W_exp} holds for any $\sigma>0$ and $\chi \in \mathscr{W}^{\mathrm{hyp}, T}_{\sigma, \omega, \Omega, \lambda}$.
\end{proposition}
\begin{proof}
     We recall that the matrices $M^{\mathrm{hyp}}_\Omega$ and $J_m$ are defined in~\eqref{Ms} and~\eqref{def:J}, respectively. Furthermore, for any $\chi \in \mathcal{W}^{\mathrm{hyp}, T}_{\sigma, \omega, \Omega, \lambda}$, the operator $\mathcal{L}^{\mathrm{hyp}}_{\omega, \Omega} \chi = J_m M^{\mathrm{hyp}}_\Omega \chi - \Lo \chi$. Hence, for all $g = (g_1,\dots, g_{2m})^\top \in \mathcal{S}^{\mathrm{exp}, T}_{(\sigma, 0), \lambda}$, we can rewrite this problem in terms of finding the unique solution $\chi = (\chi_1,\dots, \chi_{2m})^\top \in \mathcal{W}^{\mathrm{hyp}, T}_{\sigma, \omega, \Omega, \lambda}$ of the following equation
   \begin{equation}\label{proof:HE_eq_LoO_hyp}
       J_mM^{\mathrm{hyp}}_\Omega \chi - \Lo \chi = g.
   \end{equation}
   In the spirit of the proof of Proposition \ref{prop:Csigma_HEomegaOmega}, we note that $J_mM^{\mathrm{hyp}}_\Omega$ is diagonalizable with eigenvalues $\pm \Omega_j$ with $j=1,\dots,m$. Then, there exists an invertible matrix $P$ such that 
   \begin{equation}\label{proof:HE_diagMO_hyp}
       J_mM^{\mathrm{hyp}}_\Omega  = P\begin{pmatrix} -\Omega & 0 \\ 0 &  \Omega\end{pmatrix}P^{-1} \quad \mbox{with} \hspace{2mm} P =\begin{pmatrix} \mathrm{Id}_m & \mathrm{Id}_m \\  \mathrm{Id}_m &  - \mathrm{Id}_m\end{pmatrix}. 
   \end{equation}
   We recall the notation
   \begin{equation*}
       \chi=(\chi^{(1)}, \chi^{(2)})^\top, 
   \end{equation*}
   where $\chi^{(1)} = (\chi_1,\dots, \chi_{m})^\top$, and $\chi^{(2)} = (\chi_{m+1},\dots, \chi_{2m})^\top$. Furthermore, we define $X= \chi^{(1)} + \chi^{(2)}$ and $Y=\chi^{(1)} - \chi^{(2)}$. Hence, replacing~\eqref{proof:HE_diagMO_hyp} into~\eqref{proof:HE_eq_LoO_hyp}, multiplying on the left both sides of the equation by $P^{-1}$ and using the above-mentioned notation, we can rewrite~\eqref{proof:HE_eq_LoO_hyp} in terms of the following system of $2m$ decoupled equations
   \begin{equation*}
       \begin{aligned}
           -\Omega_j X_j - \Lo X_j &= g_j + g_{m+j}\\
           \Omega_j Y_j - \Lo Y_j &= g_j - g_{m+j}.
       \end{aligned}
   \end{equation*}
   for $j=1,\dots,m$.
   Using Lemma \ref{lemma:HE_interm_LoO_exp}, there exists a unique solution of the above system of equations in such a way that 
   \begin{equation}\label{systm:inv_XY_exp}
       X_j, Y_j \in \mathcal{S}^{\mathrm{exp}, T}_{(\sigma, 0), \lambda} \hspace{3mm} \mbox{and} \hspace{3mm} |X_j|^T_{\sigma, \lambda} \le {2 \over \lambda - \Omega_j}|g|^T_{\sigma, \lambda}, \hspace{2mm} |Y_j|^T_{\sigma, \lambda} \le {2 \over \lambda + \Omega_j}|g|^T_{\sigma, \lambda},
   \end{equation}
   for all $j=1,\dots,m$. Hence, going back to the original variables 
   \begin{equation}\label{eq:inv_exp_W}
       \chi^{(1)}={X+Y \over 2}, \quad \chi^{(2)}={X-Y \over 2},
   \end{equation}
    we conclude the proof of the first part of this Proposition. In order to prove~\eqref{eq:prop_dist_Lomega_W_exp}, we assume $\sigma \ge 2$ and we observe that, if $\chi \in \mathcal{W}^{\mathrm{hyp}, T}_{\sigma, \omega, \Omega, \lambda}$ then there exists $g \in \mathcal{S}^{\mathrm{exp},T}_{(\sigma, 0), \lambda}$ such that $\mathcal{L}^{\mathrm{hyp}}_{\omega, \Omega} \chi = g$. Then, we can write $\chi$ as in~\eqref{eq:inv_exp_W} where $X_j$ and $Y_j$ are the unique solutions of~\eqref{systm:inv_XY_exp} for $j=1,\dots,m$. Hence, by\eqref{eq:prop_dist_Lomega_z_W_exp} of Lemma \ref{lemma:HE_interm_LoO_exp}, we have that 
    \begin{equation*}
        \partial_q \left(\Lo X_j\right) = \Lo \left(\partial_q X_j\right),\quad \partial_q \left(\Lo Y_j\right) = \Lo \left(\partial_q Y_j\right)
    \end{equation*}
    for $j=1,\dots,m$. Now, using~\eqref{eq:inv_exp_W} and the latter, we conclude the proof of~\eqref{eq:prop_dist_Lomega_W_exp}.
    
    The proof in the real-analytic setting is identical, up to the obvious modifications, and is therefore omitted.
\end{proof}

In the following proposition, we study the invertibility of the linear operator $\mathfrak{L}^{\mathrm{hyp}}_{\omega, \Omega}$. For our purposes, it is sufficient to consider only the Hölder setting. We use the Banach spaces defined in~\eqref{def:M_Sym_exp_Holder} and~\eqref{def:Sp_exp_Holder}.
\begin{proposition}\label{prop:Csigma_HEomegaOmegaOmega_hyp}
    Given $\sigma \ge 0$, $T\ge 0$, and $\lambda > 2 \max_{1 \le j \le m} \Omega_j$, the following operator
    \begin{equation*}
        \mathfrak{L}^{\mathrm{hyp}}_{\omega, \Omega} : \mathcal{S}p^{\mathrm{hyp}, T}_{\sigma, \omega, \Omega, \lambda} \longrightarrow \mathcal{S}ym^{\mathrm{exp}, T}_{\sigma, \lambda}
    \end{equation*}
    is well-defined and invertible. Moreover, 
    \[
\| (\mathfrak{L}^{\mathrm{hyp}}_{\omega, \Omega})^{-1} \| \leq C(\lambda, \Omega),
    \]
    for some $C > 0$ depending only on $\lambda$ and $\Omega$. 
\end{proposition}
\begin{proof}
We recall that, for any $G \in \mathcal{S}p^{\mathrm{hyp}, T}_{\sigma, \omega, \Omega, \lambda} $, we have that $\mathfrak{L}^{\mathrm{hyp}}_{\omega, \Omega} G = G^\top \Mhyp + \Mhyp G + J_m \Lo G$. Hence, for every $g \in \mathcal{S}ym^{\mathrm{exp}, T}_{\sigma, \lambda}$, we can rewrite this dynamical problem in terms of finding the unique solution $G \in \mathcal{S}p^{\mathrm{hyp}, T}_{\sigma, \omega, \Omega, \lambda} $ of the following equation
\begin{equation}\label{eq:HE_hyp_trasv}
    G^\top \Mhyp + \Mhyp G + J_m \Lo G = g.
\end{equation}
For this purpose, we observe that, for any $g \in \mathcal{S}ym^{\mathrm{exp}, T}_{\sigma, \lambda}$ and $ G \in \mathcal{S}p^{\mathrm{hyp}, T}_{\sigma, \omega, \Omega, \lambda}$, we can write
    \[ g = \begin{pmatrix}
X & Y \\ Y^{\top} & -Z
\end{pmatrix},  \qquad  \begin{array}{ll}  Y : \T^n \times I_T \to \mathcal{M}_m(\R), \quad X, Z : \T^n \times I_T \to  \textup{Sym}(m, \R), \\
X_{ij}, Y_{ij}, Z_{ij} \in \mathcal{S}^{\mathrm{exp}, T}_{(\sigma, 0), \lambda}, \qquad \text{ for any } 1 \leq i, j \leq m; \end{array}\]
\[ \hat G = \begin{pmatrix}
P & Q \\ R & -P^{\top}
\end{pmatrix}, \qquad   \begin{array}{ll}  P : \T^n \times I_T \to \mathcal{M}_m(\R), \quad Q, R : \T^n \times I_T \to  \textup{Sym}(m, \R), \\ P_{ij}, Q_{ij}, R_{ij} \in \mathcal{S}^{\mathrm{exp},T}_{(\sigma, 0), \lambda} \qquad \text{ for any } 1 \leq i, j \leq m.  \end{array} \]

    Using the above notation, we can rewrite the equation~\eqref{eq:HE_hyp_trasv} as
    \begin{equation} \label{proof:lemma_Df_inv_hyp_2}
    \left\{ \begin{array}{llll}
P^{\top}\Omega + \Omega P - \Lo Q = X, \\
-R \Omega + \Omega Q + \Lo P = Y, \\
Q \Omega - \Omega R + \Lo P^\top = Y^{\top}, \\
P \Omega + \Omega P^{\top} + \Lo R = -Z.
 \end{array} \right.  \end{equation}
We denote
\begin{align*}
    2\mathfrak{A} &= P - P^\top + (Q + R), \qquad 2 \mathfrak{B} = P+P^\top - (Q-R), \\
    2 \mathfrak{C} &= P+P^\top + (Q-R), \qquad 2\mathfrak{D} = P-P^\top - (Q+R)
\end{align*}
and we observe that $\mathfrak{A}^\top = -\mathfrak{D}$, $\mathfrak{B}^\top = \mathfrak{B}$ and $\mathfrak{C}^\top = \mathfrak{C}$.

We can rewrite~\eqref{proof:lemma_Df_inv_hyp_2} in terms of the following system of $2m$ decoupled equations
\begin{equation*} 
\begin{cases}
&-\Omega\mathfrak{A} + \mathfrak{A}\Omega - \Lo \mathfrak{A} = -{1 \over 2} \left[(Y-Y^\top) + (X-Z)\right], \\ 
&-\Omega\mathfrak{B} - \mathfrak{B}\Omega - \Lo \mathfrak{B} = {1 \over 2} \left[(Y+Y^\top) - (X-Z)\right], \\ 
&\Omega \mathfrak{C} + \mathfrak{C}\Omega - \Lo \mathfrak{C} = -{1 \over 2} \left[(Y+Y^\top) + (X+Z)\right], \\ 
&\Omega \mathfrak{D} - \mathfrak{D}\Omega - \Lo \mathfrak{D} = -{1 \over 2} \left[(Y-Y^\top) - (X-Z)\right].
\end{cases}
\end{equation*}
or equivalently,
\begin{equation*} 
\begin{cases}
&-\left(\Omega_j - \Omega_k\right)\mathfrak{A}_{jk} - \Lo \mathfrak{A}_{jk} = -{1 \over 2} \left[Y_{jk}-Y_{kj} + X_{jk}-Z_{jk}\right], \\ 
&-\left(\Omega_j + \Omega_k\right)\mathfrak{B}_{jk} - \Lo \mathfrak{B}_{jk} = {1 \over 2} \left[Y_{jk}+Y_{kj} - X_{jk}-Z_{jk}\right], \\ 
&\left(\Omega_j + \Omega_k\right)\mathfrak{C}_{jk} - \Lo \mathfrak{C}_{jk} = {1 \over 2} \left[Y_{jk}+Y_{kj} + X_{jk}+Z_{jk}\right], \\ 
&\left(\Omega_j - \Omega_k\right)\mathfrak{D}_{jk} - \Lo \mathfrak{D}_{jk} = -{1 \over 2} \left[Y_{jk}-Y_{kj} - X_{jk}+Z_{jk}\right], 
\end{cases}
\end{equation*}
for $1 \le j, k \le m$. Rememebering that $\lambda > 2 \max_{1 \le j \le m} \Omega_j$, using Lemma \ref{lemma:HE_interm_LoO_exp}, we prove the existence of a unique solution, 
\begin{equation*}
    \mathfrak{A},\ \mathfrak{B},\ \mathfrak{C},\ \mathfrak{D} \in \mathcal{S}^{\mathrm{exp}, T}_{(\sigma -1, \lambda)},\ \mbox{such that} \hspace{2mm} |\mathfrak{A}|^T_{\sigma-1, \lambda} ,\ |\mathfrak{B}|^T_{\sigma, \lambda} ,\ |\mathfrak{C}|^T_{\sigma, \lambda} ,\ |\mathfrak{D}|^T_{\sigma, \lambda} \le C(\lambda, \Omega) |g|^T_{\sigma-1, \lambda},
\end{equation*}
where $C(\lambda, \Omega)$ stands for a positive constant depending on $\lambda$ and $\Omega_j$, for $j=1,\dots, m$. Finally, noting that 
\begin{equation*}
    \hat G = \begin{pmatrix}
P & Q \\ R & -P^{\top}
\end{pmatrix} = {1 \over 2}\begin{pmatrix}
\mathfrak{A} + \mathfrak{B} + \mathfrak{C} + \mathfrak{D} & \mathfrak{A} - \mathfrak{B} + \mathfrak{C} - \mathfrak{D}\\
\mathfrak{A} + \mathfrak{B} - \mathfrak{C} - \mathfrak{D} & \mathfrak{A} - \mathfrak{B} - \mathfrak{C} + \mathfrak{D}
\end{pmatrix},
\end{equation*}
we conclude the proof of this lemma. 
\end{proof}

\subsection{Norm properties}
\label{sc:norm_properties_hyp} 
In this section, we prove several properties for the norms of the Banach spaces introduced in Sections \ref{sc:functional_setting_hyp} and \ref{sc:proof_analytic_hyp}. Throughout this section, we denote by $C(\cdot)$ a generic positive constant depending on the dimensions $n$ and $m$, and the parameter in brackets.

\begin{proposition}\label{prop:Csigma_prop_norms_exp}
    Given $\sigma, \lambda, \varrho \ge 0$, $d \ge 1$, $T\ge 0$ and a non-negative integer $k$, for all $f \in \mathcal{S}^{\mathrm{exp},T}_{(\sigma, k), \lambda}$ and $g \in \mathcal{S}^{\mathrm{exp},T}_{(\sigma, k), \varrho}$ we have the following properties
\begin{enumerate}
\item For all $s\ge 0$ and  $\beta \in \N^{2(n + m)}$, if $|\beta| + s \le \sigma + k$, then  
\begin{equation*}
\left|\partial^{\beta}_{(q,p,z)} f \right|^T_{s, \lambda} \le C |f|^T_{\sigma +k, \lambda}
\end{equation*}
\item If $f \in \mathcal{S}^{\mathrm{exp},T}_{(\sigma, k), d\lambda}$ then $|f|^T_{\sigma+k, \lambda}  \le e^{-(d-1)\lambda T}|f|^T_{\sigma+k, d\lambda}$, 
\item $|fg|^T_{\sigma+k, \lambda + \varrho} \le C(\sigma, k)\left(|f|^T_{0,\lambda}|g|^T_{\sigma+k,\varrho} + |f|^T_{\sigma+k,\lambda}|g|^T_{0,\varrho}\right)$. %
\item We consider $\sigma \ge 1$,  and we assume that $g:\T^n \times B \times I_T \to \T^n \times B$.  Letting $\tilde g: \T^n \times B \times I_T \to \T^n \times B \times I_T$ such that $\tilde g(q,p,z,t) = (g(q,p,z,t), t)$. Moreover, we assume that $f \in \mathcal{S}^{\mathrm{exp},T}_{(\sigma, k), \lambda + \varrho}$. Then $f \circ \tilde g \in \mathcal{S}^{\mathrm{exp}, T}_{\sigma+k, \lambda+\varrho}$ and 
\begin{equation*}
|f \circ \tilde g|^T_{\sigma+k, \lambda+\varrho} \le C(\sigma, k) \left(|f|^T_{\sigma+k,\lambda}\left(|\partial_{(q,p,z)} g|^T_{0,\varrho}\right)^\sigma + |f|^T_{1,\lambda}|\partial_{(q,p,z)}  g|^T_{\sigma+k-1,\varrho} +  |f|^T_{0, \lambda+\varrho}  \right).
 \end{equation*}
\end{enumerate}
\end{proposition}
\begin{proof}
    We refer to~\cite{Sca22c} for the proof. 
\end{proof}

 As a direct consequence of Proposition \ref{prop:Csigma_prop_norms_exp}, we have the following.

\begin{proposition}\label{prop:Csigma_prop_norms_matrix_exp}
    Let $\sigma \ge 0$, $\lambda, \, \varrho >0$, $d \ge 1$, and $T\ge 0$, for all $F \in  \mathcal{M}^{\mathrm{exp}, T}_{\sigma, \lambda}$ and $G \in  \mathcal{M}^{\mathrm{exp}, T}_{\sigma, \varrho}$ we have the following properties.
\begin{enumerate}
\item For all $s\ge 0$ and  $\beta \in \N^{2(n + m)}$, if $|\beta| + s \le \sigma $, then  
\begin{equation*}
\left|\partial^{\beta}_{(q,p,z)} F \right|^T_{s, \lambda} \le C(n,m) |F|^T_{\sigma , \lambda}
\end{equation*}
\item \label{prop:crescita_indici_norm_hyp} If $F \in \mathcal{M}^{\mathrm{exp}, T}_{\sigma, d\lambda}$ then $|F|^T_{\sigma, \lambda}  \le e^{-(d-1)\lambda T}|F|^T_{\sigma, d\lambda}$, 
\item \label{prop:bound_norm_product_matrices_hyp} $|FG|^T_{\sigma, \lambda+\varrho} \le C(\sigma,  m)\left(|F|^T_{0,\lambda}|G|^T_{\sigma,\varrho} + |F|^T_{\sigma,\lambda}|G|^T_{0,\varrho}\right)$. %
\end{enumerate}
\end{proposition}

The following properties follow easily from the definitions together with Proposition \ref{prop:Csigma_prop_norms_matrix}.

\begin{lemma}
\label{lem:exp_exp_bounds}
      Let $\sigma \ge 0$, $\lambda, \, \varrho >0$, and $T\ge 0$. For any $B \in \mathcal{M}^{\mathrm{exp}, T}_{\sigma, 0},$ $F \in  \mathcal{M}^{\mathrm{exp}, T}_{\sigma, \lambda},$ and $G \in  \mathcal{M}^{\mathrm{exp}, T}_{\sigma, \varrho},$ the following statements hold for some constants $C = C(\sigma, m)$, depending only on $\sigma$ and $m$.

      \begin{enumerate}
          \item  $|e^B|_{\sigma, 0}^T \leq e^{C|B|_{\sigma, 0}^T}.$
          \item  $|e^F - \mathrm{Id}_{2m}|_{\sigma, \lambda}^T \leq  |F|_{\sigma, \lambda}^Te^{C|F|_{\sigma, 0}^T}$.          
          \item  $|e^F - \mathrm{Id}_{2m} - F|_{\sigma, 2\lambda}^T \leq (|F|_{\sigma, \lambda}^T)^2e^{C|F|_{\sigma, 0}^T}$.
          \item  $|\mathcal{A}(B, G)|_{\sigma, \varrho}^T \leq C|G|_{\sigma, \varrho}^Te^{C|B|_{\sigma, 0}^T}$.
          \item  $|\mathcal{A}(F, G) - G|_{\sigma, \lambda + \varrho}^T \leq C|F|_{\sigma, \lambda}^T |G|_{\sigma, \varrho}^T e^{C|F|_{\sigma, 0}^T}.$
          \item  If $\Lo B \in \mathcal{M}_{\sigma, \lambda}^T$ then $|\Lo (e^B)|_{\sigma, \lambda}^T \leq C e^{C|B|^T_{\sigma, 0}}|\Lo B|_{\sigma, \lambda}^T$.
          \item  If $\Lo B \in \mathcal{M}_{\sigma, 1}^T$ and $\Lo G \in \mathcal{M}_{\sigma, \varrho + 1}^T$ then $$|\Lo (\mathcal{A}(B, G))|_{\sigma, \varrho + 1}^T \leq C e^{C|B|^T_{\sigma, 0}}(|\Lo B|_{\sigma, 1}^T|G|_{\sigma, \varrho}^T + |\Lo G|_{\sigma, \varrho + 1}^T).$$
          \item If $\Lo F \in \mathcal{M}_{\sigma, \lambda + 1}^T$ and $\Lo G \in \mathcal{M}_{\sigma, \varrho + 1}^T$ then $$|\Lo (\mathcal{A}(F, G) - G)|_{\sigma, \lambda + \varrho + 1}^T \leq C e^{C|Fe|^T_{\sigma, 0}}(|\Lo F|_{\sigma, \lambda + 1}^T|G|_{\sigma, \varrho}^T + |F|_{\sigma, \lambda}^T |\Lo G|_{\sigma, \varrho + 1}^T).$$
      \end{enumerate}
\end{lemma}

\begin{proof}
    The proof is identical to that of Lemma \ref{lem:exp_expansion_bounds} but using Proposition \ref{prop:Csigma_prop_norms_matrix_exp} instead of Proposition \ref{prop:Csigma_prop_norms_matrix} and using the norm in the spaces of functions with exponential decay rather than the norm for functions with polynomial decay.
\end{proof}

\begin{proposition}
\label{prop:properties_analy_hyp}
Let $\sigma>0$, $\lambda, \, \varrho \ge 0$, $d \ge 1$ and $T\ge 0$. For any $f \in \mathscr{S}^{\mathrm{exp}, T}_{\sigma, \lambda}$ and $g \in \mathscr{S}^{\mathrm{exp}, T}_{\sigma, \varrho}$, we have the following properties    
\begin{enumerate}
    \item For all $0<\sigma'<\sigma$ and $\beta \in \N^{2(n+m)}$, then
    \begin{equation*}
        \left|\partial^{\beta}_{(q,p,z)} f \right|^T_{\sigma', \lambda} \le {C \over (\sigma-\sigma')^{|\beta|}} |f|^T_{\sigma, \lambda}.
    \end{equation*}
    \item If $f \in \mathscr{S}^{\mathrm{exp}, T}_{\sigma, d\lambda}$ then $|f|^T_{\sigma, \lambda} \le e^{-(d-1)\lambda T}|f|^T_{\sigma, d\lambda}$.
    \item $|fg|^T_{\sigma, \lambda +\varrho} \le |f|^T_{\sigma, \lambda}|g|^T_{\sigma, \varrho}$. In particular, if $\varrho >0$ then $|fg|_{\sigma, \lambda}^T \to 0$ as $T\to +\infty$.
    \item For all $0<\sigma'<\sigma$, we consider $g:\T^n_{\sigma'} \times B_{\sigma'} \times I_T \to \T^n_\sigma \times B_\sigma$ such that $g \in \mathscr{S}^{\mathrm{exp}, T}_{\sigma',0}$. Letting $\tilde g: \T^n_{\sigma'} \times B_{\sigma'} \times I_T \to \T^n_\sigma \times B_\sigma \times I_T$ such that $\tilde g(q,p,z,t) = (g(q,p,z,t), t)$, 
then $f \circ  \tilde g \in \mathscr{S}^{\mathrm{exp}, T}_{\sigma', \lambda}$ and 
\begin{equation*}
    |f\circ \tilde g|^T_{\sigma', \lambda} \le |f|^T_{\sigma, \lambda}.
\end{equation*}
\end{enumerate}
\end{proposition} 
\begin{proof}
    The proof is similar to that of Proposition \ref{prop:properties_analy_pol}.
\end{proof}

\appendix
\section{Hölder class of functions}\label{app:Holder}
In this section, we recall the definition of  Hölder classes of functions and some of their well-known properties. For this purpose, given an open subset $E$ of $\R^n$ and a positive integer $k \ge 0$,  we denote by $C^k(E)$ the spaces of functions $f: E \to \R$ with continuous partial derivatives $\partial^\alpha f \in C^0(E)$ for all $\alpha \in \N^n$ with $|\alpha|= |\alpha|_1 = \alpha_1+...+\alpha_n \le k$. Furthermore, for all $f \in C^k(E)$, we define the following norm
\begin{equation*}
|f|_{C^k} = \sup_{|\alpha|\le k}|\partial^\alpha f|_{C^0},
\end{equation*}
where $|\partial^\alpha f|_{C^0} = \sup_{x \in E}|\partial^\alpha f(x)|$. 

Let $\sigma = k + \mu$, where $k \in \Z$, $k \ge 0$ and $0 < \mu <1$.  We define by $C^\sigma(E)$ the space of Hölder functions $f\in C^k(E)$ verifying 
\begin{equation}
\label{def:Holder_norm}
|f|_{C^\sigma} = \sup_{|\alpha|\le k}|\partial^\alpha f|_{C^0} + \sup_{x,y\in E\\, x\ne y, |\alpha| = k}{|\partial^\alpha f(x) - \partial^\alpha f(y)| \over |x-y|^\mu}<\infty.
\end{equation}
It is well-known that $C^\sigma(E)$ endowed with the norm~\eqref{def:Holder_norm} is a Banach space. We adopt the same notation for Real-valued functions and matrices. We say that a Real-valued function or a matrix belongs to $C^\sigma(E)$ if each of its components does. In this case, the corresponding norm is defined as the maximum of the norms of its components.

We recall that $C(\cdot)$ stands for a constant depending on the parameters in brackets. The following proposition summarizes some properties of the norm~\eqref{def:Holder_norm}. 
\begin{proposition}
\label{prop: prop_Holder_norms}
We consider $f$, $g \in C^\sigma(E)$ and $\sigma \ge 0$.
\begin{enumerate}
\item For all $\beta \in \N^{n}$ and $s\ge 0$, if $|\beta| + s \le \sigma$ then  $\left|{\partial^{|\beta|} \over \partial{x_1}^{\beta_1}... \partial{x_n}^{\beta_n}} f \right|_{C^s} \le C|f|_{C^\sigma}$.\\
\item  $|fg|_{C^\sigma} \le C(\sigma)\left(|f|_{C^0}|g|_{C^\sigma} + |f|_{C^\sigma}|g|_{C^0}\right)$. 
\end{enumerate}
 Let $E_1$ be an open subset of $\R^n$ and $z:E_1\to E$ a function taking values in the domain of $f$.  In what follows $\partial z$ stands for the partial derivatives of $z$.
\begin{enumerate}
\item[(3)] If $\sigma < 1$, $f \in C^1(E)$, $z \in C^\sigma (E_1)$ then $f\circ z \in C^\sigma(E_1)$ and
$$
|f \circ z|_{C^\sigma} \le C(|f|_{C^1}|z|_{C^\sigma}+ |f|_{C^0}).
$$
\end{enumerate} 
\begin{enumerate}
\item[(4)] If $\sigma < 1$, $f \in C^\sigma(E)$, $z \in C^1 (E_1)$ then $f\circ z \in C^\sigma(E_1)$ 
$$
|f \circ z|_{C^\sigma} \le C(|f|_{C^\sigma}|\partial z|^\sigma_{C^0}+ |f|_{C^0}).
$$
\end{enumerate}
\begin{enumerate}
\item[(5)] If $\sigma \ge 1$ and $f \in C^\sigma (E)$, $z \in C^\sigma (E_1)$ then $f\circ z \in C^\sigma(E_1)$ 
$$
|f \circ z|_{C^\sigma} \le C(\sigma) \left(|f|_{C^\sigma}|\partial z|^\sigma_{C^0} + |f|_{C^1}|\partial z|_{C^{\sigma-1}}+ |f|_{C^0}\right).
$$
\end{enumerate}
\end{proposition}
\begin{proof}
    We refer to~\cite{Hor76} and~\cite{Sca22} for the proof. 
\end{proof}

The following is a quantitative version of the classical Implicit Function Theorem (see, e.g., \cite{ChiAM2}). 

\begin{theorem}
\label{thm:QIFT}
    Let $(X, |\cdot |_X), (Y, |\cdot|_Y)$, $(Z, |\cdot|_Z)$ be Banach spaces and $U \subseteq X,$ $V \subseteq Y$ be open subsets containing $0$. Let $F:U \times V \subseteq X \times Y \to Z$ be a $C^1$ map such that $F(0, 0) = 0$ and $D_2 F(0, 0): Y 
    \to Z$ is invertible. 
    
    Suppose there exist $M, r, \rho > 0$ satisfying the following.
    \begin{enumerate}
        \item   $B_X(r) \times B_Y(\rho) \subseteq U \times V$.
        \item $\| D_2 F(0, 0)^{-1}\|_{L(Z, Y)} \leq M$.
        \item $|F(x, 0)|_Y \leq \tfrac{\rho}{2M}$, for all $x \in B_X(r) $.
        \item $\|D_2 F(x, y) - D_2F(0, 0)\|_{L(Y, Z)} \leq \frac{1}{2M}$, for all $x \in B_X(r)$ and all $y \in B_Y(\rho)$.
    \end{enumerate}
    Then, there exists a $C^1$ map $g: B_X(r) \subseteq X \to B_Y(\rho) \subseteq Y$ such that
    \[ F(x, g(x)) = 0, \qquad \text{ for all } x \in B_X(r).  \]
\end{theorem}
\section*{Acknowledgments} D.S. have been partially supported by the grant PID2024-158570NB-I00  funded by \\
MCIN/AEI/10.13039/501100011033 and “ERDF A way of making Europe”.  D.S. also acknowledges partial support from the ICREA Acadèmia 2023 grant awarded to Dr. Marcel Guardia Munàrriz.

\bibliographystyle{amsalpha}
\bibliography{ref}
\end{document}